\documentclass[10pt]{amsart}

\usepackage{amsmath,amsthm,amssymb,bm}
\usepackage{mathtools}
\usepackage{mathrsfs}
\usepackage{microtype}
\usepackage{lmodern}
\usepackage[T1]{fontenc}
\usepackage[english]{babel}
\usepackage{extarrows}
\usepackage{tikz,pgffor}
\usepackage{hyperref}
\hypersetup{colorlinks=true,citecolor=blue,linkcolor=blue,urlcolor=blue,
pdfstartview=FitH }

\newcommand{\SL}{\mathrm{SL}}

\newcommand{\Qd}{\mathbb{Q}}
\newcommand{\Qbar}{\overline{\Qd}}

\newcommand{\Nd}{\mathbb{N}}
\newcommand{\Vd}{\mathbb{V}}
\newcommand{\Zd}{\mathbb{Z}}
\newcommand{\Fd}{\mathbb{F}}
\newcommand{\Rd}{\mathbb{R}}

\newcommand{\Gd}{\mathbb{G}}
\newcommand{\Pd}{\mathbb{P}}

\newcommand{\Ad}{\mathbb{A}}
\newcommand{\xv}{\mathbf{x}}
\newcommand{\vv}{\mathbf{v}}
\newcommand{\hv}{\mathbf{h}}
\newcommand{\tv}{\mathbf{t}}
\newcommand{\uv}{\mathbf{u}}
\newcommand{\yv}{\mathbf{y}}
\newcommand{\zv}{\mathbf{z}}

\newcommand{\dd}{\,{\mathrm d}}

\newcommand{\Vm}{\mathscr{V}}
\newcommand{\Wm}{\mathscr{W}}
\newcommand{\Dm}{\mathscr{D}}
\newcommand{\Rm}{\mathscr{R}}
\newcommand{\Pm}{\mathscr{P}}
\newcommand{\Nm}{\mathscr{N}}
\newcommand{\Mm}{\mathscr{M}}
\newcommand{\Om}{\mathscr{O}}
\newcommand{\Tm}{\mathscr{T}}
\newcommand{\Am}{\mathscr{A}}

\newcommand{\rleft}{\mathopen{}\mathclose\bgroup\left}
\newcommand{\rright}{\aftergroup\egroup\right}

\newcommand{\Xf}{\mathfrak{X}}

\newcommand{\cfr}{\mathfrak{c}}

\newcommand{\tX}{\widetilde{X}}

\newcommand{\stX}{\smash{\tX}}

\newcommand{\tY}{\widetilde{Y}}

\newcommand{\tT}{\widetilde{T}}
\newcommand{\tZ}{\widetilde{Z}}
\newcommand{\tpi}{\widetilde{\pi}}
\newcommand{\ie}{i.\,e.,~}
\newcommand{\eg}{e.\,g.,~}
\newcommand{\Sigmamax}{\Sigma_\mathrm{max}}
\newcommand{\Z}{\Zd}

\newcommand{\Q}{\Qd}

\DeclareMathOperator{\cone}{cone}
\DeclareMathOperator{\CF}{CF}
\DeclareMathOperator{\Gal}{Gal}
\DeclareMathOperator{\Pic}{Pic}
\DeclareMathOperator{\Eff}{Eff}
\DeclareMathOperator{\vol}{vol}
\DeclareMathOperator{\Div}{div}
\DeclareMathOperator*{\Hom}{Hom}
\DeclareMathOperator*{\Supp}{supp}
\DeclareMathOperator*{\Spec}{Spec}
\DeclareMathOperator*{\Cl}{Cl}
\DeclareMathOperator{\Res}{Res}
\DeclareMathOperator{\rank}{rk}
\DeclareMathOperator{\Fr}{Fr}

\theoremstyle{plain}
\newtheorem{theorem}{Theorem}
\newtheorem{lemma}[theorem]{Lemma}
\newtheorem{cor}[theorem]{Corollary}
\newtheorem{prop}[theorem]{Proposition}
\newtheorem{hyp}[theorem]{Hypothesis}
\theoremstyle{definition}
\newtheorem{remark}[theorem]{Remark}

\numberwithin{theorem}{section}
\numberwithin{equation}{section}
\numberwithin{table}{section}
 
\makeatletter
\@namedef{subjclassname@2020}{%
  \textup{2020} Mathematics Subject Classification}
\makeatother

\begin{document}

\author[Blomer]{Valentin Blomer}

\address{Universit\"at Bonn, Mathematisches Institut, Endenicher Allee 60, 53115 Bonn, Germany}

\email{blomer@math.uni-bonn.de}

\author[Br\"udern]{J\"org Br\"udern}

\address{Universit\"at G\"ottingen, Mathematisches Institut, Bunsenstra{\ss}e 3--5, 37073 G\"ottingen, Germany}

\email{jbruede@gwdg.de}

\author[Derenthal]{Ulrich Derenthal}

\address{Leibniz Universit\"at Hannover, Institut f\"ur Algebra,
  Zahlentheorie und Diskrete Mathematik, Welfengarten 1, 30167
  Hannover, Germany}

\email{derenthal@math.uni-hannover.de}

\author[Gagliardi]{Giuliano Gagliardi}

\address{Leibniz Universit\"at Hannover, Institut f\"ur Algebra,
  Zahlentheorie und Diskrete Mathematik, Welfengarten 1, 30167
  Hannover, Germany}
\email{gagliardi@math.uni-hannover.de}

\title[Manin--Peyre conjecture]
{The Manin--Peyre conjecture \\ for smooth spherical Fano varieties of semisimple rank one}

\thanks{First three authors partially supported by the DFG-SNF lead
  agency program (BL 915/2-2, BR 3048/2-2, DE 1646/4-2). Fourth author
  partially supported by the Israel Science Foundation (grant No.
  870/16) and the Max Planck Institute for Mathematics in Bonn.}
 
\keywords{rational points, spherical varieties, Fano threefolds, Manin--Peyre
  conjecture, Cox rings, harmonic analysis}

\begin{abstract}
  The Manin--Peyre conjecture is established for a class of smooth spherical
  Fano varieties of semisimple rank one.  This includes all smooth spherical
  Fano threefolds of type $T$ as well as some higher-dimensional smooth
  spherical Fano varieties.
\end{abstract}

\subjclass[2020]{Primary 14G05; Secondary 11D45, 14M27, 11G35}

\setcounter{tocdepth}{1}

\maketitle

\tableofcontents 

\section{Introduction}

\subsection{Manin's conjecture}

Manin's conjecture \cite{MR89m:11060} predicts an asymptotic formula for the
number of rational points of bounded height on Fano varieties. Its most
classical version is the following: let $X$ be a smooth Fano variety over
$\Qd$ whose set of rational points is Zariski dense. Let
$H \colon X(\Qd) \to \Rd$ be an anticanonical height function. For an open
  subset $U$ of $X$, let $ N_{X,U,H}(B)$ denote the number of $x \in U(\Qd)$
  with $ H(x) \le B$. Then one expects that there is a dense open subset
  $U \subseteq X$ and a positive number $c$ such that
\begin{equation}\label{MC}
  N_{X, U,H}(B) = (1+o(1)) c B(\log B)^{\rank \Pic X-1}.
\end{equation}
Peyre \cite{MR1340296} proposed a product formula for $c$, and in the sequel
we refer to this predicted value of $c$ as Peyre's constant.
It turned out that in its original form Manin's conjecture is not always
correct (see \cite{MR1401626}). The more recent thin set version (see \cite{MR2019019},
\cite[Conjectures~1.2, 5.2]{LST}) is in line with all known results hitherto. 

When the dimension is large compared to the degree of the variety, one may
apply the circle method to estimate $ N_{X,U,H}(B)$. In this way, Browning and
Heath-Brown \cite{MR3605019} confirmed Manin's conjecture whenever $X$ is
geometrically integral and the inequality
$\dim X \ge ((\deg X)-1)2^{\deg X}-1$ holds. The asymptotic formula \eqref{MC}
is also known for several classes of equivariant compactifications of
algebraic groups or homogeneous spaces: for certain horospherical varieties
(flag varieties \cite{MR89m:11060}, toric varieties \cite{MR1620682}, and
toric bundles over flag varieties \cite{MR1723811}), for wonderful
compactifications of semisimple groups of adjoint type
\cite{MR2328719,MR2482443}, for certain other wonderful varieties
\cite{MR2795511}, and for biequivariant compactifications of unipotent groups
\cite{MR3465086} (including equivariant $\Gd_\mathrm{a}^n$-compactifications
\cite{MR1906155}).  Here the proofs use harmonic analysis on adelic points.

In absence of additional structure, we only know four more low-dimensional
cases: Manin's conjecture was verified for two smooth quintic del Pezzo
surfaces \cite{MR1909606,MR2099200}, for one smooth quartic del Pezzo surface
\cite{MR2838351}, and (in the thin set version \cite{LST}) for a quadric bundle in
$\Pd^3 \times \Pd^3$ \cite{BHB}.  Not surprisingly, there are many more
results on versions of Manin's conjecture for \emph{singular} varieties,
because usually analytic techniques are easier to implement in the presence of
singularities.
  
In this paper, we take a different methodological approach and initiate a
systematic study of Manin's conjecture for varieties for which we have access
to the Cox ring, and where a universal torsor is given by a polynomial of the
shape
\begin{equation}\label{torsor}
  \sum_{i=1}^k b_i \prod_{j=1}^{J_i}   x_{ij}^{h_{ij}} = 0
\end{equation}
with integral coefficients $b_i$ and certain exponents $h_{ij} \in
\Bbb{N}$. This includes a fairly large class of interesting cases,  in particular 
numerous varieties with a torus action of complexity one or higher (see
\cite{HS,fahrner,HHW} and the references therein, for example),   most weak del Pezzo
surfaces whose universal torsor is given by one equation \cite{MR3180592},  (nontoric) spherical varieties of
semisimple rank one, as well as  several
nonspherical smooth Fano threefolds \cite{MR3348473}, and  many
other varieties. 

Our analytic approach towards Manin's conjecture, to be described later in
more detail, is insensitive to the dimension of the variety (in contrast to
the circle method) and independent of an additional group structure (in
contrast to methods based on harmonic analysis on adelic points). A 
showcase for our approach is the proof the Manin-Peyre conjecture for all smooth spherical Fano
threefolds of semisimple rank one and type $T$ in Theorem \ref{dim3}. We will
give several more examples in Theorems \ref{dim4} and \ref{thm2} to shed light
on the scope of the underlying method.

\subsection{Spherical varieties}

Let $G$ be a connected reductive group. A normal $G$-variety $X$ is called
spherical if a Borel subgroup of $G$ has a dense orbit in $X$.  Spherical
varieties have a rich theory. They include symmetric varieties, and the
corresponding space $L^2(X)$ has been the subject of intense investigation
from the point of view of (local) harmonic analysis and the (relative)
Langlands program (\eg \cite{Sa, SV}). Spherical varieties also admit a
combinatorial description. This is achieved by the recently completed Luna
program \cite{lun01,bp15b,cf2,los09a} and the Luna--Vust theory of spherical
embeddings \cite{lv83,kno91}. We recall the relevant theory in
Section~\ref{sec2} and refer to \cite{bl11,per14,tim11} as general
references. In this paper, we are interested in the size of smooth spherical
varieties in the context of Manin's conjecture.

If the acting group $G$ has semisimple rank zero, then $G$ is a torus and
Manin's conjecture is known (\cite{MR1620682}, see also \cite{MR1679841}).
The next interesting case is $G$ of semisimple rank one.  Here we may assume
$G = \SL_2\times \Gd_\mathrm{m}^r$ by passing to a finite cover (see
Section~\ref{sec:ssr1} for more details).  Let
$G/H = (\SL_2 \times \Gd_\mathrm{m}^r)/H$ be the open orbit in $X$. Let
$H'\times \Gd_\mathrm{m}^r = H\cdot \Gd_\mathrm{m}^r \subseteq \SL_2 \times
\Gd_\mathrm{m}^r$. Then the homogeneous space $\SL_2/H'$ is spherical, and
hence either $H'$ is a maximal torus (\emph{the case $T$}) or $H'$ is 
the normalizer of a maximal torus in $\SL_2$ (\emph{the case $N$}) or the
homogeneous space $\SL_2/H'$ is horospherical, in which case $X$ is isomorphic
(as an abstract variety, possibly with a different group action) to a toric
variety, so we may exclude this case from our discussion.

\subsection{Spherical Fano threefolds}

We start our discussion with dimension 3, the smallest dimension where
nonhorospherical spherical varieties of semisimple rank one exist.  A complete
classification of nontoric smooth spherical Fano threefolds over
$\overline{\Qd}$ was established by Hofscheier \cite{hofscheier}, cf.\
Table~\ref{tab:classification_spherical}. In this situation, the acting group
always has semisimple rank one, so our present setup is in fact already the
general picture, and the following discussion applies to all nontoric smooth
spherical Fano threefolds.

There are precisely four nonhorospherical examples of type $T$ that are
not equivariant $\Bbb{G}_{\rm a}^3$-compactifications. They have natural split
forms $X_1,\dots,X_4$ over $\Qd$, which we describe in Section~\ref{sec:fano3folds} in
detail; see Table~\ref{tab:v18} for an overview. In the classification of smooth
Fano threefolds by Iskovskikh \cite{I1,I2} and Mori--Mukai \cite{MR641971},
they have types III.24, III.20 (of Picard number $3$), IV.8, IV.7 (of Picard
number $4$), respectively.

In Section~\ref{sec:height}, we will define natural anticanonical height
functions $H_j \colon X_j(\Qd) \to \Rd$ using the anticanonical monomials in
their Cox rings.  We establish the Manin--Peyre conjecture in all these
cases. We write $N_j(B)$ for $N_{X_j, U_j, H_j}(B)$, where here and in all
subsequent cases, the open subset $U_j$ will be the set of all points with
nonvanishing Cox coordinates.
 
\begin{theorem}\label{dim3}
  The Manin--Peyre conjecture holds for the smooth spherical Fano
  threefolds $X_1,\dots,X_4$ of semisimple rank one and type $T$.  More
  precisely, there exist explicit constants $C_1,\ldots, C_4$ such that
  $$N_j(B) = (1 + o(1))C_j B (\log B)^{\rank \Pic X_j -1} $$
  for $1 \leq j \leq 4$. 
  The values of $C_j$ are the ones predicted by Peyre. 
\end{theorem}
 
It is a fun exercise to compute $C_j$ explicitly (cf.\ Appendix~\ref{A}), for
which the interesting and apparently previously unknown integral identities
involving sin-integrals and Fresnel integrals in Lemma~\ref{sin} play an
important role.  One obtains
\begin{displaymath}
  \begin{split}
    & C_1 = \frac{40 -\pi^2}{12} \prod_{p} (1 - p^{-2})^3,
    \quad C_3 =  \frac{5(258 - 4\pi^2)}{1296}  \prod_p\left(1-\frac 1
      p\right)^4\left(1+\frac 4 p+\frac 4{p^2}+\frac 1{p^3}\right),\\
    & C_2   = \frac{170 - \pi^2 - 96\log 2}{36}   \prod_{p} (1 - p^{-2})^3,
    \quad C_4 =  \frac{94-2\pi^2 }{72}\prod_p\left(1-\frac 1
      p\right)^4\left(1+\frac 4 p+\frac 4{p^2}+\frac 1{p^3}\right). 
\end{split}
\end{displaymath}

Theorem~\ref{dim3} is an easy consequence of Theorem~\ref{manin-cor} that
proves the Manin--Peyre conjecture for smooth split spherical Fano varieties
of arbitrary dimension with semisimple rank one and type $T$, subject to a
number of technical conditions that are straightforward to check in every
given instance.  Similar methods apply also to smooth spherical Fano varieties
of type $N$, but these have some additional features to which we return in a
subsequent paper.

Theorem~\ref{dim3} contains the first examples where Manin's conjecture is
established for smooth Fano threefolds that do not follow from general results
concerning equivariant compactifications of algebraic groups or homogeneous
spaces. Theorem~\ref{dim3} in fact confirms the Manin--Peyre conjecture for
\emph{all} classes of \emph{smooth spherical Fano threefolds of semisimple rank one
  and type $T$}. Previously the knowledge of the number of rational points on
these varieties has been much less precise. Manin \cite{MR1199203} shows that
smooth Fano threefolds have at least linear growth for rational points in
Zariski dense open subsets of bounded anticanonical height over sufficiently
large ground fields. A closer inspection of his arguments reveals in fact
lower bounds of the correct order of magnitude:
$N_j \gg B(\log B)^{\text{rk}(\text{Pic } X_j) - 1}$ in the situation of
Theorem~\ref{dim3} (cf.\ the proof of \cite[Proposition~1.4]{MR1199203} as the
$X_j$ in Theorem~\ref{dim3} are blow-ups of toric varieties).  Tanimoto
\cite[\S 7]{arXiv:1812.03423} proves the upper bounds
$N_j \ll B^{5/2 + \varepsilon}$ for $j = 1, 2, 4$ and
$N_3 \ll B^{2+ \varepsilon}$.
 
\subsection{Higher-dimensional cases}

A classification of higher-dimensional spherical varieties is currently not
available, but our methods work equally well in dimension exceeding three. For
a given dimension, there are still only finitely many cases of smooth
spherical Fano varieties of semisimple rank one, and we include some
representative examples with interesting torsor equations and high Picard
number. Many other examples are available by the same method.  The four
varieties $X_5, X_6, X_7, X_8$ that we investigate here are smooth spherical
Fano varieties of semisimple rank one and type $T$ of dimension $4, 5, 6, 7$,
respectively, with $\rank\Pic X_5=5$, $\rank\Pic X_6=3$, $\rank\Pic X_7=5$,
and $\rank\Pic X_8=6$. We refer to Section~\ref{sec:geometry_X5_X6} for their
combinatorial description and Table \ref{tab:v18} for a quick overview and
remark that for neither of these varieties, Manin's conjecture follows from
previous results (cf.\ Appendix~\ref{rem:not}).

\begin{theorem}\label{dim4}
  The Manin--Peyre conjecture holds for the smooth spherical Fano varieties
  $X_5, \ldots, X_8$ of semisimple rank one and type $T$. More
  precisely, there exist explicit constants $C_5, \ldots, C_8 > 0$ such that
 $$N_j(B) = (1 + o(1))C_j B (\log B)^{\rank\Pic X_j-1}  $$
 for $j = 5, \ldots, 8$.  The values of $C_j$ are the ones predicted by Peyre.
\end{theorem}

We remark that Theorems~\ref{dim3} and \ref{dim4} are compatible with the thin
set version of Manin's conjecture. Since our spherical varieties have a
connected stabilizer for the open orbit, their sets of rational points are not
thin \cite[Corollary~2.5]{Bor}. As in \cite[Examples~5.12, 5.13]{LST}, one can
show that our results are compatible with \cite[Conjecture~5.2]{LST}.

\subsection{The methods}

The starting point of the quantitative analysis of Fano varieties in this
paper is a good understanding of their Cox ring. We use it to pass to a
universal torsor and translate Manin's conjecture into an explicit counting
problem whose structure we describe in a moment and that is amenable to
analytic techniques. The descent to a universal torsor is a common technique
in analytic approaches to Manin's conjecture, but in many cases it proceeds by
ad hoc considerations. Here we take a more systematic approach and derive the
passage from the Cox ring to the explicit counting problem in considerable
generality. This is summarized in
Proposition~\ref{prop:countingproblem_abstract}.  Next we take the opportunity
to express Peyre's constant in terms of Cox coordinates in
Proposition~\ref{prop:peyre} as a product of a surface integral, the volume of
a polytope and an Euler product, so that a verification of the complete
Manin--Peyre conjecture is possible without additional ad hoc computations.

This first part of the paper is presented in greater generality than necessary
for the direct applications to spherical varieties, and should prove  
to be useful in other situations.

The second part of the paper is devoted to an explicit solution of counting
problems having the structure required in
Proposition~\ref{prop:countingproblem_abstract}.  In many important cases, a
universal torsor is given by a single equation of the shape \eqref{torsor}.
We may have additional variables $x_{01}, \ldots, x_{0J_{0}}$ that do not
appear in the torsor equation; for those we put formally $h_{0j} = 0$.
Equation \eqref{torsor} is then to be solved in nonzero integers $x_{ij}$.
This seemingly simple diophantine problem has to be analyzed with certain
coprimality constraints on the variables, and the variables are restricted to
a highly cuspidal region. As specified in
Proposition~\ref{prop:countingproblem_abstract}, the height condition
translates into inequalities
\begin{equation}\label{height}
  \prod_{i=0}^{k} \prod_{j=1}^{J_i} |x_{ij}|^{\alpha^\nu_{ij}} \leq B \quad ( 1\le \nu \le N)
\end{equation}
for certain nonnegative exponents\footnote{The superscript $\nu$ is not an
  exponent, but an index. This notation is chosen in accordance with the
  notation in Section~\ref{sec:charts_torsors}.} $\alpha^\nu_{ij}$.  In order
to describe the coprimality conditions on the variables $x_{ij}$ in
\eqref{torsor}, let
$S_{\rho} \subseteq \{(i,j) : i = 0, \ldots, k, j = 1, \ldots, J_i\}$
$(1\le \rho\le r)$ be a collection of sets that define $r$ conditions
\begin{equation}\label{gcd}
\gcd\{x_{ij}: (i,j)\in S_\rho\}  = 1 \quad (1\le \rho\le r).
\end{equation}
Now fix a set of coefficients $b_i$ in \eqref{torsor}, and let
$N_{\mathbf b}(B)=N(B)$ denote the number of
$x_{ij} \in \Bbb{Z} \setminus \{0\}$ ($0 \leq i \leq k$, $1 \leq i \leq J_i$)
satisfying \eqref{torsor}, \eqref{height} and \eqref{gcd}. We aim to establish
an asymptotic formula of the shape
\begin{equation}\label{manin}
N(B) = (1+o(1)) c_1 B (\log B)^{c_2}
\end{equation}
for some constants $c_1 > 0$, $c_2 \in \Bbb{N}_0$, and our method succeeds
subject to quite general conditions. Of course, for a proper solution of the
Manin--Peyre conjecture, we do not only have to establish \eqref{manin}, but to
recover the geometric and arithmetic nature of $c_1$ and $c_2$ in terms of the
Manin--Peyre predictions. This will require some natural consistency
conditions involving the exponents $h_{ij}$ in the torsor equation
\eqref{torsor} and $\alpha^\nu_{ij}$ in the height conditions \eqref{height},
cf.\ in particular \eqref{1c}, \eqref{1a} below.

We now describe in more detail the analytic machinery that yields asymptotic
formulas of type \eqref{manin} for the problem given by \eqref{torsor},
\eqref{height}, \eqref{gcd}.  Input of two types is required.
 
On the one hand, we need a preliminary upper bound of the expected order of
magnitude for the count in question. The precise requirements are formulated
in the form of Hypothesis~\ref{H2} below. In many instances, the desired
bounds can be verified by soft and elementary techniques. In particular, for
smooth spherical Fano varieties of semisimple rank one and type $T$, this can
be checked by computing dimensions and extreme points of certain polytopes,
see Proposition~\ref{propH2}.

On the other hand, we require an asymptotic formula for the number of integral
solutions of \eqref{torsor} in potentially lopsided boxes, with variables
restricted by $\frac{1}{2} X_{ij} \leq |x_{ij}| \leq X_{ij}$, say. As a
notable feature of the method, the asymptotic information is required only
when the $k$ products $\prod_j X_{ij}^{h_{ij}}$ $(1 \leq i \leq k)$ have
roughly the same size. The circle method deals with this auxiliary counting
problem in considerable generality, culminating in
Proposition~\ref{circle-method} that comes with a power saving in the shortest
variable $\min_{ij} X_{ij}$.

The method described in Section~\ref{sec8} transfers the information obtained
for counting in boxes to the strangely shaped region described by the
conditions \eqref{height}.  In \cite{BB} we presented a combinatorial method
to achieve this for certain regions of hyperbolic type. Here we use complex
analysis to do this work for us in a far more general context. A prototype of
this idea, developed only in a special (and nonsmooth) case, can be found in
\cite{BBS2}. The final result is Theorem~\ref{analytic-theorem} that we will
state once the relevant notation has been developed. Again we are working in
greater generality than needed for the immediate applications in this paper,
with future applications in mind.

In the case of smooth spherical Fano threefolds of semisimple rank one and type
$T$ (and in many other examples that can be found in
\cite{MR3348473,fahrner,HS}, for example), the torsor equation \eqref{torsor}
is of the shape ``2-by-2 determinant equals some monomial'', that is (up to
changing signs)
\begin{equation}\label{typeT}
x_{11} x_{12} + x_{21} x_{22}  +   \prod_{j=1}^{J_3}   x_{3j}^{h_{3j}} = 0.
\end{equation}
While the general transition method is independent of the shape of the torsor
equation, for the particular case \eqref{typeT}, Theorem
\ref{analytic-theorem} together with Propositions~\ref{circle-method} and
\ref{propH2} offers a ``black box'' to obtain the Manin--Peyre conjecture in
any given situation with a small amount of elementary computations. This is
formalized in Theorem~\ref{manin-cor}, which readily yields the proofs of
Theorems~\ref{dim3} and \ref{dim4} in Sections~\ref{appl1} and \ref{appl2}.

This leaves us with the task to establish an asymptotic formula for the number of solutions of the torsor equation \eqref{typeT}, with suitable constraints on the variables. The equation \eqref{typeT} involves an isolated product $x_{11}x_{12}$, one way to proceed would be to view \eqref{typeT} as a congruence modulo $x_{11}$, thus eliminating $x_{12}$. This approach is very familiar to workers in the area of divisor sums; an exemplary  and historic reference is Titchmarsh's work on the divisor problem that now bears his name. In contexts very closely related to the questions that concern us here it has been successfully applied, too, for example in work of  Le Boudec \cite{Bou}, in a collaboration of the first two authors of this paper with Salberger \cite{BBS2}, and on many other occasions. However, there are a number of disadvantages stemming from the asymmetric use of the variables $x_{11}, x_{12}, x_{21}$ and $x_{22}$. In particular,  our transition to counting solutions of \eqref{typeT} in spiky regions needs to be feeded with information on the distribution of the solutions of \eqref{typeT} with {\em all}  variables in dyadic ranges. We therefore eschew the elementary approach in favour of the circle method. The restriction to dyadic ranges is easy to implement in this environment, and the resulting leading terms in the asymptotic formulae lend themselves more easily to Peyre's predictions, too.

The following table summarizes the analytic data discussed in this subsection
for the varieties $X_1, \ldots, X_8$ featured in Theorems~\ref{dim3} and
\ref{dim4}.  Here $N$ is the number of height conditions in \eqref{height};
the total number of variables is
$J = J_0+\dots+J_3 = \dim X_i + \rank \Pic X_i + 1$.

\begin{table}[ht]
  \centering
  \begin{tabular}[ht]{ccccccc}
    \hline
    & dim & rk Pic &   torsor equation & $N$ \\
    \hline\hline
    $X_1$ & $3$ & 3 &  {$x_{11}x_{12}-x_{21}x_{22}-x_{31}x_{32}$} & 13\\
    $X_2$& $3$ & 3 &  {$x_{11}x_{12}-x_{21}x_{22}-x_{31}x_{32}x_{33}^2$} & 13\\
    $X_3$ & $3$ & 4 &   {$x_{11}x_{12}-x_{21}x_{22}-x_{31}x_{32}$}  & 14 \\
    $X_4$ & $3$ & 4 &  {$x_{11}x_{12}-x_{21}x_{22}-x_{31}x_{32}$}  & 17 \\
    \hline
    $X_5$ & $4$ & 5  &   {$x_{11}x_{12}-x_{21}x_{22}-x_{31}x_{32}x_{33}$}  & 34 \\
    $X_6$ & $5$ & 3  &   {$x_{11}x_{12}-x_{21}x_{22}-x_{31}x_{32}^2$}  & 24\\
    $X_7$ & $6$ &  5  &   {$x_{11}x_{12}-x_{21}x_{22}-x_{31}x_{32}x_{33}x_{34}x_{35}^2$}  & 80 \\
    $X_8$ & $7$ &6 & {$x_{11}x_{12}-x_{21}x_{22}-x_{31}x_{32}x_{33}^2x_{34}^2$} & 156\\
    \hline
    $\tX^\dagger$ & 3 & 4 & {$x_{11}x_{12}-x_{21}x_{22}-x_{31}x_{32}x_{33}^2$} & 13 \\
    \hline
  \end{tabular}
  \caption{}
 \label{tab:v18}
\end{table}

\subsection{Another application}

Theorem~\ref{manin-cor} offers a promising line of attack to establish Manin's
conjecture in many instances, not only those covered by Theorems~\ref{dim3}
and \ref{dim4}. As proof of concept, we include a somewhat different
application featuring a singular spherical Fano threefold. The last two
authors \cite{MR3853044} have studied some examples, and have confirmed
Manin's conjecture for two families of singular spherical Fano threefolds. One
family was given by the equation $ad-bc-z^{n+1}=0$ in weighted projective
space $\mathbb{P}(1,n,1,n,1)$, the other was the family of hypersurfaces given
by $ad-bc-y^n z^{n+1}=0$ in a certain toric variety ($n \ge 2$).  For the
counting problem on the torsor, elementary analytic techniques were enough. We
believe that this is related to the fact that all the varieties have
noncanonical (log terminal) singularities, with the exception of the first
variety for $n = 2$, which is a slightly harder case with canonical
singularities and a crepant resolution. However, for similar varieties, the
elementary counting techniques in \cite{MR3853044} do not seem to be of
strength sufficient for a proof of Manin's conjecture.

In Section~\ref{sec:Xdagger}, we use the much stronger technology developed in
this paper to discuss one such case. Let $X^\dagger$ be the anticanonical
contraction of the blow-up of the hypersurface
$\Vd(z_{11}z_{12}-z_{21}z_{22}-z_{31}z_{32})$ in $\Pd^2_\Qd \times \Pd^2_\Qd$
(with coordinates $(z_{11}:z_{21}:z_{31})$ and $(z_{12}:z_{22}:z_{32})$) in
the two curves $\Vd(z_{31}) \times \{(0:0:1)\}$ and $\Vd(z_{31}, z_{32})$.
This is a singular Fano threefold admitting a crepant resolution.

\begin{theorem}\label{thm2}
  For the singular spherical Fano threefold $X^\dagger$, there exists a
  positive number $C^{\dag}$ such that
  \begin{equation*}
     N^{\dag}(B) = (1+o(1)) C^\dag B(\log B)^3.
  \end{equation*}
  The value of $C^{\dag}$ is the one predicted by Peyre \cite{MR2019019}.
\end{theorem}

Further applications are postponed to a separate paper. \medskip

\noindent \emph{Notational remarks.} This work draws on results from various
areas of mathematics. Due to the large number of topics covered it seemed
impracticable to aim for an entirely consistent notation. Any attempt to do so
would be in conflict with traditions in the respective fields. We opt for a
pragmatic approach and use notation that, locally, seems natural to working
mathematicians.  For example, almost everywhere in the paper, the letter $B$
signals the threshold for the height of points in several counting problems,
but in Section~\ref{sec2}, a Borel subgroup of the group $G$ that occurs in
the definition of a spherical variety is denoted by $B$.  This is just one
example of double booking for symbols that are often ``frozen'' in less
interdisciplinary writings. We therefore introduce notation at the appropriate
stage of the argument.

\subsection*{Acknowledgements}

The authors thank the anonymous referees for their useful remarks and suggestions.

\part{Heights and Tamagawa measures in Cox coordinates}\label{part1}

Universal torsors were introduced and studied by Colliot-Th\'el\'ene and
Sansuc; see \cite{MR0899402}. Their first major application to Manin's
conjecture can be found in the work of Salberger \cite{MR1679841} on toric
varieties.

Cox rings were defined by Hu and Keel \cite{MR1786494}, and they
provide a global description of universal torsors; the Cox ring of a
normal irreducible algebraic variety $X$ is roughly defined as
$\Rm(X) = \bigoplus_{[D]\in\Cl(X)} \Gamma(X, \mathcal{O}_X(D))$, where
specifying the multiplication law requires some care. Moreover, a
quotient construction $\Spec\Rm(X) \supseteq \stX \to X$ is obtained.
This generalizes the homogeneous coordinate ring of $\Pd^n$ with
quotient construction $\Ad^{n+1}\setminus\{0\} \to \Pd^n$ as well as
Cox's construction for toric varieties \cite{coxtor}. For details on
toric varieties and Cox rings, we refer to the books \cite{MR2810322,adhl15} and to
\cite{arXiv:1408.5358}.

Given a variety whose Cox ring with precisely one relation is known
explicitly, we show (under mild conditions) how to write down an anticanonical
height function \eqref{eq:height_definition}, how to make the counting problem
on a universal torsor explicit
(Proposition~\ref{prop:countingproblem_abstract}), and how to express Peyre's
constant (Proposition~\ref{prop:peyre}). This is achieved in terms of the Cox
ring data, without constructing an anticanonical embedding in a projective
space, widely generalizing results from
\cite{MR1340296,MR1681100,MR1679841,BBS1,BBS2}.

\section{Varieties and universal torsors in Cox coordinates}\label{sec:charts_torsors}

In this section, we recall how a variety $X$ with precisely one relation in
its Cox ring can be described in \emph{Cox coordinates} as a hypersurface in a
toric variety (with affine charts as in Section~\ref{sec:affine_charts} that
will be used in in the following sections), and how this gives a description
of their universal torsors as hypersurfaces in affine space
(Section~\ref{sec:torsors_models}). This leads to an explicit description of
the parameterization of the rational points on $X$ by integral points on a
universal torsor (Proposition~\ref{prop:lift_to_torsor}).

Let $X$ be a smooth split projective variety over $\Qd$ with big and semiample
anticanonical class $\omega_X^\vee$ whose Picard group is free of finite
rank. (Here, \emph{split} means that the natural map from the Picard group
$\Pic X$ over the ground field to the geometric Picard group is
an isomorphism.) Assume that it has a finitely generated Cox ring $\Rm(X)$
\cite[Definition~2.6]{MR1786494}, \cite[\S~1.4]{adhl15} with precisely
one relation with integral coefficients.

In other words, $X$ has a Cox ring over $\Qd$ \cite{arXiv:1408.5358} of the
form
\begin{equation}\label{eq:cox_ring}
  \Rm(X) \cong \Qd[x_1,\dots,x_J]/(\Phi),
\end{equation}
where $x_1,\dots,x_J$ is a system of pairwise nonassociated $\Pic X$-prime
generators and the relation $\Phi \in \Zd[x_1,\dots,x_J]$ is nonzero.
According to \cite[Construction~3.2.5.3]{adhl15}, \eqref{eq:cox_ring} defines a canonical
embedding of $X$ into a (not necessarily complete) ambient toric variety
$Y^\circ$.

\begin{lemma}
  \label{lemma:small_completion}
  The toric variety $Y^\circ$ can be completed to a projective toric
  variety $Y$ such that the natural map $\Cl Y \to \Cl X=\Pic X$ is an
  isomorphism and $-K_X$ is big and semiample on $Y$.
\end{lemma}

\begin{proof}
  By \cite[Proposition~3.2.5.4(iii)]{adhl15}, we have
  $\Cl Y^\circ = \Cl X$. We consider the Gelfand--Kapranov--Zelevinsky
  (GKZ) decomposition of $Y^\circ$ (see, for example,
  \cite[\S~2.2.2]{adhl15}). According to
  \cite[Construction~3.2.5.7]{adhl15} the chambers in the GKZ
  decomposition of $Y^\circ$ which contain ample divisors on $X$ give
  rise to completions $Y$ of $Y^\circ$ with $\Cl Y^\circ = \Cl Y$. Now
  choose $Y$ corresponding to a chamber whose closure contains $-K_X$.
  Since $-K_X$ is semiample on $X$, this is possible by
  \cite[Proposition~3.3.2.9]{adhl15}. Then $-K_X$ is semiample
  on $Y$ according to \cite[Proposition~2.4.2.6]{adhl15}.

  By \cite[Propositions~3.3.2.9 and 2.4.2.6]{adhl15},
  $-K_X$ is in the relative interior of the moving cone of $Y$, hence
  $-K_X$ is big on $Y$.
\end{proof}

We assume that $Y$ is chosen as in Lemma~\ref{lemma:small_completion}.  Its
Cox ring is $\Rm(Y) = \Qd[x_1,\dots,x_J]$ \cite[Construction~3.2.5.3]{adhl15}.
Let $\Sigma$ be the fan of $Y$, and let $\Sigmamax$ be the set of maximal
cones.  The generators $x_1,\dots,x_J$ have the same grading as in $\Rm(X)$
and are in bijection to the rays $\rho \in \Sigma(1)$; we also write $x_\rho$
for $x_i$ corresponding to $\rho$. We generally write
\begin{equation}\label{JN}
J = \#\Sigma(1), \quad N = \#\Sigmamax,
\end{equation}
and we assume:
\begin{equation}\label{eq:toric_smooth}
  \text{The projective toric variety $Y$ can be chosen to be regular.}
\end{equation}

\subsection{Affine charts in Cox coordinates}\label{sec:affine_charts}

Since $\Rm(X) \cong \Qd[x_\rho : \rho \in \Sigma(1)]/(\Phi)$ with
$\Pic X$-homogeneous $\Phi$, our variety $X$ is a hypersurface defined by
$\Phi$ (in Cox coordinates) in the toric variety $Y$ (with Cox ring
$\Rm(Y)=\Qd[x_\rho : \rho \in \Sigma(1)]$). On $Y$, we can regard $X$ as a
prime divisor of class $\deg\Phi \in \Cl Y$.

We introduce further notation for the toric variety $Y$.
In Part~\ref{part1}, let $U$ be the open
torus in $Y$. For each $\rho \in \Sigma(1)$, we have a $U$-invariant Weil divisor
$D_\rho$ defined by $x_\rho$ of class $[D_\rho]=\deg(x_\rho) \in \Cl Y$
\cite[\S 4.1]{MR2810322}. Let
\begin{equation}\label{eq:D0}
D_0 \coloneqq \sum_{\rho \in \Sigma(1)} D_\rho,
\end{equation}
which is an effective divisor of
class $[D_0]=-K_Y$. For a $U$-invariant divisor $D=\sum_{\rho \in \Sigma(1)}
\lambda_\rho D_\rho$, let
\begin{equation}\label{eq:x^D}
  x^D \coloneqq \prod_{\rho \in \Sigma(1)} x_\rho^{\lambda_\rho}
\end{equation}
denote the corresponding monomial of degree $[D]$. For example,
\begin{equation}\label{eq:x^D0}
x^{D_0} = \prod_{\rho \in \Sigma(1)} x_\rho.
\end{equation}

\begin{lemma}
  \label{lemma:basiscl}
  Let $M$ and $N$ be the character and cocharacter lattices of the
  toric variety $Y$ respectively. Let
  $\rho_1, \dots, \rho_k \in \Sigma(1)$ be rays such that their
  primitive generators $u_{\rho_1}, \dots, u_{\rho_k} \in N$ form a basis of $N$.
  Then the set $\{[D_\rho] : \rho \ne \rho_1, \dots, \rho_k\}$ is a
  basis of $\Cl Y$.
\end{lemma}
\begin{proof}
  According to \cite[Before Proposition~2.1.2.7]{adhl15}
  there are two exact sequences
  \begin{align*}0 \to L \to \Z^{\Sigma(1)} \to N \to 0,\\
  0 \leftarrow \Cl(Y) \leftarrow \Z^{\Sigma(1)} \leftarrow M \leftarrow 0.
  \end{align*}
  which are dual to each other. Here $\Z^{\Sigma(1)}$ denotes the
  lattice with basis $\{e_\rho : \rho \in \Sigma(1)\}$, which is
  assumed to be dual to itself. The top right map sends $e_\rho$ to
  $u_\rho$ while the lower left map sends $e_\rho$ to $[D_\rho]$.
  Since the top right map sends $e_{\rho_1}, \dots e_{\rho_k}$ to a
  basis of $N$, the lower left map sends their complement to a basis
  of $\Cl(Y)$.
\end{proof}

It follows from Lemma~\ref{lemma:basiscl} that for each
$\sigma \in \Sigmamax$, the set $\{[D_\rho] : \rho \notin \sigma(1)\}$
is a basis of $\Cl Y$; in other words,
\begin{equation}\label{eq:sigma-basis}
  \{\deg(x_\rho) : \rho \notin \sigma(1)\}
\end{equation}
is a basis of $\Pic X$.

\begin{lemma}\label{lem:monomials_degree_L}
  For each $\sigma \in \Sigmamax$, there is a unique effective Weil divisor
  $D(\sigma)=\sum_{\rho \notin \sigma(1)} \alpha^\sigma_\rho D_\rho$ of class
  $-K_X$ whose support is contained in
  $\bigcup_{\rho \notin \sigma(1)} D_\rho$.
\end{lemma}

\begin{proof}
  For the existence, choose an effective $U$-invariant $\Qd$-Weil divisor $D$
  on $Y$ with $[D]=-K_X$. Let $M$ be the character lattice of the torus
  $U$. We write $U_\sigma \subseteq Y$ for the open subset corresponding to the cone $\sigma$.
  
  Choose $\chi_\sigma \in M_\Qd$ such that
  $(\Div \chi_\sigma)_{|U_\sigma} = D_{|U_\sigma}$. Define
  $D(\sigma) \coloneqq D-\Div \chi_\sigma$. Then $D(\sigma)$ is of class
  $-K_X$ and its support is contained in
  $\bigcup_{\rho \notin \sigma(1)} D_\rho$. Moreover, a multiple of $-K_X$
  being globally generated means that we have $\chi_\sigma \le \chi_{\sigma'}$
  on $\sigma'$ for every $\sigma' \in \Sigmamax$ \cite[Theorem
  6.1.7]{MR2810322}. Hence $D(\sigma)$ is an effective $\Qd$-divisor.

  Because of \eqref{eq:sigma-basis}, there is a unique $\Zd$-linear
  combination of the $D_\rho$ with $\rho \notin \sigma(1)$ of class $-K_X$,
  which must be equal to $D(\sigma)$.
\end{proof}

For $\sigma \in \Sigmamax$, notation \eqref{eq:x^D} gives
\begin{equation}\label{eq:height_monomials}
  x^{D(\sigma)} = \prod_{\rho \notin \sigma(1)} x_\rho^{\alpha^\sigma_\rho},
\end{equation}
where $\alpha^\sigma_\rho$ are the unique nonnegative integers satisfying
$-K_X = \sum_{\rho \notin \sigma(1)} \alpha^\sigma_\rho \deg(x_\rho)$ in $\Pic X$
(as in Lemma~\ref{lem:monomials_degree_L}).

Every $\sigma \in \Sigmamax$ defines an affine chart on $Y$ as follows. For
each $\rho' \in \Sigma(1)$, we can write
\begin{equation}\label{eq:alpha_sigma_rho_rho'}
\deg(x_{\rho'}) = \sum_{\rho \notin \sigma(1)} \alpha^\sigma_{\rho',\rho} \deg(x_{\rho})
\end{equation}
with certain $\alpha^\sigma_{\rho',\rho} \in \Zd$ by \eqref{eq:sigma-basis}.
Then
\begin{equation*}
  z^\sigma_{\rho'} \coloneqq x_{\rho'}/\prod_{\rho \notin \sigma(1)}
  x_\rho^{\alpha^\sigma_{\rho',\rho}}
\end{equation*}
is a rational section of degree $0 \in \Cl Y$, with $z^\sigma_{\rho'}=1$
for $\rho' \notin \sigma(1)$. By \cite[Theorem~1.2.18]{MR2810322},
the sections $z^\sigma_{\rho'}$ for $\rho' \in \sigma(1)$ define an isomorphism
\begin{equation}\label{eq:U^sigma}
U^\sigma \to \Ad^{\sigma(1)}_\Qd,
\end{equation}
where $U^\sigma$ is the open subset of $Y$ where $x_\rho \ne 0$ for all
$\rho \notin \sigma(1)$ (\ie the complement of
$\bigcup_{\rho \notin \sigma(1)} D_\rho$ in $Y$).

We also obtain affine charts on the open subset
\begin{equation}\label{eq:X^sigma}
X^\sigma \coloneqq X \cap U^\sigma
\end{equation}
of $X$. The image
of $X^\sigma$ in $\Ad^{\sigma(1)}_\Qd$
is defined by 
\begin{equation}\label{eq:Phi^sigma}
  \Phi^\sigma \coloneqq \Phi(z^\sigma_\rho)=
  \Phi(x_\rho)/\prod_{\rho\notin \sigma(1)}x_\rho^{\beta^\sigma_\rho},
\end{equation}
where $\beta^\sigma_\rho \in \Zd$ satisfy
\begin{equation}\label{eq:deg_Phi}
\deg\Phi=\sum_{\rho\notin \sigma(1)} \beta^\sigma_\rho\deg(x_\rho)
\end{equation}
since
$x_\rho \ne 0$ on $U^\sigma$ for $\rho \notin \sigma(1)$.  By the implicit
function theorem, for every $P \in X^\sigma(\Qd_v)$ with
$\partial \Phi^\sigma/\partial z^\sigma_{\rho_0}(P) \ne 0$ for some
$\rho_0 \in \sigma(1)$, there is an open $v$-adic neighborhood
$U_0 \subseteq X^\sigma(\Qd_v)$ such that the composition of
$X^\sigma \to \Ad_\Qd^{\sigma(1)}$ with the natural projection
$\pi^\sigma_{\rho_0} \colon \Ad_\Qd^{\sigma(1)} \to \Ad_\Qd^{\sigma(1)
  \setminus \{\rho_0\}}$ that drops the $\rho_0$-coordinate induces a chart
\begin{equation}\label{eq:chart}
  U_0 \to \Qd_v^{\sigma(1) \setminus \{\rho_0\}}.
\end{equation}
Its inverse is obtained by
computing the $\rho_0$-coordinate
$z^\sigma_{\rho_0}=\phi((z^\sigma_\rho)_{\rho \in
  \sigma(1)\setminus\{\rho_0\}})$ using the implicit function $\phi$ obtained
by solving $\Phi^\sigma$ for $z^\sigma_{\rho_0}$.

\subsection{Universal torsors and models}\label{sec:torsors_models}

Let $T \cong \mathbb{G}_{\mathrm{m},\Qd}^{\rank \Pic X}$ be the
N\'eron--Severi torus of $X$ (\ie the torus whose characters are
$\Pic X =\Cl Y$).  Cox's construction and the theory of Cox rings \cite[\S
8]{MR1679841} and \cite[\S 5.1]{MR2810322} give universal torsors
$X_0 \subset Y_0$ (with inclusion morphism $\iota_0 \colon X_0\to Y_0$) over
$X \subset Y$ (with inclusion $\iota : X\to Y$). Here $Y_0$ is the
principal universal torsor over $Y$ under $T$. Both projections $X_0\to X$ and
$Y_0 \to Y$ are called $\pi$.

We have fans $\Sigma_1 \supset \Sigma_0 \to \Sigma$ (with the sets of rays
$\Sigma_1(1)=\Sigma_0(1)$ in natural bijection to $\Sigma(1)$) corresponding
to the toric varieties
$\Ad_\Qd^J = \Ad_\Qd^{\Sigma(1)} = Y_1 \supset Y_0 \to Y$.
We have $Y_0 = Y_1 \setminus Z_Y$, where $Z_Y$ is
defined by the \emph{irrelevant ideal} \cite[\S 5.2]{MR2810322} generated by the monomials
\begin{equation}\label{eq:underline_sigma}
x^{\underline\sigma}  \coloneqq  \prod_{\rho \notin \sigma(1)} x_\rho
\end{equation}
for all
maximal cones $\sigma\in\Sigmamax$. By \cite[Proposition~5.1.6]{MR2810322}, there
are \emph{primitive collections}
\begin{equation}\label{eq:primitive_collections}
  S_1,\dots,S_r \subseteq \Sigma(1)
\end{equation}
(\ie
$S_j \not\subseteq \sigma(1)$ for all $\sigma \in \Sigma$, but for every
proper subset $S_j'$ of $S_j$, there is a $\sigma \in \Sigma$ with
$S_j' \subseteq \sigma(1)$)
such that the $r$ irreducible components of $Z_Y$ are defined by the vanishing of
$x_\rho$ for all $\rho \in S_j$.

The fans and their maps allow us to construct $\Zd$-models
$\tpi\colon \tY_1 \setminus \tZ_Y = \tY_0 \to \tY$ with an action of
$\tT \cong \mathbb{G}_{\mathrm{m},\Zd}^{\rank \Cl Y}$ on $\tY_0$ and $\tY_1$
(see \cite[Remark 8.6b and later]{MR1679841}).

The characteristic space $X_0$ is defined in $Y_0$ by $\Phi$ (interpreted as
an affine equation; see \cite[\S 1.6.3]{adhl15}). Then
$X_0 = X_1 \setminus Z_X$, where $X_1 = \Spec\Rm(X)$ is defined by $\Phi$ in
$Y_1$, and $Z_X = Z_Y \cap X_1$.

We have $\tpi \colon \tX_1 \setminus \tZ_X = \tX_0 \to \tX$ for
$\Zd$-models of $X,X_0,X_1,Z_X$ defined in $\tY,\tY_0,\tY_1,\tZ_Y$ by
$\Phi$ (regarded as an affine equation for $\tX_0,\tX_1,\tZ_X$ and as
$\Cl Y$-homogeneous for $\tX$).

\begin{prop}\label{prop:lift_to_torsor}
  We have
  \begin{align*}
    \tX_0(\Zd) &= \{\xv=(x_\rho)_{\rho \in \Sigma(1)} \in \Zd^{\Sigma(1)} :
                 \Phi(\xv)=0,\ \gcd\{x_\rho : \rho \in S_j\}=1  \text{ {\rm
                 for all  }} j=1,\dots,r\},\\
    \tX_0(\Zd_p) &= \{\xv=(x_\rho)_{\rho \in \Sigma(1)} \in \Zd_p^{\Sigma(1)}
                   : \Phi(\xv)=0,\ p \nmid \gcd\{x_\rho : \rho \in S_j\}
                   \text{ {\rm  for all }}j=1,\dots,r\}.
  \end{align*}
  The map $\tpi$ induces a $2^{\rank \Pic X} : 1$-map $\tX_0(\Zd) \to \tX(\Zd)=X(\Qd)$.
\end{prop}

\begin{proof}
  Arguing as in \cite[(11.5)]{MR1679841}, but using the 
  description of $\tZ_Y$ by the primitive collections shows 
  \begin{equation*}
    \tY_0(\Zd) = \{\yv \in \Zd^{\Sigma(1)} : \gcd\{y_\rho : \rho \in S_j\}=1
    \text{ for all } j=1,\dots,r\}.
  \end{equation*}
  Since $\tX$ is defined by $\Phi$ in $\tY$, the first result follows. The
  description of $\tX(\Zd_p)$ is obtained similarly.
  
  By \cite[Lemma~11.4]{MR1679841}, $\tpi$ induces a $2^{\rank \Cl Y} : 1$-map
  $\tY_0(\Zd) \to \tY(\Zd) = Y(\Qd)$. Restricting to the points where $\Phi$
  vanishes gives the result.
\end{proof}

\section{Heights in Cox coordinates}\label{sec:metrics_heights}

In this section, we construct an explicit adelic metrization of the
anticanonical bundle of our variety $X$ with one relation $\Phi$ in its Cox
ring (Section~\ref{sec:poincare}), using the charts from
Section~\ref{sec:affine_charts} and Poincar\'e residues. This metrization is
the basis for the construction of an anticanonical height function
(Section~\ref{sec:height}) that we use to count points, and of the Tamagawa
measure for Peyre's expected leading constant
(Section~\ref{sec:tamagawa_cox}). On the universal torsor, only the
archimedean factor of the height function remains
(Section~\ref{lem:height_torsor}). This leads to the main result of this
section: a completely explicit description of the counting problem
(Proposition~\ref{prop:countingproblem_abstract}) in terms of the Cox ring of
$X$. Section~\ref{sec:linear_algebra} contains some related linear algebra
results that will be used later.

We keep the assumptions and notation from Section~\ref{sec:charts_torsors}.

\subsection{Adelic metrization of $\omega_X^{-1}$ via Poincar\'e residues}\label{sec:poincare}

Here, we use the notation and results from Section~\ref{sec:affine_charts}.
A special case of the following can be found in \cite[\S 5]{BBS1}.  There is a
global nowhere vanishing section $s_Y$ of $\omega_Y(D_0)$ \eqref{eq:D0} whose restriction to
every open subset $U^\sigma \subset Y$ as in \eqref{eq:U^sigma} for $\sigma \in \Sigmamax$ is
$\pm \bigwedge_{\rho \in \sigma(1)} \frac{\dd z^\sigma_\rho}{z^\sigma_\rho}$
(see \cite[Proposition~8.2.3]{MR2810322}). Recall the definition of $\Phi^\sigma$ \eqref{eq:Phi^sigma}.

\begin{lemma}
  For each $\sigma \in \Sigmamax$, we define
  \begin{equation}\label{eq:varpi^sigma}
    \varpi^\sigma  \coloneqq  \frac{x^{D_0}}{x^{D(\sigma)}\Phi} s_Y \in
    \Gamma(Y,\omega_Y(D(\sigma)+X));
  \end{equation}
  this is a nowhere vanishing global section of $\omega_Y(D(\sigma)+X)$.
  On $U^\sigma$, we have
  \begin{align*}
    \varpi^\sigma  = \frac{\pm 1}{\Phi^\sigma} \bigwedge_{\rho \in \sigma(1)}
    \dd z^\sigma_\rho \in \Gamma(U^\sigma, \omega_Y(X)).
  \end{align*}
\end{lemma}

\begin{proof}
  For the first statement, note that $x^{D_0}(x^{D(\sigma)}\Phi)^{-1}$
  corresponds to the divisor $D_0-D(\sigma)-X$.

  On $U^\sigma$, we have
  \begin{equation}
    \varpi^\sigma  = \frac{\pm x^{D_0}}{x^{D(\sigma)}\Phi} \bigwedge_{\rho \in \sigma(1)}
    \frac{\dd z^\sigma_\rho}{z^\sigma_\rho} \in \Gamma(U^\sigma, \omega_Y(X))
  \end{equation}
  where
  $\Gamma(U^\sigma, \omega_Y(X)) = \Gamma(U^\sigma, \omega_Y(D(\sigma)+X))$
  since $D(\sigma)_{|U_{\sigma}} = 0$ by Lemma~\ref{lem:monomials_degree_L}. With $\beta^\sigma_\rho$ as in \eqref{eq:deg_Phi}, let
  \begin{align*}
    \lambda = \frac{x^{D_0}}{x^{D(\sigma)}\prod_{\rho\notin \sigma(1)}x_\rho^{\beta^\sigma_\rho}}.
  \end{align*}
  In view of \eqref{eq:Phi^sigma}, we obtain
  \begin{equation*}
    \varpi^\sigma  = \frac{\pm \lambda}{\Phi^\sigma} \bigwedge_{\rho \in \sigma(1)}
    \frac{\dd z^\sigma_\rho}{z^\sigma_\rho} \in \Gamma(U^\sigma, \omega_Y(X)).
  \end{equation*}

  On $U_\sigma$, we have
  \begin{equation*}
    \Div \lambda = (\Div x^{D_0})_{|U_\sigma} - (\Div
    x^{D(\sigma)})_{|U_\sigma}
    - \sum_{\rho \notin \sigma(1)} \beta_\rho^\sigma D_\rho
    = (\Div x^{D_0})_{|U_\sigma} - 0 - 0 = (\Div x^{D_0})_{|U_\sigma}.
  \end{equation*}
  We also have
  $\Div \prod_{\rho\in\sigma(1)} z_\rho^\sigma = (\Div
  x^{D_0})_{|U_\sigma}$. Therefore
  $\lambda = \prod_{\rho\in\sigma(1)} z_\rho^\sigma$ on $U_\sigma$, and we
  obtain the second statement.
\end{proof}

The Poincar\'e residue map
\begin{equation}\label{eq:poincare}
  \Res \colon \omega_Y(X) \to \iota_*\omega_X
\end{equation}
is a
homomorphism of $\Om_Y$-modules. On the smooth open subset $U^\sigma$ of $Y$,
it sends $\varpi^\sigma \in \Gamma(U^\sigma,\omega_Y(X))$ to
$\Res \varpi^\sigma \in \Gamma(U^\sigma,\iota_*\omega_X) =
\Gamma(X^\sigma,\omega_X)$, which is given by
\begin{equation}\label{eq:residue}
  \Res \varpi^\sigma = \frac{\pm 1}{\partial \Phi^\sigma/\partial
    z^\sigma_{\rho_0}} \bigwedge_{\rho \in \sigma(1)\setminus\{\rho_0\}} \dd z^\sigma_\rho
\end{equation}
on the open subset of $X^\sigma$ (see \eqref{eq:X^sigma}) where
$\partial \Phi^\sigma/\partial z^\sigma_{\rho_0} \ne 0$, for any
$\rho_0 \in \sigma(1)$.

\begin{lemma}
  The section $\Res \varpi^\sigma$ extends uniquely to a nowhere vanishing
  global section of $\omega_X(D(\sigma)\cap X)$.
\end{lemma}

\begin{proof}
  This is similar to \cite[Lemma~13]{BBS1}. Since $s_Y$ generates the
  $\Om_Y$-module $\omega_Y(D_0)$, each
  $$\varpi^\sigma = \frac{x^{D_0}}{x^{D(\sigma)} \Phi} s_Y$$ generates
  the $\Om_Y$-module $\omega_Y(X+D(\sigma))$. Since
  $\iota^*\Om_Y(D(\sigma)) = \Om_X(D(\sigma) \cap X)$ (using that
  $X \not\subseteq \Supp D(\sigma)$), the isomorphism
  $\iota^*\omega_Y(X) \to \omega_X$ adjoint to
  $\Res \colon \omega_Y(X) \to \iota_* \omega_X$ induces an isomorphism
  $\iota^*\omega_Y(X+D(\sigma)) \to \omega_X(D(\sigma) \cap X)$ that maps
  $\iota^*\varpi^\sigma$ to $\Res \varpi^\sigma$. Hence $\Res\varpi^\sigma$
  generates $\omega_X(D(\sigma) \cap X)$, \ie it is a nowhere vanishing global
  section.
\end{proof}

Therefore,
\begin{equation}\label{eq:tau^sigma}
\tau^\sigma \coloneqq (\Res \varpi^\sigma)^{-1}
\end{equation}
is a
nowhere vanishing global sections of
$\omega_X^{-1}(-D(\sigma)\cap X)$, which we can also view as a global
section of $\omega_X^{-1}$.

\begin{lemma}\label{lem:tau}
  The section $\tau^\sigma \in \Gamma(X,\omega_X^{-1})$ does not vanish
  anywhere on $X^\sigma$.
\end{lemma}

\begin{proof}
  The previous lemma shows that $\tau^\sigma$, as a global section of
  $\omega_X^{-1}$, has corresponding divisor $D(\sigma) \cap X$, whose support
  is contained in $X \cap \bigcup_{\rho \notin \sigma} D_\rho$, which is the
  complement of $X^\sigma$ \eqref{eq:X^sigma}.
\end{proof}

For any place $v$ of $\Qd$, we define a $v$-adic norm (or metric) on
$\omega_X^{-1}$ by
\begin{equation}\label{eq:def_norm}
  \|\tau(P)\|_v  \coloneqq 
  \min_{\sigma \in \Sigmamax : P \notin D(\sigma)} \left|\frac{\tau}{\tau^\sigma}(P)\right|_v
\end{equation}
for any local section $\tau$ of $\omega_X^{-1}$ not vanishing in
$P \in X(\Qd_v)$.  The next result shows that our family of local norms
$\|\cdot\|_v$ for all places $v$ is an adelic anticanonical norm as in
\cite[D\'efinition~2.3]{MR2019019}; see also \cite[Lemma~8.5]{BBS2}.

\begin{lemma}
  Let $p$ be a prime such that $\tX$ is smooth over $\Zd_p$.  On
  $\omega_X^{-1}$, the $p$-adic norm $\|\cdot\|_p$ defined by
  \eqref{eq:def_norm} coincides with the model norm $\|\cdot\|_p^*$ determined
  by $\tX$ over $\Zd_p$ as in \cite[Definition~2.9]{MR1679841}.
\end{lemma}

\begin{proof}
  Let $P \in X(\Qd_p)$, and let $\tau$ be a local section of $\omega_X^{-1}$
  not vanishing in $P$. Choose $\xi \in \Sigmamax$ such that
  $|(\tau^\xi/\tau)(P)|_p = \max_{\sigma \in \Sigmamax}
  |(\tau^\sigma/\tau)(P)|_p$, which is positive by Lemma~\ref{lem:tau} and the
  fact that the sets $X^\sigma$ cover $X$ \eqref{eq:X^sigma}; in particular, $\tau^\xi$ does not
  vanish in $P$. Hence we can compute
  \begin{equation*}
    \|\tau^\xi(P)\|_p^{-1} = \max_{\sigma \in \Sigmamax}
    \left|\frac{\tau^\sigma}{\tau^\xi}(P)\right|_p
    = \max_{\sigma \in \Sigmamax}
    \frac{|(\tau^\sigma/\tau)(P)|_p}{|(\tau^\xi/\tau)(P)|_p} = 1.
  \end{equation*}

  On the other hand, for each $\sigma \in \Sigmamax$, the section
  $\tau^\sigma$ extends to a global section $\widetilde\tau^\sigma$ of
  $\omega_{\tX/\Zd_p}^{-1}$, and $\omega_{\tX/\Zd_p}^{-1}$ is generated by the
  set of all these $\widetilde\tau^\sigma$ as an $\Om_{\tX}$-module.  The
  computation above shows for every $\sigma \in \Sigmamax$ that
  $\left|\frac{\tau^\sigma}{\tau^\xi}(P)\right|_p \le 1$, hence
  $\tau^\sigma(P) = a_\sigma \tau^\xi(P)$ for some $a_\sigma \in \Zd_p$ in the
  $\Qd_p$-module $\omega_X^{-1}(P)$, and hence also
  $\widetilde\tau^\sigma(P) = a_\sigma \widetilde\tau^\xi(P)$ in the
  $\Zd_p$-module ${\widetilde P}^*(\omega_{\tX/\Zd_p}^{-1})$. Therefore,
  ${\widetilde P}^*(\omega_{\tX/\Zd_p}^{-1})$ is generated by $\tau^\xi(P)$
  and consequently $\|\tau^\xi(P)\|_p^*=1$ by definition of the model
  norm. Finally we have
  \begin{equation*}
    \|\tau(P)\|_p = |(\tau/\tau^\xi)(P)|_p \cdot \|\tau^\xi(P)\|_p =
    |(\tau/\tau^\xi)(P)|_p \cdot \|\tau^\xi(P)\|_p^* = \|\tau(P)\|_p^*. \qedhere
  \end{equation*}  
\end{proof}

\subsection{Height function}\label{sec:height}

As in \cite[D\'efinition~2.3]{MR2019019}, our adelic anticanonical norm
$(\|\cdot\|_v)_v$ \eqref{eq:def_norm} allows us to define an anticanonical
height $H : X(\Qd) \to \Rd_{>0}$, namely
\begin{equation}\label{eq:height_definition}
  H(P) \coloneqq \prod_v \|\tau(P)\|_v^{-1}
\end{equation}
for any local section $\tau$ of $\omega_X^{-1}$ not vanishing in
$P \in X(\Qd)$; here and elsewhere, the product is taken over all places $v$
of $\Qd$. This anticanonical height on $X(\Qd)$ depends only on the choice of
Cox coordinates on $X$ \eqref{eq:cox_ring}.

In the following lemma, $x^{D(\sigma)}$ and $F_0$ are homogeneous elements of
$\Qd[x_\rho : \rho \in \Sigma(1)]$ of the same degree in $\Pic X$. Therefore,
$x^{D(\sigma)}/F_0$ can be regarded as a rational function on $X$ that can be
evaluated in $P \in X(\Qd)$ if $F_0$ does not vanish in $P$.

\begin{lemma}\label{lem:height}
 For any polynomial $F_0$ of degree $-K_X$ not vanishing in $P \in
  X(\Qd)$, one has
  \begin{equation*}
    H(P) = \prod_v \max_{\sigma \in \Sigmamax} \left|\frac{x^{D(\sigma)}}{F_0}(P)\right|_v.
  \end{equation*}
\end{lemma}

\begin{proof}
  Since the sets $X^\sigma$ as in \eqref{eq:X^sigma} for $\sigma \in \Sigmamax$ cover $X$, our point
  $P$ is contained in $X^{\xi}(\Qd)$ for some $\xi \in \Sigmamax$. By
  Lemma~\ref{lem:tau}, we can compute $H(P)$ with $\tau \coloneqq
  \tau^{\xi}$ as in \eqref{eq:tau^sigma}. We have $\varpi^\sigma =
  x^{-D(\sigma)}x^{D(\xi)}\varpi^{\xi}$ by definition \eqref{eq:varpi^sigma}. Since $\Res$ is an $\Om_Y$-module
  homomorphism \eqref{eq:poincare}, this implies
  $\tau^\sigma = x^{D(\sigma)}x^{-D(\xi)}\tau^{\xi}$. Therefore,
  \begin{equation}\label{eq:norm_max}
    \|\tau^{\xi}(P)\|_v^{-1}
    = \max_{\sigma \in \Sigmamax} \left|\frac{\tau^\sigma}{\tau^{\xi}}(P)\right|_v
    = \max_{\sigma \in \Sigmamax}\left|\frac{x^{D(\sigma)}}{x^{D(\xi)}}(P)\right|_v,
  \end{equation}
  hence our claim holds for $F_0 \coloneqq x^{D(\xi)}$. By the product
  formula, it follows for arbitrary $F_0$ not vanishing in $P$.
\end{proof}

\subsection{Heights on torsors}

We lift the height function $H$ to the universal torsor $X_0$ as in
Section~\ref{sec:torsors_models} as follows. Let
\begin{equation*}
  H_0 \colon X_0(\Qd) \to \Rd_{>0}
\end{equation*}
be the composition of $\pi \colon X_0(\Qd) \to X(\Qd)$ and the height
function $H$ defined in \eqref{eq:height_definition}. The following
is analogous to \cite[Proposition~10.14]{MR1679841}.

\begin{lemma}\label{lem:height_torsor}
  For $P_0 \in X_0(\Qd)$, we have
  \begin{equation*}
    H_0(P_0) = \prod_v \max_{\sigma \in \Sigmamax} |x^{D(\sigma)}(P_0)|_v.
  \end{equation*}
\end{lemma}

\begin{proof}
  Let $P = \pi(P_0) \in X(\Qd)$. For $F_0$ of degree $-K_X$ not
  vanishing in $P$ and $\sigma \in \Sigmamax$, we can compute
  $(x^{D(\sigma)}/F_0)(P)$ as in Lemma~\ref{lem:height},
  but we can also regard $x^{D(\sigma)}$ and $F_0$ as regular functions on
  $X_0$ that can be evaluated in $P_0$. Here we have
  $x^{D(\sigma)}(P_0)/F_0(P_0) = (x^{D(\sigma)}/F_0)(P)$.  Using
  Lemma~\ref{lem:height}, we obtain
  \begin{equation*}
    H_0(P_0)=H(P) = \prod_v \max_{\sigma \in \Sigmamax}
    \left|\frac{x^{D(\sigma)}}{F_0}(P)\right|_v
    = \prod_v \max_{\sigma \in \Sigmamax}
    \left|\frac{x^{D(\sigma)}(P_0)}{F_0(P_0)}\right|,
  \end{equation*}
  and $\prod_v |F_0(P_0)|_v = 1$ by the product formula.
\end{proof}

The next result is analogous to \cite[Proposition~11.3]{MR1679841}.

\begin{cor}\label{cor:height_torsor_integral}
  For any prime $p$ and $P_0 \in \tX_0(\Zd_p)$, we have
  \begin{equation*}
    \max_{\sigma \in \Sigmamax} |x^{D(\sigma)}(P_0)|_p=1.
  \end{equation*}
  For $P_0 \in \tX_0(\Zd)$, we have
  \begin{equation*}
    H_0(P_0) = \max_{\sigma \in \Sigmamax} |x^{D(\sigma)}(P_0)|_\infty.
  \end{equation*}
\end{cor}

\begin{proof}
  Let $p$ be a prime and $P_0 \in \tX_0(\Zd_p)$. Then $P_0 \bmod p$ is in
  $\tX_0(\Fd_p)$. Since $\tX_0$ is defined by the irrelevant ideal in $\tX_1$
  as in \eqref{eq:underline_sigma}, there is a $\xi \in \Sigmamax$ such that
  $x^{\underline{\xi}}(P_0 \bmod p) \ne 0 \in \Fd_p$.  Since the support of
  $D(\xi)$ is as in Lemma~\ref{lem:monomials_degree_L}, we have
  $x^{D(\xi)}(P_0 \bmod p) \ne 0 \in \Fd_p$, and hence
  $|x^{D(\xi)}(P_0)|_p=1$. Using $x^{D(\sigma)}(P_0) \in \Zd_p$ for all
  $\sigma \in \Sigmamax$, we conclude
  $\max_{\sigma \in \Sigmamax} |x^{D(\sigma)}(P_0)|_p=1$.

  Therefore, for $P_0 \in \tX_0(\Zd)$, only the archimedean factor in
  Lemma~\ref{lem:height_torsor} remains.
\end{proof}

\subsection{Parameterization in Cox coordinates}

The following proposition translates the analysis of $N_{X, U, H}(B)$ into a
counting problem as described in the introduction that is amenable to methods
of analytic number theory.  It parameterizes the rational points on $X$ by
integral points on the universal torsor $\tX_0$ in terms of the torsor
equation from the Cox ring \eqref{eq:cox_ring}, the height conditions from the
anticanonical monomials \eqref{eq:height_monomials} and the coprimality
conditions from the primitive collections \eqref{eq:primitive_collections}.

\begin{prop}\label{prop:countingproblem_abstract}
  Let $X$ be a variety as in the first paragraph of
  Section~\ref{sec:charts_torsors} that satisfies the
  assumption~\eqref{eq:toric_smooth}.  Let
  $U=X \setminus \bigcup_{\rho \in \Sigma(1)} D_\rho$ be the open
  subset of $X$ where all Cox coordinates $x_\rho$ are nonzero. Let
  $H$ be the anticanonical height function on $X(\Qd)$ defined in
  \eqref{eq:height_definition}. Then
  \begin{equation*}
    N_{X,U,H}(B) =  
    \frac{1}{2^{\rank\Pic X}} \#\left\{\xv \in \Zd^{\Sigma(1)}_{\ne 0} :
      \begin{aligned}
        &\Phi(\xv)=0,\,  \max_{\sigma \in \Sigmamax}|\xv^{D(\sigma)}|_\infty \le B,\\
        &\gcd\{x_\rho : \rho \in S_j\} = 1 \text{ for every $j=1,\dots,r$}
      \end{aligned}
    \right\}\text{,}
  \end{equation*}
  using the notation \eqref{eq:cox_ring}, \eqref{eq:height_monomials},
  \eqref{eq:primitive_collections}.
\end{prop}

\begin{proof}
  We combine the $2^{\rank\Pic X} : 1$-map and the description of $\tX_0(\Zd)$
  from Proposition~\ref{prop:lift_to_torsor} with the lifted height function in
  Corollary~\ref{cor:height_torsor_integral}. The preimage of $U(\Qd)$ in $\tX_0(\Zd)$ is
  the set where $x_\rho \ne 0$ for all $\rho \in \Sigma(1)$.
\end{proof}

\subsection{Some linear algebra}\label{sec:linear_algebra}

The monomials $\textbf{x}^{D(\sigma)}$ and the polynomial $\Phi$ that appear
in Proposition~\ref{prop:countingproblem_abstract} are not independent. In
this subsection, we analyze this dependence and describe it in the form of a
rank condition on a certain matrix. This will be useful later when we apply
methods from complex analysis to obtain an asymptotic formula for
$N_{X, U, H}(B)$.

We consider $\Qd^J=\Qd^{\Sigma(1)}$ \eqref{JN} with standard basis
$(e_\rho)_{\rho \in \Sigma(1)}$ indexed by the rays of $\Sigma$. Let
\begin{equation*}
  p\colon \Qd^{\Sigma(1)} \to (\Pic X)_\Qd
\end{equation*}
be the surjective linear map that sends $e_\rho$ to
$[D_\rho]=\deg(x_\rho)$ as in \eqref{eq:sigma-basis}. For
$\xv = (x_\rho)_{\rho \in \Sigma(1)} \in \Qd_v^{\Sigma(1)}$ for some place $v$
of $\Qd$ and
$\vv = (v_\rho)_{\rho \in \Sigma(1)} \in \Zd_{\ge 0}^{\Sigma(1)}$, let
$\xv^\vv \coloneqq \prod_{\rho \in \Sigma(1)} x_\rho^{v_\rho}$.

\begin{lemma}\label{lemma:lpp}
  The set $Q \coloneqq p^{-1}(-K_X) \cap \Qd_{\ge 0}^{\Sigma(1)}$ is a
  bounded polytope of dimension $J-\rank\Pic X$.  Its set $\Vm$ of vertices of
  $Q$ lies in $\Zd_{\ge 0}^{\Sigma(1)}$. Let $v$ be a place of $\Qd$.  For all
  nonzero $\xv \in \Qd_v^{\Sigma(1)}$, we have
  \begin{equation*}
    \max_{\sigma \in \Sigmamax} |\xv^{D(\sigma)}|_v = \max_{\vv \in \Vm} |\xv^{\vv}|_v.
  \end{equation*}
\end{lemma}

\begin{proof}
  In the notation of the proof of Lemma~\ref{lem:monomials_degree_L},
  write $D = \sum_{\rho}a_\rho D_\rho$. Then the
  $-\chi_\sigma$ are the vertices, and possibly (if $-K_X$ is not ample) some
  other points, of the $\rank M$-dimensional polytope
  \begin{align*}
  P_D = \{\chi \in M_\Qd : \langle n_\rho, \chi\rangle \ge -a_\rho \text{ for
    all $\rho$}\}\text{;}
  \end{align*}
  see \cite[\S 4.3 and after Lemma~9.3.9]{MR2810322}.

  Now consider the injective affine map
  $\phi\colon M_\Qd \to \Qd^{\Sigma(1)}$,
  $\chi \mapsto \sum_{\rho} (a_\rho + \langle n_\rho, \chi\rangle)e_\rho$ as
  well as the linear surjective map
  $p\colon \Qd^{\Sigma(1)} \to (\Cl Y)_\Qd$. We have
  $\rank M = J - \rank \Pic X$ and
  $\operatorname{im}(p \circ \phi) = \{-K_X\}$. Moreover, the condition
  $\phi(\chi) \in \Qd^{\Sigma(1)}_{\ge 0}$ is equivalent to
  $\langle n_\rho, \chi\rangle \ge -a_\rho$ for all $\rho$. It follows that
  $\phi$ restricts to a bijection
  $P_D \to Q = p^{-1}(-K_X) \cap \Qd^{\Sigma(1)}_{\ge 0}$.  Hence $Q$ is bounded and of
  dimension $J - \rank \Pic X$.

  As we have $p(-\chi_\sigma) = D(\sigma)$, where $D(\sigma)$ is
  interpreted as an element of $\Zd^{\Sigma(1)}$ in the obvious way,
  we obtain $\Vm \subseteq \phi(\{D(\sigma) : \sigma \in \Sigmamax\}) \subseteq Q$.
  Hence the equality
    \begin{equation*}
    \max_{\sigma \in \Sigmamax} |\xv^{D(\sigma)}|_v = \max_{\vv \in \Vm} |\xv^{\vv}|_v
  \end{equation*}
  holds, and, since $\phi(M) \subseteq \Zd^{\Sigma(1)}$, we also obtain
  $\Vm \subset \Zd_{\ge 0}^{\Sigma(1)}$.
\end{proof}

We recall \eqref{JN} and the notation \eqref{eq:height_monomials} for the
exponents $\alpha_{\rho}^{\sigma}$ occurring in $\textbf{x}^{D(\sigma)}$. We write
the defining equation $\Phi$ from \eqref{eq:cox_ring} in the form
\begin{equation}\label{Phi-sec3}
  \Phi = \sum_{i=1}^k b_i\prod_{\rho \in \Sigma(1)} x_\rho^{h_{i\rho}},
\end{equation}
(\ie $k$ is the number of monomials, and
$\hv_i = (h_{i\rho})_{\rho \in \Sigma(1)} \in \Zd_{\ge 0}^{\Sigma(1)}$ is the
exponent vector of the $i$-th term of $\Phi$). We now consider the block
matrix
\begin{equation}\label{matrix}
  \Am = \begin{pmatrix}\Am_1&\Am_2\\ \Am_3&\Am_4\end{pmatrix}
  \in \Rd^{(J+1)\times(N+k)}.
\end{equation}
Here
$\Am_1 = (\alpha_{\rho}^{\sigma})_{(\rho, \sigma) \in \Sigma(1) \times
  \Sigmamax} \in \Bbb{R}^{J \times N}$ is the height matrix for the height
function from Proposition~\ref{prop:countingproblem_abstract}. We let
$\Am_2 \in \Rd^{J \times k}$ be the matrix whose $i$-th column is
$\hv_i-\hv_k$ for $i=1,\dots,k-1$ and whose $k$-th column is
$\hv_k-(1,\dots,1)^{\top}$. Furthermore, let
$\Am_3 = (1, \dots, 1) \in \Rd^{1 \times N}$ and
$\Am_4 = (0,\dots,0,-1) \in \Rd^{1 \times k}$.

The definition of $\Am_2$ may appear to be somewhat artificial. Its purpose
will become clear in \eqref{zast} in Section~\ref{54}.

\begin{lemma}\label{rank}
  We have $\rank \Am = \rank \Am_1 = J - \rank\Pic X + 1$.
\end{lemma}

\begin{proof}
  According to Lemma~\ref{lemma:lpp}, the polytope $Q$ spans an affine
  subspace of dimension $J-\rank\Pic X$ in $\Rd^{J}$, which does not contain
  $0$ since $-K_X \ne 0$. It follows that $Q$ spans a vector space of
  dimension $J-\rank\Pic X + 1$ in $\Rd^{J}$. This shows
  $\rank \Am_1 = J - \rank\Pic X + 1$.

  Since the columns of $\Am_1$ lie in an affine subspace of $\Rd^J$ that does
  not contain $0$, a linear combination of these columns can be $0$ only if
  the sum of the coefficients is $0$. It follows that we have
  $ \rank \rleft(\begin{smallmatrix}\Am_1 \\ \Am_3\end{smallmatrix}\rright) =
  \rank \Am_1$.  Since $\Phi$ is $\Pic X$-homogeneous, the first $k-1$ columns
  of $\Am_2$ lie in $p^{-1}(0)$.  Moreover, note that the last column of
  $\Am_2$ lies in $p^{-1}(K_X)$ since
  $\deg\Phi-\sum_{\rho \in \Sigma(1)} \deg(x_\rho)=K_X$ by
  \cite[Proposition~3.3.3.2]{adhl15}. Together with the fact that the columns
  of $\Am_1$ lie in $p^{-1}(-K_X)$, we obtain
  $ \rank \Am = \rank \rleft(\begin{smallmatrix}\Am_1 \\
    \Am_3\end{smallmatrix}\rright)\text{.}$
\end{proof}

Let ${\bm \zeta} = (\zeta_1, \ldots, \zeta_k) \in \Bbb{R}^{k}$ be a vector
satisfying
\begin{equation}\label{zeta}
  \zeta_i > 0  \text{ for all }  1 \leq i \leq k, \quad
  \sum_{i=1}^kh_{i\rho} \zeta_i < 1\text{ for all }\rho\in \Sigma(1), \quad
  \sum_{i=1}^k \zeta_i = 1.
\end{equation}
This condition will reappear in Part~\ref{part2} as \eqref{zeta1}.

\begin{lemma}\label{pos}
  Let ${\bm \zeta}$ be as in \eqref{zeta},
  ${\bm \tau}_1= (1-\sum_{i=1}^k h_{i\rho}\zeta_i)_{\rho \in \Sigma(1)} =
  (1,\dots,1)-\sum_{i=1}^k \zeta_i \hv_i$, and let
  ${\bm\tau}= ({\bm \tau}_1, 1)^\top$.  The system of $J+1$ linear equations
  \begin{align*}
    \begin{pmatrix}\Am_1 \\ \Am_3 \end{pmatrix}{\bm\sigma} = {\bm\tau}
  \end{align*}
  has a solution ${\bm\sigma} \in \Rd^N_{> 0}$. 
\end{lemma}

\begin{proof}
  According to \cite[Proposition~3.3.3.2]{adhl15}, we have
  $\bm\tau_1 \in p^{-1}(-K_X)$.  It follows from
  $Q = p^{-1}(-K_X) \cap \Qd_{\ge 0}^{\Sigma(1)}$ that the relative interior
  of $Q$ satisfies
  $Q^\circ \supseteq p^{-1}(-K_X) \cap \Qd_{> 0}^{\Sigma(1)}$. Since all
  coordinates of $\bm\tau_1$ are positive, we obtain $\bm\tau_1 \in
  Q^\circ$. Since the columns of $\Am_1$ are the vertices of $Q$, the column
  $\tau_1^\top$ can be written as a linear combination of the columns of
  $\Am_1$ with strictly positive coefficients whose sum is $1$. The existence
  of ${\bm\sigma} \in \Rd^N_{> 0}$ as required follows.
\end{proof}
 
\section{Tamagawa numbers in Cox coordinates}\label{sec:tamagawa_cox}

In this section, we use the adelic metrization (see
Section~\ref{sec:poincare}) of the anticanonical bundle on our variety $X$ to
make the local measures (Section~\ref{sec:local_measures}) explicit that are
used in the Tamagawa number (Section~\ref{sec:tamagawa_number}) in Peyre's
constant. We lift the $p$-adic measures to the universal torsor
(Section~\ref{sec:measures_torsor}), which allows as to express the $p$-adic
densities in the Tamagawa number in terms of the number of points on the
universal torsor modulo $p^\ell$, which is the number of solutions
modulo $p^\ell$ of the relation $\Phi$ in the Cox ring
(Section~\ref{sec:points_mod_p^l}). Furthermore, we rewrite the real density
and Peyre's constant $\alpha$ (Section~\ref{sec:real_density}) in a way that
will appear in our analytic method in Part~\ref{part2}. In total, we obtain a
description of Peyre's constant for $X$ in terms of the Cox ring of
$X$ (Proposition~\ref{prop:peyre}).

We continue to work in the setting of Sections~\ref{sec:charts_torsors} and
\ref{sec:metrics_heights}. Additionally, we assume that $X$ is an almost Fano
variety (\eg a smooth Fano variety) as in \cite[D\'efinition~3.1]{MR2019019}
(\ie $X$ is smooth, projective and geometrically integral with
$H^1(X,\Om_X) = H^2(X,\Om_X) = 0$, free geometric Picard group of finite rank,
and big $\omega_X^\vee$).

\subsection{Local measures}\label{sec:local_measures}

By \cite[(2.2.1)]{MR1340296}, \cite[Notations~4.3]{MR2019019} and
\cite[Theorem~1.10]{MR1679841}, the $v$-adic norm $\|\cdot\|_v$ on
$\omega_X^{-1}$ defined in \eqref{eq:def_norm} induces a measure $\mu_v$ on
$X(\Qd_v)$. We express it using the Poincar\'e residues from
Section~\ref{sec:poincare} and the affine charts from
Section~\ref{sec:affine_charts}; in particular, recall \eqref{eq:height_monomials}, \eqref{eq:X^sigma}, \eqref{eq:varpi^sigma}, \eqref{eq:tau^sigma}. See \cite[(5.8), (5.9)]{BBS1} for an example
of the next result.

\begin{prop}\label{prop:local_measure}
  Let $\xi \in \Sigmamax$. For a Borel subset $N_v$ of $X^{\xi}(\Qd_v)$, we have
  \begin{equation}\label{eq:local_measure_abstract}
    \mu_v(N_v)
    =\int_{N_v} \frac{|\Res\varpi^{\xi}|_v}{\max_{\sigma \in \Sigmamax}
      |\tau^\sigma\Res\varpi^{\xi}|_v}
    =\int_{N_v} \frac{|\Res\varpi^{\xi}|_v}{\max_{\sigma \in \Sigmamax}
      |x^{D(\sigma)}/x^{D(\xi)}|_v},
  \end{equation}
  where $|\Res\varpi^{\xi}|_v$ is the $v$-adic density on
  $X^{\xi}(\Qd_v)$ of the volume form $\Res\varpi^{\xi}$
  on $X^{\xi}$.

  Let $\rho_0 \in \xi(1)$.  If $N_v$ is contained in a sufficiently small open
  $v$-adic neighborhood of a point $P$ in $X^{\xi}(\Qd_v)$ with
  $\partial \Phi^{\xi}/\partial z^{\xi}_{\rho_0}(P) \ne 0$, then
  \begin{equation}\label{eq:local_measure_explicit}
    \mu_v(N_v)=\int_{\pi^{{\xi}}_{\rho_0}(N_v)} \frac{\bigwedge_{\rho \in
        {\xi(1)} \setminus \{\rho_0\}} \dd z^{\xi}_\rho}
    {|\partial \Phi^{\xi}/\partial z^{\xi}_{\rho_0}(\zv^{\xi})|_v \max_{\sigma
        \in \Sigmamax}|x^{D(\sigma)}(\zv^{\xi})|_v}
  \end{equation}
  in the affine coordinates
  $\zv^{\xi} = (z^{\xi}_{\rho})_{\rho \in {\xi(1)}}$, where
  $\pi^{\xi}_{\rho_0} \colon U^{\xi}(\Qd_v)=\Qd_v^{{\xi(1)}} \to
  \Qd_v^{\xi(1)\setminus\{\rho_0\}}$ is the natural projection and
  $z^{\xi}_{\rho_0}$ is expressed in terms of the other coordinates using the
  implicit function for $\Phi^{\xi}$.
\end{prop}

\begin{proof}
  As in \eqref{eq:chart}, the implicit function theorem gives a $v$-adic neighborhood
  $U_0 \subseteq X^\xi(\Qd_v)$ of $P$ and an implicit function
  $\phi \colon V \to \Qd_v$ for
  $V = \pi^\xi_{\rho_0}(U_0) \subseteq \Qd_v^{\xi(1)\setminus\{\rho_0\}}$ such that
  $\Phi^\xi(\zv^\xi)=0$ for all $\zv^\xi \in X^\xi(\Qd_v)$ with
  $z^\xi_{\rho_0}$ the image of
  $(z^\xi_\rho)_{\rho \in \xi(1)\setminus\{\rho_0\}} \in V$ under $\phi$. We
  work with $\|\tau^\xi(P)\|_v$ as in \eqref{eq:tau^sigma} and use
  $x^{D(\xi)}(\zv^\xi)=1$ (see \eqref{eq:height_monomials}) in our affine
  coordinates on $X^\xi(\Qd_v)$.  Then the formulas in
  \cite[(2.2.1)]{MR1340296} and \cite[Theorem~1.10]{MR1679841} give
  \eqref{eq:local_measure_explicit} for $N_v \subseteq U_0$.  Indeed, our chart
  is
  \begin{equation*}
    \pi \coloneqq \pi^\xi_{\rho_0} \colon U_0 \to V \subseteq
    \Qd_v^{\xi(1)\setminus\{\rho_0\}}.
  \end{equation*}
  In this chart, by \eqref{eq:residue},
  the image of the local canonical section
  $\bigwedge_{\rho \in \xi(1) \setminus \{\rho_0\}} \dd z^\xi_\rho$ under
  \begin{equation*}
    \omega(\pi) \colon \pi^*\omega_{\Ad_\Qd^{\xi(1) \setminus
        \{\rho_0\}}} \to \omega_X
  \end{equation*}
  is $\partial \Phi^\xi/\partial z^\xi_{\rho_0}
  \cdot \Res\varpi^\xi$.  
  This implies that the image of the local anticanonical section
  $\bigwedge_{\rho \in \xi(1) \setminus \{\rho_0\}}
  \frac{\partial}{\partial z^\xi_{\rho}}$ under
  \begin{equation*}
    {}^t\omega(\pi)^{-1} \colon \pi^*\omega^{-1}_{\Ad_\Qd^{\xi(1)
        \setminus \{\rho_0\}}} \to \omega^{-1}_X
  \end{equation*}
  is
  $(\partial \Phi^\xi/\partial z^\xi_{\rho_0})^{-1} \cdot
  \tau^\xi$. Therefore, $\mu_v(N_v)$ for $N_v \subseteq U_0$ as defined
  in [Peyre95, (2.2.1)] is the integral over $\pi(N_v)$ of
  \begin{align*}
    \omega_v
    &= \|((\partial \Phi^\xi/\partial
      z^\xi_{\rho_0})^{-1}\cdot\tau^\xi)(\pi^{-1}((z^\xi_\rho)_{\rho
      \in \xi(1)\setminus\{\rho_0\}}))\|_v \bigwedge_{\rho \in \xi(1)
      \setminus \{\rho_0\}} \dd z^\xi_\rho\\
    &=|\partial \Phi^\xi/\partial z^\xi_{\rho_0}(\zv^\xi)|_v^{-1} \cdot
      \|\tau^\xi(\zv^\xi)\|_v\bigwedge_{\rho \in \xi(1) \setminus \{\rho_0\}}
      \dd z^\xi_\rho
  \end{align*}
  Using \eqref{eq:norm_max} together with $x^{D(\xi)}(\zv^\xi)=1$, we obtain
  \eqref{eq:local_measure_explicit}.
  
  By \eqref{eq:residue}, we see that the right hand side of
  \eqref{eq:local_measure_abstract} coincides with
  \eqref{eq:local_measure_explicit} for $N_v \subseteq U_0$. Since $X$ is
  smooth, $X^\xi(\Qd_v)$ can be covered with such $U_0$, hence $\mu_v(N_v)$ is
  equal to the right hand side for all $N_v \subseteq X^\xi(\Qd_v)$. Since
  $\varpi^\sigma/\varpi^\xi = x^{D(\xi)}/x^{D(\sigma)}$ by definition \eqref{eq:varpi^sigma}, we have
  $\tau^\sigma \Res\varpi^\xi = \tau^\sigma/\tau^\xi =
  x^{D(\sigma)}/x^{D(\xi)}$ by \eqref{eq:tau^sigma}, and hence the integrals in
  \eqref{eq:local_measure_abstract} are equal.
\end{proof}

\subsection{Tamagawa number}\label{sec:tamagawa_number}

Here we use some standard notation as in \cite[\S 2]{MR1340296}, \cite[\S
4]{MR2019019}. Let $S$ be a sufficiently large finite set of finite places of
$\Qd$ as in \cite[Notations~4.5]{MR2019019}. For any prime $p \in S$, let
\begin{equation*}
L_p(s,\Pic\overline{X}) \coloneqq \det(1-p^{-s}\Fr_p \mid \Pic(X_{\overline
  \Fd_p}) \otimes \Qd)^{-1}.
\end{equation*}
Since $X$ is split,
$L_p(s, \Pic\overline{X}) = (1- p^{-s})^{-\rank \Pic X}$, hence
\begin{equation*}
  L_S(s,\Pic\overline{X})  \coloneqq  \prod_{p \notin S} L_p(s,\Pic\overline{X}) =
  \zeta(s)^{\rank \Pic X}\prod_{p \in S}(1- p^{-s})^{\rank \Pic X}.
\end{equation*}
Therefore,
$\lim_{s \to 1} (s-1)^{\rank \Pic X} L_S(s,\Pic\overline{X}) = \prod_{p \in
  S}(1- p^{-1})^{\rank \Pic X}$, and the convergence factors are
\begin{equation*}
  \lambda_p^{-1}  \coloneqq  L_p(1,\Pic\overline X)^{-1} = \left(1-p^{-1}\right)^{\rank\Pic X}
\end{equation*}
for $p \notin S$ and $\lambda_p^{-1} \coloneqq 1$ for $p \in S$. Hence Peyre's
Tamagawa number \cite[D\'efinition~4.5]{MR2019019} is
\begin{equation}\label{eq:tamagawa}
  \tau_H(X) =\mu_\infty(X(\Rd)) \prod_p (1 - p^{-1})^{\rank\Pic X} \mu_p(X(\Qd_p)).
\end{equation}
The Euler product converges by \cite[Remarque~4.6]{MR2019019}.

\subsection{Measures on the torsor}\label{sec:measures_torsor}

By \cite[Proposition~8.2.3]{MR2810322}, we have a rational $\#\Sigma(1)$-form
\begin{equation*}
  s_{Y_0} = \bigwedge_{\rho \in \Sigma_0(1)} \frac{\dd y_\rho}{y_\rho}
\end{equation*}
on the toric principal universal torsor
$Y_0 \subset Y_1 = \Ad^{\Sigma_0(1)}_\Qd$ as in
Section~\ref{sec:torsors_models}, with coordinates $y_\rho$ for
$\rho \in \Sigma_0(1)$, using our bijection $\Sigma_0(1) \to \Sigma(1)$. Now
we regard $\Phi$ and $y^D$ (defined as in \eqref{eq:x^D} for $U$-invariant
divisors $D$ on $Y$) as polynomials in $y_\rho$ and as functions on $Y_0$. As
in \cite[(5.12)]{BBS1} and using the notation \eqref{eq:x^D0}, \eqref{eq:height_monomials}, we define 
\begin{equation*}
  \varpi_{Y_0}^\sigma = \frac{y^{D_0}}{y^{D(\sigma)}\Phi} s_{Y_0}
\end{equation*}
for each $\sigma \in \Sigmamax$, and
\begin{equation*}
  \varpi_{Y_0} = \frac{1}{\Phi}\bigwedge_{\rho \in \Sigma_0(1)} \dd y_\rho.
\end{equation*}
We have
\begin{equation}\label{eq:varpi_Y0^sigma}
\varpi_{Y_0}^\sigma = \varpi_{Y_0}/y^{D(\sigma)}
\end{equation}
on the open subset $Y_0^\sigma \coloneqq \pi^{-1}(U^\sigma)$ of $Y_0$, see \eqref{eq:U^sigma}.

We have
\begin{equation*}
\varpi_{Y_0}^\sigma \in \Gamma(Y_0^\sigma, \omega_{Y_0}(X_0))
\end{equation*}
with Poincar\'e
residue $\Res\varpi_{Y_0}^\sigma \in \Gamma(X_0^\sigma,\omega_{X_0})$
on
$X_0^\sigma = \pi^{-1}(X^\sigma) = X_0 \cap Y_0^\sigma$.
As in Section~\ref{sec:local_measures}, we obtain a $v$-adic
measure $m_v$ on $X_0(\Qd_v)$ defined by
\begin{equation*}
  m_v(M_v) = \int_{M_v} \frac{|\Res\varpi_{Y_0}^\xi|_v}
  {\max_{\sigma \in \Sigmamax} |y^{D(\sigma)}/y^{D(\xi)}|_v}
\end{equation*}
for a Borel subset $M_v$ of $X_0^\xi(\Qd_v)$.
Alternatively, we can write
\begin{equation*}
  m_v(M_v) = \int_{M_v} \frac{|\Res\varpi_{Y_0}|_v}{\max_{\sigma \in \Sigmamax} |y^{D(\sigma)}|_v}
\end{equation*}
because $\varpi_{Y_0} \in \Gamma(Y_0,\omega_{Y_0}(X_0))$ has a residue form
$\Res\varpi_{Y_0} \in \Gamma(X_0,\omega_{X_0})$ that restricts to
$y^{D(\xi)}\Res\varpi_{Y_0}^{\xi}$ on $X_0^{\xi}$ by \eqref{eq:varpi_Y0^sigma}. If $M_v$
is sufficiently small, this is explicitly
\begin{equation}\label{eq:measure_torsor_explicit}
  m_v(M_v) = \int_{\pi_{\rho_0}(M_v)} \frac{\bigwedge_{\rho \in \Sigma_0(1) \setminus \{\rho_0\}} \dd y_\rho}
  {|\partial \Phi/\partial x_{\rho_0}(\yv)|_v \max_{\sigma \in \Sigmamax} |\yv^{D(\sigma)}|_v}
\end{equation}
in the coordinates $\yv=(y_\rho)_{\rho \in \Sigma_0(1)}$, where $\pi_{\rho_0}$ is the
projection to all coordinates $y_\rho$ with $\rho\ne\rho_0$ and where
$y_{\rho_0}$ is expressed in terms of these coordinates using the
implicit function theorem.

\begin{lemma}
  Let $D_0^{Y_0} = \pi^* D_0$ be the sum of the prime divisors defined
  by $y_\rho=0$ for $\rho \in \Sigma_0(1)$. Then there is a unique
  nowhere vanishing global section
  $s_{Y_0/Y} \in \Gamma(Y_0,\omega_{Y_0/Y})$ such that
  $s_{Y_0} = s_{Y_0/Y} \otimes \pi^*s_Y$ via the natural isomorphism
  $\omega_{Y_0}(D_0^{Y_0}) = \omega_{Y_0/Y} \otimes \pi^*
  \omega_Y(D_0)$.
  
  Let $s_{X_0/X}$ be the image of $\iota_0^*s_{Y_0/Y}$ under the isomorphism
  $\Gamma(X_0,\iota_0^*\omega_{Y_0/Y}) \to \Gamma(X_0,\omega_{X_0/X})$, and
  $s_{X_0/X}^\sigma$ be the restriction of $s_{X_0/X}$ to $X_0^\sigma$. Then
  $\Res\varpi_{Y_0}^\sigma = s_{X_0/X}^\sigma \otimes
  \pi^*\Res\varpi^\sigma$ under the canonical isomorphism
  $\omega_{X_0} = \omega_{X_0/X} \otimes \pi^*\omega_X$.
\end{lemma}

\begin{proof}
  See \cite[Lemma~16]{BBS1}.
\end{proof}

\begin{lemma}\label{lem:measure_variety_torsor_p-adic}
  For any prime $p$, we have $m_p(\tX_0(\Zd_p)) = (1-p^{-1})^{\rank \Pic X} \mu_p(X(\Qd_p))$.
\end{lemma}

\begin{proof}
  Our proof follows \cite[Lemma~18]{BBS1}. 
  By \cite[pp.~126--127]{MR1679841}, the map $\pi\colon X_0 \to X$ induces an
  $v$-adic analytic torsor $\pi_v \colon X_0(\Qd_v) \to X(\Qd_v)$ under $T(\Qd_v)$. By
  \cite[Theorem~1.22]{MR1679841} and the previous lemma, the relative volume form
  $s_{X_0/X}$ defines $v$-adic measures on the fibers of $\pi_v$ over
  $X(\Qd_v)$. Integrating along these fibers gives a linear functional
  $\Lambda_v \colon C_c(X_0(\Qd_v)) \to C_c(X(\Qd_v))$.

Let $\chi_p\colon X_0(\Qd_p) \to \{0,1\}$ be the characteristic function
  of $\tX_0(\Zd_p) \subset \tX_0(\Qd_p) = X_0(\Qd_p)$. Since
  $\chi_p \in C_c(X_0(\Qd_p))$, we have
  $m_p(\tX_0(\Zd_p)) = \int_{X(\Qd_p)} \Lambda_p(\chi_p) \mu_p$.

  We claim that $(\Lambda_p(\chi_p))(P) = (1-p^{-1})^{\rank \Pic X}$
  for every $P \in X(\Qd_p) = \tX(\Zd_p)$. Indeed, we have
  $s_{\tY_0} = s_{\tY_0/\tY} \otimes \pi^* s_{\tY}$, where $s_{\tY_0/\tY}$
  is the extension of $s_{Y_0/Y}$ to a $\tT$-equivariant generator of
  $\omega_{\tY_0/\tY}$. Furthermore, $s_{X_0/X}$ extends to a
  $\tT$-equivariant generator $s_{\tX_0/\tX}$ of $\omega_{\tX_0/\tX}$. For
  a point $P \in \tX(\Zd_p)$, the torsor $\tX_0\to \tX$ can be pulled
  back to $(\tX_0)_P \to P$, and hence $s_{\tX_0/\tX}$ pulls back to a
  $\tT_{\Zd_p}$-equivariant global section $s_{(\tX_0)_P}$ on
  $\omega_{(\tX_0)_P/\Zd_p}$. But the torsor over $P$ is trivial, and
  $\tT \cong \mathbb{G}_{\mathrm{m}}^r$ with $r = \rank\Pic X$, hence
  there are affine coordinates $(t_1,\dots,t_r)$ for the affine
  $\Zd_p$-scheme $(\tX_0)_P$ with
  $s_{(\tX_0)_P} = \dd t_1/t_1 \wedge \dots \wedge \dd
  t_r/t_r$. Therefore,
  \begin{equation*}
    (\Lambda_p(\chi_p))(P) = \int_{(\tX_0)_P(\Zd_p)} |s_{(\tX_0)_P}|_p =
    \Big(\int_{\Zd_p^\times} \frac{\dd t}{t}\Big)^r =  (1- p^{-1})^r.\qedhere
  \end{equation*}
\end{proof}

\subsection{Comparison to the number of points modulo $p^\ell$}\label{sec:points_mod_p^l}

In this section, we describe $ \mu_p(X(\Qd_p))$ in terms of congruences. In
the special case $Y=\Pd^n_\Qd$, this was worked out in
\cite[Lemma~3.2]{MR1681100}.

Let $p$ be a prime.  For $\ell \in \Zd_{>0}$, using notation
\eqref{eq:primitive_collections}, we have
\begin{equation*}
   \tX_0(\Zd/p^\ell\Zd) = \{\xv \in (\Zd/p^\ell\Zd)^{\Sigma(1)} : \Phi(\xv)=0 \in \Zd/p^\ell \Zd,
  \  p \nmid \gcd\{x_\rho : \rho \in S_j\} \text{ for all } j=1,\dots,r\}
\end{equation*}
as in Proposition~\ref{prop:lift_to_torsor} and define
\begin{equation}\label{eq:c_p}
  c_p  \coloneqq  \lim_{\ell \to \infty} \frac{\# \tX_0(\Zd/p^{\ell}\Zd)}{(p^\ell)^{\#\Sigma(1)-1}} \,\, \text{ and
  } \,\, c_{\mathrm{fin}}  \coloneqq  \prod_p c_p.
\end{equation}
We will see in Proposition~\ref{prop:measure_torsor_mod_p^l} that the
sequence defining $c_p$ becomes stationary;  in particular, the limit $\ell \rightarrow \infty$ exists.  The convergence of
$c_{\mathrm{fin}}$ will follow from
Proposition~\ref{prop:p-adic_density};  see \eqref{eq:tamagawa}.
For $\xv \in  \tX_0(\Zd/p^{\ell}\Zd)$, let
\begin{equation*}
  \tX_0(\Zd_p)_\xv \coloneqq \{\yv \in \tX_0(\Zd_p)\mid \yv \equiv \xv \bmod{p^\ell}\}.
\end{equation*}

\begin{lemma}\label{lem:partial_derivatives}
  There is an $\ell_1 \in \Zd_{>0}$ such that the following holds for
  all $\ell \ge \ell_1$: for any $\xv \in  \tX_0(\Zd/p^{\ell}\Zd)$, there is a nonnegative integer $c_\xv<\ell_1$
  and an $\rho_\xv \in \Sigma(1)$ such that for all $\yv \in \tX_0(\Zd_p)_\xv$ one has
  \begin{equation*}
    \inf_{\rho \in \Sigma(1)}\{v_p(\partial \Phi/\partial x_\rho(\yv))\} = v_p(\partial \Phi/\partial x_{\rho_\xv}(\yv)) = c_\xv.
  \end{equation*}
  
\end{lemma}

\begin{proof}
  Since $X$ is smooth, $X_0$ is also smooth.  Hence for any
  $\yv \in X_0(\Qd_p)$, we have $\partial\Phi/\partial x_\rho(\yv) \ne 0$ for
  some $\rho \in \Sigma(1)$. In particular, for any $\yv \in \tX_0(\Zd_p)$,
  the valuation $v_p(\partial\Phi/\partial x_\rho(\yv))$ is finite for some
  $\rho$. Hence
  $I_p(\yv) \coloneqq \inf_{\rho \in \Sigma(1)}\{v_p(\partial \Phi/\partial
  x_\rho(\yv))\}$ is finite.

  There is an $\ell_1$ such that $I_p(\yv)<\ell_1$ for all
  $\yv \in \tX_0(\Zd_p)$.  To see this, assume the contrary. Then there is a
  sequence $\yv_1,\yv_2,\ldots \in \tX_0(\Zd_p)$ with $I_p(\yv_j)\ge j$ for all
  $j$. The description of $\tX_0(\Zd_p)$ in
  Proposition~\ref{prop:lift_to_torsor} shows that this sequence has an
  accumulation point $\yv_0 \in \tX_0(\Zd_p)$: infinitely many $\yv_i$ have
  the same first $p$-adic digits, infinitely many of these have the same
  second $p$-adic digits, and so on; we obtain $\yv_0$ by using these $p$-adic
  digits; $\Phi(\yv_0)=0$ since $\Phi$ is continuous, and $\yv_0$ satisfies
  the coprimality conditions since these depend only on the first $p$-adic
  digits. Passing to a subsequence, we may assume that $\yv_0$ is the limit of
  the sequence $(\yv_j)_j$. Then
  $\partial \Phi/\partial x_\rho(\yv_0) = \lim_{j \to \infty} \partial
  \Phi/\partial x_\rho(\yv_j) = 0$ for all $\rho \in \Sigma(1)$. This
  contradicts the  smoothness of $X$ over $\Qd_p$.

  Let $\ell \ge \ell_1$ and $\xv \in  \tX_0(\Zd/p^{\ell}\Zd)$. For any
  $\yv \in \tX_0(\Zd_p)_\xv$, the first $\ell$ digits of
  $\partial \Phi/\partial x_\rho(\yv)$ depend only on $\xv$, and since
  $I_p(\yv) < \ell_1 \le \ell$, at least one of these digits is
  nonzero for some $\rho \in \Sigma(1)$. We choose $c_\xv$ and $\rho_\xv$ such that digit number
  $c_\xv$ (\ie the coefficient of $p^{c_\xv}$ in the $p$-adic expansion) of
  $\partial \Phi/\partial x_{\rho_\xv}(\yv)$ is nonzero, while all lower
  digits of $\partial \Phi/\partial x_\rho(\yv)$ for all $\rho \in \Sigma(1)$ are zero.
\end{proof}

\begin{prop}\label{prop:measure_torsor_mod_p^l}
  For every prime $p$, there is an $\ell_0 \in \Zd_{>0}$ such that for all
  $\ell \ge \ell_0$ we have 
  \begin{equation*}
    m_p(\tX_0(\Zd_p)) = \frac{\# \tX_0(\Zd/p^{\ell}\Zd)}{(p^\ell)^{\dim X_0}}.
  \end{equation*}
\end{prop}

\begin{proof}
  Let $\ell_1$ be as in Lemma~\ref{lem:partial_derivatives}. For
  $\xv \in  \tX_0(\Zd/p^{\ell_1}\Zd)$ and $\ell \ge \ell_1$, let
  \begin{equation*}
     \tX_0(\Zd/p^{\ell}\Zd)_\xv \coloneqq \{\yv \in (\Zd/p^{\ell}\Zd)^{\Sigma(1)} \mid \Phi(\yv)=0 \in
    \Zd/p^{\ell}\Zd,\ \yv \equiv \xv \bmod{p^{\ell_1}}\}.
  \end{equation*}
  We will see that
  \begin{equation}\label{eq:measure_eq_number_x}
    m_p(\tX_0(\Zd_p)_\xv) = \frac{\# \tX_0(\Zd/p^{\ell}\Zd)_\xv}{(p^{\ell})^{\#\Sigma(1)-1}}
  \end{equation}
  for all $\ell \ge \ell_1+c_\xv$ with $c_\xv<\ell_1$ as in
  Lemma~\ref{lem:partial_derivatives}. Since $\tX_0(\Zd_p)$ is the disjoint
  union of the sets $\tX_0(\Zd_p)_\xv$ and $ \tX_0(\Zd/p^{\ell}\Zd)$ is the disjoint union of the sets
  $ \tX_0(\Zd/p^{\ell}\Zd)_\xv$ for $\xv \in  \tX_0(\Zd/p^{\ell_1}\Zd)$, our result follows for all
  $\ell \ge \ell_0  \coloneqq  2\ell_1-1$.

  For the proof of \eqref{eq:measure_eq_number_x}, we fix
  $\xv \in  \tX_0(\Zd/p^{\ell_1}\Zd)$ and let $c_\xv, \rho_\xv$ be as in
  Lemma~\ref{lem:partial_derivatives}.  
  We claim that $\Phi(\yv) \bmod{p^{\ell_1+c_\xv}}$ is the same for all
  $\yv \in \Zd_p^{\Sigma(1)}$ with $\yv \equiv \xv \bmod{p^{\ell_1}}$; we write
  $\Phi^*(\xv)$ for this value in $\Zd/p^{\ell_1+c_\xv}\Zd$.  Indeed, for
  $\yv,\yv' \in \Zd_p^{\Sigma(1)}$, we have
  \begin{equation*}
    \Phi(\yv') = \Phi(\yv) + \sum_{\rho \in \Sigma(1)} (y_\rho'-y_\rho)\cdot \partial\Phi/\partial x_\rho(\yv)
    + \sum_{\rho',\rho'' \in \Sigma(1)} \Psi_{\rho',\rho''}(\yv,\yv')(y_{\rho'}'-y_{\rho'})(y_{\rho''}'-y_{\rho''})
  \end{equation*}
  for certain polynomials
  $\Psi_{\rho',\rho''} \in \Zd_p[X_\rho,X'_\rho : \rho \in \Sigma(1)]$ by Taylor expansion.
  If $\yv' \equiv \yv \bmod{p^{\ell_1}}$, we conclude
  $\Phi(\yv') \equiv \Phi(\yv) \bmod{p^{\ell_1+c_\xv}}$.

  If $\Phi^*(\xv) \ne 0 \in \Zd/p^{\ell_1+c_\xv}\Zd$, then there is no
  $\yv \in \Zd_p^{\Sigma(1)}$ with $\yv \equiv \xv \bmod{p^{\ell_1}}$ and
  $\Phi(\yv)=0$, hence the set $\tX_0(\Zd_p)_\xv$ is empty, and the same holds
  for $ \tX_0(\Zd/p^{\ell}\Zd)_\xv$ for all $\ell \ge \ell_1+c_\xv$ for similar reasons.

  Now assume $\Phi^*(\xv) = 0 \in \Zd/p^{\ell_1+c_\xv}\Zd$.  By Hensel's lemma, the
  map $\pi_{\rho_\xv}$ that drops the $\rho_\xv$-coordinate defines an isomorphism
  from the integration domain $\tX_0(\Zd_p)_\xv$ to the set
  \begin{align*}
    &\{(y_\rho)_{\rho \in \Sigma(1) \setminus \{\rho_\xv\}} \in \Zd_p^{\Sigma(1)
      \setminus \{\rho_\xv\}} \mid y_\rho \equiv x_\rho \bmod{p^{\ell_1}}\text{
      for all }\rho \in \Sigma(1) \setminus \{\rho_\xv\}\}\\
    ={}&\{(x_\rho+z_\rho)_{\rho \in \Sigma(1) \setminus \{\rho_\xv\}} \mid z_\rho \in
         p^{\ell_1}\Zd_p\} \cong (p^{\ell_1}\Zd_p)^{\Sigma(1) \setminus \{\rho_\xv\}}
  \end{align*}
  Therefore, by \eqref{eq:measure_torsor_explicit} and the first statement in
  Corollary~\ref{cor:height_torsor_integral}, 
  \begin{equation*}
    m_p(\tX_0(\Zd_p)_\xv) = \int_{\pi_{\rho_\xv}(\tX_0(\Zd_p)_\xv)}
    \frac{\bigwedge_{\rho \in \Sigma(1)\setminus \{\rho_\xv\}} \dd y_\rho}
    {|\partial\Phi/\partial x_{\rho_\xv}(\yv)|_p},
  \end{equation*}
  where $y_{\rho_\xv}$ is expressed in terms of the other coordinates using
  $\pi_{\rho_\xv}^{-1}$.  We have
  $|\partial\Phi/\partial x_{\rho_\xv}(\yv)|_p = p^{-c_\xv}$ on the integration
  domain (Lemma~\ref{lem:partial_derivatives}).  Thus,
  \begin{equation*}
    m_p(\tX_0(\Zd_p)_\xv)
    = \int_{(p^{\ell_1}\Zd_p)^{\Sigma(1) \setminus  \{\rho_\xv\}} }
    \frac{\bigwedge_{\rho \in \Sigma(1) \setminus \{\rho_\xv\}} \dd z_\rho}{p^{-c_\xv}}
    = p^{c_\xv-\ell_1(\#\Sigma(1)-1)}.
  \end{equation*}
  On the other hand, by the discussion above,
  $\Phi^*(\xv)=0 \in \Zd/p^{\ell_1+c_\xv}\Zd$ means $\Phi(\yv)=0 \in \Zd/p^{\ell_1+c_\xv}\Zd$
  for all $\yv \equiv \xv \bmod{p^{\ell_1}}$. Therefore,
  \begin{align*}
    \frac{\# \tX_0(\Zd/p^{\ell_1+c_\xv}\Zd)_\xv}{(p^{\ell_1+c_\xv})^{\#\Sigma(1)-1}} = \frac{p^{c_\xv \#\Sigma(1)}}{(p^{\ell_1+c_\xv})^{\#\Sigma(1)-1}} = p^{c_\xv-\ell_1(\#\Sigma(1)-1)}.
  \end{align*}
  Using Hensel's lemma as before, we see  that
  $\# \tX_0(\Zd/p^{\ell}\Zd)_\xv/(p^\ell)^{\#\Sigma(1)-1}$ has the same value
  for all $\ell \ge \ell_1+c_\xv$. This completes the proof of
  \eqref{eq:measure_eq_number_x}.
\end{proof}

\begin{prop}\label{prop:p-adic_density}
  We have
  \begin{equation*}
   (1 - p^{-1})^{\rank \Pic X}\mu_p(X(\Qd_p)) =  c_p.
  \end{equation*}
\end{prop}

\begin{proof}
  We combine Lemma~\ref{lem:measure_variety_torsor_p-adic} and
  Proposition~\ref{prop:measure_torsor_mod_p^l} with \eqref{eq:c_p}.
\end{proof}

\subsection{The real density}\label{sec:real_density}

In this section, we compute the real density and Peyre's $\alpha$-constant in
terms of quantities that come up naturally in the analytic method in
Sections~\ref{sec8} and \ref{sec9}. For the case $Y = \Pd^n_\Qd$, see \cite[\S
5.4]{MR1340296}.

For any $\sigma \in \Sigmamax$, we can write
\begin{align*}
  -K_X = \sum_{\rho \notin \sigma(1)} \alpha^\sigma_\rho \deg(x_\rho)
\end{align*}
with $\alpha^\sigma_\rho \in \Zd$ by Lemma~\ref{lem:monomials_degree_L}. In
this section, we assume for convenience:
\begin{equation}\label{eq:assumption_real_density_strong}
  \begin{aligned}
    &\text{Every variable $x_\rho$ appears in at most one monomial  of $\Phi$.}\\
    &\text{There are $\sigma \in \Sigmamax$, $\rho_0 \in \sigma(1)$ and $\rho_1 \in \Sigma(1)
      \setminus \sigma(1)$ such that $\alpha^\sigma_{\rho_1} \ne 0$, }\\
    &\text{the variable $x_{\rho_0}$ appears
      with exponent $1$ in $\Phi$, and }\\
    &\text{no $x_\rho$ with $\rho \in \sigma(1) \cup \{\rho_1\} \setminus
      \{\rho_0\}$ appears in the same monomial  of $\Phi$ as $x_{\rho_0}$,}
  \end{aligned}
\end{equation}
This assumption will be satisfied and easy to check in all our applications. 
It implies assumption~\eqref{simplifying} below and hence will allow us to
compare Peyre's real density with $c_\infty$ as in Section~\ref{sec9}.

We fix $\sigma,\rho_0,\rho_1$ as in
\eqref{eq:assumption_real_density_strong}. Let $\sigma(1)' \coloneqq \sigma(1) \cup
\{\rho_1\}$. When we write $\rho \notin \sigma(1)'$, we mean $\rho \in
\Sigma(1) \setminus \sigma(1)'$. Because of $\alpha^\sigma_{\rho_1} \ne 0$
and \eqref{eq:sigma-basis},
$\{\deg(x_\rho) : \rho \notin \sigma(1)'\} \cup \{K_X\}$ is an
$\Rd$-basis of $(\Pic X)_\Rd$. Hence we can define the real numbers
$b_{\rho,\rho'}$ and $b_{\rho'}$ to satisfy
\begin{align*}
  \deg(x_{\rho'}) = - b_{\rho'}K_X -\sum_{\rho \notin \sigma(1)'} b_{\rho,\rho'}\deg(x_\rho)
\end{align*}
for $\rho' \in \sigma(1)'$.

We consider the height matrix
$\Am_1 = (\alpha_{\rho}^{\sigma})_{(\rho, \sigma) \in \Sigma(1) \times
  \Sigmamax} \in \Rd^{\Sigma(1) \times \Sigmamax} = \Bbb{R}^{J \times N}$ as in \eqref{matrix}.  Let $Z_\rho$ for
$\rho \in \Sigma(1)$ be the rows of this matrix. The following shows that our
definition of $b_{\rho,\rho'}$ and $b_{\rho'}$ is consistent with definitions
\eqref{beta} and \eqref{beta0} that will be needed in Section~\ref{sec8}.

\begin{lemma}\label{lem19}
  We have
  \begin{align*}
    Z_{\rho} = \sum_{\rho' \in \sigma(1)'} b_{\rho,\rho'} Z_{\rho'}
    \quad  \text{and} \quad 
    (1,\dots,1) = \sum_{\rho' \in \sigma(1)'} b_{\rho'} Z_{\rho'}
  \end{align*}
  for all $\rho \notin \sigma(1)'$. In particular, with
  \begin{equation}\label{eq:rkPic}
    R = 2+\dim X = J-\rank\Pic X + 1,
  \end{equation}
  the $R$ rows $\{Z_{\rho'} : \rho' \in \sigma(1)'\}$ form a maximal linearly
  independent subset.
\end{lemma}

\begin{proof}
  As in \eqref{matrix}, let
  $\Am_3 = (1,\dots,1) \in \Bbb{R}^{1\times \Sigmamax} = \Rd^{1 \times N}$.
  Let $\{e_\rho : \rho \in \Sigma(1)\} \cup \{e_0\}$ be the standard basis of
  $\Rd^{\Sigma(1)} \times \Rd$. We define $\deg(e_\rho) = \deg(x_\rho)$ for
  $\rho \in \Sigma(1)$ and $\deg(e_0) = K_X$. Consider the sequence of linear
  maps
  \begin{align*}
    \Rd^{\Sigmamax}
    \xlongrightarrow{\begin{pmatrix}\Am_1\\\Am_3\end{pmatrix}}
    \Rd^{\Sigma(1)} \times \Rd
    \xlongrightarrow{\deg} (\Pic X)_\Rd\xlongrightarrow{} 0\text{.}
  \end{align*}
  The second map is surjective, and the image of the first is contained in the
  kernel of the second. Since  we have
  $\rank \Am_1 = \#\Sigma(1) + 1 - \rank \Pic X$ by Lemma~\ref{rank}, this sequence is
  exact. It follows that the dual sequence
  \begin{align*}
    \Rd^{\Sigmamax}
    \xlongleftarrow{\begin{pmatrix}\Am_1^\top & \Am_3^\top\end{pmatrix}}
    \Rd^{\Sigma(1)} \times \Rd
    \xlongleftarrow{\deg^\vee} (\Pic X)_\Rd^\vee\xlongleftarrow{} 0
  \end{align*}
  is exact as well. Let $\{d^\vee_\rho : \rho \notin \sigma(1)'\} \cup \{K_X^\vee\}$ be the
  $\Rd$-basis of $(\Pic X)_\Rd^\vee$ dual to the $\Rd$-basis of $(\Pic X)_\Rd$
  given above. We have
  \begin{align*}
    \deg^\vee(d_\rho^\vee) = e_\rho - \sum_{\rho' \in \sigma(1)'} b_{\rho,\rho'} e_{\rho'}
    \quad  \text{and} \quad     \deg^\vee(K_X^\vee) = e_0 - \sum_{\rho' \in \sigma(1)'} b_{\rho'} e_{\rho'}
  \end{align*}
  for all $\rho \notin \sigma(1)'$. Since these elements lie in the kernel
  of the leftmost map in the dual exact sequence, this gives the required
  relations between the rows of the matrix $\Am_1$ and the row
  $\Am_3$.
\end{proof}

We compare the factor $\alpha(X)$ of Peyre's constant as in
\cite[D\'efinition~2.4]{MR1340296} to
\begin{equation}\label{cast-final-first}
 c^{\ast}   \coloneqq    \text{vol}\Big\{ \textbf{r} \in [0,
 \infty]^{\Sigma(1)\setminus\sigma(1)'}  : b_{\rho'}  -\sum_{\rho \notin \sigma(1)'}
 r_{\rho}  b_{\rho,\rho'}   \geq 0 \text{ for all }  \rho' \in \sigma(1)' \Big\},
\end{equation}
which will appear in \eqref{cast-final}.

\begin{lemma}\label{lem:alpha}
  We have
  \begin{equation*}
    \alpha(X) = \frac{1}{|\alpha^\sigma_{\rho_1}|} c^{\ast}.
  \end{equation*}
\end{lemma}

\begin{proof}
  Let $\vol_{\Zd}$ be the volume on $(\Pic X)_\Rd$ defined by the
  lattice $\Pic X$, and let $\vol_\Rd$ be the volume on $(\Pic X)_\Rd$
  defined by the basis $\{K_X\} \cup \{\deg(x_\rho) : \rho \notin \sigma(1)'\}$. Since
  the determinant of the transformation matrix is $-\alpha^\sigma_{\rho_1}$, we have
  $\vol_\Zd = |\alpha^\sigma_{\rho_1}|\vol_\Rd$. For the corresponding dual volumes
  on $(\Pic X)^\vee_\Rd$, we have
  $\vol^\vee_\Zd =  |\alpha^\sigma_{\rho_1}|^{-1}\vol^\vee_\Rd$.

  Peyre considers the unique $(\rank \Pic X-1)$-form $\vol_{\mathrm{P}}$ on
  $(\Pic X)_\Rd^\vee$ such that
  $\vol_{\mathrm{P}} \wedge K_X = \vol^\vee_\Zd$. We also consider the
  form  $\vol_V = \bigwedge_{\rho \notin \sigma(1)'} \deg(x_\rho)$. Note
  that we have $\vol_V \wedge K_X = \vol_\Rd^\vee$. It follows that we have
  $\vol_{\mathrm P} = |\alpha^\sigma_{\rho_1}|^{-1} \vol_V$. These
  forms can be restricted to volumes on any affine subspace parallel to the
  subspace $V = \{\phi \in (\Pic X)_\Rd^\vee : \langle\phi, K_X\rangle = 0\}$.
  Hence
  \begin{align*}
    \alpha(X)
    &= \vol_P{}\{r \in (\Eff X)^\vee : \langle r, K_X\rangle = -1\}\\
    &= |\alpha^\sigma_{\rho_1}|^{-1}\vol_V{}\{r \in (\Pic X)_\Rd^\vee :
      \langle r, K_X \rangle = -1, \langle r, \deg x_\rho\rangle \ge
      0 \text{ for all $\rho \in \Sigma(1)$}\}\\
    &= |\alpha^\sigma_{\rho_1}|^{-1}\vol_V{}\rleft\{r_0 K_X^\vee + \sum_{\rho
      \notin \sigma(1)'} r_\rho d^\vee_\rho :
      \begin{aligned}
        &r_0 = -1, r_\rho \ge 0\text{ for all }\rho \notin \sigma(1)',\\
        & b_{\rho'}-\textstyle\sum_{\rho \notin \sigma(1)'}
        r_{\rho}b_{\rho,\rho'} \ge 0 \text{ for all } \rho' \in \sigma(1)'
      \end{aligned}
          \rright\}, 
  \end{align*}
  and the claim follows.
\end{proof}

Next we analyze Peyre's real density $\mu_\infty(X(\Rd))$ as given in
Proposition~\ref{prop:local_measure}.  By our assumption
\eqref{eq:assumption_real_density_strong}, the equation $\Phi = 0$ can be
solved for $x_{\rho_0}$ when all $x_\rho$ with $\rho \notin \sigma(1)'$ are
nonzero; here, the implicit function $\phi$ is a rational function in
$\{x_\rho : \rho \in \Sigma(1) \setminus \{\rho_0\}\}$ whose total
$\Pic X$-degree is $\deg(x_{\rho_0})$. Whenever
$S \subseteq \sigma(1)' \setminus \{\rho_0\}$ and $\uv = (u_\rho) \in \Rd^S$,
we write $\phi(\uv,\mathbf{1})$ for
$\phi((x_\rho)_{\rho \in \Sigma(1) \setminus \{\rho_0\}})$ with
$x_\rho=u_\rho$ for $\rho \in S$ and $x_\rho=1$ otherwise; this is a
polynomial expression in $\uv$. Using notation \eqref{eq:height_monomials}, we write
\begin{equation*}
  H_\infty(\xv)  \coloneqq  \max_{\sigma' \in \Sigmamax}|\xv^{D(\sigma')}|
\end{equation*}
for any $\xv \in \Rd^{\Sigma(1)}$.

For the computation of $\mu_\infty(X(\Rd))$, we work with
\eqref{eq:local_measure_explicit} and the chart \eqref{eq:chart} from the subset of
$X^\sigma(\Rd)$ to $\Rd^{\sigma(1) \setminus \{\rho_0\}}$ that drops the
$\rho_0$-coordinate.  Its inverse is induced by the map
\begin{equation*}
  f\colon 
  \Rd^{\sigma(1) \setminus \{\rho_0\}} \to \Rd^{\Sigma(1)},\quad
  \zv =(z_\rho)
  \mapsto (x_\rho)
  \text{ with } x_\rho \coloneqq 
  \begin{cases}
    \phi(\zv,\mathbf{1}), & \rho=\rho_0,\\
    z_\rho, & \rho \in \sigma(1) \setminus \{\rho_0\},\\
    1, & \rho \notin \sigma(1)
  \end{cases}
\end{equation*}
if we interpret the right hand side in Cox coordinates.  Since
$f(\Rd^{\sigma(1) \setminus \{\rho_0\}})$ and $X(\Rd)$ differ by a set of
measure zero, Peyre's real density can be expressed as
\begin{equation}\label{eq:omegaInftyPeyre}
  \omega_\infty  \coloneqq  \mu_\infty(X(\Rd)) = \int_{\zv\in \Rd^{\sigma(1)
      \setminus \{\rho_0\}}} \frac{\dd\zv}{|\partial\Phi/\partial
    x_{\rho_0}(f(\zv))|\cdot H_\infty(f(\zv))}.
\end{equation}

Using the map 
\begin{equation*}
  g\colon 
  \Rd^{\sigma(1)' \setminus \{\rho_0\}} \to \Rd^{\Sigma(1)},\quad
  \tv = (t_\rho) \mapsto (x_\rho)
  \text{ with }x_\rho \coloneqq 
  \begin{cases}
    \phi(\tv,\mathbf{1}),  & \rho=\rho_0,\\
    t_\rho, & \rho \in \sigma(1)' \setminus \{\rho_0\},\\
    1, & \rho \notin \sigma(1)',
  \end{cases}
\end{equation*}
we define
\begin{equation}\label{cinf-first}
  c_\infty  \coloneqq  2^{\#\Sigma(1)-\#\sigma(1)-1} \int_{\tv \in
    \Rd^{\sigma(1)' \setminus \{\rho_0\}},\ H_\infty(g(\tv))
    \le 1} \frac{\dd \tv}{|\partial\Phi/\partial x_{\rho_0}(g(\tv))|},
\end{equation}
which will reappear in \eqref{cinf-final} and \eqref{eq:c_infty}.

To compare $\omega_\infty$ and $c_\infty$, we use the following substitution.

\begin{lemma}\label{lem:transform}
  Let $\Psi$ be a $\Pic X$-homogeneous rational function
  in $\{x_\rho : \rho \in \Sigma(1)\}$ of degree
  \begin{equation*}
    \sum_{\rho \notin \sigma(1)} \alpha^\sigma_{\Psi,\rho}\deg(x_\rho).
  \end{equation*}
  Let $\alpha^\sigma_{\rho',\rho} \in \Zd$ for $\rho' \in \Sigma(1)$ and
  $\rho \notin \sigma(1)$ be as in \eqref{eq:alpha_sigma_rho_rho'}.  Then the
  substitution $z_{\rho'}=t_{\rho_1}^{-\alpha^\sigma_{\rho',\rho_1}}t_{\rho'}$
  for $\rho' \in \sigma(1) \setminus \{\rho_0\}$ gives
  $\Psi(f(\zv))=t_{\rho_1}^{-\alpha^\sigma_{\Psi,\rho_1}}\Psi(g(\tv))$.  In
  particular,
  $\phi(\zv,\mathbf{1})=t_{\rho_1}^{-\alpha^\sigma_{\rho_0,\rho_1}}\phi(\tv,\mathbf{1})$.

  If $t_{\rho_1}$ appears in $\phi(\tv,\mathbf{1})$ with odd exponent, then
  there is another $t_\rho$ with odd exponent in the same monomial or there is
  a $t_\rho$ with odd exponent in each of the other monomials of
  $\phi(\tv,\mathbf{1})$.
\end{lemma}

\begin{proof}
  Consider the case $\Psi=x_\rho$ first. For
  $\rho \in \sigma(1)\setminus \{\rho_0\}$, the claim holds by definition of
  the substitution. For $\rho = \rho_1$, we have
  $\Psi(f(\zv))=1=t_{\rho_1}^{-1}\cdot t_{\rho_1} =
  t_{\rho_1}^{-\alpha^\sigma_{\Psi,\rho_1}}\Psi(g(\tv))$. For
  $\rho \notin \sigma(1)'$, we have
  $\Psi(f(\zv))=1\cdot
  1=t_{\rho_1}^{-\alpha^\sigma_{\Psi,\rho_1}}\Psi(g(\tv))$. Therefore, the
  claim holds for all monomials and hence also for all homogeneous polynomials
  and all homogeneous rational functions in
  $\{x_\rho : \rho \in \Sigma(1) \setminus \{\rho_0\}\}$. In particular, in
  the case $\Psi=x_{\rho_0}$, since $\phi$ is such a rational function of
  degree $\deg(x_{\rho_0})$, the substitution gives
  $\Psi(f(\zv))=\phi(\zv,\mathbf{1}) =
  t_{\rho_1}^{-\alpha^\sigma_{\rho_0,\rho_1}}\phi(\tv,\mathbf{1}) =
  t_{\rho_1}^{-\alpha^\sigma_{\Psi,\rho_1}}\Psi(g(\tv))$. Now the claim
  follows for all monomials, homogeneous polynomials, and finally all
  homogeneous rational functions in $\{x_\rho : \rho \in \Sigma(1)\}$.

  Let $\psi$ be the numerator of $\phi$. Because of
  \eqref{eq:assumption_real_density_strong}, $t_{\rho_1}$ appears in at most
  one monomial of $\psi(\tv,\mathbf{1})$; we assume that it appears in the
  first monomial with odd exponent. Therefore, either the exponent of
  $t_{\rho_1}$ in the first monomial of
  $t_{\rho_1}^{-\alpha^\sigma_{\psi,\rho_1}}\psi(\tv,\mathbf{1})$ is odd, or
  the exponents of $t_{\rho_1}$ in all other monomials of this expression are
  odd. But since our substitution gives
  $\psi(\zv,\mathbf{1})=t_{\rho_1}^{-\alpha^\sigma_{\psi,\rho_1}}\psi(\tv,\mathbf{1})$,
  the exponent of $t_{\rho_1}$ in a certain monomial of
  $t_{\rho_1}^{-\alpha^\sigma_{\psi,\rho_1}}\psi(\tv,\mathbf{1})$ can only be
  odd if there is a $z_\rho$ with odd exponent in the corresponding monomial
  of $\psi(\zv,\mathbf{1})$, and then the exponent of $t_\rho$ in this
  monomial of $\psi(\tv,\mathbf{1})$ is also odd.
\end{proof}

\begin{prop}\label{prop:omega_infty-c_infty}
  We have
  \begin{equation*}
     \mu_\infty(X(\Rd)) = \frac{|\alpha^\sigma_{\rho_1}|}{2^{\rank \Pic X}} c_\infty.
  \end{equation*}
\end{prop}

\begin{proof}
  Our starting point is \eqref{eq:omegaInftyPeyre}. We use the identity (for
  positive real $s$)
  \begin{equation*}
    \frac{1}{s} = \int_{z_{\rho_1} > 0,\ sz_{\rho_1}\le 1} \dd z_{\rho_1}
  \end{equation*}
  to deduce
  \begin{equation*}
    \omega_\infty = \int_{(\zv,z_{\rho_1}) \in \Rd^{\sigma(1) \setminus
        \{\rho_0\}} \times \Rd_{>0},\ H_\infty(f(\zv))\cdot z_{\rho_1} \le 1}
    \frac{\dd\zv\dd z_{\rho_1}}{|\partial\Phi/\partial x_{\rho_0}(f(\zv))|}.
  \end{equation*}
  We use the transformation $z_{\rho_1} = t_{\rho_1}^{\alpha^\sigma_{\rho_1}}$
  (with positive $t_{\rho_1}$) and the transformations from
  Lemma~\ref{lem:transform}. The latter give
  $H_\infty(f(\zv)) = t_{\rho_1}^{-\alpha^\sigma_{\rho_1}}H_\infty(g(\tv))$
  since all monomials appearing in the definition of the anticanonical height
  function $H_\infty$ have degree $-K_X$; therefore,
  $H_\infty(f(\zv))\cdot z_{\rho_1} = H_\infty(g(\tv))$.  Furthermore,
  $|\partial\Phi/\partial x_{\rho_0}(f(\zv))| =
  |t_{\rho_1}^{-\alpha^\sigma_{\partial\Phi/\partial
      x_{\rho_0},\rho_1}}\partial\Phi/\partial x_{\rho_0}(g(\tv))|$ (even without using the
  observation that these are the same constant by
  \eqref{eq:assumption_real_density_strong}).  We obtain
  $\dd z_{\rho_1} = |\alpha^\sigma_{\rho_1}
  t_{\rho_1}^{\alpha^\sigma_{\rho_1}-1}| \dd t_{\rho_1}$ and
  \begin{equation*}
    \dd \zv = |t_{\rho_1}^{-\sum_{\rho' \in \sigma(1) \setminus
        \{\rho_0\}}\alpha^\sigma_{\rho',\rho_1}}|\bigwedge_{\rho' \in \sigma(1)
      \setminus \{\rho_0\}} \dd t_{\rho'}.
  \end{equation*}
  The integration domain is unchanged.

  We have $-K_X=\sum_{\rho' \in \Sigma(1)} \deg(x_{\rho'}) - \deg(\Phi)$ by
  \cite[Proposition~3.3.3.2]{adhl15}, and
  $\deg(\partial\Phi/\partial
  x_{\rho_0})=\deg(\Phi)-\deg(x_{\rho_0})$. Therefore,
  $\alpha^\sigma_{\rho_1}=\sum_{\rho' \in \Sigma(1)}
  \alpha^\sigma_{\rho',\rho_1}-\alpha^\sigma_{\Phi,\rho_1}$ and
  $\alpha^\sigma_{\partial\Phi/\partial
    x_{\rho_0},\rho_1}=\alpha^\sigma_{\Phi,\rho_1}-\alpha^\sigma_{\rho_0,\rho_1}$. Since
  $\alpha^\sigma_{\rho',\rho}=\delta_{\rho'=\rho}$ for all
  $\rho',\rho \notin \sigma(1)$, we conclude that
  $$\alpha^\sigma_{\rho_1}=\sum_{\rho' \in \sigma(1) \setminus
    \{\rho_0\}}\alpha^\sigma_{\rho',\rho_1}+1-\alpha^\sigma_{\partial\Phi/\partial
    x_{\rho_0},\rho_1}.$$ This
  shows that the powers of $t_{\rho_1}$ cancel out, so that
  $\dd\zv\dd z_{\rho_1}/|\partial\Phi/\partial x_{\rho_0}(f(\zv))| = \dd
  \tv/|\partial\Phi/\partial x_{\rho_0}(g(\tv))|$. Therefore,
  \begin{equation*}
    \omega_\infty = |\alpha^\sigma_{\rho_1}|
    \int_{\tv \in \Rd^{\sigma(1) \setminus \{\rho_0\}} \times \Rd_{>0},\
      H_\infty(g(\tv))\le 1}
    \frac{\dd \tv}{|\partial\Phi/\partial x_{\rho_0}(g(\tv))|}.
  \end{equation*}
  
  We claim that 
  \begin{equation*}
    \omega_\infty^-  \coloneqq  |\alpha^\sigma_{\rho_1}|
    \int_{\tv \in \Rd^{\sigma(1) \setminus \{\rho_0\}} \times \Rd_{<0},\
      H_\infty(g(\tv))\le 1}
    \frac{\dd \tv}{|\partial\Phi/\partial x_{\rho_0}(g(\tv))|}
  \end{equation*}
  has the same value as $\omega_\infty$. Indeed, $\phi(\tv,\mathbf{1})$ (the
  $\rho_0$-component of $g(\tv)$) is the only place where the sign of
  $t_{\rho_1}$ might matter. Our claim is clearly true if $t_{\rho_1}$ does
  not appear in $\phi(\tv,\mathbf{1})$ or if $t_{\rho_1}$ has an even exponent
  in $\phi(\tv,\mathbf{1})$. If $t_{\rho_1}$ appears in $\phi(\tv,\mathbf{1})$
  with odd exponent, then the change of variables
  $t_{\rho_1}' \coloneqq -t_{\rho_1}$ and $t_\rho' \coloneqq -t_\rho$ for all
  $t_\rho$ appearing in the final statement of Lemma~\ref{lem:transform} in
  $\omega_\infty^-$ shows that $\omega_\infty^-=\omega_\infty$.  Therefore,
  \begin{equation*}
    \mu_\infty(X(\Rd)) = \omega_\infty = \frac{1}{2}(\omega_\infty+\omega_\infty^-) =
    \frac{|\alpha^\sigma_{\rho_1}|}{2}
    \int_{\tv \in \Rd^{\sigma(1) \setminus \{\rho_0\}} \times \Rd_{\ne 0},\
      H_\infty(g(\tv))\le 1}
    \frac{\dd \tv}{|\partial\Phi/\partial x_{\rho_0}(g(\tv))|}.
  \end{equation*} 
  Since $\rank \Pic X = \#\Sigma(1)-\#\sigma(1)$ and replacing
  $\Rd^{\sigma(1) \setminus \{\rho_0\}} \times \Rd_{\ne 0}$ by
  $\Rd^{\sigma(1)' \setminus \{\rho_0\}}$ does not change the integral, this
  completes the proof.
\end{proof}

\subsection{Peyre's constant in Cox coordinates}\label{sec:peyres_constant}

\begin{prop}\label{prop:peyre}
  Let $X$ be a split almost Fano variety over $\Qd$ with semiample
  $\omega_X^\vee$ that has a finitely generated Cox ring $\Rm(X)$ with
  precisely one relation $\Phi$ with integral coefficients and satisfies the
  assumptions \eqref{eq:toric_smooth} and
  \eqref{eq:assumption_real_density_strong}.  Then Peyre's constant for $X$
  with respect to the anticanonical height $H$ as in
  \eqref{eq:height_definition} is
  \begin{equation*}
    c = \frac{1}{2^{\rank\Pic X}}c^\ast c_\infty c_{\mathrm{fin}},
  \end{equation*}
  using the notation~\eqref{eq:c_p}, \eqref{cast-final-first},
  \eqref{cinf-first}.
\end{prop}

\begin{proof}
  According to \cite[5.1]{MR2019019}, Peyre's constant for $X$ is
  $c = \alpha(X) \beta(X) \tau_H(X)$. Here the cohomological constant is
  \begin{equation*}
    \beta(X)=\#H^1(\Gal(\overline{\Qd}/\Qd), \Pic(X \otimes_{\Qd} \overline{\Qd})) = 1
  \end{equation*}
  since $X$ is split. Recall \eqref{eq:tamagawa} for $\tau_H(X)$. By
  Lemma~\ref{lem:alpha} and Proposition~\ref{prop:omega_infty-c_infty},
  $\alpha(X)\mu_\infty(X(\Rd))= c^\ast c_\infty$. Furthermore, we use
  Proposition~\ref{prop:p-adic_density} for the $p$-adic densities.
\end{proof}

\part{The asymptotic formula}\label{part2}

This part, culminating in Theorem~\ref{analytic-theorem}, is devoted to a
proof of the asymptotic formula \eqref{manin} for the counting problem
described by \eqref{torsor}, \eqref{height} and \eqref{gcd}, subject to
certain conditions to be specified in due course. The nature of our results will be similar to
Proposition~\ref{prop:countingproblem_abstract}, except that we
specialize the general polynomial $\Phi$ to a polynomial of the shape
\eqref{torsor}. In other words, every variable appears in at most one
monomial, and for better readability in comparison with \eqref{Phi-sec3}, we
relabel the variables and their exponents as in \eqref{torsor}.  In the
notation of \eqref{torsor}, we have
$$J = J_0 + J_1 + \dots + J_k$$
variables, where $J_0$ is the number of variables that do not occur in any of
the monomials.  As mentioned in the introduction, the particular shape
\eqref{torsor} is not an atypical situation; it appears sufficiently often in
practice that it deserves special attention. In Section~\ref{sec9}, we will
also show that if the conditions \eqref{torsor}--\eqref{gcd} come from an
algebraic variety satisfying the hypotheses of Proposition~\ref{prop:peyre},
then the leading constant in \eqref{manin} agrees with Peyre's prediction, as
computed in Proposition~\ref{prop:peyre}.

Before we begin, we fix some notation for use in the remainder of the paper.
Vector operations are to be understood componentwise. In particular, just like
the common addition of vectors, for
$\mathbf{x} = (x_1, \ldots, x_n)\in \Bbb{C}^n$,
$\mathbf{y} = (y_1, \ldots, y_n)\in \Bbb{C}^n$, we write
$\mathbf{x} \cdot \mathbf{y} = (x_1y_1, \ldots, x_ny_n) \in \Bbb{C}^n$. If
$\mathbf{x} \in \Bbb{R}^n_{>0}$, $\mathbf{y} \in \Bbb{C}^n$, we write
$\mathbf{x}^{\mathbf{y}} = x_1^{y_1} \cdots x_n^{y_n}$.  We also use this
notation when $\mathbf{x} \in \Bbb{R}^n$ and $\mathbf y \in \mathbb N^n$. We
put $\langle \mathbf x \rangle = x_1x_2\cdots x_n$.  We write
$|\,\cdot\,|_{1}$ for the usual 1-norm, and $|\,\cdot\,|$ denotes the maximum
norm. For $q \in \Bbb{N}$, we write $\mu(q)$ for the M\"obius function of $q$, the Euler totient is denoted $\phi(q)$, and we write
$\underset{a \bmod{q}}{\left.\sum \right.^{\ast}}$ for a sum over reduced
residue classes modulo $q$. The greatest common divisor of non-zero integers $a$, $b$ is denoted by $(a,b)$; confusion with elements of $\mathbb Z^2$ should not arise. The lowest common multiple is $[a,b]$. 
As usual, $e(x) = e^{2 \pi i x}$ for $x\in \Bbb{R}$.  Finally we apply the following convention
concerning the letter $\varepsilon$: whenever $\varepsilon$ occurs in a
statement, it is asserted that the statement is true for any positive real
number $\varepsilon$. Note that this allows implicit constants in Landau or
Vinogradov symbols to depend on $\varepsilon$, and that one may conclude from
$A_1\ll B^\varepsilon$ and $A_2\ll B^\varepsilon$ that one has
$A_1A_2\ll B^\varepsilon$, for example.

\def\eps{\varepsilon}
\section{Diophantine analysis of the torsor}\label{dioph}

In this section and the next, we study the torsor equation \eqref{torsor} with
its variables \emph{restricted to boxes}. For the number of its integral
solutions, we seek an asymptotic expansion whose leading term features a
product of local densities. All estimates are required uniformly relative to
the coefficients $b_1,\ldots,b_k\in {\mathbb Z}\setminus\{0\}$ that occur in
\eqref{torsor}. We assume $k \geq 3$ throughout.

The building blocks of the local densities are Gau\ss\ sums and their
continuous analogues, and we begin by defining the former. Let
$\mathbf h=(h_1,\ldots,h_n)\in\mathbb N^n$ be a ``chain of exponents''. In the following, all implied constants may depend on $\mathbf h$. Then,
for $a\in\mathbb Z$, $q\in\mathbb N$ let
\begin{equation}\label{E1}
  E(q,a;{\mathbf h}) = q^{-n} \sum_{\substack{1\le x_j \le q\\ 1\le j\le n} }
  e\Big(\frac{ax_1^{h_1}x_2^{h_2}\cdots x_n^{h_n}}{q}\Big)
  = q^{-n} \sum_{\substack{1\le x_j \le q\\ 1\le j\le n} } e\Big(\frac{a
    {\mathbf x}^{\mathbf h}}{q}\Big) .
\end{equation}
For a continuous counterpart, let $\mathbf Y \in [\frac12, \infty)^n$, put
${\mathscr Y}=\{\mathbf y\in\mathbb R^n : \frac12 Y_j< |y_j| \le Y_j\;\; (1\le
j\le n)\}$ and define
\begin{equation}\label{E2}
  I(\beta, {\mathbf Y};\mathbf h) = \int_{\mathscr Y} e(\beta
  y_1^{h_1}y_2^{h_2}\cdots y_n^{h_n})\,\mathrm d\mathbf y.
\end{equation} 
This exponential integral satisfies the simple bound
\begin{equation}\label{E3}
  I(\beta, {\mathbf Y};\mathbf h) \ll \langle \mathbf
  Y\rangle (1+ {\mathbf Y}^{\mathbf h}|\beta|)^{-1}.
\end{equation}
Indeed, if $n=1$, then integration by parts yields the bound $O(Y^{1-h}|\beta|^{-1})$, which together with the trivial bound $O(Y)$ confirms 
\eqref{E3}. If
$n>1$, then one uses the obvious relation
$$  I(\beta, {\mathbf Y};\mathbf h) = \int_{\frac12 Y_1\le |y|\le Y_1} I(\beta
y^{h_1}, (Y_2,\ldots, Y_{n}); (h_2,\ldots, h_{n}))\,\mathrm dy $$ together
with induction. With \eqref{E3} in hand for $n-1$ in place of $n$, one infers
\eqref{E3} for $n$ from
$$ I(\beta, {\mathbf Y};\mathbf h) \ll Y_2 Y_3\cdots Y_{n}  \int_{\frac12
  Y_1\le |y|\le Y_1} (1+Y_2^{h_2}\cdots
Y_{n}^{h_{n}}|y^{h_1}\beta|)^{-1}\,\mathrm d y. $$

We now describe the counting problem at the core of this section.  For
$\mathbf{b} \in (\Bbb{Z}\setminus\{0\})^k$ and
$\mathbf{X}=(X_{ij}) \in [1, \infty)^{J}$, let
$\mathscr{N}_{\mathbf{b}}(\mathbf{X})$ denote the number of solutions
$\mathbf{x} \in \Bbb{Z}^J$ to \eqref{torsor} satisfying
$\frac{1}{2}X_{ij} \leq |x_{ij}| \leq X_{ij}$.  Associated with each summand
in \eqref{torsor} are a chain of exponents
$\mathbf h_i=(h_{i1},\ldots,h_{iJ_i})$ and boxing vectors
$\mathbf X_i=(X_{i1},\ldots,X_{iJ_i})$. In the interest of brevity, we now put
\begin{equation}\label{E4} 
  E_i(q,a) = E(q,a; \mathbf h_i), \quad I_i(\beta, \mathbf X)= I(\beta, \mathbf
  X_i; \mathbf h_i) \quad (1\le i \le k).
\end{equation}
The {\em singular integral} for this counting problem is then defined by 
\begin{equation}\label{E5}
  \mathscr{I}_{\mathbf{b}}(\mathbf{X}) = \langle\mathbf
  X_0\rangle\int_{-\infty}^\infty I_1(b_1\beta,\mathbf X) I_2(b_2\beta,\mathbf
  X)\cdots I_k(b_k\beta, \mathbf X)\,\mathrm d\beta,
\end{equation}
and the {\em singular series} is 
\begin{equation}\label{E6}
  {\mathscr E}_{\mathbf b} = \sum_{q=1}^{\infty} \underset{a
    \bmod{q}}{\left.\sum \right.^{\ast}} E_1(q,ab_1)E_2(q,ab_2)\cdots
  E_k(q,ab_k).
\end{equation} 
By \eqref{E3}, the singular integral converges absolutely provided only that
$k\ge 2$. Unfortunately, it is not as easy to determine whether the singular
series converges; this depends on the chains of exponents in a subtle
manner. However, we note that an argument paralleling that in the proof of
\cite[Lemma~2.11]{Va} shows that the sum
\begin{equation} \label{E7} 
 \underset{a \bmod{q}}{\left.\sum \right.^{\ast}} E_1(q,ab_1)E_2(q,ab_2)\cdots
 E_k(q,ab_k)
\end{equation}
is a multiplicative function of $q$. Hence, based on the hypothesis that the
singular series is absolutely convergent, one has the alternative
representation
\begin{equation*} 
  {\mathscr E}_{\mathbf b} = \prod_p \sum_{l=0}^\infty
  \underset{a \bmod{p^l}}{\left.\sum \right.^{\ast}}
  E_1(p^l,ab_1)E_2(p^l,ab_2)\cdots E_k(p^l,ab_k).
\end{equation*}
By orthogonality of additive characters, the partial sums $0 \leq l \leq L$
count congruences modulo $p^L$, and (still under the assumption of absolute
convergence) we can therefore express the singular series as a product of
``local densities'':
\begin{equation}\label{localdensities}
  {\mathscr E}_{\mathbf b} = \prod_p  \lim_{L \rightarrow \infty}
  \frac{1}{p^{L(J_1 + \dots +J_k - 1)}} \#\Big\{(\textbf{x}_1, \ldots,
  \textbf{x}_k)  \bmod{ p^L} :  b_1 \textbf{x}_1^{\textbf{h}_1} + \dots +
  b_k\textbf{x}_k^{\textbf{h}_k} \equiv 0  \bmod{ p^L}\Big\} .
\end{equation}

The transition method to be detailed in Section~\ref{sec8} works with the
proviso that the product
$ {\mathscr E}_{\mathrm b} \mathscr{I}_{\mathbf{b}}(\mathbf{X})$ is a good
approximation to $\mathscr{N}_{\mathbf{b}}(\mathbf{X})$. We detail these
requirements as follows; note that \eqref{zeta1} is \eqref{zeta} specialized
to the equation \eqref{torsor}.

\begin{hyp}\label{H1}
  The singular series $\mathscr{E}_{\mathbf{b}}$
  converges absolutely. There are real numbers $\beta_1,\dots,\beta_k \le 1$
  with
  \begin{equation}\label{E}
    \mathscr{E}_{\mathbf{b}} \ll  |b_1|^{\beta_1} |b_2|^{\beta_2}\cdots |b_k|^{\beta_k}.
  \end{equation}
  Further, there exists $\bm\zeta\in \mathbb R^{k}$ with
  \begin{equation}\label{zeta1}
    \zeta_i > 0  \text{ for all }  1 \leq i \leq k, \quad
    h_{ij} \zeta_i < 1\text{ for all }i, j, \quad
    \sum_{i=1}^k \zeta_i = 1,
  \end{equation}
  and there exist real numbers $0 < \lambda \leq 1$, $\delta_1>0$ and $C\ge 0$ with the property  that
  whenever $\mathbf{X} \in [1, \infty)^{J}$ obeys the condition that
  \begin{equation}\label{samesize}
    \min_{1 \leq i \leq k} \mathbf X_i^{\mathbf h_i} \geq \bigl(\max_{1 \leq i
      \leq k} \mathbf X_i^{\mathbf h_i}\big)^{1-\lambda}, 
  \end{equation}
  then uniformly in $\mathbf b\in(\mathbb Z\setminus\{0\})^k$, one has 
  \begin{equation}\label{errorterm}
    \mathscr{N}_{\mathbf{b}}(\mathbf{X}) - \mathscr{E}_{  \mathbf{b}}
    \mathscr{I}_{\mathbf{b}}(\mathbf{X})
    \ll  |b_1 \cdots b_k|^C (\min_{ij}X_{ij})^{-\delta_1}   \prod_{i=0}^k
    \prod_{j=1}^{J_i} X_{ij}^{1 - h_{ij}\zeta_i + \varepsilon},
  \end{equation}
  wherein we wrote $\zeta_0=h_{0j}=0$ $(1\le j\le J_0)$.
\end{hyp}

In the situation of \eqref{typeT}, Hypothesis~\ref{H1} is in fact a theorem.

\begin{prop}\label{circle-method}
  Suppose that $k=3$, $J_1\ge J_2\ge 2$ and $h_{ij} = 1$ for $i = 1, 2$,
  $1\le j \le J_i$. Then Hypothesis~\ref{H1} is true.
 \end{prop} 
  
We prove this in the next section.
As the proof will show, much more is true. We are free to choose $\bm\zeta$ according to \eqref{zeta1}, and one can specify the parameters $\bm\beta$, $\lambda$ and $C$. In terms of the number $\omega$ defined in \eqref{omegabound} below, one may take
$$ \lambda = 2^{-4-|\mathbf h_3|_1} \omega, \quad   C=300/\omega $$
and
\begin{equation} \label{lambdabeta}  \bm \beta = \Big(\frac{1}{2}(1-\mu)+\varepsilon,
  \frac{1}{2}(1-\mu)+\varepsilon, \mu\Big),
\end{equation}
for any $\varepsilon > 0$, and any $\mu$ with $\varepsilon < \mu < |\mathbf h_3|^{-1}$.   

In the rest of this section we prepare the proof of Proposition \ref{circle-method}
with some 
bounds for the local factors, and we begin with an upper bound for the
singular integral. At the same time, we compare the singular integral with a
truncated version of it. To define the latter, let $Z_0$ be the maximum of the
numbers $\mathbf X_i^{\mathbf h_i}$ $(1\le i\le k)$, and let $Q\ge 1$. Then
put
$$  \mathscr{I}_{\mathbf{b}}(\mathbf{X}, Q) =  \langle\mathbf
  X_0\rangle \int_{-QZ_0^{-1}}^{QZ_0^{-1}}
I_1(b_1\beta,\mathbf X) I_2(b_2\beta,\mathbf X)\cdots I_k(b_k\beta, \mathbf
X)\,\mathrm d\beta. $$

\begin{lemma}\label{singint}
  Let $k\ge 3$, let $\zeta_0=0$, and let $\zeta_i$ $(1\le i\le k)$ be positive real numbers with
  $\zeta_1+\zeta_2+\dots+\zeta_k=1$. Then
  $$ \mathscr{I}_{\mathbf{b}}(\mathbf{X}) \ll |b_1|^{-\zeta_1} \cdots
  |b_k|^{-\zeta_k}  \prod_{i=0}^k \prod_{j=1}^{J_i} X_{ij}^{1 -
    h_{ij}\zeta_i}.$$
  Further, there is a number $\delta>0$ such that whenever $Q\ge 1$ one has
  $$ \mathscr{I}_{\mathbf{b}}(\mathbf{X})-\mathscr{I}_{\mathbf{b}}(\mathbf{X},Q)
  \ll Q^{-\delta}  \prod_{i=0}^k \prod_{j=1}^{J_i} X_{ij}^{1 - h_{ij}\zeta_i}.$$
\end{lemma}

\begin{proof}
  By H\"older's inequality,
  $$ \int_{-\infty}^\infty \prod_{i=1}^k (1+\mathbf X_i^{\mathbf h_i}
  |b_i\beta|)^{-1}\,\mathrm d\beta \le \prod_{i=1}^k \Big(
  \int_{-\infty}^\infty (1+\mathbf X_i^{\mathbf h_i}
  |b_i\beta|)^{-1/\zeta_i}\,\mathrm d\beta\Big)^{\zeta_i}, $$ and by
  \eqref{E5} and \eqref{E3} the first statement in the lemma is immediate. For
  the second, one picks $\iota$ with $Z_0=\mathbf X_\iota^{\mathbf h_\iota}$
  and observes that
  $$   \int_{QZ_0^{-1}}^\infty  (1+\mathbf X_\iota^{\mathbf h_\iota}
  |b_\iota\beta|)^{-1/\zeta_\iota}\,\mathrm d\beta \ll Q^{1-(1/\zeta_\iota)}
  \mathbf X_\iota^{-\mathbf h_\iota}.$$
  If this bound is used within the preceding application of H\"older's
  inequality, one arrives at the second statement in the lemma.
\end{proof}

We continue with some general remarks on Gau\ss\ sums.

\begin{lemma}\label{Gauss}
  Let $\mathbf h\in\mathbb N^n$. Let $b\in\mathbb Z$, $q\in\mathbb N$ and
  $q'=q/(q,b)$, $b'=b/(q,b)$. Then $E(q,b;\mathbf h)= E(q',b';\mathbf h)$.  If
  $n\ge 2$, $h_1=1$ and $(b,q)=1$, then
  $$ E(q,b,\mathbf h) = q^{1-n} \#\{x_2,\ldots,x_n:\, 1\le x_j\le q, \;
  x_2^{h_2}x_3^{h_3}\cdots x_n^{h_n}\equiv 0 \bmod q\}. $$
  Further
  $$ E(q,b, (1,\ldots, 1)) = q^{1-n}\sum_{\substack{ d_j\mid q \\ q\mid
      d_2d_3\cdots d_n}} \varphi\Big(\frac{q}{d_2}\Big)\cdots
  \varphi\Big(\frac{q}{d_n}\Big). $$ In particular,
  $E(q,b, (1,\ldots, 1)) \ll q^{\varepsilon-1}$ and $E(q,b,(1,1))=q^{-1}$.
\end{lemma}

\begin{proof}
  We have $b/q=b'/q'$ whence $e(bx_1^{h_1}\cdots x_n^{h_n}/q)$ has period $q'$
  in all $x_j$.  Summing over all $x_j$ modulo $q$ gives the first statement
  at once. The second statement follows from \eqref{E1} and orthogonality,
  after carrying out the sum over $x_1$. If we specialize the second statement
  to $h_j=1$ for all $j$ and sort the $x_j$ according to the values of
  $d_j=(x_j,q)$, then we arrive at the formula for $E(q,b, (1,\ldots, 1))$,
  from which the remaining claims are immediate.
\end{proof}

\begin{lemma}\label{Gaussaverage}
  Let $\mathbf h\in\mathbb N^n$ with
  $h_1\le h_2\le \dots\le h_n$. Then, for each $b\in\mathbb Z$, the sum
  $$ D(q,b,\mathbf h) =    \underset{a \bmod{q}}{\left.\sum \right.^{\ast}} E(q,ab,\mathbf h) $$
  is multiplicative as a function of $q$, and one has
  $D(q,b,\mathbf h) \ll (q,b)^{1/h_n} q^{1+\varepsilon-1/h_n}$.
\end{lemma}

\begin{proof}
  Within this proof the numbers $h_j$ are fixed. Therefore, we remove $\mathbf h$ from the notation temporarily. Thus $D(q,b)$ abbreviates $D(q,b,\mathbf h)$, for example.

  By \eqref{E7},
  the function $D(q,b)$ is multiplicative in $q$, and we proceed to evaluate it 
  for $q=p^l$ with $p$ prime and
  $l\in\mathbb N$. 
  Let $M_b(q)$ denote the number of $\mathbf x\in (\mathbb Z/ q\mathbb Z)^n$
  with $bx_1^{h_1}\cdots x_n^{h_n} \equiv 0 \bmod q$. Now, first applying Lemma~\ref{Gauss}, and then \eqref{E1} and orthogonality, one confirms the identities  
  $$ D(p^l,b) = \underset{a \bmod{p^l}}{\sum} E(p^l,ab,\mathbf h) - \underset{a
    \bmod{p^{l-1}}}{ \sum  } E(p^{l-1},ab,\mathbf h)= p^{l(1-n)} M_b(p^l) - p^{(l-1)(1-n)} M_b(p^{l-1}). $$
Let $\beta$ be the number with $p^\beta\mid b$ and $p^{\beta+1}\nmid b$.  
  Obviously, if $l\le \beta$, then $M_b(p^l) = p^{ln}$, and the preceding formula gives $D (p^l,b) =\phi(p^l)$. 
  If $l>\beta$, then $M_b(p^l)$ is the number of solutions of $x_1^{h_1}\cdots x_n^{h_n} \equiv 0 \bmod p^{l-\beta}$ with $1\le x_j\le p^l$ $(1\le j\le n)$. Thus $M_b(p^l)= p^{\beta n} M_1(p^{l-\beta})$. We now estimate $M_1(p^\sigma)$. Consider $x_1,\ldots, x_n$ with $p^{\nu_j}\mid x_j$. The congruence $x_1^{h_1}\cdots x_n^{h_n} \bmod p^\sigma$ is equivalent with 
\begin{equation}\label{jb1}  h_1\nu_1 +\dots + h_n\nu_n\ge \sigma.
\end{equation} 
Thus, for a fixed tuple $\nu_1,\ldots \nu_n$, there are at most $p^{n\sigma-\nu_1-\dots-\nu_n}$ solutions counted by $M_1(p^\sigma)$. Further, if \eqref{jb1} holds, then
$$ \nu_1+\dots+\nu_n\ge \frac1{h_n} (h_1\nu_1+ \dots+ h_n\nu_n) \ge \frac{\sigma}{h_n}. $$ Since the number of tuples $\nu_1,\ldots,\nu_n$ that arise here certainly does not exceed $\sigma^n$, we deduce that $M_1(p^\sigma) \le \sigma^n p^{n\sigma - \lceil \sigma/h_n\rceil}$. This implies $M_b(p^l)\le l^n p^{ln - \lceil(l-\beta)/h_n\rceil}$. On inserting this bound in the identity for $D(p^l,b)$, one first confirms the desired estimate for $D(q,b)$ for prime powers $q$, and then for general $q$ by multiplicativity.
\end{proof} 

We now use these results to discuss the singular series that arises in
Proposition~\ref{circle-method}. Then we have $k=3$, $J_1\ge J_2\ge 2$, and we
may use the last clause of Lemma~\ref{Gauss} with $\mathbf h_1$ and
$\mathbf h_2$. Further, we put $h= \max h_{3j}$ and use
Lemma~\ref{Gaussaverage} to confirm that
\begin{equation}\label{innergauss}
  \underset{a \bmod{q}}{\left.\sum \right.^{\ast}} E_1(q,ab_1)E_2(q,ab_2)E_3(q,ab_3)
  \ll q^{\eps-1-1/h} (q,b_1)(q,b_2)(q,b_3)^{1/h}.
\end{equation}
It is now immediate that the singular series converges absolutely.  Further, on using crude bounds of the type $(x,y)\le x^\sigma y^{1-\sigma}$ with $0\le \sigma\le 1$, it follows from \eqref{innergauss} that whenever 
$0<\varepsilon < \mu < 1/h$ one has
 from \eqref{innergauss} that
\begin{align} \label{murks}
    &  \sum_{q=1}^\infty \Big| \underset{a \bmod{q}}{\left.\sum
      \right.^{\ast}} E_1(q,ab_1)E_2(q,ab_2)E_3(q,ab_3)\Big| \ll
    \sum_{q=1}^\infty  q^{\eps-1-\mu} (q,b_1)(q,b_2)b_3^{\mu}\notag \\
    & \ll b_3^{\mu} \sum_{c_1 \mid b_1} \sum_{c_2 \mid b_2} (c_1c_2)^{\varepsilon -\mu}  (c_1, c_2)^{1+\mu - \varepsilon } \leq b_3^{\mu} \sum_{c_1 \mid
      b_1} \sum_{c_2 \mid b_2} (c_1c_2)^{\frac{1}{2} (1 - \mu + \varepsilon)}
    \ll b_3^{\mu}(b_1b_2)^{\frac{1}{2}(1 - \mu) + \varepsilon}.
\end{align}
This establishes all the statements in Proposition~\ref{circle-method} that
concern the singular series, and it also confirms the comment following Proposition~\ref{circle-method} about an admissible choice of $\bm\beta$.

\section{The circle method}

\subsection{Weyl sums}

In this section, we apply the circle method to establish
Proposition~\ref{circle-method}. We prepare the ground with a discussion of
the generalized Weyl sums
\begin{equation*}
  W(\alpha,\mathbf Y;\mathbf h)
  = \sum_{\mathbf y \in \mathbb Z^n \cap \mathscr Y} e(\alpha {\mathbf y}^{\mathbf h}).
\end{equation*}
Here and in the sequel, we continue to use the notation from the previous
section, and in particular, $\mathbf h$, $\mathbf Y$ and $ \mathscr Y$ are as
in \eqref{E2}. The upper bound for the mean square
\begin{equation}\label{W2}
  \int_0^1 |W(\alpha,\mathbf Y;\mathbf h)|^2 \,\mathrm d \alpha
  \ll \langle\mathbf Y\rangle^{1+\varepsilon}
\end{equation}
is pivotal,  and is readily checked:
by orthogonality, the integral in question equals the number of solutions of the 
diophantine equation ${\mathbf x}^{\mathbf h}={\mathbf y}^{\mathbf h}$ with $\mathbf x, \mathbf y \in \mathbb Z^n \cap \mathscr Y$. There are $\langle\mathbf Y\rangle$ choices for $\mathbf x$, and $y_1\cdots y_n$ is a divisor of ${\mathbf x}^{\mathbf h}$, leaving $\langle\mathbf Y\rangle^\varepsilon$ choices for $\mathbf y$, once $\mathbf x$ is chosen. 

The next result is a version of Weyl's inequality.

\begin{lemma}\label{Weyl}
  Let $\alpha\in\mathbb R$, $a\in\mathbb Z$, $q\in\mathbb N$ and
  $|q\alpha -a |\le q^{-1}$. Suppose that $Y_1\ge Y_2\ge \dots\ge Y_n$. Then
  $$ |W(\alpha,\mathbf Y;\mathbf h)|^{2^{|\mathbf h|_1 -n}} \ll \langle
  \mathbf Y\rangle ^{2^{|\mathbf h|_1 -n+\varepsilon}} \Big(\frac1{q} +
  \frac1{Y_n}+\frac{q}{\mathbf Y^{\mathbf h}}\Big).$$
\end{lemma}

\begin{proof}
  If $n=1$ this is the familiar form of Weyl's inequality. If $n\ge 2$, then
  we apply repeated Weyl differencing. Let $h\in\mathbb N$. On combining
  \cite[Lemma 2.4]{Va} with \cite[Exercise 2.8.1]{Va}, one has
  $$ \Big| \sum_{X<x\le 2X} e(\beta x^h)\Big|^{2^{h-1}} \le (2X)^{2^{h-1}-h}
  \sum_{\substack{|u_j|\le X\\ 1\le j < h}} \sum_{x\in I(\mathbf u)}
  e\big(h!\beta u_1u_2\dots u_{h-1}(x + {\textstyle \frac12}|\mathbf
  u|_1)\big)$$ where the $I(\mathbf u)$ are certain subintervals of
  $[X,2X]$. Note here that the sum on the right is real and non-negative. One trivially has
  $$  \Big| \sum_{-2X\leq x< -X} e(\beta x^h)\Big|= \Big| \sum_{X<x\le 2X}
  e(\beta x^h)\Big|, $$
  and hence it follows that
  \begin{equation}\label{wstep}
    \Big| \sum_{X<|x|\le 2X} e(\beta x^h)\Big|^{2^{h-1}} \ll X^{2^{h-1}-h}
    \sum_{\substack{|u_j|\le X\\ 1\le j < h}}
    \sum_{x\in I(\mathbf u)} e\big(h!\beta u_1u_2\dots u_{h-1}
    (x + {\textstyle \frac12}|\mathbf u|_1)\big).
  \end{equation}

  By H\"older's inequality,
  $$ |W(\alpha,\mathbf Y;\mathbf h)|^{2^{h_1 -1}} \le (Y_2\cdots Y_n)^{2^{h_1-1}-1}
  \sum_{\substack{\frac12 Y_\nu< |y_\nu|\le Y_\nu \\ 2\le \nu\le n}}
  \Big|\sum_{\frac12 Y_1<|y_1|\le Y_1} e(\alpha y_1^{h_1}y_2^{h_2}\cdots
  y_n^{h_n})\Big|^{ 2^{h_1 -1}}. $$ We apply \eqref{wstep} with
  $\beta=\alpha y_2^{h_2}\cdots y_n^{h_n}$ to the sum over $y _1$. We write
  $\mathbf h'=(h_2,h_3,\ldots, h_n)$, $\mathbf Y'=(Y_2,Y_3,\ldots, Y_n)$ and
  then find that
  $$  |W(\alpha,\mathbf Y;\mathbf h)|^{2^{h_1 -1}} \ll
  Y_1^{2^{h_1-1}-h_1} \langle \mathbf Y'\rangle^{2^{h_1 -1}-1}
  \sum_{\substack{|u_j|\le Y_1\\ 1\le j < h_1}} \sum_{y\in I_1(\mathbf u)}
  W(h_1! \alpha u_1u_2\cdots u_{h_1-1}(y+ {\textstyle \frac12}|\mathbf u|_1) ,
  \mathbf Y'; \mathbf h') $$ where $I_1(\mathbf u)$ are certain subintervals
  of $[\frac12 Y_1, Y_1]$. Now we apply H\"older's inequality again to bring
  in $|W(\beta,\mathbf Y'; \mathbf h')|^{2^{h_2-1}}$. We may then estimate the
  sum over $y_2$ by \eqref{wstep}. Repeated use of this process produces the
  inequality
  \begin{equation}\label{W5}
    |W(\alpha,\mathbf Y;\mathbf h)|^{2^{h_1-1}\cdots 2^{h_n-1}}
    \ll \langle \mathbf Y\rangle^{2^{h_1+\dots+h_n-n}} {\mathbf Y}^{-\mathbf h}
    \sum_{\mathbf u_1,\ldots, \mathbf u_n} \sum_{\substack{y_\nu\in
        I_\nu(\mathbf u_\nu) \\ 1\le \nu <n}} \Big| \sum_{y_n\in I_n(\mathbf
      u_n)} e(\alpha vy_n)\Big|
  \end{equation}
  in which $\mathbf u_\nu \in \Bbb{Z}^{h_{\nu} - 1}$ runs over integer vectors
  with $|\mathbf u_\nu|\le Y_\nu$ for $1\le \nu\le n$, the
  $I_\nu(\mathbf u_\nu)$ are certain subintervals of $[\frac12 Y_\nu, Y_\nu]$
  and
  $$ v= h_1! h_2!\cdots h_n!\langle \mathbf u_1\rangle\cdots
  \langle \mathbf u_n\rangle y_1y_2\cdots y_{n-1}. $$ Note that $v=0$ will
  occur in \eqref{W5} only when one of the $\mathbf u_\nu$ has a zero entry, so
  that the total contribution to \eqref{W5} from summands with $v=0$ does not
  exceed $ \langle \mathbf Y\rangle^{2^{h_1+\dots+h_n-n}} Y_n^{-1}$, which is
  acceptable. For nonzero $v$, the innermost sum in \eqref{W5} does not exceed
  $\min(Y_n, \|\alpha v\|^{-1})$. Further, we have
  $v\ll \mathbf Y^{\mathbf h}Y_n^{-1}$, and a divisor function estimate shows that
  there are no more than $O(|v|^\varepsilon)$ choices for $\mathbf u_\nu$,
  $y_\nu$ that correspond to the same $v$. This shows that
  $$ |W(\alpha,\mathbf Y;\mathbf h)|^{2^{h_1-1}\cdots 2^{h_n-1}}
  \ll \langle \mathbf Y\rangle^{2^{|\mathbf h|_1-n}} Y_n^{-1} + \langle
  \mathbf Y\rangle^{2^{|\mathbf h|_1-n}+\varepsilon} {\mathbf Y}^{-\mathbf
    h}\!  \sum_{1\le v \ll \mathbf Y^{\mathbf h}Y_n^{-1}}\! \min(Y_n, \|\alpha
  v\|^{-1}). $$ Reference to \cite[Lemma~2.1]{Va} completes the proof.
\end{proof}

We complement this result with an approximate evaluation of $W$.

\begin{lemma}\label{Weylapprox}
  Let $\alpha\in\mathbb R$, $a\in\mathbb Z$, $q\in\mathbb N$ and
  $\alpha = (a/q) +\beta$. Suppose that $Y_1\ge Y_2\ge \dots\ge Y_n$. Then
  $$  W(\alpha,\mathbf Y;\mathbf h)= E(q,a;\mathbf h) I(\beta,\mathbf Y;\mathbf h)
  + O\big(Y_1Y_2\cdots Y_{n-1}q(1+\mathbf Y^{\mathbf h}|\beta|)\big). $$
\end{lemma}

\begin{proof}
  The case $n=1$ is a rough and elementary version of \cite[Theorem
  4.1]{Va}. We now induct on $n$ and suppose that the lemma is already
  available with $n-1$ in place of $n$. As before, we write
  $\mathbf Y'=(Y_2,Y_3,\ldots,Y_n)$ etc., isolate the sum over $y_1$ and
  invoke the induction hypothesis with $\alpha y_1^{h_1}$ for $\alpha$. This
  yields
  \begin{align*}  
    W(\alpha,\mathbf Y;\mathbf h)
    &= \sum_{\frac12 Y_1 <|y_1|\le Y_1}
      \Big( E(q,ay_1^{h_1};\mathbf h')I(\beta y_1^{h_1},\mathbf Y'; \mathbf h')
      + O\big(Y_2\cdots Y_{n-1}q(1+{\mathbf Y'}^{\mathbf h'} |y_1|^{h_1}|\beta|)\big)\Big)
    \\
    & =  \sum_{\frac12 Y_1 <|y_1|\le Y_1}
      E(q,ay_1^{h_1};\mathbf h')I(\beta y_1^{h_1},\mathbf Y'; \mathbf h')
      + O\big(Y_1 Y_2\cdots Y_{n-1}q(1+{\mathbf Y}^{\mathbf h} |\beta|)\big).
  \end{align*}
  In view of \eqref{E1} and \eqref{E2} we may rewrite the sum over $y_1$ on
  the right hand side as
  $$ q^{1-n}\sum_{\substack{1\le x_\nu\le q\\ 2\le \nu\le n}} 
  \int_{{\mathscr Y}'} \sum_{\frac12 Y_1 <|y_1|\le Y_1}
  e\Big(y_1^{h_1}\Big(\beta \mathbf y'^{\mathbf h'} + \frac{a\mathbf
    x'^{\mathbf h'}}{q}\Big)\Big)\,\mathrm d\mathbf y', $$ where $\mathscr Y'$
  is the analogue of $\mathscr Y$ in the coordinates $\mathbf y'$.  We may now
  apply the case $n=1$ with $\beta \mathbf y'^{\mathbf h'}$ for $\beta$ and
  $a\mathbf x'^{\mathbf h'}$ for $a$ to conclude that
  \begin{align*}
    \sum_{\frac12 Y_1 <|y_1|\le Y_1}
    &e\Big(y_1^{h_1}\Big(\beta \mathbf y'^{\mathbf h'} + \frac{a\mathbf
      x'^{\mathbf h'}}{q}\Big)\Big)\\
    & = q^{-1} \sum_{x_1=1}^q e\Big(\frac{ax_1^{h_1} x'^{\mathbf h'}}{q}\Big)
  \int_{\frac12 Y_1<|y_1|\le Y_1} e(\beta y_1^{h_1}\mathbf y'^{\mathbf
      h'})\,\mathrm dy_1
      + O\big(q+q Y_1^{h_1}|y'^{\mathbf h'}\beta|\big).
  \end{align*}
  The induction is now completed by inserting this last formula into the two
  preceding displays.
\end{proof}

\def\om{\omega}
\subsection{Towards the circle method}\label{sec6.2}

We are ready to embark on the proof of Proposition~\ref{circle-method}. We
work in the broader framework of Hypothesis~\ref{H1} in large parts of the
argument, but will restrict to the situation described in
Proposition~\ref{circle-method} whenever the bounds for Gauss sums are entering the argument. We hope that the wider
scope of our presentation will be helpful in related investigations.

We begin with a general remark concerning the ``dummy variables'' $x_{0j}$
that do not occur explicitly in the torsor equation. Suppose that
Hypothesis~\ref{H1} has been established for a given torsor equation, without
any dummy variables, that is, with $J_0=0$. Now consider the same torsor
equation with $J_0\ge 1$ dummy variables. For this new problem, the count
$\mathscr N_{\mathbf b}(\mathbf X)$ factorizes as
$\mathscr N_{\mathbf b}(\mathbf X)=W_0(\mathbf X_0)\mathscr N^*$, say, where
$\mathscr N^*$ is the number of solutions counted by
$\mathscr N_{\mathbf b}(\mathbf X)$ but with the variables $\mathbf x_0$
ignored, and $W_0(\mathbf X_0)$ is the number of
$\mathbf x_0\in\mathbb Z^{J_0}$ with $\frac 12 X_{0j}<|x_{0j}| \le X_{0j}$ for
$1\le j\le J_0$. A trivial lattice point count yields
$$ W_0(\mathbf X_0) = \langle \mathbf X_0\rangle + O(\langle \mathbf X_0\rangle
(\min X_{0j})^{-1}), $$ and if one multiplies this with the asymptotic formula
for $\mathscr N^*$ that we have assumed to be available to us, then one
derives the claims in Hypothesis~\ref{H1} with dummy variables. This shows
that it suffices to address the problem of verifying Hypothesis~\ref{H1} only
in the case where $J_0=0$, and we will assume this for the rest of this
section.
  
To launch the circle method argument, recall the definition of
$\mathscr N_{\mathbf b}(\mathbf X) $ in the paragraph encapsulating displays
\eqref{E4}--\eqref{E6}. In the notation of that section, we define
$$ W_i(\alpha, \mathbf X) = W(\alpha, \mathbf X_i; \mathbf h_i) \quad (1\le i\le k) . $$
By orthogonality,
\begin{equation*}
  \mathscr N_{\mathbf b}(\mathbf X) = 
  \int_0^1 W_1(b_1 \alpha, \mathbf X) \cdots  W_k(b_k \alpha, \mathbf X)\,\mathrm d\alpha .
\end{equation*}
 Our main parameters are
\begin{equation*}
  Z= \min_{1\le i\le k} \mathbf X_i^{\mathbf h_i}, \quad Z_0 
  = \max_{1\le i\le k} \mathbf X_i^{\mathbf h_i}, \quad M = \min_{ij} X_{ij},
\end{equation*}
and we find it convenient to renumber variables to ensure that
\begin{equation}\label{C4}
  X_{i1}\le X_{i2}\le \dots \le X_{iJ_i} \quad (1\le i\le k).
\end{equation}
Once and for all, fix positive numbers $\zeta_i$ as in \eqref{zeta1},  and the number $\om$ defined by
\begin{equation}\label{omegabound}
  \om^{-1} = 40 k \max_{1\le i\le k} J_i |\mathbf h_i|.
\end{equation}
In particular we have $0< \om \le 1/120$. Hence, the intervals $\mathfrak M(q,a)$, defined as the set of $\alpha\in\mathbb R$ with
$|\alpha-(a/q)|\le Z^{\om-1}$,  are disjoint as $a, q$ range over
$1\le a\le q\le Z^\om$, $(a,q)=1$. The union of these intervals we denote by
$\mathfrak M$. Let
$\mathfrak m = [Z^{\om-1}, 1+Z^{\om-1}]\setminus \mathfrak M$. On writing
$$ \mathscr N_{\mathfrak A} = \int_{\mathfrak A} 
W_1(b_1 \alpha, \mathbf X) \cdots  W_k(b_k \alpha, \mathbf X)\,\mathrm d\alpha $$
one has
\begin{equation}\label{C5}
  \mathscr N_{\mathbf b}(\mathbf X)    =\mathscr N_{\mathfrak M} +\mathscr N_{\mathfrak m}.
\end{equation}

The circle method treatment depends on the relative size of $M$ and $Z$. We
first give a proof of Proposition~\ref{circle-method} in the case where
$M\ge Z^{10k\om}$ (the {\em tame} case).

\subsection{The tame case: major arcs}\label{sec6.3}

For $\alpha\in\mathfrak M$, there is a unique pair $a,q$ with
$1\le a\le q\le Z^\om$, $(a,q)=1$ and a number $\beta\in\mathbb R$ with
$|\beta|\le Z^{\om-1}$ and $\alpha =(a/q)+\beta$. By Lemma~\ref{Weylapprox},
\begin{equation}\label{Sapprox}
  W_i(b_i\alpha, \mathbf X) = E_i(q,ab_i)I_i(\beta b_i,\mathbf X_i) +
  O(\langle\mathbf X^\dag_i\rangle q (1+\mathbf X_i^{\mathbf h_i}|b_i\beta|))
\end{equation}
where, temporarily, $\mathbf X^\dag_i=(X_{i2},\ldots,X_{iJ_i})$ is the vector
that is $\mathbf X_i$ with its smallest entry deleted. Since we are in the
tame case, this implies that
$\langle\mathbf X^\dag_i\rangle \le \langle\mathbf X_i\rangle
Z^{-10k\om}$. Further, by hypothesis and \eqref{samesize}, we have
$\mathbf X_i^{\mathbf h_i} \le Z_0\le Z^{1/(1-\lambda)}$.  Now, since $\lambda \le \om/2$, it follows that $(1-  \lambda)^{-1}  \leq 1 + \omega$, and therefore
\begin{equation}\label{lam1}
   \mathbf X_i^{\mathbf h_i}\le Z_0 \le Z^{1+\om}\quad (1\le i\le k).
 \end{equation}    
We shall use these bounds frequently. Here we apply \eqref{lam1} to obtain the estimate
$$W_i(b_i\alpha, \mathbf X) = E_i(q,ab_i)I_i(\beta b_i,\mathbf X_i) +
O(\langle\mathbf X_i\rangle Z^{-9k\om}|b_i|). $$
Noting the trivial bounds
$$ W_i(b_i\alpha, \textbf{X})\ll \langle\mathbf X_i\rangle, \qquad 
E_i(q,ab_i)I_i(\beta b_i,\mathbf X_i) \ll \langle\mathbf X_i\rangle $$ and the
 identity
$$ W_1W_2 \cdots W_k - T_1T_2\cdots T_k = \sum_{i=1}^k
(W_i-T_i)W_1\cdots W_{i-1}T_{i+1}\cdots T_k, $$
we conclude that
$$ \prod_{i=1}^k W_i(b_i\alpha, \mathbf X) =
\prod_{i=1}^k E_i(q,ab_i)I_i(\beta b_i,\mathbf X_i) +O( \langle\mathbf
X_1\rangle\cdots \langle\mathbf X_k\rangle |\mathbf b|_1 Z^{-9k\om}). $$ We
integrate this over $\mathfrak M$. Since the measure of $\mathfrak M$ is
$O(Z^{3\om-1})$, the error will contribute an amount not exceeding
$$ \langle\mathbf X_1\rangle\cdots \langle\mathbf X_k\rangle |\mathbf b|_1 Z^{-8k\om-1}\le
\langle\mathbf X_1\rangle\cdots \langle\mathbf X_k\rangle |\mathbf b|_1
M^{-1/5}Z^{-6k\om-1}\le \langle\mathbf X_1\rangle\cdots \langle\mathbf
X_k\rangle |\mathbf b|_1 M^{-1/5}Z_0^{-1}.  $$ It follows that
\begin{equation}\label{C6}
  {\mathscr N}_{\mathfrak M} = {\mathscr E}_{\mathbf b} (Z^{\om}) {\mathscr
    I}_{\mathbf b}(\mathbf X, Z^\om)+O(\langle\mathbf X_1\rangle\cdots
  \langle\mathbf X_k\rangle |\mathbf b|_1
  M^{-1/5}Z_0^{-1}),
\end{equation}
where
$$ {\mathscr E}_{\mathbf b} (Q) = \sum_{q\le Q}  \underset{a
  \bmod{q}}{\left.\sum \right.^{\ast}} E_1(q,ab_1)E_2(q,ab_2)\cdots
E_k(q,ab_k).$$
Note here that the error estimate in \eqref{C6} is good enough to be absorbed in the error term in \eqref{errorterm}.

We are now required to complete the singular series. At this stage, we have to
be content with the setup in Proposition~\ref{circle-method}, but then have
recourse to \eqref{innergauss}, which provides us with the bound
$$ {\mathscr E}_{\mathbf b} (Z^{\om}) ={\mathscr E}_{\mathbf b} +O(Z^{-\om/(2h)}|b_1b_2b_3|). $$
In combination with Lemma~\ref{singint}, we then infer that there is a number $\delta>0$ with 
$$
{\mathscr E}_{\mathbf b} (Z^{\om}) {\mathscr I}_{\mathbf b}(\mathbf X,Z^\om)
=
{\mathscr E}_{\mathbf b}  {\mathscr I}_{\mathbf b}(\mathbf X)
+O( |b_1b_2b_3|Z^{-\om \delta} \langle\mathbf X_1\rangle \langle\mathbf X_2\rangle \langle\mathbf X_3\rangle
\mathbf X_1^{-\zeta_1\mathbf h_1}\mathbf X_2^{-\zeta_2\mathbf h_2}\mathbf X_3^{-\zeta_3\mathbf h_3}).
$$
It follows that in the tame case, there is indeed a number $\delta_1>0$ such that
\begin{equation}\label{C9}
{\mathscr N}_{\mathfrak M}={\mathscr E}_{\mathbf b}  {\mathscr I}_{\mathbf b}(\mathbf X) + O( |b_1b_2b_3|M^{-\delta_1} \langle\mathbf X_1\rangle\langle\mathbf X_2\rangle \langle\mathbf X_3\rangle
\mathbf X_1^{-\zeta_1\mathbf h_1}\mathbf X_2^{-\zeta_2\mathbf h_2}\mathbf X_3^{-\zeta_3\mathbf h_3}).
\end{equation}

\subsection{The tame case: minor arcs}\label{sec6.4}

In our treatment of the minor arcs, we again work subject to the conditions in
Proposition~\ref{circle-method}. There are two cases.

First suppose that $|b_3|\le Z^{\om/2}$. We apply Weyl's inequality to
$W_3(b_3\alpha, \textbf{X})$. Let
$$ H= 2^{h_{31}+\dots+h_{3J_3}-J_3} .$$
We claim that uniformly for $\alpha\in\mathfrak m$, one has
\begin{equation}\label{C9a}
  W_3(b_3\alpha,\mathbf X) \ll \langle \mathbf X_3\rangle Z^{-\om/(3H)}.
\end{equation}
Indeed, if $Z$ is large and $\alpha\in\mathbb R$ is such that
$|W_3(b_3\alpha,\mathbf X)|\ge \langle \mathbf X_3\rangle Z^{-\om/(3H)}$, then
a familiar coupling of Lemma~\ref{Weyl} with Dirichlet's theorem on
diophantine approximation shows that there are coprime numbers $a$, $q$ with
$|qb_3\alpha -a |\le Z^{\om/2}\mathbf X_3^{-\mathbf h_3} \le Z^{(\om/2)-1}$
and $1\le q\le Z^{\om/2}$. But then $1\le |b_3|q\le Z^{\om}$, and hence
$\alpha$ cannot be in $\mathfrak m$.

By \eqref{W2} and an obvious substitution,
$$ \int_0^1 |W_i(b_i\alpha,\mathbf X)|^2\,\mathrm d\alpha \ll \langle \mathbf
X_i\rangle^{1+\eps}. $$
Hence, by Schwarz's inequality and \eqref{C9a},
$$ \mathscr N_{\mathfrak m} \ll \big( \langle \mathbf X_1\rangle \langle
\mathbf X_2\rangle \big)^{1/2+\eps} \sup_{\alpha\in\mathfrak m}
|W_3(b_3\alpha,\mathbf X)| \ll  \langle \mathbf X_1\rangle \langle \mathbf
X_2\rangle \langle \mathbf X_3\rangle Z^{\varepsilon-1-\om/(3H)}. $$ 
We have $\lambda \le \om/(12H)$, and so
\begin{equation}\label{lam2}
  (1-\lambda) (1 + \omega/(3H)) \geq 1 + \omega/(6H).
\end{equation}
Hence $Z^{-1-(\om/3H)}\ll Z_0^{-1-\om/(6H)}$, which shows that
$\mathscr N_{\mathfrak m}$ is an acceptable error in
Proposition~\ref{circle-method}. This combines with \eqref{C5} to complete the
proof of Proposition~\ref{circle-method} in the case under consideration.

Next consider the case where $|b_3|>Z^{\om/2}$. Here the claim in
Proposition~\ref{circle-method} reduces to a trivial upper bound, as we now explain. The triangle inequality give $|W_i(\alpha)|\le \langle \mathbf X_i\rangle$, and therefore, the integral representation of $\mathscr N_{\mathbf b}(\mathbf X)$
gives 
$\mathscr N_{\mathbf b}(\mathbf X) \le \langle \mathbf X_1\rangle\langle \mathbf X_2\rangle\langle \mathbf X_3\rangle $. Similarly, on combing \eqref{murks} with Lemma \ref{singint}, we have the crude bound 
$$ \mathscr{E}_{  \mathbf{b}}
    \mathscr{I}_{\mathbf{b}}(\mathbf{X}) \ll |b_1b_2b_3|^{1/2} \langle \mathbf X_1\rangle\langle \mathbf X_2\rangle\langle \mathbf X_3\rangle. $$
We take $C=300/\om$ in \eqref{errorterm}. Then $|b_3|^C\ge Z^{150}$, and so
$$     |b_1b_2b_3|^{1/2} \langle \mathbf X_1\rangle\langle \mathbf X_2\rangle\langle \mathbf X_3\rangle \le |b_1b_2b_3|^C  Z_0^{-2} $$
which is more than is required to confirm \eqref{errorterm} in this final case.
It should be noted that the discussion of the case $|b_3|>Z^{\om/2}$ did not use that we are in the tame case, but applies in general. Also, we have now completed the proof of Proposition~\ref{circle-method} in the tame case.

\subsection{Major arcs again}

It remains to deal with the case where $M<Z^{10k\om}$. We assume this
inequality from now on. Again, we work in the broader framework of
Sections~\ref{sec6.2} and~\ref{sec6.3}, and refine the circle method approach
to cover the current situation as well. We say that a variable $x_{ij}$ is
small if $X_{ij}<Z^{10k\om}$. By hypothesis, there is at least one small
variable. Also, by \eqref{C4}, there is a number $J'_i$ such that the $x_{ij}$
with $j\le J'_i$ are small, and those with $j>J'_i$ are not. We proceed to show that
\begin{equation} \label{smally} \prod_{j\le J'_i} X_{ij}  \le \langle \mathbf X_i\rangle^{1/4}. 
\end{equation}
To see this, note that the definition of $J'_i$ gives 
\begin{equation}
\label{smallvar}    \prod_{j\le J'_i} X_{ij}  \le Z^{10k\om J'_i} \le Z^{10k\om J_i}.
\end{equation}
But $Z\le \mathbf X_i^{\mathbf h_i} \le \langle  \mathbf X_i\rangle^{|\mathbf h_i|}$. We insert this in the previous display and apply the inequality 
$$ 10 k\om J_i |\mathbf h_i|\le \frac14 $$
(which is immediate from \eqref{omegabound}) to derive \eqref{smally}.

The significance of \eqref{smally} is that it implies that
for each $i$, there are variables $x_{ij}$ that are not small. This is
important throughout this section. We put
$$ \mathbf X'_i = (X_{i1},\ldots,X_{iJ'_i}), \quad
\mathbf X''_i = (X_{i,J'_i+1},\ldots,X_{iJ_i}), \quad \mathbf X_i= (\mathbf
X'_i,\mathbf X''_i), $$ where $\mathbf X'_i$ is void if $x_{i1}$ is not
small. In the same way, we dissect the variable
$\mathbf x_i = (\mathbf x'_i,\mathbf x''_i)$ and the chain of exponents
$\mathbf h_i = (\mathbf h'_i,\mathbf h''_i)$.  By orthogonality, we then have
\begin{equation}\label{C10}
  \mathscr N_{\mathbf b}(\mathbf X) = \sum_{(\mathbf x'_1,\ldots,\mathbf x'_k)
    \in \mathscr{Y}'\cap \mathbb Z^{J'}}\int_0^1
  W(b_1\alpha {\mathbf x'_1}^{\mathbf h'_1}, \mathbf X''_1; \mathbf h''_1)
  \dots
  W(b_k\alpha {\mathbf x'_k}^{\mathbf h'_k}, \mathbf X''_k; \mathbf h''_k)
  \,\mathrm d\alpha,
\end{equation}
where $J'=J'_1+\cdots+J'_k$ and 
\begin{equation}\label{C11}
  \mathscr{Y}'  \coloneqq  \{ \textbf{x}' \in \Bbb{R}^{J'} :
  {\textstyle\frac12 } X_{ij} <|x_{ij}| \le X_{ij} \text{ for } 1\le i\le k,  1\le j\le J'_i\}.
\end{equation}

We apply the circle method to the integral in \eqref{C10}. By
Lemma~\ref{Weylapprox}, when $\alpha = (a/q)+\beta$, one finds that subject to
\eqref{C11}, one has
\begin{equation*}
  W\big(b_i\alpha {\mathbf x'_i}^{\mathbf h'_i}, \mathbf X''_i; \mathbf h''_i\big)=
  E\big(q,ab_i {\mathbf x'_i}^{\mathbf h'_i};  \mathbf h''_i\big)
  I\big(\beta b_i  {\mathbf x'_i}^{\mathbf h'_i},  \mathbf X''_i; \mathbf h''_i\big)
  +O\big(\langle  \mathbf X''_i\rangle Z^{-10k\om } q(1+|b_i \beta|  \mathbf
    X_i^{\mathbf h_i})\big).
\end{equation*}
Here it is worth recalling that $ \mathbf X''_i$ is not void and has all its
components at least as large as $Z^{10k\om}$. We now apply \eqref{lam1} to confirm that for $\alpha \in \mathfrak M$,
the error in the preceding display does not exceed
$$
\langle \mathbf X''_i\rangle Z^{\om-10k\om } + 
\langle \mathbf X''_i\rangle Z^{-10k\om }|b_i| Z^{2\om-1} \mathbf X_i^{\mathbf
  h_i} \le \langle \mathbf X''_i\rangle |b_i|Z^{3\om -10k\om } \le \langle
\mathbf X''_i\rangle |b_i|Z^{-9k\om }.$$ Let ${\sf S}$ denote the integrand in
\eqref{C10}, and let ${\sf M}$ denote the product of the expressions
\begin{equation*}
 E(q,ab_i {\mathbf x'_i}^{\mathbf h'_i}, \mathbf h''_i) I(\beta b_i {\mathbf
  x'_i}^{\mathbf h'_i}, \mathbf X''_i; \mathbf h''_i), 
\end{equation*}
with $1\le i\le
k$. Then, following the discussion in the initial part of Section~\ref{sec6.3}, we obtain
\begin{equation}\label{C11a}
  {\sf S}- {\sf M}\ll \langle  \mathbf X''_1\rangle\cdots \langle  \mathbf
  X''_k\rangle |\mathbf b|_1 Z^{-9k\om}.
\end{equation}
We integrate over $\mathfrak M$ and sum over the integral points in ${\mathscr Y}'$. Then, again
as in Section~\ref{sec6.3}, this gives
\begin{equation}\label{C12}
  \mathscr N_{\mathbf b}(\mathbf X) =  \sum_{ (\mathbf x'_1,\ldots,\mathbf
    x'_k) \in  \mathscr{Y}'\cap\mathbb Z^{J'}} \mathscr E' \mathscr I' + \mathscr N^\dag +
  O(\langle  \mathbf X_1\rangle\cdots \langle  \mathbf X_k\rangle |\mathbf b|_1
  Z^{-8k\om-1}),
\end{equation}
where
\begin{align}
  \mathscr E'
  & = \sum_{q\le Z^\om}\underset{a \bmod{q}}{\left.\sum \right.^{\ast}}
    E(q,ab_1 {\mathbf x'_1}^{\mathbf h'_1},  \mathbf h''_1)\cdots
    E(q,ab_k {\mathbf x'_k}^{\mathbf h'_k},  \mathbf h''_k)\label{C13}, \\
  \mathscr I'
  &= \int_{-Z^{\om-1}}^{Z^{\om-1}}
    I(\beta b_1  {\mathbf x'_1}^{\mathbf h'_1},  \mathbf X''_1; \mathbf h''_1)\cdots
    I(\beta b_k  {\mathbf x'_k}^{\mathbf h'_k},  \mathbf X''_k; \mathbf h''_k)
    \,\mathrm d\beta, \nonumber
\end{align}
and where $\mathscr N^\dag$ is the same expression as in \eqref{C10} but with
integration over the minor arcs $\mathfrak m$. Exchanging the sum with the
integral in \eqref{C10}, we see that
$\mathscr N^\dag = \mathscr N_{\mathfrak m}$.  Note that the error in
\eqref{C12} also occurred in Section~\ref{sec6.3} and, in the display preceding \eqref{C6}, was shown to be of
acceptable size.

The difficulty now is that the moduli $q$ in \eqref{C13} are too large for the
small variables to be arranged in residue classes modulo $q$. We therefore
prune the sum over $q$. In preparation for this manoeuvre, we bound
$\mathscr I'$ uniformly in $\mathbf x'_i$. Whenever
$\mathbf x'_i \in \mathscr{Y}'$, one finds from \eqref{E3} that
\begin{equation*}
  I(\beta b_i {\mathbf x'_i}^{\mathbf h'_i}, \mathbf X''_i ; \mathbf h''_i) \ll
   \langle\mathbf X''_i\rangle (1+ {\mathbf X''_i}^{\mathbf h''_i}| {\mathbf x'_i}^{\mathbf h'_i}
   b_i\beta|)^{-1}
   \ll \langle \mathbf X''_i \rangle(1+ \mathbf X_i^{\mathbf h_i}
   |b_i\beta|)^{-1}.
 \end{equation*}
 Hence, by H\"older's inequality, 
 \begin{equation}\label{C15}
   \mathscr I' \ll \prod_{i=1}^k \langle   \mathbf X''_i\rangle
   \Big(\int_{-\infty}^\infty (1+ \mathbf X_i^{\mathbf h_i}
   |b_i\beta|)^{-1/\zeta_i}\,\mathrm d\beta\Big)^{\zeta_i}
   \ll \prod_{i=1}^k \langle   \mathbf X''_i\rangle  \mathbf X_i^{-\zeta_i \mathbf h_i}.
   \end{equation}
   Now let $\mathscr E^\dag$ be the portion of the sum defining $\mathscr E$
   where $q\le M^{1/8}$, and let $\mathscr E^\ddag$ be the portion with
   $M^{1/8}<q\le Z^\om$. Then $\mathscr E'= \mathscr E^\dag+\mathscr E^\ddag$,
   and \eqref{C13} and \eqref{C15} yield
\begin{equation}\label{C16}
  \sum_{ (\mathbf x'_1,\ldots,\mathbf x'_k) \in  \mathscr{Y}'} \mathscr E^\ddag \mathscr I'
  \ll \Big(\prod_{i=1}^k \langle   \mathbf X''_i\rangle  \mathbf X_i^{-\zeta_i
    \mathbf h_i}\Big) \sum_{M^{1/8}<q<Z^{\om}}
  \sum_{ (\mathbf x'_1,\ldots,\mathbf x'_k)\in  \mathscr{Y}' }  
  \Big|\underset{a \bmod{q}}{\left.\sum \right.^{\ast}}
  \prod_{i=1}^k E(q,ab_i {\mathbf x'_i}^{\mathbf h'_i}; \mathbf h''_i)\Big|.
\end{equation}

At this point, we require a workable upper bound for the innermost sum. In the
situation of Proposition~\ref{circle-method}, we have $k=3$, and such a bound is provided by
\eqref{innergauss}. With $h=\max h_{3j}$, this yields
\begin{equation}\label{C17}
\underset{a \bmod{q}}{\left.\sum \right.^{\ast}}
\prod_{i=1}^3 E(q,ab_i {\mathbf x'_i}^{\mathbf h'_i}; \mathbf h''_i)
\ll \frac{(q,b_1 \langle \mathbf x'_1\rangle) (q,b_2 \langle \mathbf
  x'_2\rangle)(q,b_3  {\mathbf x'_3}^{\mathbf h'_3})^{1/h} }{q^{1+1/h}}.
\end{equation}
Now
$ (q,b_1 \langle \mathbf x'_1\rangle) \le |b_1| (q,x_{11})\cdots(q,x_{1J'_1})
$ and likewise for $(q,b_2 \langle \mathbf x'_2\rangle)$.  Similarly,
$$ (q,b_3  {\mathbf x'_1}^{\mathbf h'_3})^{1/h}\le
 |b_3| (q,x_{31}^{h_{31}})^{1/h}\cdots(q,x_{3J'_3}^{h_{3J'_3}})^{1/h}
\le  |b_3| (q,x_{31})\cdots(q,x_{3J'_3}). $$
We may sum \eqref{C17} over   $\mathbf x'_i \in \mathscr{Y}'$, using the simple bound
$$ \sum_{x\le X} (q,x) \ll q^\varepsilon X. $$
It then follows that the right hand side of \eqref{C16} does not exceed
\begin{equation}\label{C17a}
  \ll  \Big(\prod_{i=1}^3 |b_i|\langle   \mathbf X'_i\rangle  \langle
  \mathbf X''_i\rangle  \mathbf X_i^{-\zeta_i \mathbf h_i}\Big)
  \sum_{M^{1/8}<q<Z^{\om}} q^{\varepsilon-1-1/h}
  \ll  M^{-1/(9h)} |b_1b_2b_3| \prod_{i=1}^3   \langle   \mathbf X_i\rangle
  \mathbf X_i^{-\zeta_i \mathbf h_i}.
\end{equation}
In the specific situation of Proposition~\ref{circle-method}, this is an
acceptable error term.

We now turn to the product $\mathscr E^\dag \mathscr I'$. Here we prune the
range of integration. Let
$$ \mathscr I^\dag = \int_{-M^{1/8}Z_0^{-1}}^{M^{1/8}Z_0^{-1}}
I(\beta b_1 \mathbf {x'}_1^{\mathbf h'_1}, \mathbf X''_1; \mathbf h''_1)\cdots
I(\beta b_k \mathbf {x'}_k^{\mathbf h'_k}, \mathbf X''_k; \mathbf
h''_k)\,\mathrm d\beta, $$ and let $\mathscr I^\ddag$ be the complementary
integral over $ M^{1/8}Z_0^{-1}< |\beta|\le Z^{\om-1}$ so that
$\mathscr I'=\mathscr I^\dag+\mathscr I^\ddag$. To obtain an upper bound for
$\mathscr I^\ddag$, choose an index $\iota$ with
$Z_0 = \mathbf X_\iota^{\mathbf h_\iota}$. Then
$$ \int_{M^{1/8}Z_0^{-1}}^\infty (1+\mathbf X_\iota^{\mathbf
  h_\iota}|b_\iota\beta|)^{-1/\zeta_\iota}\, \mathrm d\beta \ll \mathbf
X_\iota^{-\mathbf h_\iota} M^{(\zeta_\iota-1)/8}, $$ and since
$\zeta_\iota<1$, we observe that the exponent of $M$ is negative. With this
adjustment, the argument in \eqref{C15} shows that uniformly for
$\mathbf x'_i \in \mathscr{Y}'$ one has
\begin{equation}\label{C18}
  \mathscr I^\ddag \ll M^{(\zeta_\iota-1)\zeta_\iota/8} \prod_{i=1}^k \langle
  \mathbf X''_i\rangle  \mathbf X_i^{-\zeta_i \mathbf h_i}.
\end{equation}
We can now imitate the argument from \eqref{C16}--\eqref{C17a}, this time
applying \eqref{C18} and summing over $q\le M^{1/8}$. In the cases covered by
Proposition~\ref{circle-method}, this yields
$$ \sum_{(\mathbf x'_1,\ldots,\mathbf x'_3) \in  \mathscr{Y}'} \mathscr E^\dag \mathscr I^\ddag
\ll M^{(\zeta_\iota-1)\zeta_\iota/9} |b_1b_2b_3| \prod_{i=1}^3 \langle \mathbf X_i\rangle
\mathbf X_i^{-\zeta_i \mathbf h_i}, $$ which can be absorbed in the error term
when $\delta_1 <  \frac19 \min (1-\zeta_i)\zeta_i$. On collecting
together, we deduce from \eqref{C12} and the discussion above that
\begin{equation}\label{C19}
  \mathscr N_{\mathbf b}(\mathbf X) = \sum_{(\mathbf x'_1,\ldots,\mathbf x'_k)
    \in  \mathscr{Y}'} \mathscr E^\dag \mathscr I^\dag + \mathscr N_{\mathfrak
    m} +O(F),
\end{equation}
where $F$ is an acceptable error provided that $C>1$ and  $\delta_1 $ is small enough.

It would now be possible to exchange the sums over $\mathbf x'_i$ with the
summations present in the definition of $\mathscr E^\dag$, and to evaluate
these sums by arranging the $x_{ij}$ in arithmetic progressions, as suggested
earlier. However, we prefer an indirect argument that is technically
simpler. Let $\mathfrak N$ denote the union of the pairwise disjoint intervals
$|\alpha-(a/q)|\le M^{1/8}Z_0^{-1}$ with $1\le a\le q\le M^{1/8}$ and
$(a,q)=1$. Observe that $\mathfrak N \subset \mathfrak M$. Hence, integrating
\eqref{C11a} over $\mathfrak N$ we find that
\begin{equation}\label{C20}
  \sum_{ (\mathbf x'_1,\ldots,\mathbf x'_k) \in  \mathscr{Y}'}\int_{\mathfrak N} 
  W(b_1\alpha {\mathbf x'_1}^{\mathbf h'_1}, \mathbf X''_1; \mathbf h''_1)
  \cdots
  W(b_k\alpha {\mathbf x'_k}^{\mathbf h'_k}, \mathbf X''_k; \mathbf h''_k)
  \,\mathrm d\alpha
  =
  \sum_{ (\mathbf x'_1,\ldots,\mathbf x'_k) \in  \mathscr{Y}'} \mathscr E^\dag
  \mathscr I^\dag + O(F')
\end{equation}
where $F'$ is an error that certainly does not exceed the error present in
\eqref{C12} because the measure of $\mathfrak N$ is smaller than that of
$\mathfrak M$. Exchanging sum and integral, it transpires that the left hand
side of \eqref{C20} is simply the major arc contribution
$\mathscr N_{\mathfrak N}$.  To evaluate the latter, we can run an argument
from Section~\ref{sec6.3} with $\mathfrak N$ in place of $\mathfrak M$. The
bound \eqref{Sapprox} becomes
$$
W_i(b_i\alpha, \mathbf X) = E_i(q,ab_i)I_i(\beta b_i,\mathbf X_i) +
O(\langle\mathbf X_i\rangle M^{-3/4} |b_i\beta|),
$$ 
and then the result in \eqref{C6} changes to 
$$
{\mathscr N}_{\mathfrak N} = {\mathscr E}_{\mathbf b} (M^{1/8}) {\mathscr
  I}_{\mathbf b}(\mathbf X, M^{1/8})+O(\langle\mathbf X_1\rangle\cdots
\langle\mathbf X_k\rangle |\mathbf b|_1 M^{-3/8}Z_0^{-1}).
$$
We can now complete the singular series and the singular integral as in
Section~\ref{sec6.3}.  The argument that produced \eqref{C9} now delivers
exactly the same asymptotics for $\mathscr N_{\mathfrak N}$. Via \eqref{C19}
and \eqref{C20}, it follows that
$ \mathscr N_{\mathbf b}(\mathbf X) = {\mathscr E}_{\mathbf b} {\mathscr
  I}_{\mathbf b}(\mathbf X) + \mathscr N_{\mathfrak m} + O(F'') $ where $F''$
is an error acceptable to Hypothesis~\ref{H1}. Consequently, it remains to
estimate the contribution from the minor arcs.

\subsection{Minor arcs again}

The argument of Section~\ref{sec6.4} yields an acceptable bound for
$\mathscr N_{\mathfrak m}$ provided that the estimate \eqref{C9a} remains
valid in cases that are not tame. Hence we now complete the proof of
Proposition~\ref{circle-method} by showing that indeed \eqref{C9a} holds in
the wider context, uniformly for $\alpha\in\mathfrak m$ and
$1\le |b_3|\le Z^{\om/2}$. In doing so, we may suppose that $x_{31}$ is small,
for otherwise our previous argument leading to \eqref{C9a} still applies. We
write
$$ T(\alpha,\mathbf x'_3) =W(b_3\alpha {\mathbf x'_3}^{\mathbf h'_3},\mathbf X''_3; \mathbf h''_3). $$
Then
$$ W_3(b_3\alpha, \mathbf X) = \sum_{\mathbf x'_3} T(\alpha,\mathbf x'_3), $$
with the sum extending over $\frac{1}{2} X_{3j} \leq |x_{3j}| \leq X_{3j}$ $(1\le j\le J'_3)$.

We apply Weyl's inequality to $ T(\alpha,\mathbf x'_3)$. Let
$K=2^{|\mathbf h''_3|_1 - J_3+J'_3}$, and note that all entries in
$\mathbf X''_3$ are at least as large as $Z^\om$. Hence, whenever the real
number $\gamma$ and $c\in\mathbb Z$ and $t\in\mathbb N$ are such that
$|t\gamma-c|\le t^{-1}$, then by Lemma~\ref{Weyl}, one has
\begin{equation}\label{C30}
  |W(\gamma, \mathbf X''_3;\mathbf h''_3)|^K \ll \langle\mathbf
  X''_3\rangle^{K+\varepsilon} \Big(\frac1{t} + \frac1{Z^\om} +
  \frac{t}{{\mathbf X''_3}^{\mathbf h''_3}}\Big). 
\end{equation}
By Dirichlet's theorem on diophantine approximation, there are $c$ and $t$
with $t\le Z^{-\om} {\mathbf X''_3}^{\mathbf h''_3}$ and
$|t\gamma-c|\le Z^\om {\mathbf X''_3}^{-\mathbf h''_3} $. Then, on applying a
familiar transference principle (see \cite[Exercise 2.8.2]{Va}) to
\eqref{C30}, we find that
$$ |W(\gamma, \mathbf X''_3;\mathbf h''_3)|^K \ll 
\langle\mathbf X''_3\rangle^{K+\varepsilon} \Big(\frac1{Z^\om} + \frac{1}{t+
  {\mathbf X''_3}^{\mathbf h''_3}|t\gamma-c|}\Big). $$ Since there is a variable that is not small, we have
$K<H$, and hence that $K\le H/2$. Consequently, for a given $\mathbf x'_3$, we either have
$T(\alpha,\mathbf x'_3) \ll \langle\mathbf X''_3\rangle Z^{-\om/(3H)}$ or
there are $t=t(\mathbf x'_3)$ and $c=c( \mathbf x'_3)$ with $t\le Z^{\om/3}$
and
\begin{equation}\label{C31}
  \Big|b_3\alpha {\mathbf x'_3}^{\mathbf h'_3} - \frac{c}{t}\Big| \le
  \frac{Z^{\om/3}}{t{\mathbf X''_3}^{\mathbf h''_3}}.
\end{equation}
Let $\mathscr X$ be the set of all $\mathbf x'_3$ where the latter case
occurs.  Then
\begin{equation}\label{C32}
 W_3(b_3\alpha, \mathbf X) \ll \langle\mathbf X_3\rangle Z^{-\om/(3H)}
+ \langle\mathbf X''_3\rangle \sum_{\mathbf x'_3 \in\mathscr X}
\big(t+  {\mathbf X''_3}^{\mathbf h''_3}|tb_3\alpha {\mathbf x'_3}^{\mathbf h'_3}-c|\big)^{-1/H}.
\end{equation}

We write $Q={\mathbf X'_3}^{\mathbf h'_3}Z^\om$ and apply Dirichlet's theorem
again to find coprime numbers $a$, $q$ with $1\le q \le Q$ and
$|qb_3\alpha -a|\le Q^{-1}$. On comparing this approximation to $b_3\alpha$
with that given by \eqref{C31}, we find that whenever
$\mathbf x'_3\in\mathscr X$, then
\begin{equation}\label{C41} |at {\mathbf x'_3}^{\mathbf h'_3} -cq | \le QZ^{\om/3} {\mathbf
  X''_3}^{-\mathbf h''_3} + Q^{-1}t {\mathbf X'_3}^{\mathbf h'_3} .
  \end{equation}
But $t\le Z^{\om/3}$, and therefore, the second summand on the right does not
exceed $Z^{-\om/2}$. For the first summand, we 
note that
\begin{equation}
\label{C42} QZ^{\om/3} {\mathbf
  X''_3}^{-\mathbf h''_3} = Z^{4\om/3} {\mathbf X'_3}^{2\mathbf h'_3} {\mathbf X_3}^{-\mathbf h_3} \le Z^{4\om/3 -1} {\mathbf X'_3}^{2\mathbf h'_3}. 
\end{equation}
Further, by \eqref{smallvar}, we have $\langle\mathbf X'_3\rangle \le Z^{10k\om J_3}$, and hence that ${\mathbf X'_3}^{2\mathbf h'_3}\le \langle\mathbf X'_3\rangle^{2|\mathbf h|} \le Z^{20k\om J_3|\mathbf h_3|}$. However, it is immediate from \eqref{omegabound} that
$$ \frac43 \om + 20k\om J_3|\mathbf h_3| <1 , $$
so that the expression in \eqref{C42} tends to zero as $Z\to \infty$. By \eqref{C41}, we see that for large $Z$
we must have $at {\mathbf x'_3}^{\mathbf h'_3} =cq$.  Hence
$ t= q/(q, {\mathbf x'_3}^{\mathbf h'_3})$, and \eqref{C32} simplifies to
$$  W_3(b_3\alpha, \mathbf X) \ll \langle\mathbf X_3\rangle Z^{-\om/(3H)}
+ \langle\mathbf X''_3\rangle \sum_{\mathbf x'_3 \in\mathscr X} (q, {\mathbf
  x'_3}^{\mathbf h'_3})^{1/H} \big(q+ \mathbf X_3^{-\mathbf
  h_3}|b_3||qb_3\alpha-a|\big)^{-1/H}.$$ Here we can sum over {\em all}
$\mathbf x'_3$ and apply an argument paralleling that leading from \eqref{C17}
to \eqref{C17a}. This produces
$$ 
W_3(b_3\alpha, \mathbf X) \ll \langle\mathbf X_3\rangle Z^{-\om/(3H)} +
\langle\mathbf X_3\rangle q^\varepsilon \big(q+ \mathbf X_3^{-\mathbf
  h_3}|qb_3\alpha-a|\big)^{-1/H}.$$ The bound \eqref{C9a} is now evident, and
the proof of Proposition~\ref{circle-method} is complete.

\section{Upper bound estimates}

\subsection{The upper bound hypothesis}

As we mentioned in the introduction, not only asymptotic information of the
type encoded in Hypothesis~\ref{H1} is required as an input for the transition
method in Section~\ref{sec8}, but also certain upper bound estimates that are
needed, for example, to handle the contribution to the count that comes from
solutions of \eqref{torsor} where the summands are very unbalanced. Again, we
formulate the requirements as a hypothesis that can then be checked in the
particular cases at hand. We recall the definition of the block matrix
\begin{equation}\label{newA}
  \mathscr{A} = \left(
    \begin{matrix}
      \mathscr{A}_1 & \mathscr{A}_2\\
      \mathscr{A}_3 & \mathscr{A}_4
    \end{matrix}
  \right) \in \Bbb{R}^{(J+1) \times (N+k)}
\end{equation}
in \eqref{matrix}.  In the slightly simpler setup  of the torsor
equation \eqref{torsor} and the height conditions \eqref{height} we have
\begin{equation}\label{newA1}
  \mathscr{A}_1 = (\alpha_{ij}^{\nu}) \in \Rd_{\ge 0}^{J\times N}
\end{equation}
with $0 \leq i \leq k$, $1 \leq j \leq J_i$, $1 \leq \nu \leq N$ and  
\begin{equation}\label{A2new}
  \Am_2 = (e_{ij}^\mu) \in \Rd^{J\times k}
  \text{ with $e_{ij}^\mu=\begin{cases}\delta_{\mu=i}h_{ij}&\text{$i<k$, $\mu < k$,}
      \\-h_{kj}&\text{$i=k$, $\mu < k$,}\\
      -1&\text{$i<k$, $\mu = k$,}\\
      h_{kj}-1
      &\text{$i=k$, $\mu = k$.}
    \end{cases}$}
\end{equation}
This notation is more convenient for the analytic manipulations in the
following sections.

Throughout we assume that
\begin{equation}\label{1c}
  \text{rk}(\mathscr{A}_1) = \text{rk}(\mathscr{A}) = R \quad \text{(say).}
\end{equation}
In our applications, this will be satisfied by Lemma~\ref{rank}, and $R$ plays
by Lemma~\ref{lem19} the same role as in \eqref{eq:rkPic}.  We define
\begin{equation}\label{defc2}
  c_2 = J-R,
\end{equation}
so that by \eqref{eq:rkPic} this choice of $c_2$ is the expected exponent in
\eqref{manin}.  For any vector ${\bm \zeta}$ satisfying the properties
specified in \eqref{zeta1}, where we allow more generally also
$\zeta_i \geq 0$, and for arbitrary $\zeta_0 > 0$, we also assume that the
system of $J+1$ linear equations
\begin{equation}\label{1a}
  \begin{split}
    \left(\begin{matrix} \mathscr{A}_1\\ \mathscr{A}_3\end{matrix}\right) {\bm
      \sigma} = \Big(1 - h_{01}\zeta_0, \ldots, 1 - h_{kJ_k}\zeta_k,
    1\Big)^{\top}
  \end{split}
\end{equation}
in $N$ variables has a solution ${\bm \sigma} \in \Bbb{R}_{>0}^{N}$.  In our
applications, this is ensured by Lemma~\ref{pos} (whose proof also works for
$\zeta_i \geq 0$).

\begin{remark}
  The condition $\rank \Am = \rank \Am_1$ puts some restrictions on the height
  matrix $\Am_1$. For instance, no row of $\mathscr{A}_1$ can vanish
  completely (since every column of $\mathscr{A}_2$ is linearly dependent on
  the columns of $\mathscr{A}_1$). For future reference, we remark that this
  implies that the set of conditions \eqref{height} for
  $x_{ij} \in \Bbb{Z} \setminus \{0\}$ implies $|x_{ij}| \leq B$ for all
  $(i, j)$.
\end{remark}

Now let $H \geq 1$, $0 < \lambda \leq 1$ and
$\mathbf{b}, \mathbf{y} \in \Bbb{N}^{J}$.  Let
$N_{\mathbf{b}, \mathbf{y}}(B, H, \lambda)$ be the number of solutions
$\mathbf{x} \in (\Bbb{Z}\setminus \{0\})^J$ satisfying the conditions
\begin{equation}\label{equation}
  \begin{split}
    &  \sum_{i=1}^k  \prod_{j=1}^{J_i}  (b_{ij} x_{ij})^{h_{ij}} = 0,\quad
    \prod_{i=0}^k \prod_{j=1}^{J_i} | y_{ij} x_{ij}|^{\alpha^\nu_{ij}} \leq B
    \quad (1 \leq \nu \leq N),
  \end{split}
\end{equation}
and at least one of the inequalities
\begin{equation}\label{Hlambda}
  \begin{split}
    & \min_{ij} |x_{ij}| \leq H, \quad \quad \min_{1 \leq i \leq k} \prod_{j =
      1}^{J_i} |x_{ij}|^{h_{ij}} < \Bigl(\max_{1 \leq i \leq k} \prod_{j =
      1}^{J_i} |2x_{ij}|^{h_{ij}}\Bigr)^{1-\lambda}.
  \end{split}
\end{equation}
Note that for $\textbf{x} \in (\Bbb{Z}\setminus \{0\})^J$ satisfying
\eqref{equation}, the first condition in \eqref{Hlambda} is always satisfied
for $H=B$ and the second condition in \eqref{Hlambda} is never satisfied for
$\lambda = 1$.  Let $\mathscr{S}_{\textbf{y}}(B, H, \lambda)$ denote the set
of all $\mathbf{x} \in [1, \infty)^J$ that satisfy \eqref{Hlambda} and the $N$
inequalities in the second part of \eqref{equation}. As in \eqref{gcd}, we
denote by $S_{\rho}$, $1 \leq \rho \leq r$, subsets of the set of pairs
$(i, j)$ with $0 \leq i \leq k$, $1 \leq j \leq J_i$ corresponding to the
coprimality conditions.

\begin{hyp}\label{H2} 
  Let $c_2$ be the number introduced in \eqref{defc2} and let $\lambda$ 
  be as in Hypothesis~\ref{H1}.  Suppose that there
  exist
  ${\bm \eta} = (\eta_{ij}) \in \Bbb{R}_{> 0}^J$ and $ \delta_2,
  \delta_2^{\ast} > 0$ with the following properties:
 \begin{equation}\label{3}
   C_1({\bm \eta}) \colon  \quad 
   \sum_{(i, j) \in S_{\rho}}\eta_{ij} \geq 1+\delta_2 \quad \text{for all} \quad 1 \leq \rho \leq r, 
 \end{equation}
 \begin{equation}\label{2a}
  N_{\mathbf{b}, \mathbf{b} \cdot \mathbf{y}}(B, H,  \lambda)
   \ll  B (\log B)^{c_2  -1+\varepsilon} (1+\log H)  \mathbf{b}^{-{\bm  \eta}}
   \langle\mathbf{y}\rangle^{-\delta_2^{\ast}}
 \end{equation}
 and
 \begin{equation}\label{continuous}
   \int_{\mathscr{S}_{\textbf{y}}(B, H, \lambda)} \prod_{ij}
   x_{ij}^{-h_{ij}\zeta_i}
   \dd\mathbf{x} \ll B (\log B)^{c_2  -1+\varepsilon} (1+\log H)
   \langle \mathbf{y}\rangle^{-\delta_2^{\ast}}
 \end{equation} 
 for any $\varepsilon > 0$ and some ${\bm \zeta}$ satisfying \eqref{zeta1}. 
\end{hyp}

The bound \eqref{2a} is the desired upper bound $B(\log B)^{c_2+\varepsilon}$
with some saving in the coefficients $\textbf{b}$, $\textbf{y}$ and with some
extra logarithmic saving in the situation of condition \eqref{Hlambda}, \ie if
one variable is short (that is, $\log H = o((\log B)^{1+\varepsilon})$) or the
blocks $\prod_j |x_{ij}|^{h_{ij}}$ for $1 \leq i \leq k$ are unbalanced in
size (so that the second assumption in \eqref{Hlambda} holds and we may choose
$H$ very small even if all $x_{ij}$ are large).

\subsection{Reduction to linear algebra}

Our main applications involve the torsor equation \eqref{typeT}. In this case,
the verification of Hypothesis~\ref{H2} can be checked simply by a linear
program. This will be established in Proposition~\ref{propH2} below. We start
with two elementary lemmas. Here $(., ., .)$ denotes the greatest common divisor, $[., ., .]$ denotes the least common multiple and $\tau$ is the divisor function.

\begin{lemma}\label{lattice}
  Let $\mathbf{v}\in \Bbb{Z}^3$ be primitive and let $H_1, H_2, H_3 > 0$. Then
  the number of primitive ${\mathbf u} \in \Bbb{Z}^3$ that satisfy
  $u_1v_1 + u_2v_2 + u_3v_3 = 0$ and that lie in the box $|u_i| \leq H_i$
  $(1 \leq i \leq 3)$ is $O(1+H_1H_2|v_3|^{-1})$.
\end{lemma}

This is \cite[Lemma  3]{HB}. 

\begin{lemma}\label{gcd-lemma}
  Let $\alpha, \beta, \gamma \in \Bbb{N}$, $A, B, X_1, \ldots, X_r \geq 1$,
  $h_1, \ldots, h_r \in \Bbb{N}$ with $h_1 \leq \dots \leq h_r$.  Then
  $$\sum_{a \leq A} \sum_{b \leq B} \sum_{\substack{x_j \leq X_j\\ 1 \leq j
      \leq r}}
  (\alpha a, \beta b, \gamma \mathbf{x}^{\textbf{h}}) \ll (\alpha, \beta,
  \gamma)^{1/h_r}(\alpha, \beta)^{1-1/h_r}
  \tau(\alpha)\tau(\beta)\tau(\gamma)\tau_r(\alpha\beta\gamma) A B \langle
  \textbf{X} \rangle. $$
\end{lemma}

\begin{proof}
  The left hand side of the formula is at most
  \begin{displaymath}
    \begin{split}
      & \sum_f f \sum_{\substack{a \leq A\\ f \mid \alpha a}}
      \sum_{\substack{b \leq B\\ f \mid \beta b}}  \sum_{\substack{x_j \leq
          X_j\, (1 \leq j \leq r)\\ f \mid \gamma
          \textbf{x}^{\textbf{h}}}}1\leq AB \sum_f \frac{(f, \alpha)(f,
        \beta)}{f} \sum_{f_1 \cdots f_r = f/(f, \gamma)}  \sum_{\substack{x_j
          \leq X_j\, (1 \leq j \leq r)\\ f_j \mid x_j^{h_j}  }}1 \\
      &\leq AB\langle \textbf{X} \rangle \sum_f \frac{(f, \alpha)(f, \beta)
        (f, \gamma)^{1/h_r}\tau_r(f)} {f^{1 + 1/h_r}} \leq \zeta(1 +
      1/h_r)^rAB\langle \textbf{X} \rangle \sum_{a \mid \alpha} \sum_{b \mid
        \beta} \sum_{c \mid \gamma} \frac{abc^{1/h_r}\tau_r([a, b, c]) }{[a,
        b, c]^{1 + 1/h_r}}.
    \end{split}
  \end{displaymath}
  Since
  $abc^{\delta}[a, b, c]^{-1-\delta} \leq (a, b)^{1-\delta}(a, b, c)^{\delta}$
  for $0 \leq \delta \leq 1$, the lemma follows.
\end{proof}
 
We apply the previous two lemmas to analyze the number of solutions
$\textbf{x} \in (\Bbb{Z} \setminus\{0\})^J$ to the first equation in
\eqref{equation} in the special case where $k = 3$, $J_1 = J_2 = 2$ and
$h_{11} = h_{12} = h_{21} = h_{22} = 1$, cf.\ \eqref{typeT}.  In this case,
the equation reads
\begin{equation}\label{thisequation}
  b_{11} b_{12} x_{11} x_{12} + b_{21} b_{22} x_{21} x_{22} +
  \prod_{j=1}^{J_3} (b_{3j} x_{3j})^{h_{3j}} = 0.
\end{equation}
Without loss of generality, assume
\begin{equation}\label{h3jsorted}
  \text{$h_{31} \leq  \dots \leq h_{3J_3}$, and let $\nu$ be the largest
    index with $h_{3 \nu} = 1$.}
\end{equation}
If no such index exists, we put $\nu = 0$.  For notational simplicity, we write
\begin{equation}\label{defmu}
  \mu = 1 - h_{3J_3}^{-1} \in [0, 1).
\end{equation}  
Suppose first that $\nu \geq 1$.  Let us temporarily restrict to $\textbf{x}$
satisfying
\begin{equation}\label{copr}
(x_{11}x_{12}, x_{21}x_{22}, x_{31} \cdots x_{3\nu}) = 1. 
\end{equation}
For $X_{ij} \leq |x_{ij}| \leq 2 X_{ij}$ in dyadic boxes, by
Lemma~\ref{lattice} with $x_{12}, x_{22}, x_{31}$ in the roles of
$u_1, u_2, u_3$ and
$$v_3 = \frac{ x_{31}^{-1}\prod_j (b_{3j}x_{3j})^{h_{3j}}}{ \big(b_{11}b_{12}x_{11}, b_{21}b_{22}x_{21}, x_{31}^{-1}\prod_j (b_{3j}
  x_{3j})^{h_{3j}}\big)}$$
(since $\textbf{v}$ must be primitive) 
 and Lemma~\ref{gcd-lemma}, the number of such solutions to
\eqref{thisequation} is
\begin{displaymath}
\begin{split} 
  & \ll \langle \textbf{X}_0\rangle \underset{\substack{X_{11} \leq x_{11}
      \leq 2X_{11}\\ X_{21} \leq x_{21} \leq 2 X_{21}}}{\sum\sum}
  \sum_{\substack{X_{3j} \leq x_{3j}\leq  2X_{3j}\\ 2 \leq j \leq J_3}} \Big(1
  + \frac{X_{12}X_{22} }{x_{31}^{-1}\prod_j (b_{3j}x_{3j})^{h_{3j}}}
  \Big(b_{11}b_{12}x_{11}, b_{21}b_{22}x_{21}, x_{31}^{-1}\prod_j (b_{3j}
  x_{3j})^{h_{3j}}\Big)\Big)\\
  & \ll \langle \textbf{X}_0 \rangle \Big(X_{11}X_{21} \frac{\langle
    \textbf{X}_3\rangle}{X_{31}} + |\textbf{b}|^{\varepsilon}\Big(
  \frac{(b_{11}b_{12}, b_{21}b_{22}) }{\textbf{b}_3^{\textbf{h}_3}}
  \Big)^{\mu}X_{11}X_{12}X_{21}X_{22} \prod_{j} X_{3j}^{1 - h_{3j}}\Big)
\end{split}
\end{displaymath}
for every $\varepsilon > 0$ and $\mu$ as in \eqref{defmu}. By symmetry, this improves itself to
\begin{equation}\label{sym1}
  \langle \textbf{X}_0 \rangle \Big(  \frac{\min(X_{11}, X_{12}) \min(X_{21},
    X_{22})  \langle \textbf{X}_3\rangle }{\max(X_{31}, \ldots, X_{3 \nu})} +
  |\textbf{b}|^{\varepsilon}\Big( \frac{(b_{11}b_{12}, b_{21}b_{22})
  }{\textbf{b}_3^{\textbf{h}_3}} \Big)^{\mu} X_{11}X_{12}X_{21}X_{22}
  \prod_{j} X_{3j}^{1 - h_{3j}}\Big).
 \end{equation}
Permuting the roles of $u_1, u_2, u_3$ in Lemma~\ref{lattice}, we obtain similarly the bound 
\begin{displaymath}
  \begin{split} 
    & \ll \langle \textbf{X}_0\rangle \underset{\substack{X_{11} \leq x_{11}
        \leq 2X_{11}\\ X_{21} \leq x_{21} \leq 2 X_{21}}}{\sum\sum}
    \sum_{\substack{X_{3j} \leq x_{3j}\leq 2X_{3j}\\ 2 \leq j \leq J_3}} \Big(1
    + \frac{X_{12}X_{31}}{b_{21}b_{22}x_{21}} \Big(b_{11}b_{12}x_{11},
    b_{21}b_{22}x_{21}, \prod_j (b_{3j} x_{3j})^{h_{3j}}\Big)\Big)\\
    & \ll \langle \textbf{X}_0 \rangle \Big(X_{11}X_{21} X_{32}\cdots X_{3J_3} +
    |\textbf{b}|^{\varepsilon} X_{11}X_{12}\langle \textbf{X}_3 \rangle \Big).
  \end{split}
\end{displaymath}
Again by symmetry, this improves itself to
$$ \langle \textbf{X}_0  \rangle  \Big(\frac{\min(X_{11}, X_{12}) \min(X_{21}
  ,X_{22})  \langle \textbf{X}_3\rangle }{\max(X_{31}, \ldots, X_{3 \nu})}+
|\textbf{b}|^{\varepsilon} \min(X_{11}X_{12}, X_{21}X_{22}) \langle
\textbf{X}_3 \rangle \Big).$$
Together with \eqref{sym1}, we now see that the number of
$\textbf{x} \in (\Bbb{Z} \setminus \{0\})^J$ satisfying \eqref{thisequation},
\eqref{copr} and $X_{ij} \leq |x_{ij}| \leq 2 X_{ij}$ does not exceed
\begin{equation}\label{sym2}
  \begin{split}
    |\textbf{b}|^{\varepsilon}  \langle \textbf{X}_0 \rangle
    \Big(&\frac{\min(X_{11}, X_{12}) \min(X_{21}, X_{22})  \langle
      \textbf{X}_3\rangle }{\max(X_{31}, \ldots, X_{3 \nu})} \\
    &+ \frac{X_{11}X_{12}X_{21}X_{22}\langle \textbf{X}_3\rangle }{\max(X_{11}
      X_{12}, X_{21}X_{22}, (\textbf{b}_3^{\textbf{h}_3}(b_{11}b_{12},
      b_{21}b_{22})^{-1})^{\mu} \textbf{X}_3^{\textbf{h}_3})}   \Big)
  \end{split}
\end{equation}
We now replace the minima and maxima in \eqref{sym2} by suitable geometric
means. With future applications in mind, we keep the result as general as is
possible.

For $\ell = 1, 2$ and ${\bm \tau}^{(\ell)} = (\tau^{(\ell)}_{ij}) \in \Bbb{R}_{> 0}^J$ with
\begin{equation}\label{tau1}
  \begin{split}
    &\tau^{(\ell)}_{0j} = 1, \quad  \tau^{(\ell)}_{11} + \tau^{(\ell)}_{12}  \geq
    1, \quad  \tau^{(\ell)}_{21} + \tau^{(\ell)}_{22} \geq 1, \quad
    \sum_{j=1}^{\nu}  \tau^{(\ell)}_{3j} \geq  \nu-1, \quad \tau^{(\ell)}_{3j} =
    1\, (j > \nu),\\
    & \min(\tau^{(\ell)}_{11}, \tau^{(\ell)}_{12}) +  \min(\tau^{(\ell)}_{21},
    \tau^{(\ell)}_{22}) +  \min(\tau^{(\ell)}_{31}, \ldots, \tau^{(\ell)}_{3\nu})
    > 1
  \end{split}
\end{equation}
(where $\nu$ is as in \eqref{h3jsorted}), we have
$$\frac{\langle \textbf{X}_0\rangle\min(X_{11}, X_{12}) \min(X_{21}, X_{22})
  \langle \textbf{X}_3\rangle }{\max(X_{31}, \ldots, X_{3 \nu})}  \leq
\textbf{X}^{{\bm \tau}^{(\ell)}}.$$
(The second line in \eqref{tau1} is not needed here, but will be required
later when we remove condition \eqref{copr}.) Let $\bm \zeta, \bm \zeta'$ satisfy
\eqref{zeta1} and let $\zeta_0, \zeta'_0\in \Bbb{R}$ be arbitrary. Then
$$\frac{\langle\textbf{X}_0\rangle X_{11}X_{12}X_{21}X_{22}\langle
  \textbf{X}_3\rangle }{\max(X_{11} X_{12}, X_{21}X_{22},
  (\textbf{b}_3^{\textbf{h}_3}(b_{11}b_{12}, b_{21}b_{22})^{-1})^{\mu}
  \textbf{X}_3^{\textbf{h}_3})}  \leq     \Big(\frac{(b_{11}b_{12}
  b_{21}b_{22})^{1/2}}{   \textbf{b}_3^{\textbf{h}_3} } \Big)^{\mu\zeta_3'}
\prod_{ij} X_{ij}^{1 - h_{ij} \zeta'_i}  .  $$
Thus  we can bound \eqref{sym2} by
$$ |\textbf{b}|^{\varepsilon } \Big(  \textbf{X}^{{\bm \tau}^{(1)}} +
\Big(\frac{(b_{11}b_{12} b_{21}b_{22})^{1/2}}{  \textbf{b}_3^{\textbf{h}_3} }
\Big)^{\mu\zeta'_3}   \prod_{ij} X_{ij}^{1 - h_{ij} \zeta'_i} \Big)   $$
and also by
$$ |\textbf{b}|^{\varepsilon+1}    \Big(  \textbf{X}^{{\bm \tau}^{(2)}} +
\prod_{ij} X_{ij}^{1 - h_{ij} \zeta_i}    \Big)   $$
and so, for any  $0 < \alpha\leq 1$,  by
\begin{equation}\label{before-d}
  |\textbf{b}|^{\varepsilon+\alpha} \Big(  \textbf{X}^{{\bm \tau}^{(1)}} +
  \Big(\frac{(b_{11}b_{12} b_{21}b_{22})^{1/2}}{  \textbf{b}_3^{\textbf{h}_3}
  }  \Big)^{\mu\zeta'_3}   \prod_{ij} X_{ij}^{1 - h_{ij} \zeta'_i}
  \Big)^{1-\alpha}   \Big(  \textbf{X}^{{\bm \tau}^{(2)}} + \prod_{ij}
  X_{ij}^{1 - h_{ij} \zeta_i}    \Big)^{ \alpha}. 
\end{equation}
We will apply this with $\alpha$ very small (but fixed). The idea of this
maneuver is to separate the $\textbf{b}$- and $\textbf{y}$-decay in \eqref{2a}
from the bound in $B$ and $H$. Before we proceed with the estimation, we
remove the condition \eqref{copr}. Let us therefore assume that
$(x_{11}x_{12}, x_{21}x_{22}, x_{31} \cdots x_{3\nu}) = d$. Then we can apply
the previous analysis with $X_{ij}/d_{ij}$ in place of $X_{ij}$ for numbers
$d_{ij}$ satisfying
$d_{11} d_{12} = d_{21} d_{22} = d_{31} \cdots d_{3\nu} = d$ for
$i = 1, 2, 3$.  The second line in \eqref{tau1} and \eqref{zeta1} (recall that
$h_{11} = h_{12} = h_{21} = h_{22} = h_{31} = \dots = h_{3\nu} = 1$) ensure
that summing \eqref{before-d} over all $d$ (and all such combinations of
$d_{ij}$) yields a convergent sum. Thus the bound \eqref{before-d} remains
true for the number of all $\textbf{x} \in (\Bbb{Z} \setminus \{0\})^J$
satisfying   \eqref{thisequation} and
$X_{ij} \leq |x_{ij}| \leq 2X_{ij}$.

We are currently working under the assumption $\nu \geq 1$, but this is only
for notational convenience. Indeed, if $\nu = 0$, we apply Lemma~\ref{lattice}
with one of $u_1, u_2, u_3$ equal to 1, and in \eqref{sym2} we agree on the
convention that the maximum of the empty set is 1. Condition \eqref{copr} is
automatically satisfied in this case (the empty product being defined as 1),
and hence the second line in \eqref{tau1} is not needed, so that we may define
as usual the minimum of the empty set as $\infty$. With these conventions,
\eqref{before-d} remains true also if $\nu = 0$.

We now invoke the $N$ inequalities in \eqref{equation}. We choose
$$ \bm \zeta' = (\zeta_1', \zeta_2', \zeta_3')= \Big( \frac{1}{2} - \frac{1}{5 h_{3J_3}},
\frac{1}{2} - \frac{1}{5 h_{3J_3}}, \frac{2}{5h_{3J_3}}\Big)$$ and
\begin{equation}\label{tau1final}  
   \bm \tau^{(1)} =  \big(1 -  h_{01}\zeta''_0, \ldots, 1 -  h_{kJ_k}\zeta''_k\big)
\end{equation}
where $\bm\zeta'' = (\zeta_1'', \zeta_2'', \zeta_3'')$ satisfies 
\begin{equation*}
  \bm\zeta'' = (\zeta_1'', \zeta_2'', \zeta_3'') =
  \begin{cases}
    (1/3, 1/3, 1/3), & h_{3J_3} = 1,\\
    (1/2, 1/2, 0), & h_{3J_3} > 1.
  \end{cases}
\end{equation*}
Then ${\bm \tau}^{(1)}$ satisfies \eqref{tau1}. By \eqref{1a}, there exists
${\bm \sigma}^{(1)} \in \Bbb{R}_{> 0}^N$ with
\begin{equation}\label{sigma}
  |{\bm \sigma}^{(1)}|_1  \leq 1, \quad \mathscr{A}_1 {\bm \sigma}^{(1)}  = {\bm \tau}^{(1)}. 
\end{equation}
Such a vector also exists if ${\bm \tau}^{(1)}$ is replaced by
${\bm \tau}= (1 - h_{00}\zeta_0', \ldots, 1 - h_{3J_3}\zeta'_3)$.

Now, taking suitable combinations of the $N$ inequalities of the second
condition in \eqref{equation}, we see that every $\textbf{x}$ satisfying these
also satisfies
$$\prod_{ij} |x_{ij}|^{\tau_{ij}^{(1)}} \leq B \textbf{y}^{- {\bm \tau}^{(1)}}, \quad
\prod_{ij} |x_{ij}|^{1 - h_{ij} \zeta_i'} \leq B \prod_{ij}
y_{ij}^{h_{ij}\zeta_i' - 1}.$$ Define
\begin{equation*}
  \begin{split}
    {\bm \zeta}^{\ast} & = 
    \Big(\zeta_1' - \frac{1}{2} \mu \zeta_3',    \zeta_2' - \frac{1}{2} \mu
    \zeta_3', \  \zeta_3'(1 + \mu) \Big) = \Big( \frac{1}{2} -
    \frac{1}{5(1+\mu) h_{3J_3}},  \frac{1}{2} - \frac{1}{5 (1+\mu)h_{3J_3}},
    \frac{2}{5(1+\mu)h_{3J_3}}\Big)
  \end{split}
\end{equation*}
with $\mu$ as in \eqref{defmu} and
$\tilde{\bm \tau} = (1 - h_{ij} \zeta_i^{\ast})_{ij}$.  We summarize our
findings in the following lemma.

\begin{lemma}\label{hypo1}
  In the situation of equation~\eqref{thisequation}, suppose that
  $\mathbf{b}, \mathbf{y} \in \Bbb{N}^J$, $1
  \leq H \leq B$, $0 < \alpha, \lambda \leq 1$,  $\tau_{\ast}  \coloneqq 
  \min_{ij}(\tau_{ij}^{(1)}, 1-h_{ij}\zeta_i') > 0$. Let
  ${\bm
    \zeta}$ satisfy \eqref{zeta1} and   ${\bm \tau}^{(2)} \in
  \Bbb{R}_{> 0}^J$ as in \eqref{tau1}. 
  Then    
  \begin{equation}\label{X1}
    \begin{split}
      N_{\textbf{b}, \textbf{b} \cdot \textbf{y}}(B, H, \lambda) \ll&
      |\textbf{b}|^{\varepsilon + \alpha} \Big( \langle \textbf{y}
      \rangle^{-\tau_{\ast}} \big( \textbf{b}^{-{\bm \tau}^{(1)}} +
      \textbf{b}^{-{\tilde{\bm \tau}}} \big) B\Big)^{1-\alpha}
      \left.\sum_{\textbf{X}}\right.^{\ast}\Big( \textbf{X}^{{\bm
          \tau}^{(2)}\alpha} + \prod_{ij} X_{ij}^{(1 - h_{ij} \zeta_i)\alpha}
      \Big)
    \end{split}
  \end{equation}
  where $\textbf{X} = (X_{ij})$ and the asterisk indicates that each
  $X_{ij} = 2^{\xi_{ij}}$ runs over powers of 2 and is subject to
  $\prod_{ij} X_{ij}^{\alpha^\nu_{ij}} \leq B$ for $1 \leq \nu \leq N$ and at
  least one of the inequalities
  \begin{equation*}
    \min_{ij} X_{ij} \leq H, \quad \quad \min_{1 \leq i \leq k} \prod_{j =
      1}^{J_i} X_{ij}^{h_{ij}} < \Bigl(\max_{1 \leq i \leq k} \prod_{j =
      1}^{J_i} (2X_{ij})^{h_{ij}}\Bigr)^{1-\lambda}.
  \end{equation*}
\end{lemma}

Similarly, but in a much simpler way, we derive the continuous analogue 
\begin{equation}\label{X2}
  \int_{\mathscr{S}_{\textbf{y}}(B, H, \lambda)} \prod_{ij}
  x_{ij}^{-h_{ij}\zeta_i} \dd\mathbf{x} \ll \big(\langle \textbf{y}
  \rangle^{-\tau^{\dag}}   B\big)^{1-\alpha}
  \left.\sum_{\textbf{X}}\right.^{\ast}\prod_{ij} X_{ij}^{(1 - h_{ij}
    \zeta_i)\alpha}
\end{equation}
with $\tau^{\dag} = \min_{ij}(1-h_{ij} \zeta_i  ) > 0$ and the sum is subject
to the same conditions. \\

As mentioned above, we will choose $\alpha$ in \eqref{X1} very small. The key
property of ${\bm \tau^{(1)}}$ and ${\tilde{\bm \tau}}$ is that all their
entries are $\geq 1/2$ where equality is only possible for ${\bm \tau^{(1)}}$
at indices $(ij)$ with $i\in \{1, 2\}$ if $h_{3J_3} \geq 2$. Since
$|S_{\rho}| \geq 2$ for all $1 \leq \rho \leq r$, we conclude that the
conditions
$$C_1\big((1-\alpha){\bm \tau^{(1)}}\big), \quad C_1\big((1-\alpha)\tilde{\bm \tau}\big)$$
in \eqref{3} hold for sufficiently small $\alpha > 0$ provided that
\begin{equation}\label{fail}
  \max_{ij} h_{ij} = 1  \,\,\text{or there exists no $\rho$ with } S_{\rho} =
  \{(i_1, j_1), (i_2, j_2)\}, i_1, i_2 \in \{1, 2\}.
\end{equation}

We now transform the $X$-sums in \eqref{X1} and \eqref{X2}. For an arbitrary
vector ${\bm \tau} \in \Bbb{R}_{\geq 0}^J$, we rewrite a sum
$\sum^{\ast}_{\textbf{X}} \textbf{X}^{{\bm \tau}\alpha}$ of the type appearing
in \eqref{X1} and \eqref{X2} as
\begin{equation}\label{typical}
  \underset{{\bm \xi} \in \Bbb{N}_0^J}{\left.\sum \right.^{\ast}} B^{\alpha \,
    \tilde{\bm \xi}^{\top} {\bm \tau}},
  \quad\quad \tilde{\bm \xi} = \frac{\log 2}{\log B} {\bm \xi},
\end{equation}
and now $\sum^{\ast}$ indicates that the sum is subject to  
\begin{equation}\label{poly1}
  \mathscr{A}_1^{\top} \tilde{\bm \xi} \leq (1, \ldots, 1)^{\top} \in \Bbb{R}^N
\end{equation}
(the inequality being understood componentwise) and at least one of the inequalities
\begin{align}\label{poly2}
  &   \tilde{\xi}_{ij} \leq \frac{\log H}{\log B} \quad \text{for some } i, j,\\
  \label{poly3}  
  & \min_{1 \leq i \leq k}\sum_{j=1}^{J_i} \tilde{\xi}_{ij} h_{ij} < \max_{1
    \leq i \leq k}\sum_{j=1}^{J_i} \Big(\tilde{\xi}_{ij} +\frac{\log 2}{\log
    B}\Big) h_{ij}(1-\lambda). 
\end{align}
For future reference, we note that 
\begin{equation}\label{ref}
\max_{1 \leq i \leq k}\sum_{j=1}^{J_i} \Big(\tilde{\xi}_{ij} +\frac{\log
  2}{\log B}\Big) h_{ij}(1-\lambda) = \max_{1 \leq i \leq k}\sum_{j=1}^{J_i}
\tilde{\xi}_{ij}   h_{ij}(1-\lambda) + O\Big(\frac{1}{\log B}\Big).
\end{equation}
For $0 \leq i \leq k$, $1 \leq j \leq J_i$, $0 < \lambda \leq 1$ and a
permutation $\pi \in S_k$, we consider the closed, convex polytopes
\begin{equation}\label{polytope}
  \begin{split}  
    \mathscr{P} & = \{ {\bm \psi} \in \Bbb{R}^J : {\bm \psi} \geq 0, \,
    \mathscr{A}_1^{\top}  {\bm \psi} \leq (1, \ldots, 1)^{\top}\},\\
    \mathscr{P}_{ij}& = \{ {\bm \psi} \in  \mathscr{P} : \psi_{ij} = 0\},\\
    \mathscr{P}(\lambda, \pi) &= \Big\{ {\bm \psi} \in \mathscr{P} :
    \sum_{j=1}^{J_{\pi(1)}} \psi_{\pi(1), j} h_{\pi(1), j} \leq \dots \leq
    \sum_{j=1}^{J_{\pi(k)}} \psi_{\pi(k), j} h_{\pi(k), j}, \\
    & \quad\quad\quad\quad\quad \sum_{j=1}^{J_{\pi(1)}} \psi_{\pi(1), j}
    h_{\pi(1), j} \leq(1-\lambda) \sum_{j=1}^{J_{\pi(k)}} \psi_{\pi(k), j}
    h_{\pi(k), j}\Big\}.
  \end{split}
\end{equation}    
We assume that
\begin{equation}\label{simplex1}
  C_2({\bm \tau}) \colon \quad
  \max\{ {\bm \psi}^{\top} {\bm \tau} :  {\bm \psi} \in \mathscr{P}\} = 1.
\end{equation}
The intersection of the hyperplane
$\mathscr{H} \colon {\bm \psi}^{\top} {\bm \tau} = 1$ with any of the above
polytopes is again a closed convex polytope, and we assume that the dimensions
satisfy
\begin{equation}\label{simplex2}
  C_3({\bm \tau}) \colon \quad
  \begin{array}{l}
    \dim(\mathscr{H} \cap \mathscr{P} ) \leq c_2,\\
    \dim(\mathscr{H} \cap \mathscr{P}_{ij} ) \leq c_2 - 1,
    \quad  0 \leq i \leq k, 1 \leq j \leq J_i,\\
    \dim(\mathscr{H} \cap \mathscr{P}(\lambda,
    \pi) ) \leq c_2 - 1, \quad \pi \in S_k.
  \end{array}
\end{equation}
With this notation and the assumptions \eqref{simplex1} and \eqref{simplex2},
we return to \eqref{typical}. Clearly the sum has $O((\log B)^J)$ terms, so
the contribution of ${\bm \xi}$ with
$$\tilde{\bm \xi}^{\top} {\bm \tau} \leq 1 - \frac{J \log\log B}{\alpha \log
  B} $$ to \eqref{typical} is $O(B^{\alpha})$. By \eqref{simplex1}, we may now
restrict to
\begin{equation}\label{restrict}
  1 - \frac{J \log\log B}{\alpha \log B}  \leq \tilde{\bm \xi}^{\top} {\bm \tau} \leq 1
\end{equation}
in the sense that
\begin{equation}\label{sensethat}
\underset{{\bm \xi} \in \Bbb{N}_0^J}{\left.\sum \right.^{\ast}}
B^{\alpha \, \tilde{\bm \xi}^{\top} {\bm \tau}} \ll B^{\alpha}\big(1 +\#\mathscr{X}_1 + \#\mathscr{X}_2\big)
\end{equation}
where
$$\mathscr{X}_1 =  \{{\bm \xi} \in \Bbb{N}_0^J : \eqref{poly1}, \eqref{poly2}, \eqref{restrict}
\} , \quad \mathscr{X}_2 =  \{{\bm \xi} \in \Bbb{N}_0^J : \eqref{poly1}, \eqref{poly3},
\eqref{restrict} \}.$$ 
We define
$$\mathscr{Y}_1 =  \{{\bm \xi} \in \Bbb{R}_{\ge 0}^J : \eqref{poly1}, \eqref{poly2}, \eqref{restrict}
\} , \quad \mathscr{Y}_2 =  \{{\bm \xi} \in \Bbb{R}_{\ge 0}^J : \eqref{poly1}, \eqref{poly3},
\eqref{restrict} \}$$ 
and bound  $\#\mathscr{X}_1$ resp.\  $\#\mathscr{X}_2$ by the Lipschitz principle, \ie  by the volume and the volume of the boundary of $\mathscr{Y}_1$ resp.\ $\mathscr{Y}_2$ (or a superset thereof).  By the third condition in \eqref{simplex2} as well as \eqref{ref} and \eqref{restrict} we see that $\mathscr{Y}_2$ is contained in an $O_{\alpha}(\log\log B)$ neighborhood of a union of polytopes of dimension at most $c_2-1$ and side lengths $O(\log B)$, so that
$$\#\mathscr{X}_2 \ll_{\alpha, \lambda} (\log B)^{c_2-1}(\log\log B)^{J-(c_2-1)}  \ll (\log B)^{c_2 - 1+\varepsilon}.$$
Similarly, by the first two conditions in \eqref{simplex2} and \eqref{restrict} we see that $\mathscr{Y}_2$ is contained in an $O_{\alpha}(\log\log B)$ neighborhood of a union of parallelepipeds of dimension at most $c_2$, where  at most $c_2-1$ of the side lengths of each parallelepiped are of size $O(\log B)$ and the remaining ones (if any) are of size $O(\log H)$.  We conclude
$$\#\mathscr{X}_1 \ll_{\alpha} (\log B)^{c_2-1}(\log H + \log\log B) (\log\log B)^{J-c_2}  \ll (\log B)^{c_2 - 1+\varepsilon} (1 + \log H).$$
We substitute the bounds for $\#\mathscr{X}_1$, $\#\mathscr{X}_2$ into
\eqref{sensethat} and use this in \eqref{X1} and \eqref{X2}. From Lemma
\ref{hypo1} we conclude the following result.
   
\begin{prop}\label{propH2}
  In the situation of equation~\eqref{thisequation}, let $\lambda$ 
   be as in Hypothesis~\ref{H1} and $\bm \zeta$ as in \eqref{zeta1}.  Define the matrix
  $\mathscr{A}_1$ as in \eqref{newA1} and the polytopes
  $\mathscr{P}, \mathscr{P}_{ij}, \mathscr{P}(\lambda, \pi)$ as in
  \eqref{polytope}. Choose ${\bm \tau}^{(2)}$ satisfying
  \eqref{tau1}.
  Suppose that \eqref{fail} holds as well as 
  the conditions 
  \begin{equation}\label{ass2}
    C_2( {\bm \tau}^{(2)}), \quad C_3( {\bm \tau}^{(2)}), \quad
    C_2((1 - h_{ij} \zeta_{i})_{ij}), \quad C_3((1 - h_{ij} \zeta_{i})_{ij})
  \end{equation}
  hold as in \eqref{simplex1}, \eqref{simplex2}. Then Hypothesis~\ref{H2} is true. 
\end{prop}
    
Condition \eqref{ass2} requires a linear program. In principle this can be done
by hand (we show this in a special case in Appendix~\ref{A}), but a
straightforward computer-assisted verification is more time-efficient. We can
replace \eqref{fail} by the following condition: there exist vectors
${\bm \tau}^{(1)} \in \Bbb{R}^J$, ${\bm \sigma} \in \Bbb{R}^N$ satisfying
\eqref{tau1final} and \eqref{sigma} such that $C_1({\bm \tau}^{(1)})$ holds.

\section{The transition method}\label{sec8}

In this section, we describe a method that derives an asymptotic formula for
$N(B)$ as in \eqref{manin} from the input provided by Hypotheses~\ref{H1} and
\ref{H2}. In fact, we will only need these hypotheses for certain choices of parameters to be discussed in a moment. Our main result will be formulated at the end of the section. In the
interest of brevity, we now choose $b_1=\dots=b_k=1$ in \eqref{torsor}. No
extra difficulties arise should one wish to handle the more general case, but
a more elaborate notation would be needed. All equations that occur in the
examples treated in this paper may be interpreted to have coefficients $1$
only.

We begin with some more notation. We continue to use the vector operations
introduced in Section~\ref{dioph}. In addition, if
$\mathscr{R} \subseteq \Bbb{R}^n$ and $\mathbf{x} \in \Bbb{R}^n$, then
$\mathbf{x}\cdot \mathscr{R} = \{\mathbf{x} \cdot \mathbf{y} : \mathbf{y} \in
\mathscr{R}\} \subseteq \Bbb{R}^n$.  For
$\textbf{v}= (v_1, \ldots, v_n) \in \Bbb{R}^n$, we write
\begin{equation}\label{tilde}
  \widetilde{\textbf{v}} = (2^{v_1}, \ldots, 2^{v_n}) \in \Bbb{R}^n.
\end{equation}
For $\mathbf{g} \in \Bbb{N}^r$, we write
$\mu(\mathbf{g}) = \prod_{\rho=1}^r \mu(g_\rho)$  where $\mu$ denotes the M\"obius function. We write
$\mathbf{1} = (1, \ldots, 1)$, the dimension of the vector being understood
from the context.

For $0 < \Delta < 1$ let $f_{\Delta} \colon[0, \infty) \rightarrow [0, 1]$ be
a smooth function with
\begin{equation}\label{smooth0}
  \text{supp}(f_{\Delta}) \subseteq [0, 1 + \Delta), \quad f_{\Delta} = 1
  \text{ on } [0, 1], \quad \frac{d^j}{dx^j} f_{\Delta}(x) \ll_j \Delta^{-j}
\end{equation}
whose Mellin transform $\widehat{f}_{\Delta}$ obeys, once $\delta_3>0$ and
$A\geq 0$ are fixed, the inequality
\begin{equation}\label{smooth}
  \frac{\dd^j}{\dd s^j} \widehat{f}_{\Delta}(s) \ll_{j, A, \delta_3}
  \frac{(1+\Delta |s|)^{-A}}{|s|}
\end{equation}
for all $j \in \Bbb{N}_0$, uniformly in $\delta_3 \leq \Re s < 2$.  A
construction of $f_{\Delta}$ is given in \cite[(2.3)]{BBS1}.  From
\eqref{smooth}, we infer the useful estimate
\begin{equation}\label{useful}
  \mathscr{D}\Bigl(\mathbf{s}^{\mathbf{a}} \prod_{\nu=1}^N
  \widehat{f}_{\Delta}(s_{\nu})\Big)
  \ll  \Delta^{-\| \mathbf{a}\|_1 - c} |\mathbf{s}|^{-c} \langle \textbf{s}
  \rangle^{-1}
\end{equation}
for $\mathbf{s} = (s_1, \ldots, s_{N}) \in \Bbb{C}^N$ with
$2>\Re s_{\nu} \geq \delta_3 > 0$, $\mathbf{a} \in \Bbb{N}_0^N$, $c \geq 1$
and any linear differential operator $\mathscr{D}$ with constant coefficients
in $s_1, \ldots, s_N$, the implied constant being dependent on
$\textbf{a}, N, c, \mathscr{D}$.

We write $\int^{(n)}$ for an iterated $n$-fold Mellin--Barnes integral. The
lines of integration will be clear from the context or otherwise specified in
the text.  If all $n$ integrations are over the same line $(c)$, then we write
this as $\int_{(c)}^{(n)}$.

We continue to work subject to the conditions \eqref{1c}, \eqref{1a}. Also, we
suppose that Hypotheses~\ref{H1} and \ref{H2} are available to us.  With
$\beta_i$ as in Hypothesis~\ref{H1} and $S_{\rho}$ as in \eqref{gcd}, we
suppose that there is some $\delta_4 >0$ with
\begin{equation}\label{1b}
  \sum_{(i,j) \in S_{\rho}} (1 - \beta_i h_{ij}) \geq  1 + \delta_4
  \,\,\, (1 \leq \rho \leq r)
  \quad \text{ and } \quad \beta_ih_{ij} \leq 1 \,\, (1 \leq i \leq k, 1 \leq
  j \leq J_i).
\end{equation}

In order to efficiently work with the asymptotic formula in
Hypothesis~\ref{H1}, it is necessary to rewrite the singular integral as a
Mellin transform. With ${\bm \zeta}$ as in Hypothesis~\ref{H1} (in particular
satisfying \eqref{zeta1}), we assume that
\begin{equation}\label{J-cond}
  J_i \geq 2 \quad \text{whenever}\quad  \zeta_i \geq 1/2.
\end{equation}
We also define 
\begin{equation*}
  J^* = J_1 + \dots  +J_k
\end{equation*}
for the number of variables appearing in the torsor equation.

\begin{lemma}\label{todo}
  Let $\mathbf{b} \in (\Bbb{Z}\setminus\{0\})^k$ and
  $\mathbf{X} \in [1/2, \infty)^{J}$.  For $1 \leq i \leq k$, put
  \begin{equation}\label{Ki}
    \mathscr{K}_i(z) =
    \begin{cases}
      \Gamma(z) \cos(\pi z/2), & h_{ij} \text{ odd for some } 1 \leq j \leq J_i,\\
      \Gamma(z) \exp({\rm i} \pi z/2),  &h_{ij} \text{ even for all } 1 \leq j \leq J_i.
    \end{cases}
  \end{equation}
  Then, on writing 
  $z_k = 1 - z_1 - \dots - z_{k-1}$, one has
  \begin{equation*}
    \mathscr{I}_{\mathbf{b}}(\mathbf{X}) =\frac{2^{J^*}}{\pi}
    \langle \textbf{X}_0\rangle\int_{(\zeta_1)} \cdots \int_{(\zeta_{k-1})}
    \prod_{i=1}^k \frac{ \mathscr{K}_i(z_i) }{b_i^{z_i}}
    \prod_{j=1}^{J_i} \Bigl(X_{ij}^{1 - h_{ij} z_i }
    \frac{1 - 2^{h_{ij}z_i-1}}{1 - h_{ij} z_i  }\Bigr)
    \frac{\dd z_1 \cdots \dd z_{k-1}}{(2\pi {\rm i})^{k-1}} . 
  \end{equation*}
\end{lemma}

Note that \eqref{zeta1} implies that $\Re z_k = \zeta_k$.

\begin{proof}
  We start with the absolutely convergent Mellin identity
  $$e(w) = \int_{\mathscr{C}} \Gamma(s)   \exp\left(\frac{1}{2}\text{sgn}(w)
    {\rm i} \pi s\right) |2\pi w|^{-s} \frac{\dd s}{2\pi {\rm i}}$$
  for $w \in \Bbb{R} \setminus \{0\}$ and $\mathscr{C}$ the contour
  \begin{equation*}
    \textstyle(-1-{\rm i}\infty, -1-{\rm i}] \cup [-1-{\rm i}, \frac{1}{k} -
    {\rm i}] \cup [ \frac{1}{k}- {\rm i},  \frac{1}{k} + {\rm i}] \cup [
    \frac{1}{k} + {\rm i}, -1 + {\rm i}] \cup [-1 + {\rm i}] \cup [-1 + {\rm
      i} \infty),
  \end{equation*} 
  which can simply be checked by moving the contour to the left and comparing
  power series. Integrating this over $\mathscr{Y}$ as in \eqref{E2} based on
  $$\int_{\frac{1}{2} Y \leq y \leq Y} y^{-hs} \dd y = \frac{1- 2^{hs}}{1-hs} Y^{1-hs}$$
  and using the definition \eqref{E4}, we obtain
  \begin{equation}\label{twosided}
    I_i(b_i \beta, \textbf{X}_i) = 2^{J_i}  \int_{\mathscr{C}}\frac{
      \mathscr{K}_i(z_i)  }{(2\pi|b_i\beta|)^{z_i}} \prod_{j=1}^{J_i}
    \Bigl(X_{ij}^{1 - h_{ij} z_i }  \frac{1 - 2^{h_{ij}z_i-1}}{1 - h_{ij} z_i
    }\Bigr) \frac{\dd z_i}{2\pi {\rm i}}
  \end{equation}
  for every $i$. Note that $\text{sgn}(\textbf{y}_i^{\textbf{h}_i})$ is always
  1 if and only if $h_{ij}$ is even for all $1 \leq j \leq J_i$.  At this
  point, we can straighten the contour and replace it with
  $\Re z_i = \zeta_i$. The expression is still absolutely convergent, provided
  that \eqref{J-cond} holds.  We insert this formula into \eqref{E5} for
  $i = 1, \ldots, k-1$ getting
  \begin{displaymath}
    \begin{split}
      \mathscr{I}_{\textbf{b}}(\textbf{X}) = \langle\textbf{X}_0\rangle
      \int_{-\infty}^{\infty}
      &2^{J_1 + \dots + J_{k-1}} \int^{(k-1)}_{\Re z_i = \zeta_{i}}
      \prod_{i=1}^{k-1} \frac{ \mathscr{K}_i(z_i)  }{(2\pi|b_i|)^{z_i}}
      \prod_{j=1}^{J_i}  \Bigl(X_{ij}^{1 - h_{ij} z_i }  \frac{1 -
        2^{h_{ij}z_i-1}}{1 - h_{ij} z_i  }\Bigr)
      \frac{\dd \textbf{z}}{(2\pi {\rm i})^{k-1}} \\
      &\times I_{k}(b_k\beta, \textbf{X}_k) |\beta|^{-z_1 - \dots - z_{k-1}}
      d\beta .
    \end{split}
  \end{displaymath}
  The integral in $\beta$ is still absolutely convergent, by \eqref{E3} and
  \eqref{zeta1}. It is the two-sided Mellin transform of
  $ I_{k}(b_k\beta, \textbf{X}_k)$ in $\beta$ at
  $z_k = 1- z_1- \dots - z_{k-1}$. An evaluation can be read off from
  \eqref{twosided} by Mellin inversion, and the lemma follows.
\end{proof}

We are now prepared to describe our method in detail.

\subsection{Step 1: Initial manipulations}\label{sec51}

Let $\chi \colon (\Bbb{Z} \setminus \{0\})^J \rightarrow [0, 1]$ be the
characteristic function on the set of solutions to the torsor equation
 \eqref{torsor} subject to $b_1 = \dots = b_k = 1$, and let
$\psi \colon (\Bbb{Z} \setminus \{0\})^J \rightarrow [0, 1]$ be the
characteristic function on $J$-tuples of nonzero integers satisfying
the coprimality conditions \eqref{gcd}. For $1 \leq \nu \leq N$, let
\begin{equation}\label{defP}
  P_{\nu}(\textbf{x}) = \prod_{ij} |x_{ij}|^{\alpha_{ij}^{\nu}}
\end{equation}
denote the monomials appearing in the height conditions 
\eqref{height}.  We start with some smoothing. Let $0 < \Delta < 1/10$ and
define
$$  F_{\Delta, B}(\mathbf{x}) =
\prod_{\nu = 1}^{N} f_{\Delta}
\left(\frac{P_{\nu}(\mathbf{x})}{{B}}\right). $$ Then the counting function
$$N_{\Delta}({B}) =  \sum_{\mathbf{x} \in (\Bbb{Z}\setminus\{0\})^J}
\psi(\mathbf{x})
\chi(\mathbf{x}) F_{\Delta, B}(\mathbf{x}) $$
satisfies
\begin{equation}\label{sandwich}
N_{\Delta}({ B}(1 - \Delta)) \leq N(B)  \leq N_{\Delta}({B}). 
\end{equation}
We remove the coprimality conditions encoded in $\psi$ by M\"obius
inversion. As in \cite[Lemma 2.1]{BBS2}, we have
$$N_{\Delta}({B}) =  \sum_{\mathbf{g} \in \Bbb{N}^r} \mu(\mathbf{g})
\sum_{\mathbf{x} \in (\Bbb{Z}\setminus\{0\})^J}   \chi({\bm \gamma} \cdot
\mathbf{x}) F_{\Delta, B}({\bm \gamma} \cdot  \mathbf{x}),$$
where for given $\mathbf{g} \in \Bbb{N}^r$, we wrote
\begin{equation}\label{gamma}
  {\bm \gamma} = (\gamma_{ij}) \in\Bbb{N}^J, \quad
  \gamma_{ij} = {\rm lcm}\{g_{\rho} \mid (i, j) \in S_{\rho}\}.
\end{equation}
for $0 \leq i \leq k$, $1 \leq j \leq J_i$.  In the following we will need \eqref{2a} of Hypothesis \ref{H2} only for $\textbf{b} = \bm \gamma$. For later purposes, we state the
following elementary lemma.
\begin{lemma}\label{kgV}
  For ${\bm \gamma}\in \Bbb{N}^J$ as in \eqref{gamma}, $\delta > 0$,
  $1 \leq \rho \leq r$, and ${\bm \eta} = (\eta_{ij}) \in \Bbb{R}^J_{\geq 0}$,
  the series
  $$\sum_{\textbf{g} \in \Bbb{N}^r}  {\bm \gamma}^{-{\bm \eta}}  g_{\rho}^{\delta}$$
  is convergent provided that 
  $$\sum_{(i, j) \in S_{\rho}}  \eta_{ij} > 1 + \delta$$
  holds for all $1 \leq \rho \leq r$.
\end{lemma}

\begin{proof}
  Suppose that $\sum_{(i, j) \in S_{\rho}} \eta_{ij} \geq 1+\delta+ \delta_0$
  for all $\rho$ and some $\delta_0 > 0$. The sum in question can be written
  as an Euler product, and a typical Euler factor has the form
  $$\sum_{{\bm \alpha} \in \Bbb{N}^r_0}
  p^{f(\bm \alpha)}, \quad f(\bm \alpha) = \delta \alpha_\rho -\sum_{i, j} \eta_{ij}\max_{(i, j) \in
      S_t} \alpha_t.$$
This is  $$1 + O\Big(\sum_{\alpha=1}^{\infty} \frac{(1+\alpha)^r}{p^{\alpha(1+
     \delta_0)}}\Big).$$
  The statement is now clear. 
\end{proof}

For $1 \leq T \leq B$, we define
$$N_{\Delta, T}({B}) =  \sum_{|\mathbf{g}| \leq T  } \mu(\mathbf{g})
\sum_{\mathbf{x} \in (\Bbb{Z}\setminus\{0\})^J}   \chi({\bm \gamma} \cdot
\mathbf{x}) F_{\Delta, B}({\bm \gamma} \cdot  \mathbf{x}).$$
By \eqref{2a}, \eqref{3} (recall $\Delta \leq 1/10$) and Lemma~\ref{kgV}, and 
by an estimate that is often called Rankin's trick,
\begin{equation}\label{error1}
  \begin{split}
    |N_{\Delta, T}({B}) - N_{\Delta}({B}) | &\leq \sum_{|\mathbf{g}| > T }
    N_{{\bm \gamma}, {\bm \gamma}}(2B, 2B, 1)
    \ll B(\log B)^{c_2 + \varepsilon} \sum_{|\mathbf{g}| > T }    {\bm \gamma}^{-{\bm \eta}}\\
    & \leq B(\log B)^{c_2 + \varepsilon}\sum_{\mathbf{g} } {\bm \gamma}^{-{\bm
        \eta}} \Bigl(\frac{|\mathbf{g}|}{T}\Bigr)^{\delta_2 - \varepsilon}\ll
    B(\log B)^{c_2 + \varepsilon} T^{-\delta_2 }.
  \end{split}
\end{equation}
Next we write each factor $f_{\Delta}$ in the definition of $F_{\Delta, B}$ as
its own Mellin inverse, so that
$$N_{\Delta, T}({B}) =  \sum_{|\mathbf{g}| \leq T  } \mu(\mathbf{g})
\int_{(1)}^{(N)}  \sum_{\mathbf{x} \in (\Bbb{Z}\setminus\{0\})^J}  \frac{
  \chi({\bm \gamma} \cdot \mathbf{x}) }{{\bm \gamma}^\mathbf{v}  }\prod_{ij}
|x_{ij}|^{-v_{ij}} \prod_{\nu=1}^N\Bigl(
\widehat{f}_{\Delta}(s_{\nu})B^{s_{\nu}}\Bigr) \frac{\dd\mathbf{s}}{(2\pi {\rm
    i})^{N}}$$
where
\begin{equation}\label{vs}
  \mathbf{v}  = (v_{ij}) = \mathscr{A}_1 \mathbf{s} \in \Bbb{C}^J
\end{equation}
and $\mathscr{A}_1 = (\alpha_{ij}^{\nu}) \in \Bbb{R}^{J\times N}$ is as
before.  By partial summation, we obtain
\begin{displaymath}
  \begin{split}
    \sum_{\mathbf{x} \in (\Bbb{Z}\setminus\{0\})^J}  \frac{ \chi({\bm \gamma} \cdot \mathbf{x}) }{{\bm \gamma}^\mathbf{v} } \prod_{ij} |x_{ij}|^{-v_{ij}} &  = \frac{1}{{\bm \gamma}^{\mathbf{v}}} \Bigl(\prod_{i, j} v_{ij}\Bigr)\int_{[1, \infty)^J} \sum_{0 < |x_{ij}| \leq X_{ij}}  \chi({\bm \gamma} \cdot \mathbf{x}) \mathbf{X}^{-\mathbf{v} - \mathbf{1}} \dd\mathbf{X}\\
    &= \frac{1}{{\bm \gamma}^{\mathbf{v}}} \Bigl(\prod_{i, j} \frac{v_{ij}}{1 -
      2^{-v_{ij}}}\Bigr)\int_{[1, \infty)^J} \sum_{\frac{1}{2} X_{ij} < |x_{ij}|
      \leq X_{ij}} \chi({\bm \gamma} \cdot \mathbf{x}) \mathbf{X}^{-\mathbf{v} -
      \mathbf{1}} \dd\mathbf{X},
  \end{split}
\end{displaymath}
so that  
\begin{equation*}
  \begin{split}
    N_{\Delta, T}({B})  =  \sum_{|\mathbf{g}| \leq T  } \mu(\mathbf{g})
    \int_{(1)}^{(N)}  \frac{1}{{\bm \gamma}^{\mathbf{v}}} \Bigl(\prod_{i, j}
    \frac{v_{ij}}{1 - 2^{-v_{ij}}}\Bigr)\int_{[1,
      \infty)^J}\frac{\mathscr{N}_{{\bm
          \gamma}^{\ast}}(\mathbf{X})}{\mathbf{X}^{\mathbf{v} + \mathbf{1}}}
    \dd\mathbf{X}
    \prod_{\nu=1}^N\Bigl(  \widehat{f}_{\Delta}(s_{\nu})B^{s_{\nu}}\Bigr)
    \frac{\dd\mathbf{s}}{(2\pi {\rm i})^{N}}
  \end{split}
\end{equation*}
in the notation of Hypothesis~\ref{H1}, where 
\begin{equation}\label{gammaast}
  {\bm \gamma}^{\ast} = \Big(\prod_{j=1}^{J_i} \gamma_{ij}^{h_{ij}}\Big)_{1
    \leq i \leq k} \in \Bbb{N}^k.
\end{equation}
We emphasize that we need \eqref{E} of Hypothesis \ref{H1} only for $\textbf{b} = {\bm \gamma}^{\ast}$. 

\subsection{Step 2: Removing the cusps}

We would like to insert the asymptotic formula from Hypothesis~\ref{H1}. This
gives a meaningful error term only if $\min X_{ij}$ is not too small, and the
formula is only applicable if \eqref{samesize} holds. Thus, for
$0 < \delta<1, 0<\lambda \le 1$ we define the set
$$\mathscr{R}_{\delta, \lambda} = \Bigl\{\mathbf{X} = (\textbf{X}_1, \ldots,
\textbf{X}_{k}) \in [1, \infty)^J : \min_{i, j} X_{ij} \geq \max
X_{ij}^{\delta}, \, \min_{1 \leq i \leq k}   \textbf{X}_i^{\textbf{h}_i}  \geq
\big(\max_{1 \leq i \leq k}
\textbf{X}_i^{\textbf{h}_i}\big)^{1-\lambda}\Bigr\}.$$
Correspondingly we put 
\begin{equation}\label{defNDeltaTdeltalambda}
  N_{\Delta, T, \delta, \lambda} =  \sum_{|\mathbf{g}| \leq T  }
  \mu(\mathbf{g}) \int_{(1)}^{(N)}  \frac{1}{{\bm \gamma}^{\mathbf{v}}}
  \Bigl(\prod_{i, j} \frac{v_{ij}}{1 -
    2^{-v_{ij}}}\Bigr)\int_{\mathscr{R}_{\delta,
      \lambda}}\frac{\mathscr{N}_{{\bm
        \gamma}^{\ast}}(\mathbf{X})}{\mathbf{X}^{\mathbf{v} + \mathbf{1}}}
  \dd\mathbf{X}
  \prod_{\nu=1}^N\Bigl(  \widehat{f}_{\Delta}(s_{\nu})B^{s_{\nu}}\Bigr)
  \frac{\dd\mathbf{s}}{(2\pi {\rm i})^{N}}.
\end{equation}
While $ \lambda $ is fixed, $\delta$ is allowed to depend on $B$ and will
later be chosen as a negative power of $\log B$. In particular, all subsequent
estimates will be uniform in $\delta$.

\begin{lemma}\label{lemma2}
  We have
  $$N_{\Delta, T}({B})  - N_{\Delta, T, \delta, \lambda} \ll
  T^r B(\log B)^{c_2 + \varepsilon} (\delta + (\log B)^{-1}).$$
\end{lemma}

\begin{proof}
  This is essentially \cite[Lemma 5.1]{BBS2}. The idea is to revert all steps
  from Section~\ref{sec51} and apply the bound \eqref{2a}. By a change of
  variables, we have
  \begin{displaymath}
    \begin{split}
      N_{\Delta, T, \delta, \lambda}  = &  \sum_{|\mathbf{g}| \leq T  }
      \mu(\mathbf{g}) \int_{(1)}^{(N)}  \frac{1}{{\bm \gamma}^{\mathbf{v}}}
      \Bigl(\prod_{i, j} \frac{v_{ij}}{1 - 2^{-v_{ij}}}\Bigr)\sum_{{\bm
          \sigma} \in \{0, 1\}^J} (-1)^{|{\bm \sigma}|_1} \\
      &\times \int_{ \widetilde{-{\bm \sigma}} \cdot \mathscr{R}_{\delta,
          \lambda}} \sum_{0< |x_{ij}| \leq X_{ij}} \chi({\bm \gamma} \cdot
      \mathbf{x})(\widetilde{{\bm \sigma}} \cdot \mathbf{X})^{-\mathbf{v}}
      \frac{\dd\mathbf{X}}{\langle \textbf{X} \rangle} \prod_{\nu=1}^N\Bigl(
      \widehat{f}_{\Delta}(s_{\nu})B^{s_{\nu}}\Bigr)
      \frac{\dd\mathbf{s}}{(2\pi {\rm i})^{N}},\end{split}
  \end{displaymath}
  where we recall the notation \eqref{tilde}. By partial summation, this equals
  \begin{displaymath}
    \begin{split}
      & \sum_{|\mathbf{g}| \leq T } \mu(\mathbf{g}) \int_{(1)}^{(N)}
      \Bigl(\prod_{i, j} \frac{1}{1 - 2^{-v_{ij}}}\Bigr)\sum_{{\bm \sigma} \in
        \{0, 1\}^J} (-1)^{|{\bm \sigma}|_1} 2^{-\sum_{ij} \sigma_{ij} v_{ij}}
      \sum_{ \mathbf{x} \in \widetilde{-{\bm \sigma}} \cdot \mathscr{R}_{\delta,
          \lambda}} \frac{ \chi({\bm \gamma} \cdot \mathbf{x})}{{\bm
          \gamma}^{\mathbf{v}}\mathbf{x}^{\mathbf{v}}} \prod_{\nu=1}^N\Bigl(
      \widehat{f}_{\Delta}(s_{\nu})B^{s_{\nu}}\Bigr) \frac{\dd\mathbf{s}}{(2\pi
        {\rm i})^{N}}.
    \end{split}
  \end{displaymath}
  We conclude that
  \begin{displaymath}
    \begin{split}
      |N_{\Delta, T}({B})  - N_{\Delta, T, \delta, \lambda} |   \leq &
      \sum_{|\mathbf{g}| \leq T  } \sum_{{\bm \sigma} \in \{0, 1\}^J}
      \Big|\int_{(1)}^{(N)}    \Bigl(\prod_{i, j} \frac{1}{1 -
        2^{-v_{ij}}}\Bigr)    \\
      &\times 2^{-\sum_{ij} \sigma_{ij} v_{ij}} \sum_{ \mathbf{x} \in (\Bbb{Z}
        \setminus\{0\})^J \setminus \widetilde{-{\bm \sigma}} \cdot
        \mathscr{R}_{\delta, \lambda}} \frac{ \chi({\bm \gamma} \cdot
        \mathbf{x})}{{\bm \gamma}^{\mathbf{v}}\mathbf{x}^{\mathbf{v}}}
      \prod_{\nu=1}^N\Bigl( \widehat{f}_{\Delta}(s_{\nu})B^{s_{\nu}}\Bigr)
      \frac{\dd\mathbf{s}}{(2\pi {\rm i})^{N}}\Big|.
    \end{split}
  \end{displaymath}
  Finally we write each factor $(1 - 2^{-v_{ij}})$ as a geometric series and
  apply Mellin inversion to recast the right hand side as
  \begin{displaymath}
    \sum_{|\mathbf{g}| \leq T  } \sum_{{\bm \sigma} \in \{0, 1\}^J}
    \sum_{\mathbf{k} \in \Bbb{N}_0^J} \sum_{ \mathbf{x} \in (\Bbb{Z}
      \setminus\{0\})^J \setminus \widetilde{-{\bm \sigma}} \cdot
      \mathscr{R}_{\delta, \lambda}}  \chi({\bm \gamma} \cdot
    \mathbf{x})F_{\Delta, B}({\bm \gamma} \cdot (\widetilde{\textbf{k} +
      {\bm \sigma}} ) \cdot \mathbf{x}).      
  \end{displaymath}
  Note that any
  $\mathbf{x} \not\in \widetilde{-{\bm \sigma}} \cdot \mathscr{R}_{\delta,
    \lambda}$ in the support of
  $F_{\Delta, B}({\bm \gamma} \cdot (\widetilde{\textbf{k} + {\bm \sigma}} )
  \cdot \mathbf{x})$ satisfies
  $$\min_{ij} |x_{ij}| \leq ( (1+\Delta)B)^{\delta}\quad \text{or} \quad
  \min_{1 \leq i \leq k} \prod_{j = 1}^{J_i} |x_{ij}|^{h_{ij}} \leq
  \Bigl(\max_{1 \leq i \leq k} \prod_{j = 1}^{J_i}
  |2x_{ij}|^{h_{ij}}\Bigr)^{1-\lambda},$$ so that
  $$|N_{\Delta, T}({B})  - N_{\Delta, T, \delta, \lambda} |   \leq 2^J
  \sum_{|\mathbf{g}| \leq T  } \sum_{\mathbf{k} \in \Bbb{N}_0^J}  N_{{\bm
      \gamma},   {\bm \gamma } \cdot \widetilde{\textbf{k}}}((1+\Delta)B,
  ((1+\Delta)B)^{\delta}, \lambda)$$
  by \eqref{Hlambda}.  The lemma follows from \eqref{2a}. Note that
  $\delta_2^{\ast} > 0$ in \eqref{2a} ensures that the $\textbf{k}$-sum
  converges.
\end{proof}

\subsection{Step 3: The error term in the asymptotic formula}
 
We insert Hypothesis~\ref{H1} into \eqref{defNDeltaTdeltalambda}.  For
convenience, we now write
$\Psi_{\mathbf b}(\mathbf X) = N_{\mathbf b}(\mathbf X)- \mathscr E_{\mathbf
  b}\mathscr I_{\mathbf b}(\mathbf X)$.  In this section, we estimate the
contribution of the error $\Psi_{\mathbf b}(\mathbf X)$, which amounts to
bounding
$$E_{\Delta, T, \delta, \lambda} =  \sum_{|\mathbf{g}| \leq T  }  \Bigl|
\int_{(1)}^{(N)} \frac{1}{{\bm \gamma}^{\mathbf{v}}} \Bigl(\prod_{i, j}
\frac{v_{ij}}{1 - 2^{-v_{ij}}}\Bigr)\int_{\mathscr{R}_{\delta, \lambda}}
\frac{\Psi_{{\bm \gamma}^{\ast}}(\mathbf{X})}{ \mathbf{X}^{\mathbf{v}
    +\mathbf{1}}} \dd\mathbf{X} \prod_{\nu=1}^N\Bigl(
\widehat{f}_{\Delta}(s_{\nu})B^{s_{\nu}}\Bigr) \frac{\dd\mathbf{s}}{(2\pi {\rm
    i})^{N}}\Bigr|. $$ For $\mathbf{X} \in \mathscr{R}_{\delta, \lambda}$, we
use \eqref{errorterm} and
$\min X_{ij}^{-\delta \delta_1} \leq \prod_{ij} X_{ij}^{-\delta\delta_1/J}$ to
conclude that
$$\Psi_{{\bm \gamma}^{\ast}}(\mathbf{X}) \ll {\bm \gamma}^{C\textbf{h}}
\Big(\prod_{i=0}^k \prod_{j=1}^{J_i} X_{ij}^{1- h_{ij}\zeta_i + \varepsilon -
  \delta\delta_1/J}\Big).$$
Thus the $\mathbf{X}$-integral is absolutely convergent provided that
\begin{equation}\label{provided}
  \Re v_{ij} > 1 - h_{ij}\zeta_i - \delta\delta_1/J
\end{equation}  
holds for each $i, j$. We now choose appropriate contours for the
$\mathbf{s}$-integral. By \eqref{vs}, the choice
$\Re \mathbf{s} = {\bm \sigma} = (\sigma_{\nu}) \in\Bbb{R}_{>0}^N$ as in
\eqref{1a} is admissible to ensure \eqref{provided}. These contours stay also
to the right of the poles of $\widehat{f}_{\Delta}$ at $s = 0$ (and in fact
inside the validity of \eqref{smooth} and \eqref{useful} if $\delta_3$ is
sufficiently small) and to the right of the poles of $(1 - 2^{-v_{ij}})^{-1}$
at $\Re v_{ij} = 0$ by \eqref{zeta1} if $\delta$ is sufficiently small. By
\eqref{1a}, this ${\bm \sigma}$ satisfies $\sum \sigma_{\nu} = 1$.  We now
shift each $s_{\nu}$-contour to
$\Re s_{\nu} = \sigma_{\nu} - \delta\delta_1/(2JA)$, where
$$A = \max_{ij} \sum_\nu \alpha^\nu_{ij}.$$ Then $ \Re v_{ij} \geq 1 -
h_{ij}\zeta_i - \delta\delta_1/(2J)$ in
accordance with \eqref{provided}, and poles of any $ (1 - 2^{-v_{ij}})^{-1}$
or $\widehat{f}_{\Delta}(s_{\nu})$ remain on the left of the lines of
integration provided that $\delta$ is less than a sufficiently small constant
(it will later tend to zero as $B \rightarrow \infty$). Having shifted the
$\mathbf{s}$-contour in this way, we estimate trivially. The
$\mathscr{R}_{\delta, \lambda}$-integral is $\ll \delta^{-J}$, so that
\begin{equation}\label{error2}
  \begin{split}
    E_{\Delta, T, \delta, \lambda}& \ll \delta^{-J} B^{1
      -\frac{\delta\delta_1N}{2JA} }\sum_{|\mathbf{g}| \leq T  }  {\bm
      \gamma}^{C\textbf{h}}    \int^{(N)} \Big| \langle \textbf{v} \rangle
    \prod_{\nu} \widehat{f}_{\Delta}(s_{\nu}) \Big| \, |\dd\mathbf{s}|\\
    & \ll T^{CS+r} \delta^{-J} B^{1 -\frac{\delta\delta_1N}{2JA} }
    \Delta^{-J+\varepsilon}
  \end{split}
\end{equation}
by \eqref{useful} (which is still applicable if $\delta_3$ is sufficiently
small) with $\mathscr{D} = {\rm id}$, $c = \varepsilon$,
$\| \mathbf{a} \|_1 = J$, where
\begin{equation}\label{S}
  S = \sum_{\rho=1}^r \sum_{(i, j) \in S_{\rho}} h_{ij}. 
\end{equation}

\subsection{Step 4: Inserting the asymptotic formula}\label{54}

We now insert the main term in Hypothesis~\ref{H1} into
\eqref{defNDeltaTdeltalambda}. In order to compute this properly, we re-insert
the cuspidal contribution and replace the range
$\mathscr{R}_{\delta, \lambda}$ of integration with $[1,\infty)^J$. In this
section, we estimate the error
\begin{displaymath}
  E^{\ast}_{\Delta, T, \delta, \lambda} = \sum_{|\mathbf{g}| \leq T }\Bigl|
  \int_{(1)}^{(N)} \frac{1}{{\bm \gamma}^{\mathbf{v}}} \Bigl(\prod_{i, j}
  \frac{v_{ij}}{1 - 2^{-v_{ij}}}\Bigr)\int_{[1, \infty)^J \setminus
    \mathscr{R}_{\delta, \lambda}} \frac{\mathscr{E}_{{\bm \gamma}^{\ast}}
    \mathscr{I}_{{\bm \gamma}^{\ast}}(\mathbf{X}) }{\mathbf{X}^{\mathbf{v} +
      \mathbf{1}}} \dd\mathbf{X} \prod_{\nu=1}^N\Bigl(
  \widehat{f}_{\Delta}(s_{\nu})B^{s_{\nu}}\Bigr) \frac{\dd\mathbf{s}}{(2\pi
    {\rm i})^{N}}\Bigr|.
\end{displaymath}
We interchange the $\mathbf{s}$- and $\mathbf{X}$-integral and compute the
$\mathbf{s}$-integral first. Writing as before each $(1 - 2^{-v_{ij}})^{-1}$
as a geometric series, we obtain
$$ \int_{(1)}^{(N)}  \frac{1}{{\bm \gamma}^{\mathbf{v}}\mathbf{X}^{\mathbf{v}
  } }\Bigl(\prod_{i, j} \frac{v_{ij}}{1 - 2^{-v_{ij}}}\Bigr)
\prod_{\nu=1}^N\Bigl( \widehat{f}_{\Delta}(s_{\nu})B^{s_{\nu}}\Bigr)
\frac{\dd\mathbf{s}}{(2\pi {\rm i})^{N}} = \sum_{\mathbf{k} \in
  \Bbb{N}_0^J}\int_{(1)}^{(N)} (\widetilde{\textbf{k}} \cdot {\bm \gamma}\cdot
\mathbf{X})^{-\mathbf{v}} \langle \textbf{v} \rangle \prod_{\nu=1}^N\Bigl(
\widehat{f}_{\Delta}(s_{\nu})B^{s_{\nu}}\Bigr) \frac{\dd\mathbf{s}}{(2\pi {\rm
    i})^{N}}, $$ and
$\langle \textbf{v} \rangle \prod_{\nu} (
\widehat{f}_{\Delta}(s_{\nu})B^{s_{\nu}} ) $ is a linear combination of terms
of the form
$\prod_{\nu=1}^N s_{\nu}^{a_{\nu}} \widehat{f}_{\Delta}(s_{\nu})B^{s_{\nu}}$
for vectors $\mathbf{a} = (a_{\nu}) \in \Bbb{N}_0^N$ with
$\| \mathbf{a} \|_1 = J$. The inverse Mellin transform of
$s^a \widehat{f}_{\Delta}(s)$ is ${\tt D}^af_{\Delta}$ where ${\tt D}$ is the
differential operator $f(x) \mapsto -x f'(x)$. Hence defining
$$F^{(\mathbf{a})}_{\Delta, B}(\mathbf{x}) =  \prod_{\nu = 1}^{N} {\tt
  D}^{a_{\nu}} f_{\Delta} \left(\frac{|P_{\nu}(\mathbf{x})|}{{B}}\right)$$
with $P_{\nu}$ as in \eqref{defP}, we see that
$E^{\ast}_{\Delta, T, \delta, \lambda}$ is bounded by a linear combination of
terms of the form
\begin{displaymath}
  \begin{split}
    & \sum_{|\mathbf{g}| \leq T  } \int_{[1, \infty)^J \setminus
      \mathscr{R}_{\delta, \lambda}}  \frac{|\mathscr{E}_{{\bm \gamma}^{\ast}}
      \mathscr{I}_{{\bm \gamma}^{\ast}}(\mathbf{X})|}{\langle \textbf{X} \rangle}
    \sum_{\mathbf{k} \in \Bbb{N}_0^J} |F^{(\mathbf{a})}_{\Delta,
      B}(\widetilde{\textbf{k}} \cdot {\bm \gamma}\cdot
    \mathbf{X})|\dd\mathbf{X}\\
    & \ll   \Delta^{-J} \sum_{|\mathbf{g}| \leq T  }{\bm \gamma}^{\textbf{h}}
    \sum_{\mathbf{k} \in \Bbb{N}_0^J}  \int_{[1, \infty)^J \setminus
      \mathscr{R}_{\delta, \lambda}}  \Bigl(\prod_{ij}
    X_{ij}^{-h_{ij}\zeta_i}\Bigl)F_{0, B(1+\Delta)}(\widetilde{\textbf{k}} \cdot
    {\bm \gamma}\cdot \mathbf{X})\dd\mathbf{X}
  \end{split}
\end{displaymath}
by Lemma~\ref{singint}, \eqref{E} and \eqref{smooth0}. By \eqref{continuous}
with $\textbf{b} = (1, \ldots, 1)$,
$\textbf{y} = \widetilde{\textbf{k}} \cdot {\bm \gamma}$ and
$H = ((1+\Delta) B)^{\delta}$, we obtain
\begin{equation}\label{error3}
  E^{\ast}_{\Delta, T, \delta, \lambda} \ll
  T^{S+r} \Delta^{-J}B(\log B)^{c_2+\varepsilon}(\delta + (\log B)^{-1})
\end{equation}
with $S$ as in \eqref{S}. Again $\delta_2^{\ast} > 0$ in \eqref{continuous}
ensures that the $\textbf{k}$-sum converges. Combining Lemma~\ref{lemma2},
\eqref{error2} and \eqref{error3} and choosing
$\delta = (\log B)^{-1+\varepsilon}$, we have shown
\begin{equation}\label{step5}
    N_{\Delta, T}({B}) = N^{(1)}_{\Delta, T}({B}) + O(T^{S+r} \Delta^{-J}B(\log B)^{c_2-1+\varepsilon})
\end{equation}
where
$$N^{(1)}_{\Delta, T}({B}) =  \sum_{|\mathbf{g}| \leq T  } \mu(\mathbf{g})   \int_{(1)}^{(N)}  \frac{1}{{\bm \gamma}^{\mathbf{v}}} \Bigl(\prod_{i, j} \frac{v_{ij}}{1 - 2^{-v_{ij}}}\Bigr)\int_{[1, \infty)^J  }  \frac{\mathscr{E}_{{\bm \gamma}^{\ast}} \mathscr{I}_{{\bm \gamma}^{\ast}}(\mathbf{X}) }{\mathbf{X}^{\mathbf{v} + \mathbf{1}}} \dd\mathbf{X}
\prod_{\nu=1}^N\Bigl( \widehat{f}_{\Delta}(s_{\nu})B^{s_{\nu}}\Bigr)
\frac{\dd\mathbf{s}}{(2\pi {\rm i})^{N}}.$$ We insert Lemma~\ref{todo} and
integrate over $\mathbf{X}$. This gives
\begin{displaymath}
  \begin{split}
    N^{(1)}_{\Delta, T}({B}) = \frac{2^{J^{\ast}}}{\pi} &\sum_{|\mathbf{g}| \leq T  } \mu(\mathbf{g})   \int_{(1)}^{(N)} \int_{\Re z_i = \zeta_i}^{(k-1)}  \frac{ \mathscr{E}_{{\bm \gamma}^{\ast}}}{{\bm \gamma}^{\mathbf{v}} ({\bm \gamma}^{\ast})^{\mathbf{z}}} \Bigl(\prod_{i=1}^{k}  \mathscr{K}_i(z_i)  \prod_{j=1}^{J_i}  \frac{1 - 2^{h_{ij}z_i-1}}{1 - h_{ij} z_i  } \Bigr)   \\
    &\times \Bigl(\prod_{i=0}^k \prod_{j=1}^{J_i} \frac{v_{ij}}{(1 - 2^{-v_{ij}})w_{ij}}\Bigr)    \prod_{\nu=1}^N\Bigl(  \widehat{f}_{\Delta}(s_{\nu})B^{s_{\nu}}\Bigr) \frac{\dd\mathbf{z}}{(2\pi {\rm i})^{k-1}}
    \frac{\dd\mathbf{s}}{(2\pi {\rm i})^{N}}.
  \end{split}
\end{displaymath}
where $w_{ij} = v_{ij} + h_{ij} z_i -1$ and we recall our convention
$z_k = 1 - z_1 - \dots - z_{k-1}$. If we write
$\mathbf{w} = (w_{ij}) \in \Bbb{C}^J$, then by \eqref{vs} and \eqref{A2new},
we have
\begin{equation}\label{zast}
  \mathbf{w} = \mathscr{A}_1 \mathbf{s} + \mathscr{A}_2 \mathbf{z}^{\ast}, \quad \mathbf{z}^{\ast} = (z_1, \ldots, z_{k-1}, 1).
\end{equation}
This explains the seemingly artificial definition of $\mathscr{A}_2$.  We can
simplify this first by recalling the definition \eqref{gammaast} of
${\bm \gamma}^{\ast}$, which implies
${\bm \gamma}^{\mathbf{v}} ({\bm \gamma}^{\ast})^{\mathbf{z}} = {\bm
  \gamma}^{\mathbf{w} + \mathbf{1}}.$ Next we use our convention $h_{0j} = 0$
and insert a redundant factor
$2^{J_0} \prod_{j=1}^{J_0} (1 - 2^{h_{0j}z_0 - 1})$. We also write
$\kappa = k-1$. In this way, we can recast $ N^{(1)}_{\Delta, T}({B})$ as
\begin{displaymath}
  \frac{2^J}{\pi} \sum_{|\mathbf{g}| \leq T } \mu(\mathbf{g})
  \int_{(1)}^{(N)} \int_{\Re z_i = \zeta_i}^{(\kappa)} \frac{
    \mathscr{E}_{{\bm \gamma}^{\ast}}}{{\bm \gamma}^{\mathbf{w} +
      \mathbf{1}} } \Bigl(\prod_{i=1}^{k} \mathscr{K}_i(z_i)\Bigr)
  \frac{1}{\langle \textbf{w}\rangle }
  \frac{\phi(\mathbf{v})}{\phi(\mathbf{v} - \mathbf{w})}
  \prod_{\nu=1}^N\Bigl( \widehat{f}_{\Delta}(s_{\nu})B^{s_{\nu}}\Bigr)
  \frac{\dd\mathbf{z}}{(2\pi {\rm i})^{\kappa}} \frac{\dd\mathbf{s}}{(2\pi
    {\rm i})^{N}}
\end{displaymath}
where 
\begin{equation}\label{defphi}
  \phi(\mathbf{v}) = \prod_{i=0}^k \prod_{j=1}^{J_i}  \frac{v_{ij}}{1 - 2^{-v_{ij}}}.
\end{equation}

\subsection{Step 5: Contour shifts}\label{peyre}

In this section, we evaluate asymptotically $N^{(1)}_{\Delta, T}({B})$ by
contour shifts. Let ${\bm \sigma} = (\sigma_{\nu}) \in \Bbb{R}_{>0}^N$ be as
in \eqref{1a}. For some small $\varepsilon > 0$, we shift the
$\mathbf{s}$-contour to $\Re s_{\nu} = \sigma_{\nu} + \varepsilon$ without
crossing any poles. Shifting a little further to the left will pick up the
poles at $\mathbf{w} = 0$, whose residues produce the main term for $N(B)$. To
make this transparent, we make a change of variables as follows.

By \eqref{1c} we have
$\text{rk}(\mathscr{A}) =\text{rk}(\mathscr{A}_1\, \mathscr{A}_2) = R$, so we
can choose $R$ linearly independent members of the linear forms $w_{ij}$ in
$\mathbf{s}$ and $\mathbf{z}^{\ast} = (z_1, \ldots, z_{k-1}, 1)$, say
$w^{(1)}, \ldots, w^{(R)}$, and then the remaining $w_{ij}$ are linearly
dependent. Since also $\text{rk}(\mathscr{A}_1) = R$, we may, for fixed
$\mathbf{z}$, change variables in the $\mathbf{s}$-integral by completing the
$R$ functions $w^{(1)} , \ldots, w^{(R)} $ to a basis in any way such that the
determinant of the Jacobian is $\pm 1$. We call the new variables
$\mathbf{y} = (y_1, \ldots, y_N)$.

We can describe this also in terms of matrices. We pick a maximal linearly
independent set of $R$ rows $Z_{1}, \ldots, Z_{R}$ of the matrix
$(\mathscr{A}_1\, \mathscr{A}_2)$.  Let $Z_{R+1}, \ldots, Z_{J}$ denote the
remaining rows of $(\mathscr{A}_1\, \mathscr{A}_2)$ and let
$\mathscr{B} = (b_{kl}) \in \Bbb{R}^{(J-R) \times R}$ be the unique matrix
satisfying
\begin{equation}\label{beta}
  \mathscr{B}  \left(\begin{smallmatrix}  Z_{1}  \\ \vdots \\   Z_{R} 
    \end{smallmatrix}\right) = \left(\begin{smallmatrix} Z_{R+1} \\ \vdots \\
      Z_{J} \end{smallmatrix}\right).
\end{equation}
That is, $\mathscr{B}$ expresses the remaining $w_{ij}$ in terms of the
selected linearly independent set. Again by \eqref{1c}, we can also write the
last row $(\mathscr{A}_3\, \mathscr{A}_4)$ of $\mathscr{A}$ as a linear
combination of $Z_1, \ldots, Z_R$, say
\begin{equation}\label{beta0}
 \sum_{\ell = 1}^R   b_{\ell} Z_{\ell} =  (\mathscr{A}_3\, \mathscr{A}_4).
\end{equation}
The coefficients $b_{kl}$ and $b_{\ell}$ play the same role as in Lemma~\ref{lem19}. 
Choose a matrix  
\begin{equation}\label{mathcalC}
  \mathscr{C} = (\mathscr{C}_1\, \mathscr{C}_2) = \left(\begin{smallmatrix}
      Z_{1}  \\ \vdots \\   Z_{R}  \\  \boxed{  \,\,\,\,
        \begin{smallmatrix} \\ \ast \\  \\\end{smallmatrix}
        \,\,\,\, }\boxed{  \,\,\,\,
        \begin{smallmatrix} \\ 0 \\  \\\end{smallmatrix}
        \,\,\,\,  } \end{smallmatrix}\right)
  \in \Bbb{R}^{N \times (N+k)} ,  \quad
  (\mathscr{C}_1 \in \Bbb{R}^{N \times N}, \mathscr{C}_2 \in \mathscr{R}^{N \times k}),  
\end{equation}
with $\boxed{\ast} \in \Bbb{R}^{(N-R)\times N}$ chosen such that
$\mathscr{C}_1 \in \Bbb{R}^{N \times N}$ satisfies $\det \mathscr{C}_1 =
1$. This is possible since $\text{rk}(\mathscr{A}_1) = R$ by \eqref{1c}. Given
$\mathbf{s} \in \Bbb{C}^N$, $\mathbf{z} \in \Bbb{C}^{k-1}$, we define the
vector
\begin{equation}\label{defy}
  (y_1, \ldots, y_N)^{\top} 
  =  \mathbf{y} = \mathbf{y}(\mathbf{s}, \mathbf{z}^{\ast}) = \mathscr{C}
  (\mathbf{s}, \mathbf{z}^{\ast})^{\top}  = \mathscr{C}_1\mathbf{s} ^{\top}+
  \mathscr{C}_2 {\mathbf{z}^{\ast}}^{\top} .
\end{equation} 
We write 
\begin{equation*}
  {\bm \eta} = \mathbf{y}({\bm \sigma}, (\zeta_1, \ldots, \zeta_{k-1}, 1)) \in
  \Bbb{R}^N, \quad {\bm \eta}^{\ast} = \mathbf{y}({\bm \sigma} +\varepsilon
  \cdot \mathbf{1}, (\zeta_1, \ldots, \zeta_{k-1}, 1)) \in \Bbb{R}^N
\end{equation*}
with ${\bm \sigma}$ as in \eqref{1a} and some fixed $\varepsilon > 0$. In the
new variables $\mathbf{y}$, the path of integration
$\Re s_{\nu} = \sigma_{\nu} + \varepsilon$ becomes
$\Re y_{\nu} = \eta^{\ast}_{\nu}$.  Moreover, by \eqref{beta} and
\eqref{beta0}, we have
\begin{equation}\label{linear1}
  \langle \textbf{w} \rangle  =y_1 \cdots y_{R} \prod_{\iota=1}^{J-R}
  \mathscr{L}_\iota(\mathbf{y}), \quad 
  \mathscr{L}_{\iota}(\mathbf{y}) = \sum_{\ell=1}^Rb_{\iota\ell} y_{\ell}
\end{equation}
and 
\begin{equation}\label{linear2}
  -1 + \sum_{\nu=1}^N s_{\nu} = \mathscr{L}(\mathbf{y}), \quad 
  \mathscr{L}(\mathbf{y}) = \sum_{\ell=1}^Rb_{\ell} y_{\ell}.
\end{equation}
Thus  we can recast $ N^{(1)}_{\Delta, T}({B})$ as 
\begin{equation}\label{reca}
  \begin{split}
    \frac{2^J}{\pi} \sum_{|\mathbf{g}| \leq T  } \mu(\mathbf{g}) \int_{\Re z_i
      = \zeta_i}^{(\kappa)} &   \int_{\Re y_{\nu} = \eta_{\nu}^{\ast}}^{(N)}
    \frac{ \mathscr{E}_{{\bm \gamma}^{\ast}}}{{\bm \gamma}^{\mathbf{w} +
        \mathbf{1}}  }\frac{\phi(\mathbf{v})}{\phi(\mathbf{v}  - \mathbf{w})}
    \Bigl(  \prod_{\nu=1}^N
    \widehat{f}_{\Delta}(s_{\nu})\Bigr)\Bigl(\prod_{i=1}^{k}
    \mathscr{K}_i(z_i)\Bigr) \\
    & \times \frac{B^{1+\mathscr{L}(\mathbf{y})}}{y_1 \cdots y_{R}
      \prod_{\iota=1}^{J-R} \mathscr{L}_\iota(\mathbf{y}) }
    \frac{\dd\mathbf{y}}{(2\pi {\rm i})^{N}} \frac{\dd\mathbf{z}}{(2\pi {\rm
        i})^{\kappa}},
  \end{split}
\end{equation}
where now $\mathbf{s}, \mathbf{v}, \mathbf{w}$ are linear forms in
$\mathbf{y}, \mathbf{z}^{\ast}$ given by \eqref{vs}, \eqref{zast},
\eqref{beta} and \eqref{defy}. We now shift the $y_1, \ldots, y_{R}$-contours
appropriately within a sufficiently small $\varepsilon$-neighborhood of
${\bm \eta}$ (in which in particular
$ \phi(\mathbf{v})/\phi(\mathbf{v} - \mathbf{w}) \prod_{\nu}
\widehat{f}_{\Delta}(s_{\nu}) $ is holomorphic), always keeping
$\Re z_i = \zeta_i$. Recalling definitions \eqref{defphi} and \eqref{Ki} as
well as $\textbf{v} - \textbf{w} = (1 - h_{ij} z_{ij})_{ij}$, we record the
bound
\begin{equation}\label{residuebound}
  \begin{split}
    \mathscr{D}\Bigg( \frac{ \mathscr{E}_{{\bm \gamma}^{\ast}}}{{\bm
        \gamma}^{\mathbf{w}+\mathbf{1}} } \Bigl(\phi(\mathbf{v})
    \prod_{\nu=1}^N \widehat{f}_{\Delta}(s_{\nu})\Bigr)&
    \Bigl(\frac{1}{\phi(\mathbf{v} - \mathbf{w})}\prod_{i=1}^{k}
    \mathscr{K}_i(z_i)\Bigr) \Bigg)
    \ll T^S \Delta^{-J-c}|\mathbf{s}|^{-c}_{\infty}
    \Bigl(\prod_{i=1}^k |z_i|^{\zeta_i-\frac{1}{2} -J_i+\varepsilon} \Bigr)\\
    & = T^S\Delta^{-J-c} \Bigl(\prod_{i=1}^k |z_i|^{\zeta_i-\frac{1}{2}
      -J_i+\varepsilon} \Bigr)\big|\mathscr{C}_1^{-1}\mathbf{y} -
    \mathscr{C}_1^{-1}(\mathscr{C}_2\mathbf{z}^{\ast} )\big|^{-c}_{\infty}
  \end{split}
\end{equation}
that holds for any fixed linear differential operator $\mathscr{D}$ with
constant coefficients in $s_1, \ldots, s_{N}, z_1, \ldots, z_{k-1}$ and any
fixed $c > 0$. This follows from Stirling's formula, \eqref{useful},
\eqref{E} and \eqref{S}. In particular, choosing $c > N$ and recalling
\eqref{J-cond}, this expression is absolutely integrable over $\mathbf{z}$ and
$\mathbf{y}$. We return to \eqref{reca} and evaluate the
$(y_1, \ldots, y_R)$-integral asymptotically by appropriate contour
shifts. The integrals that arise are of the form
\begin{equation*}
  B (\log B)^{\alpha_0} \int^{(R)} \frac{B^{\ell(\tilde{\textbf{y}})}
    H(\tilde{\textbf{y}})}{\ell_1(\tilde{\textbf{y}}) \cdots
    \ell_{J_0}(\tilde{\textbf{y}}) }
  \frac{{\rm d}\tilde{\textbf{y}}}{(2\pi i)^{R_0}}
\end{equation*}
where $\alpha_0 \in \Bbb{N}_0$, $\ell_1, \ldots, \ell_{J_0}$ are linear forms
in $R_0$ variables spanning a vector space of dimension $R_0$, $\ell$ is a
linear form, the contours of integration are in an $\varepsilon$-neighborhood
of $\Re y_{\nu} = 0$ and $H$ is a holomorphic function in this region
satisfying the bound \eqref{residuebound}; initially we have $R_0 = R$,
$J_0 = J$, $\alpha_0 = 0$. As long as $\Re \ell(\tilde{\textbf{y}}) > 0$, we
can shift one of the variables to the left (if appearing with positive
coefficient) or to the right (if appearing with negative coefficient), getting
a small power saving in $B$ in the remaining integral and picking up the
residues on the way. Inductively we see that in each step
$J_0 - R_0 +\alpha_0$ is nonincreasing. Recalling the definition of $c_2$ in
\eqref{defc2}, we obtain eventually
\begin{equation}\label{N1}
 \begin{split}
   N^{(1)}_{\Delta, T}({B}) = & c^{\ast} c_{\text{fin}}(T) c_{\infty}(\Delta)B
   (\log B)^{c_2} + O(T^{S+r+\varepsilon} \Delta^{-J-N-\varepsilon}B(\log
   B)^{c_2-1})
 \end{split}
\end{equation}
for  some constant $c^{\ast} \in \Bbb{Q}$ (to be computed in a moment) and 
\begin{equation}\label{defcinf}
  \begin{split}
    &c_{\text{fin}}(T) =   \sum_{|\mathbf{g}| \leq T  } \mu(\mathbf{g})
    \frac{ \mathscr{E}_{{\bm \gamma}^{\ast}}}{\langle {\bm \gamma} \rangle  }
    , \\
    &c_{\infty}(\Delta) =\frac{2^J}{\pi} \int_{\Re z_i = \zeta_i}^{(\kappa)}
    \int_{\Re y_{\nu} = \eta^{\ast}_{\nu}}^{(N-R)} \Bigl( \prod_{\nu=1}^N
    \widehat{f}_{\Delta}(s_{\nu})|_{y_1 = \dots = y_R = 0}\Bigr)
    \Bigl(\prod_{i=1}^{k} \mathscr{K}_i(z_i)\Bigr) \frac{\dd y_{R+1} \cdots
      \dd y_N}{(2\pi {\rm i})^{N-R}} \frac{\dd\mathbf{z}}{(2\pi {\rm
        i})^{\kappa}}.
  \end{split}
\end{equation}
That the multiple integral in the formula for $c_{\infty}(\Delta)$ is
absolutely convergent follows again from \eqref{residuebound}.  Combining
\eqref{N1} with \eqref{error1} and \eqref{step5}, we have shown
\begin{equation}\label{eventually}
  N_{\Delta }(B) =  c^{\ast}  c_{\text{fin}} (T) c_{\infty}(\Delta) B (\log B)^{c_2}
  + O\big(B(\log B)^{c_2-1+\varepsilon}
  (T^{S+r} \Delta^{-J-N-\varepsilon} + T^{-\delta_2}\log B)\big)
\end{equation}
for any $1 < T  < B$.  

\subsection{Step 6: Computing the leading constant}

We proceed to compute explicitly the leading constant in
\eqref{eventually}. In this subsection, we consider $c^{\ast}$ and
$c_{\text{fin}}(T)$, and we start with the former. To this end, we observe
that in the course of the contour shifts, only the polar behavior at
$\textbf{w}=0$ is relevant, so that
$$c^{\ast} = \lim_{B \rightarrow \infty} \frac{1}{(\log B)^{c_2}}
\int^{(R)} B^{\mathscr{L}(y)} \prod_{\ell=1}^R F(y_\ell) \prod_{\iota=1}^{J-R}
\mathscr{L}_{\iota}(\textbf{y})^{-1} \frac{\dd \textbf{y}}{(2\pi {\rm i})^R}$$
for any function $F$ that is holomorphic except for a simple pole at $0$ with
residue 1, provided the integral is absolutely convergent. We choose
$F = \widehat{f}_{\Delta_0}$ for some $\Delta_0 > 0$ as in
\eqref{smooth0}--\eqref{smooth}, recall the notation
\eqref{linear1}--\eqref{linear2}, and insert the formula
$s^{-1} = \int_0^1 t^{s-1} \dd t$ for $\Re s > 0$. In this way we get the
absolutely convergent expression
\begin{displaymath}
  \begin{split}
    c^{\ast} & = \lim_{B \rightarrow \infty} \frac{1}{(\log B)^{c_2}} \int^{(R)}
    B^{\mathscr{L}(y)} \prod_{\ell=1}^R \widehat{f}_{\Delta_0}(y_\ell) \int_{[0,
      1]^{J-R}} \prod_{ \iota=1}^{J-R}
    t_{\iota}^{\mathscr{L}_{\iota}(\textbf{y})-1} \dd \textbf{t} \,
    \frac{\dd \textbf{y}}{(2\pi {\rm i})^R}\\
    & = \lim_{B \rightarrow \infty} \int^{(R)} B^{\mathscr{L}(y)}
    \prod_{\ell=1}^R \widehat{f}_{\Delta_0}(y_\ell) \int_{[0, \infty]^{J-R}}
    \prod_{\iota=1}^{J-R} B^{ -r_{\iota} \mathscr{L}_{\iota}(\textbf{y})} \dd
    \textbf{r}
    \, \frac{\dd \textbf{y}}{(2\pi {\rm i})^R}\\
    & = \lim_{B \rightarrow \infty} \int_{[0, \infty]^{J-R}} \int^{(R)} \Big(
    \prod_{\ell=1}^R \widehat{f}_{\Delta_0}(y_\ell) \Big)B^{\sum_\ell (b_{\ell}
      -\sum_{\iota} r_{\iota} b_{\iota \ell} )y_{\ell} } \frac{\dd
      \textbf{y}}{(2\pi {\rm i})^R}\, \dd \textbf{r} \\
    & = \lim_{B \rightarrow \infty} \int_{[0, \infty]^{J-R}} \prod_{\ell=1}^R
    f_{\Delta_0}\big(B^{ - b_{\ell} +\sum_{\iota} r_{\iota} b_{\iota \ell} }
    \big) \dd \textbf{r}.
  \end{split}
\end{displaymath}
Here we used a change of variables along with $c_2 = J-R$ in the first step,
cf.\ \eqref{defc2}, and Mellin inversion in the last step.  This formula holds
for every $\Delta_0 > 0$, so we can take the limit $\Delta_0\rightarrow 0$
getting
\begin{equation}\label{cast-final}
  c^{\ast}  =   \text{vol}\Big\{ \textbf{r} \in [0, \infty]^{J-R}  :
  b_{\ell}  -\sum_{\iota=1}^{J-R}  r_{\iota}  b_{\iota \ell}   \geq 0
  \text{ for all }  1 \leq \ell \leq R \Big\} . 
\end{equation}

Next we investigate $c_{\text{fin}}(T)$.  We can complete the $\mathbf{g}$-sum
at the cost of an error
$$\sum_{|\mathbf{g}| > T  }\Bigl|  \frac{ \mathscr{E}_{{\bm
      \gamma}^{\ast}}}{\langle {\bm \gamma} \rangle   } \Bigr| \ll
\sum_{\mathbf{g}  }  \Bigl( \prod_{ij} \gamma_{ij}^{-1+h_{ij}\beta_i}
\Bigr)\Bigl(\frac{|\mathbf{g}|}{T}\Bigr)^{\delta_4-\varepsilon } \ll
T^{-\delta_4+\varepsilon} $$
by \eqref{E}, \eqref{gamma}, \eqref{gammaast}, \eqref{1b} and Lemma~\ref{kgV},
so that
\begin{equation}\label{cfin-final}
  c_{\text{fin}}(T) =   c_{\text{fin}} + O(T^{-\delta_4+\varepsilon }),
  \quad c_{\text{fin}} = \sum_{\mathbf{g}   } \mu(\mathbf{g})
  \frac{ \mathscr{E}_{{\bm \gamma}^{\ast}}}{\langle {\bm \gamma} \rangle }.
\end{equation}
Using \eqref{localdensities}, we can rewrite $ c_{\text{fin}} $ in terms of
local densities (note that the sum is absolutely convergent). Recall that
$\textbf{g} = (g_1, \ldots, g_r)$ is indexed by the coprimality conditions
$S_1, \ldots, S_r$ in \eqref{gcd}. For a given choice of
$\alpha_1, \ldots, \alpha_r \in \{0, 1\}$, let
$$S(\bm \alpha) = \bigcup _{\alpha_{\rho} = 1} S_{\rho}, \quad \delta(ij, {\bm
  \alpha}) = \begin{cases} 1, & (i, j) \in S(\bm \alpha),\\ 0, & (i, j)
  \not\in S(\bm \alpha). \end{cases}$$ Then
$$c_{\text{fin}} = \prod_p \sum_{{\bm \alpha} \in \{0, 1\}^r} \frac{
  (-1)^{|{\bm \alpha}|_1} }{ p^{\#S(\bm \alpha) }} \cdot  \lim_{L \rightarrow
  \infty}\frac{1}{p^{L (J-1)}} \#\Big\{\textbf{x}  \bmod{ p^{L}} :
\sum_{i=1}^k \prod_{j=1}^{J_i} (p^{\delta(ij, {\bm \alpha})} x_{ij})^{h_{ij}}
\equiv 0 \bmod{p^{L}}\Big\}.$$
By inclusion-exclusion, this equals
\begin{equation}\label{cfin-final-euler}
  c_{\text{fin}}=   \prod_p \lim_{L \rightarrow \infty}\frac{1}{p^{L (J-1)}}
  \#\Bigg\{\textbf{x}  \bmod{ p^{L}} :
  \begin{array}{l}
    \displaystyle \sum_{i=1}^k \prod_{j=1}^{J_i}
    x_{ij}^{h_{ij}} \equiv 0 \bmod{p^{L}},\\
    (\{x_{ij} : (i, j) \in S_{\rho}\}, p) = 1 \text{ for } 1 \leq \rho \leq r
  \end{array}\Bigg\}.
\end{equation}
Combining \eqref{eventually} and \eqref{cfin-final}, we conclude 
\begin{equation*}
  N_{\Delta}(B) =  c^{\ast}  c_{\text{fin}}  c_{\infty}(\Delta) B (\log B)^{c_2}
  + O\Big(B(\log B)^{c_2-1-\delta_0} \Delta^{-J-N-\varepsilon}\Big)
\end{equation*}
for $\delta_0 = \min(\delta_2, \min(\delta_4, 1)(S+r+1)^{-1})> 0$, upon
choosing $T = (\log B)^{1/(S + r + 1)}$.  Since $N_{\Delta}(B)$ is obviously
nonincreasing in $\Delta$, we conclude from \eqref{sandwich} and the previous
display that
$N(B) =(1+o(1)) c^{\ast} c_{\text{fin}} c_{\infty} B (\log B)^{c_2}$ as
$B\rightarrow \infty$ with
\begin{equation}\label{cinftylimit}
  c_{\infty} = \lim_{\Delta \rightarrow 0}c_{\infty}(\Delta),
\end{equation}   
and this limit must exist.  We have proved

\begin{theorem}\label{analytic-theorem}
  Suppose that we are given a diophantine equation \eqref{torsor} with
  $b_1=\dots=b_k=1$ and height conditions \eqref{height} whose variables are
  restricted by coprimality conditions \eqref{gcd}. Suppose that
  Hypotheses~\ref{H1} and \ref{H2} and \eqref{1c}, \eqref{1a}, \eqref{1b},
  \eqref{J-cond} hold. Then we have the asymptotic formula
  \begin{equation}\label{rough}
    N(B) =(1+o(1)) c^{\ast}  c_{\text{{\rm fin}}}  c_{\infty}  B (\log B)^{c_2},
    \quad B \rightarrow \infty.
  \end{equation}
  Here $c^{\ast}$ is given in \eqref{cast-final} (using the notation
  \eqref{linear1}--\eqref{linear2}), $c_{\text{{\rm fin}}}$ in
  \eqref{cfin-final-euler}, $c_{\infty}$ in \eqref{cinftylimit} and
  \eqref{defcinf}, and $c_2$ in \eqref{defc2}.
\end{theorem}

More precisely, we need \eqref{E} of Hypothesis \ref{H1} only for $\textbf{b} = {\bm \gamma}^{\ast}$ and \eqref{2a} of Hypothesis \ref{H2} only for $\textbf{b} = {\bm \gamma}$. 

\section{The Manin--Peyre conjecture}\label{sec9}

In Sections~\ref{dioph}--\ref{sec8}, we established an asymptotic formula for
a certain counting problem, subject to several hypotheses. By design, we
presented this in an axiomatic style without recourse to the underlying
geometry. In the section, we relate the asymptotic formula in
Theorem~\ref{analytic-theorem} to the Manin--Peyre conjecture.  In particular,
we compute $c_{\infty}$ explicitly, and we will show (under conditions that
are easy to check) that the leading constant
$ c^{\ast} c_{\text{{\rm fin}}} c_{\infty} $ agrees with Peyre's constant for
almost Fano varieties as in Part~\ref{part1}. This applies in
particular to the spherical Fano varieties in Part~\ref{part3} of the paper.

\subsection{Geometric interpretation of $c_{\infty}$}

In this subsection, we establish the following alternative formulation of the
constant $c_{\infty}$.  Recall -- cf.\ \eqref{mathcalC} -- that the first $R$
rows of $\mathscr{C} = (\mathscr{C}_1 \mathscr{C}_2)$ are $R$ linearly
  independent  rows of $(\mathscr{A}_1 \mathscr{A}_2)$, let's say indexed by
a set $I$ of pairs $(i, j)$ with $0 \leq i \leq k$, $1 \leq j \leq J_i$ with
$|I| = R$. Let
\begin{equation}\label{def-cap-Phi}
\Phi^{\ast}(\textbf{t}) =  \sum_{i = 1}^k\prod_{(i, j) \in I}  t_{ij}^{h_{ij}},
\end{equation}
and let $\mathscr{F}$ be the affine $(R-1)$-dimensional hypersurface
$\Phi^{\ast}(\textbf{t}) = 0$ over $\Rd$. Let $\chi_I$ be the characteristic
function on the set
$$\prod_{(i, j) \in I} |t_{ij}|^{\alpha^\mu_{ij}} \leq 1, \quad 1 \leq \mu \leq N.$$
In order to avoid technical difficulties that are irrelevant for the
applications we have in mind, we make the simplifying assumption that
\begin{equation}\label{simplifying}
  \text{one of the $k$ monomials in $\Phi^{\ast}$ consists of only one
    variable,  which has   exponent 1.}
\end{equation}  
Without loss of generality, we can assume that this is the first
monomial. (Assumption \eqref{simplifying} can be removed if necessary and
follows from assumption \eqref{eq:assumption_real_density_strong}.)

\begin{lemma}
  Suppose that $\{(1, j) \in I\} = \{(1, 1)\}$ and $h_{11} = 1$.  Then
  $c_{\infty}$ is given by the surface integral
  \begin{equation}\label{cinf-final}
    c_{\infty}    = 2^{J-R} \int_{\mathscr{F}} \frac{ \chi_I(\textbf{t}) }{ \|
      \nabla \Phi^{\ast}(\textbf{t})\|}\, {\rm d}\mathscr{F}\textbf{t}.
  \end{equation}
\end{lemma}

\begin{proof}  
  We return to the definition \eqref{defcinf} of $c_{\infty}(\Delta)$ and
  compute the $\textbf{y}$-integral for fixed $\textbf{z}$.  Let us write
  $\widehat{F}(\textbf{y}) = \prod_{\nu=1}^N
  \widehat{f}_{\Delta}(s_{\nu})$. We recall from \eqref{defy} that
  $ \mathbf{y} =\mathscr{C}_1\mathbf{s} + \mathscr{C}_2 \mathbf{z}^{\ast} $
  with $\det \mathscr{C}_1 = 1$, and we view $\textbf{s}$ as a function of
  $\textbf{y}$ (for fixed $\textbf{z}$).  By Mellin inversion one confirms the
  formula
  $$  \int_{\Re y_{\nu} = \eta^{\ast}_{\nu}}^{(N-R)}   \widehat{F}(0, \ldots,
  0, y_{R+1},\ldots y_N) \frac{\dd y_{R+1} \cdots \dd y_N}{(2\pi {\rm
      i})^{N-R}} = \int_{\Bbb{R}_{>0}^R} \int^{(N)}_{\Re y_{\nu} =
    \eta^{\ast}_{\nu}} \widehat{F}(\textbf{y}) t_1^{y_1} \cdots t_R^{y_R}
  \frac{\dd \textbf{y}}{(2\pi {\rm i})^{N}} \frac{\dd\textbf{t}}{\langle
    \textbf{t} \rangle} .$$ Note that by Mellin inversion, the
  $\textbf{t}$-integral on the right hand side is absolutely convergent, even
  though the combined $\textbf{y}, \textbf{t}$-integral is not. (This formula
  is a distributional version of the ``identity''
  $\int_0^{\infty} t^{y-1} \dd t = \delta_{y=0}$.)  Let us write
  $\mathscr{C} = (\mathscr{C}_1\, \mathscr{C}_2) = (c_{\nu \mu})\in \Bbb{R}^{N
    \times (N+k)}$ and
  $\mathscr{C}_2\textbf{z}^{\ast} = \tilde{\textbf{z}} \in \Bbb{C}^N$. We
  change back to $\textbf{s}$-variables and compute the $\textbf{s}$-integral
  in the preceding display by Mellin inversion, getting
  $$ \int_{\Bbb{R}_{>0}^R}  \prod_{\mu= 1}^Nf_{\Delta} \Big(\prod_{\ell =
    1}^Rt_{\ell}^{-c_{\ell, \mu} } \Big) t_1^{\tilde{z}_1} \cdots
  t_R^{\tilde{z}_R} \frac{\dd\textbf{t}}{\langle \textbf{t} \rangle} .$$ By
  construction this integral is absolutely convergent for every fixed
  $\textbf{z}$ with $\Re z_i = \zeta_i$. Plugging back into the definition, we
  obtain
  $$c_{\infty}(\Delta) = \frac{2^J}{\pi}  \int^{(\kappa)}_{\Re z_i = \zeta_i}
  \prod_{i=1}^{k} \mathscr{K}_i(z_i) \int_{\Bbb{R}_{>0}^R}\prod_{\mu=
    1}^Nf_{\Delta} \Big(\prod_{\ell = 1}^Rt_{\ell}^{-c_{\ell, \mu} } \Big)
  t_1^{\tilde{z}_1} \cdots t_R^{\tilde{z}_R} \frac{\dd\textbf{t}}{\langle
    \textbf{t} \rangle} \frac{\dd\mathbf{z}}{(2\pi {\rm i})^{\kappa}} .$$ Here
  the $\textbf{z}$-integral is absolutely convergent since the multiple
  integral in \eqref{defcinf} was absolutely convergent. The combined
  $\textbf{t}, \textbf{z}$-integral, however, is not absolutely convergent.
  Recall that $\kappa = k-1$, $z_k = 1 - z_1 - \dots - z_{\kappa}$ and
  $\mathscr{K}_i(z)$ was defined in \eqref{Ki} with inverse Mellin transform
  $x \mapsto K_i(x)$, say, where $K_i(x) = \cos(x)$ or $\exp({\rm i} x)$.  In
  order to avoid convergence problems, we define, for $\varepsilon > 0$, the
  function
  \begin{equation}\label{defKiepsilon}
    K_i^{(\varepsilon)}(x) = K_i(x) e^{-(\varepsilon x)^2} =
    \begin{cases}
      \cos(x)e^{-(\varepsilon x)^2}  , & h_{ij} \text{ odd for some } 1 \leq j \leq J_i,\\
      e^{ix}e^{-(\varepsilon x)^2} ,  &h_{ij} \text{ even for all } 1 \leq j
      \leq J_i.
    \end{cases}
  \end{equation}
  and its Mellin transform
  $\mathscr{K}^{(\varepsilon)}_i(z) = \int_0^{\infty} K^{(\varepsilon)}_i(x)
  x^{z-1} \dd x$. This can be expressed explicitly in terms of confluent
  hypergeometric functions by \cite[3.462.1]{GR}, but we do not need this. It
  suffices to know that $\mathscr{K}^{(\varepsilon)}_i(z)$ is holomorphic in
  $\Re z > 0$, rapidly decaying on vertical lines, and we have the pointwise
  limit
  $\lim_{\varepsilon \rightarrow 0} \mathscr{K}^{(\varepsilon)}_i(z) =
  \mathscr{K}_i(z)$ for $0 < \Re z < 1$. The latter follows elementarily with
  one integration by parts by writing
  $$\int_0^{\infty} (K_i(x) - K_i^{(\varepsilon)}(x)) x^{z-1} {\dd x} =
  \int_0^{\varepsilon^{-1/2}} + \int_{{\varepsilon^{-1/2}}}^{\infty} \ll
  \varepsilon^{1/2} + \varepsilon^{1/2} \rightarrow 0$$ for
  $\varepsilon \rightarrow 0$. Correspondingly we write
  \begin{equation*}
    c^{(\varepsilon)}_{\infty}(\Delta) = \frac{2^J}{\pi}  \int^{(\kappa)}_{\Re
      z_i = \zeta_i} \prod_{i=1}^{k}  \mathscr{K}^{(\varepsilon)}_i(z_i)
    \int_{\Bbb{R}_{>0}^R}\prod_{\mu= 1}^Nf_{\Delta} \Big(\prod_{\ell =
      1}^Rt_{\ell}^{-c_{\ell, \mu} }  \Big)  t_1^{\tilde{z}_1} \cdots
    t_R^{\tilde{z}_R} \frac{\dd\textbf{t}}{\langle \textbf{t} \rangle}
    \frac{\dd\mathbf{z}}{(2\pi {\rm i})^{\kappa}}.
  \end{equation*}
  This multiple integral is now absolutely convergent, and by dominated convergence we have 
  \begin{equation}\label{dominatedc}
    c_{\infty}(\Delta) = \lim_{\varepsilon \rightarrow 0} c^{(\varepsilon)}_{\infty}(\Delta).
  \end{equation}
  We interchange the $\textbf{t}$- and $\textbf{z}$-integral, fix $\textbf{t}$
  and compute the $\textbf{z}$-integral.  Mellin inversion yields
  \begin{equation*}
    \mathscr{K}^{(\varepsilon)}_k(1 - z_1 - \dots - z_{\kappa}) =
    \int_0^{\infty} \int_{(\frac{1}{2}\zeta_k)}
    \mathscr{K}^{(\varepsilon)}_k(z_k) x^{ - z_1 - \dots - z_{k}}  \frac{\dd
      z_k}{2\pi {\rm i}} \dd x
  \end{equation*}
  for $\Re z_i = \zeta_i$, $1\leq i \leq \kappa$. Note that on the right hand
  side $\Re(z_1 + \dots + z_k) < 1$ (which is why we chose
  $\Re z_k = \frac{1}{2} \zeta_k$).  Again the double integral is not
  absolutely convergent, but the $x$-integral is absolutely convergent. In
  particular, after substituting this into the definition of
  $c^{(\varepsilon)}_{\infty}(\Delta)$, we may interchange the $x$-integral
  and the $z_1, \ldots, z_{\kappa}$-integral to conclude
  $$c^{(\varepsilon)}_{\infty}(\Delta) = \frac{2^J}{\pi} \int_{\Bbb{R}_{>0}^R}
  \int_0^{\infty} \int^{(k)} \prod_{i=1}^{k}
  \mathscr{K}^{(\varepsilon)}_i(z_i)\prod_{\mu= 1}^Nf_{\Delta}
  \Big(\prod_{\ell = 1}^Rt_{\ell}^{-c_{\ell, \mu} } \Big) t_1^{\tilde{z}_1}
  \cdots t_R^{\tilde{z}_R}x^{ - z_1 - \dots - z_{k}}
  \frac{\dd\mathbf{z}}{(2\pi {\rm i})^{k}} \dd x \frac{\dd\textbf{t}}{\langle
    \textbf{t} \rangle},$$ where $\Re z_i = \zeta_i$, $1 \leq i \leq \kappa$,
  $\Re z_k = \frac{1}{2}\zeta_k$.  By Mellin inversion, we can now compute
  each of the $z_1, \ldots, z_{\kappa}$-integrals. We recall our notation
  $\tilde{\textbf{z}} = \mathscr{C}_2 \textbf{z}^{\ast}$, so
  $$\tilde{z}_j = \sum_{i=1}^{\kappa} c_{j, N+i} z_i + c_{j, N+k}.$$ This gives
  $$c^{(\varepsilon)}_{\infty}(\Delta)   = \frac{2^J}{\pi}\int_{\Bbb{R}_{>0}^R}
  \int_0^{\infty} \Big[\prod_{\mu= 1}^Nf_{\Delta} \Big(\prod_{\ell =
    1}^Rt_{\ell}^{-c_{\ell, \mu} } \Big) \Big] \Big[K^{(\varepsilon)}_k(x)
  \prod_{i=1}^{\kappa} K^{(\varepsilon)}_{i}\Big(x \prod_{\nu=1}^R
  t_{\nu}^{-c_{\nu, N+i}}\Big)\Big] \prod_{\nu=1}^R t_{\nu}^{c_{\nu, N+k}} \dd
  x \frac{\dd\textbf{t}}{\langle \textbf{t} \rangle}.  $$ Changing variables
  $t_{\nu} \mapsto t_{\nu}^{-1}$ and then
  $x \mapsto 2 \pi x \prod_{\nu=1}^R t_{\nu}^{1+c_{\nu, N+k}}$, this becomes
  $$2^J  \int_{\Bbb{R}_{>0}^R}  \int_{-\infty}^{\infty}  \Big[\prod_{\mu=
    1}^Nf_{\Delta} \Big(\prod_{\ell = 1}^Rt_{\ell}^{c_{\ell, \mu} } \Big)
  \Big] \Big[K^{(\varepsilon)}_k(2\pi x\prod_{\nu=1}^R t_{\nu}^{1+c_{\nu,
      N+k}}) \prod_{i=1}^{\kappa} K^{(\varepsilon)}_{i}\Big(2\pi x
  \prod_{\nu=1}^R t_{\nu}^{c_{\nu, N+i}+1 + c_{\nu, N+k}}\Big)\Big] \dd x\,
  \dd\textbf{t}. $$ We re-index the variables $t_{\nu}$ as $t_{ij}$ with
  $(i, j) \in I$, as described prior to the statement of the lemma.  By the
  definition of $(\mathscr{A}_1 \mathscr{A}_2)$ in \eqref{matrix}, we then
  have
  $$   \prod_{\nu=1}^R t_{\nu}^{c_{\nu, N+i} + 1 + c_{\nu, N+k}} = \prod_{(i,
    j) \in I} t_{ij}^{h_{ij}} \quad (1 \leq i \leq \kappa), \quad\quad
  \prod_{\nu=1}^R t_{\nu}^{1+c_{\nu, N+k}} = \prod_{(k, j) \in I}
  t_{kj}^{h_{kj}}, $$ so that
  \begin{displaymath}
    \begin{split}
      c^{(\varepsilon)}_{\infty}(\Delta) & = 2^J \int_{-\infty}^{\infty}
      \int_{\Bbb{R}_{>0}^R} \Big[\prod_{\mu= 1}^Nf_{\Delta} \Big(\prod_{(i, j)
        \in I} t_{ij}^{\alpha^\mu_{ij} } \Big) \Big] \Big[ \prod_{i=1}^{k}
      K^{(\varepsilon)}_{i}\Big(2\pi x \prod_{(i, j) \in I} t_{ij}^{h_{ij}}
      \Big)\Big] \dd x\, \dd\textbf{t} .
    \end{split}
  \end{displaymath}
  By symmetry, we may extend $\textbf{t}$-integral to all of $\Bbb{R}^R$,
  recall \eqref{defKiepsilon} and write
  \begin{displaymath}
    \begin{split}
      c^{(\varepsilon)}_{\infty}(\Delta)  & = 2^{J-R}
      \int_{-\infty}^{\infty}  \int_{\Bbb{R}^R} \Psi_{\Delta}(\textbf{t})
      e\big(x \Phi^{\ast}(\textbf{t})\big) \exp\big(-(\pi \varepsilon x)^2
      \tilde{\Phi}(\textbf{t})\big)   \dd x
      \, \dd\textbf{t}     
    \end{split}
  \end{displaymath}
  with $\Phi^{\ast}$ as in \eqref{def-cap-Phi} and
  $$\Psi_{\Delta}(\textbf{t}) = \prod_{\mu= 1}^Nf_{\Delta} \Big(\prod_{(i, j)
    \in I}|t_{ij}|^{\alpha^\mu_{ij} } \Big), \quad \ \tilde{\Phi}(\textbf{t}) =
  4\sum_{i = 1}^k\prod_{(i, j) \in I} t_{ij}^{2h_{ij}}.$$ We compute the
  $x$-integral, getting
  \begin{displaymath}
    \begin{split}
      c^{(\varepsilon)}_{\infty}(\Delta) & = \frac{2^{J-R}}{\sqrt{\pi}
        \varepsilon} \int_{\Bbb{R}^R} \Psi_{\Delta}(\textbf{t}) \exp\Big(-
      \frac{(\Phi^{\ast})^2(\textbf{t})}{\varepsilon^2
        \tilde{\Phi}(\textbf{t})}\Big) \frac{ \dd\textbf{t}
      }{\sqrt{\tilde{\Phi}(\textbf{t})}}.
    \end{split}
  \end{displaymath}
  By construction, this is absolutely convergent for every fixed
  $\varepsilon > 0$, and the limit as $\varepsilon \rightarrow 0$ exists by
  \eqref{dominatedc}. Let $\mathscr{U} \coloneqq \{\textbf{t} \in
  \Bbb{R}^R : |(\Phi^{\ast})^2(\textbf{t})/\tilde{\Phi}(\textbf{t}) | \leq 1/25\}$.
  Writing
  $$ \exp\Big(- \frac{(\Phi^{\ast})^2(\textbf{t})}{\varepsilon^2
    \tilde{\Phi}(\textbf{t})}\Big) = \exp\Big(-
  \frac{(\Phi^{\ast})^2(\textbf{t})}{ \tilde{\Phi}(\textbf{t})}\Big) \exp\Big(
  (1 - \varepsilon^{-2}) \frac{(\Phi^{\ast})^2(\textbf{t})}{
    \tilde{\Phi}(\textbf{t})}\Big),$$ we obtain
  \begin{equation*}
    c^{(\varepsilon)}_{\infty}(\Delta) =  \frac{2^{J-R}}{\sqrt{\pi}
      \varepsilon}   \int_{\mathscr{U}} \Psi_{\Delta}(\textbf{t})    \exp\Big(-
    \frac{(\Phi^{\ast})^2(\textbf{t})}{\varepsilon^2
      \tilde{\Phi}(\textbf{t})}\Big)  \frac{ \dd\textbf{t}
    }{\sqrt{\tilde{\Phi}(\textbf{t})}} + O\Big(\frac{1}{\varepsilon}e^{(1 -
      \varepsilon^{-2})/25}\Big).
  \end{equation*}
  We consider now the equation
  \begin{equation}\label{implicit}
    \Phi^{\ast}(\textbf{t})/\sqrt{\tilde{\Phi}(\textbf{t})} - u = 0
  \end{equation}
  for $|u| \leq 1/5$. It is only at this point that we use
  \eqref{simplifying}. We write $\textbf{t} = (t_{11}, \textbf{t}')$ and
  $$\Phi^{\ast}(\textbf{t}) = t_{11} + (\Phi^{\ast})'(\textbf{t}'), \quad
  \tilde{\Phi}(\textbf{t}) = 4t_{11}^2 + \tilde{\Phi}'(\textbf{t}').$$
  Then for $u = 0$, the equation \eqref{implicit} has the unique solution
  $t_{11} = -(\Phi^{\ast})'(\textbf{t}')$, while for $0 < |u| \leq 1/5$, both
  $u$ and $-u$ lead to two solutions
  $$t_{11}  = \frac{-(\Phi^{\ast})'(\textbf{t}') \pm |u|
    \sqrt{4(\Phi^{\ast})'(\textbf{t}')^2 + \tilde{\Phi}'(\textbf{t}')(1 -
      4u^2)}}{1 - 4u^2} \eqqcolon \phi^{\pm}_u(\textbf{t}').$$
  For $u=0$, we have $\phi_0^+ = \phi_0^-$, and for notational simplicity we
  write $\phi_0^{\pm} = \phi = -(\Phi^{\ast})'$.  Changing variables, we
  obtain
  $$\frac{2^{J-R}}{\sqrt{\pi} \varepsilon}   \int_{\mathscr{U}}
  \Psi_{\Delta}(\textbf{t})  \exp\Big(-
  \frac{(\Phi^{\ast})^2(\textbf{t})}{\varepsilon^2
    \tilde{\Phi}(\textbf{t})}\Big)  \frac{ \dd\textbf{t}
  }{\sqrt{\tilde{\Phi}(\textbf{t})}} = \frac{2^{J-R}}{\sqrt{\pi} \varepsilon}
  \int_{-1/5}^{1/5}  \exp\Big(- \frac{u^2}{\varepsilon^2 }\Big) \Theta(u)
  {\dd}u, $$
  where
  $$ \Theta(u) =\int_{\Bbb{R}^{R-1}} \Xi (\phi^{+}_u(\textbf{t}'), \textbf{t}')
  {\dd}\textbf{t}' , \quad \Xi = \frac{2 \tilde{\Phi} \Psi_{\Delta}
  }{|2\tilde{\Phi} \Phi^{\ast}_{t_{11}} - \Phi^{\ast} \tilde{\Phi}_{t_{11}} |
  }.$$ By a Taylor expansion, we have $\Theta(u) = \Theta(0) + O(|u|)$ for
  $|u| \leq 1/5$, so that
  \begin{equation*}
    \begin{split}
      c_{\infty}(\Delta) = \lim_{\varepsilon \rightarrow
        0}\frac{2^{J-R}}{\sqrt{\pi} \varepsilon}   \int_{-\eta}^{\eta}
      \exp\Big(- \frac{u^2}{\varepsilon^2 }\Big) \Theta(u)  {\dd}u &=
      2^{J-R}\Theta(0) =2^{J-R} \int_{\Bbb{R}^{R-1}} \Xi (\phi(\textbf{t}'),
      \textbf{t}') {\dd}\textbf{t}'  \\
      &= 2^{J-R} \int_{\Bbb{R}^{R-1}} \frac{ \Psi_{\Delta} (\phi(\textbf{t}'),
        \textbf{t}') }{ | \Phi^{\ast}_{t_{11}} (\phi(\textbf{t}'),
        \textbf{t}')| } {\dd}\textbf{t}' .
    \end{split}
  \end{equation*}
  Here we can let $\Delta \rightarrow 0$, obtaining 
  \begin{equation}\label{eq:c_infty}
    \begin{split}
      c_{\infty} = 2^{J-R} \int_{\Bbb{R}^{R-1}} \frac{ \chi_I
        (\phi(\textbf{t}'), \textbf{t}') }{ | \Phi^{\ast}_{t_{11}}
        (\phi(\textbf{t}'), \textbf{t}')| } {\dd}\textbf{t}' .
    \end{split}
  \end{equation}
  (Note that the denominator is $1$ by \eqref{simplifying}, but that this
  formula should also hold without this assumption.)  We write this more
  symmetrically as follows. If $t_{ij}$ is any component of $\textbf{t}'$,
  then by implicit differentiation, we have
  $$\phi_{t_{ij}}(\textbf{t}) =
  -\frac{\Phi^{\ast}_{t_{ij}}(\phi(\textbf{t}'),\textbf{t}')}{\Phi^{\ast}_{t_{11}}(\phi(\textbf{t}'),
    \textbf{t}')},$$
  so that we can write $c_{\infty}$ as a surface integral
  \begin{displaymath}
    \begin{split}
      2^{J-R} \int_{\Bbb{R}^{R-1}} \frac{ \chi_I(\phi(\textbf{t}'),
        \textbf{t}') }{ | \Phi^{\ast}_{t_{11}} (\phi(\textbf{t}'),
        \textbf{t}')| } d\textbf{t}' = 2^{J-R}\int_{\mathscr{F}}
      \frac{\chi_I(\textbf{t}) }{\| \nabla \Phi^{\ast}(\textbf{t})\|}
      d\mathscr{F}(\textbf{t})
    \end{split}
  \end{displaymath}
  as claimed.  
\end{proof}

\subsection{Comparison with the Manin--Peyre conjecture}

\begin{theorem}\label{thm:manin-peyre}
  Let $X,H$ be as in Proposition~\ref{prop:peyre}.  Suppose that the
  corresponding counting problem for $U \subset X$ given by
  Proposition~\ref{prop:countingproblem_abstract} satisfies all assumptions of
  Theorem~\ref{analytic-theorem}.  Then the Manin--Peyre conjecture holds for
  $X$ with respect to $H$, that is,
  \begin{equation*}
    N_{X,U,H}(B) =(1+o(1)) c B(\log B)^{\rank \Pic X - 1}
  \end{equation*}
  with Peyre's constant $c$.
\end{theorem}

\begin{proof}
  By Proposition~\ref{prop:countingproblem_abstract},
  \begin{equation*}
    N_{X,U,H}(B) = 2^{-\rank \Pic X} N(B)
  \end{equation*}
  for $N(B)$ as in \eqref{manin}. Formula \eqref{rough} in
  Theorem~\ref{analytic-theorem} states that
  \begin{equation*}
    N(B) =(1+o(1)) c^{\ast}  c_{\text{{\rm fin}}}  c_{\infty}  B (\log B)^{c_2}.
  \end{equation*}
  Comparing definition \eqref{eq:c_p} with expression \eqref{cfin-final-euler}
  for $c_{\mathrm{fin}}$, the definitions \eqref{cast-final-first} and
  \eqref{cast-final} of $c^\ast$, and definition \eqref{cinf-first} with
  expression \eqref{eq:c_infty} for $c_\infty$ (which are both valid since
  assumption \eqref{eq:assumption_real_density_strong} implies
  \eqref{simplifying}), then Proposition~\ref{prop:peyre} shows that the
  leading constant for $N_{X,U,H}(B)$ is Peyre's constant, and
  $c_2 = J-R = \rank \Pic X - 1$ by \eqref{eq:rkPic}, \eqref{defc2} and
  Lemma~\ref{rank}. Therefore, Proposition~\ref{prop:countingproblem_abstract}
  combined with \eqref{rough} agrees with the Manin--Peyre conjecture.
\end{proof}

The following part provides numerous applications and shows how to apply this
in practice.

\part{Application to spherical varieties}\label{part3}

Having established the relevant theory in Part~\ref{part1} and
Part~\ref{part2} of the paper, we are now prepared to prove Manin's conjecture
for concrete families of varieties. In particular, as a consequence of
Theorem~\ref{manin-cor}, we obtain Manin's conjecture for all smooth spherical
Fano threefolds of semisimple rank one and type $T$.

\section{Spherical varieties}\label{sec2}
 
\subsection{Luna--Vust invariants}

Let $G$ be a connected reductive group over $\Qbar$.  Let $\Qbar(X)$ be the
function field of a spherical $G$-variety $X$ over $\Qbar$. Only in this
section and in Section~\ref{sec22}, let $B$ denote a Borel subgroup of $G$
with character group $\Xf(B)$. The \emph{weight lattice} is defined as
\begin{align*}
  \Mm = \rleft\{\chi \in \Xf(B) :
  \begin{aligned}
    \text{there exists $f_\chi \in \Qbar(X)^\times$ such that}\\
    \text{$b\cdot f_\chi = \chi(b)\cdot {f_\chi}$ for every $b \in B$}
  \end{aligned}
  \rright\}\text{,}
\end{align*}
Note that for every $\chi\in\Mm$, the function $f_\chi$ is uniquely determined
up to a constant factor because of the dense $B$-orbit in $X$. The \emph{set
  of colors} $\Dm$ is the set of $B$-invariant prime divisors on $X$ that are
not $G$-invariant. Moreover, we have the \emph{valuation cone}
$\Vm \subseteq \Nm_{\Qd} = \Hom(\Mm, \Qd)$, which can be identified with the
$\Qd$-valued $G$-invariant discrete valuations on $\Qbar(X)^\times$. By
Losev's uniqueness theorem \cite[Theorem~1]{los09a}, the combinatorial
invariants $(\Mm, \Vm, \Dm)$ uniquely determine the birational class of (\ie
the open $G$-orbit in) the spherical $G$-variety $X$ over $\Qbar$.

Now let $\Delta$ be the set of all $B$-invariant prime divisors on $X$. There
is a map $\mathfrak{c} \colon \Delta \to \Nm_\Qd$ defined by
$\langle \mathfrak{c}(D), \chi \rangle = \nu_D(f_\chi)$, where $\nu_D$ is the
valuation on $\Qbar(X)^\times$ induced by the prime divisor $D$.  For every
$G$-orbit $Z \subseteq X$, we define
$\Wm_Z = \{D \in \Delta : Z \subseteq D\}$. Then the collection
\begin{align*}
  \CF X = \{(\cone(\mathfrak{c}(\Wm_Z)), \Wm_Z \cap \Dm) :
  Z \subseteq X \text{ is a $G$-orbit}\}
\end{align*}
is called the \emph{colored fan} of $X$. According to the Luna--Vust
theory of spherical embeddings \cite{lv83, kno91}, the colored fan
$\CF X$ uniquely determines the spherical $G$-variety~$X$ over
$\Qbar$ among those in the same birational class.

The divisor class group $\Cl X$ can be computed from $\CF X$: by
\cite[Proposition~4.1.1]{bri07}, the maps $\Mm \to \Zd^\Delta$,
$\chi \mapsto \Div f_\chi$ and $\Zd^\Delta \to \Cl X$,
$D \mapsto [D]$ fit into the exact sequence
$\Mm \to \Zd^\Delta \to \Cl X \to 0$.

Spherical varieties with $\Vm = \Nm_{\Qd}$ are called
\emph{horospherical}.  These include flag varieties and toric
varieties. In the latter case, $G=B=T$ is a torus, and we have
$\Vm = \Nm_{\Qd}$ and $\Dm = \emptyset$.

\subsection{Semisimple rank one}\label{sec:ssr1}

Let $X$ be a spherical $G$-variety over $\Qbar$. If the connected reductive
group~$G$ has semisimple rank one, we may assume
$G = \SL_2\times \Gd_\mathrm{m}^r$ by passing to a finite cover. As a further
simplification, we replace the action by a \emph{smart action} as introduced
in \cite[Definition~4.3]{ab04}. As before, let
$G/H = (\SL_2 \times \Gd_\mathrm{m}^r)/H$ be the open orbit in $X$. Let
$H'\times \Gd_\mathrm{m}^r = H\cdot \Gd_\mathrm{m}^r \subseteq \SL_2 \times
\Gd_\mathrm{m}^r$. Then the homogeneous space $\SL_2/H'$ is spherical, and
hence either $H'$ is a maximal torus in $\SL_2$ (\emph{the case $T$}) or $H'$
is the normalizer of a maximal torus in $\SL_2$ (\emph{the case $N$}) or the
homogeneous space $\SL_2/H'$ is horospherical.  Since the action is smart, in
the horospherical case $H'$ is either a Borel subgroup in $\SL_2$ (\emph{the
  case~$B$}) or the whole group $\SL_2$ (\emph{the case $G$}).

Now let $T \subset G = \SL_2 \times \Gd_\mathrm{m}^r$ be a maximal torus, and
let $\alpha \in \Xf(T) \cong \Xf(B)$ be the simple root with respect to a
Borel subgroup $B \subset G$. It follows from the general theory of spherical
varieties that in the cases $T$ and $N$, we always have
$\Vm = \{v \in \Nm_\Qd : \langle v, \alpha \rangle \le 0\}$. The colored cones
of the form $(\Q_{\ge0}\cdot u, \emptyset) \in \CF X$, where
$u \in \Mm \cap \Vm$ is a primitive element, correspond to the $G$-invariant
prime divisors in $X$. Let $(\Q_{\ge0} \cdot u_{0j}, \emptyset) \in \CF X$ for
$j = 1, \dots, J_0$ be those with $u \in \Vm \cap (-\Vm)$, and let
$(\Q_{\ge0} \cdot u_{3j}, \emptyset) \in \CF X$ for $j = 1, \dots, J_3$ be
those with $u \notin \Vm \cap (-\Vm)$.  We denote by $D_{ij}$ the
$G$-invariant prime divisor in $X$ corresponding to
$(\Q_{\ge0} \cdot u_{ij}, \emptyset) \in \CF X$.  Then we have
$\cfr(D_{ij}) = u_{ij}$.
  
We define $h_{3j} = -\langle u_{3j}, \alpha \rangle$. The following
descriptions of the Cox rings in the different cases can be explicitly
obtained from \cite[Theorem~4.3.2]{bri07} or \cite[Theorem~3.6]{gag14}.

\smallskip Case $T$: There are two colors $D_{11}, D_{12} \in \Dm$, and we
have $\cfr(D_{11}) + \cfr(D_{12}) = \alpha^\vee|_{\Mm}$.
The Cox ring is given by
\begin{equation}\label{eq:coxring_type_T}
  \Rm(X) = \Qbar[x_{01}, \dots, x_{0J_0}, x_{11}, x_{12}, x_{21}, x_{22}, x_{31}, \dots, x_{3J_3}]
  /(x_{11}x_{12}-x_{21}x_{22}-x_{31}^{h_{31}} \cdots x_{3J_3}^{h_{3J_3}}), 
\end{equation}
cf.\ \eqref{typeT}, with
\begin{align*}
  \deg(x_{11}) &= \deg(x_{21}) = [D_{11}] \in \Cl X\text{,}\quad
  \deg(x_{12}) = \deg(x_{22}) = [D_{12}] \in \Cl X\text{, and} \\
  \deg(x_{ij}) &= [D_{ij}] \in \Cl X \text{ for $i \in \{0,3\}$.}
\end{align*}

\smallskip Case $N$: There is one color $D_{11} \in \Dm$, and we have
$\cfr(D_{11}) = \tfrac{1}{2}\alpha^\vee|_{\Mm}$.
The Cox ring is given by
\begin{equation*}
  \Rm(X) = \Qbar[x_{01}, \dots, x_{0J_0}, x_{11}, x_{12}, x_{21}, x_{31}, \dots, x_{3J_3}]
  /(x_{11}x_{12}-x_{21}^2-x_{31}^{h_{31}} \cdots x_{3J_3}^{h_{3J_3}})
\end{equation*}
with
\begin{equation*}
  \deg(x_{11}) = \deg(x_{12}) = \deg(x_{21}) = [D_{11}] \in \Cl X\text{,}\quad
  \deg(x_{ij}) = [D_{ij}] \in \Cl X \text{ for $i \in \{0,3\}$.}
\end{equation*}

\smallskip

Case $B$: We mention this case only for completeness since $X$ is isomorphic
to a toric variety here (as an abstract variety with a different group
action). There is one color $D_{11} \in \Dm$, and we have
$\cfr(D_{11}) = \alpha^\vee|_{\Mm}$.
The Cox ring is given by
$ \Rm(X) = \Qbar[x_{01}, \dots, x_{0J_0}, x_{11}, x_{12}]$ with
\begin{equation*}
  \deg(x_{11}) = \deg(x_{12}) = [D_{11}] \in \Cl X\text{,}\quad
  \deg(x_{0j}) = [D_{0j}] \in \Cl X \text{.}
\end{equation*}

\smallskip

Case $G$: We mention this case only for completeness since $X$ is a toric
$\Gd^r_m$-variety here. We have $\Dm = \emptyset$.  The Cox ring is given by
$\Rm(X) = \Qbar[x_{01}, \dots, x_{0J_0}]$ with
$ \deg(x_{0j}) = [D_{0j}] \in \Cl X. $

\subsection{Ambient toric varieties}\label{sec:ambient}

Every quasiprojective variety $X$ with finitely generated Cox ring may
be embedded into a toric variety $Y^\circ$ with nice properties, as
described in \cite[3.2.5]{adhl15}.

For a spherical variety $X$, this is explicitly described in
\cite{gag17}. According to \cite[Theorem~4.3.2]{bri07}, the Cox ring of $X$ is
generated by the union of sets $x_{D1}, \dots, x_{Dr_D} \in \Rm(X)$ for every
$D \in \Delta$. We have $r_D = 1$ if $D \notin \Dm$ and $r_D \ge 2$ if
$D \in \Dm$. Each $x_{Di}$ corresponds to a ray $\rho_{Di}$ in the fan
$\Sigma^\circ$ of the ambient toric variety $Y^\circ$.

Even if $X$ is projective, the quasiprojective toric variety $Y^\circ$ might
not be projective.  This is the case if and only if the colored cones in
$\CF X$ do not cover $\Nm_\Qd$.

Any $\Wm \subseteq \Delta$ defines a pair
$(\cone(\mathfrak{c}(\Wm)), \Wm \cap \Dm)$. If $\cone(\mathfrak{c}(\Wm))$ is
strictly convex, we call the pair a \emph{supported colored cone} if
$\cone(\mathfrak{c}(\Wm))^\circ \cap \Vm \ne \emptyset$ and an
\emph{unsupported colored cone} if
$\cone(\mathfrak{c}(\Wm))^\circ \cap \Vm = \emptyset$. If we can extend
$\CF X$ by some of these unsupported colored cones to a collection
$(\CF X)_{\rm ext}$ such that every face (in the sense of
\cite[Definition~15.3]{tim11}) of a colored cone is again in
$(\CF X)_{\rm ext}$, such that different colored cones intersect in faces, and
such that the colored cones cover the whole space $\Nm_\Qd$, then
$(\CF X)_{\rm ext}$ yields a toric variety $Y$ that completes $Y^\circ$.

We recall
here how to obtain the fan $\Sigma$ of the toric variety $Y$ from the
(possibly extended) colored fan $(\CF X)_{\rm ext}$. Let
$\Psi_D = \{\rho_{D1}, \dots, \rho_{Dr_D}\}$, and define
$\Psi_D^j = \Psi_D \setminus \{\rho_{Dj}\}$ for every
$1 \le j \le r_D$.
For every subset $\Wm \subseteq \Delta$, consider the sets of cones
\begin{align*}
  \Phi(\Wm) = \bigg\{\cone\bigg(\bigcup_{D \in \Wm}\Psi_D \cup
  \bigcup_{D \in \Delta\setminus \Wm}\Psi_D^{j(D)}\bigg) : j \in
  \Nd^{\Delta\setminus \Wm},\ 1 \le j(D) \le r_D \bigg\}\text{.}
\end{align*}
Then we have
\begin{equation}\label{eq:Sigma_max}
  \Sigma = \bigcup_{(\cone(\mathfrak{c}(\Wm)),
  \Wm \cap \Dm) \in (\CF X)_{\rm ext}} \Phi(\Wm)
\quad\text{and}\quad
  \Sigmamax = \bigcup_{(\cone(\mathfrak{c}(\Wm)),
  \Wm \cap \Dm) \in (\CF X)_{\rm ext,max}} \Phi(\Wm)\text{.}
\end{equation}

\subsection{Manin's conjecture}

We present now the main result of this paper, which implies all theorems
stated in the introduction.

\begin{theorem}\label{manin-cor}
  Let $X$ be a smooth split spherical almost Fano variety of semisimple rank
  one and type $T$ over $\Qd$ with semiample $\omega_X^\vee$ satisfying
  \eqref{eq:toric_smooth} whose colored fan $\CF X$ contains a maximal cone
  without colors.
 
  The corresponding counting problem as in
  Proposition~\ref{prop:countingproblem_abstract} features a torsor equation
  \eqref{typeT} with exponents $h_{ij}$, a height matrix $\mathscr{A}$ as in
  \eqref{newA} and coprimality conditions $S_1, \ldots S_r$ as in
  \eqref{gcd}. Choose $\bm \zeta $ satisfying \eqref{zeta1} and \eqref{J-cond}, let $\lambda$ be
  as in \eqref{lambdabeta} and choose $\bm \tau^{(2)}$ as in \eqref{tau1}.
  
  With these data, assume that \eqref{fail} and \eqref{ass2} hold.
  Then the Manin--Peyre conjecture holds for $X$ with respect to the
  anticanonical height function \eqref{eq:height_definition}.
\end{theorem}

\begin{proof}
  It is enough to check all assumptions of Theorem~\ref{thm:manin-peyre}.
  
  We observe that $X$ is as in Proposition~\ref{prop:peyre} by our
  assumptions. In particular by \eqref{eq:coxring_type_T}, its Cox ring is as
  required.  By \eqref{eq:Sigma_max}, a maximal cone without colors in $\CF X$
  gives four maximal cones $\sigma \in \Sigmamax$ such that the variables
  corresponding to the rays of $\sigma$ include precisely one of
  $x_{11},x_{21}$ and precisely one of $x_{12},x_{22}$ in
  \eqref{eq:coxring_type_T}; it is not hard to see that one of these four
  cones satisfies \eqref{eq:assumption_real_density_strong}.

  Next we check that Theorem~\ref{analytic-theorem} applies. The counting
  problem is of the required form by
  Proposition~\ref{prop:countingproblem_abstract} and
  \eqref{eq:coxring_type_T}. Hypothesis~\ref{H1} holds by
  Proposition~\ref{circle-method}, whose assumptions are satisfied by
  \eqref{eq:coxring_type_T} and which allows us to choose 
  $${\bm \beta} = \Big( \frac{1}{2} - \frac{1}{5\max_{ij} h_{ij}},
  \frac{1}{2} - \frac{1}{5\max_{ij} h_{ij}}, \frac{2}{5\max_{ij}
    h_{ij}}\Big), $$ so that \eqref{1b} holds. Condition \eqref{J-cond} means $\zeta_3 < 1/2$ which is consistent with \eqref{zeta1}.  
  Hypothesis~\ref{H2}
  holds by Proposition~\ref{propH2}.  The conditions \eqref{1c}, \eqref{1a}
  hold by Lemmas~\ref{rank} and \ref{pos}.
\end{proof}

The assumption \eqref{eq:toric_smooth} can be read off of the colored fan
$\CF X$, using the method described in Section~\ref{sec:ambient}. The
existence of a maximal cone without colors in $\CF X$ is straightforward to
check and clearly holds in all our examples below; alternatively,
\eqref{eq:assumption_real_density_strong} can be checked directly.  As
mentioned after Proposition \ref{propH2}, if \eqref{fail} fails, we can apply
an alternative, but slightly more complicated criterion.  Assumption
\eqref{ass2} requires elementary linear algebra (and can be checked quickly by
computer if desired).

\begin{remark}\label{simpl} 
  If the torsor equation is
  $x_{11} x_{12} + x_{21} x_{22} + x_{31} x_{33} = 0$, we can use
  \cite[Proposition 1.2]{BBS2} instead of Proposition~\ref{circle-method} to
  verify Hypothesis~\ref{H1}, which conveniently yields again
  ${\bm \beta} = (1/3 + \varepsilon, 1/3 + \varepsilon, 1/3 + \varepsilon)$
  and more importantly
  \begin{equation*}
    \lambda = 1. 
  \end{equation*}
  The advantage is that the third line of \eqref{simplex2} is trivially
  satisfied (the polytope is empty), so that checking \eqref{ass2} requires a
  little less computational effort.
\end{remark}
 
\section{Spherical Fano threefolds}\label{sec:fano3folds}

\subsection{Geometry}\label{sec22}

According to \cite[\S 6.3]{hofscheier}, all horospherical smooth Fano
threefolds are either toric or flag varieties. Furthermore, there are nine
smooth Fano threefolds over $\Qbar$ that are spherical, but not horospherical;
they are equipped with an action of $G=\SL_2\times \Gd_\mathrm{m}$.  The
notation $T$ and $N$ in \cite[Table~6.5]{hofscheier} and in our
Table~\ref{tab:classification_spherical} refers to the cases in
Section~\ref{sec:ssr1}.

\begin{table}[ht]
  \centering
  \begin{tabular}[ht]{ccccccc}
    \hline
    rk Pic & Hofscheier & Mori--Mukai & torsor equation & remark\\
    \hline\hline
    2 & $T_1 12$ & II.31 &  {$x_{11}x_{12}-x_{21}x_{22}-x_{31}x_{32}^2$} & eq. $\Gd_\mathrm{a}^3$-cpct.\\
    2 & $N_1 6, N_1 7$ & II.30 & {$x_{11}x_{12}-x_{21}^2-x_{31}x_{32}$} & eq. $\Gd_\mathrm{a}^3$-cpct.\\
    2 & $N_1 8$ & II.29 &  {$x_{11}x_{12}-x_{21}^2-x_{31}x_{32}^2x_{33}$} & \\
    \hline
    3 & $T_1 18$ & III.24 & {$x_{11}x_{12}-x_{21}x_{22}-x_{31}x_{32}$} & variety $X_1$ \\
    3 & $T_1 21$ & III.20  & {$x_{11}x_{12}-x_{21}x_{22}-x_{31}x_{32}x_{33}^2$} &  variety $X_2$\\
    3 & $N_0 3$ & III.22  &  $x_{11}x_{12}-x_{21}^2-x_{31}x_{32}$ & \\
    3 & $N_1 9$ & III.19 &  $x_{11}x_{12}-x_{21}^2-x_{31}x_{32}$ & \\
    \hline
    4 & $T_0 3$ &  IV.8  &  $x_{11}x_{12}-x_{21}x_{22}-x_{31}x_{32}$ &  variety $X_3$ \\
    4 & $T_1 22$ & IV.7  & $x_{11}x_{12}-x_{21}x_{22}-x_{31}x_{32}$ &  variety  $X_4$\\
    \hline
  \end{tabular}
  \caption{Smooth Fano threefolds that are spherical, but not horospherical}
  \label{tab:classification_spherical}
\end{table}

We proceed to describe the four $T$ cases $X_1, \dots, X_4 $ in Table
\ref{tab:classification_spherical} that are not equivariant
$\Gd_\mathrm{a}^3$-compactifications \cite{arXiv:1802.08090} in more
detail. In each case, we first construct a split form over $\Qd$ following the
elementary description from the Mori--Mukai classification, and then we give
the description using the Luna--Vust theory of spherical embeddings from
Hofscheier's list. Finally we describe in each case an ambient toric variety
$Y_i$ satisfying \eqref{eq:toric_smooth} that can be used with
Sections~\ref{sec:charts_torsors}--\ref{sec:tamagawa_cox}.
  
Let $\varepsilon_1 \in \Xf(B)$ be a primitive character of $\Gd_\mathrm{m}$
composed with the natural inclusion $\Xf(\Gd_\mathrm{m}) \to \Xf(B)$.

\smallskip
  
\subsubsection{$X_1$ of type III.24 and $X_4$ of type IV.7}\label{IV.7}

Consider $\Pd^2_\Qd \times \Pd^2_\Qd$ with coordinates
$(z_{11}:z_{21}:z_{31})$ and $(z_{12}:z_{22}:z_{32})$, and the hypersurface
$W_4 = \Vd(z_{11}z_{12}-z_{21}z_{22}-z_{31}z_{32}) \subset \Pd^2_\Qd \times
\Pd^2_\Qd$ of bidegree $(1,1)$. This is a smooth Fano threefold of type
II.32. It contains the curves
\begin{align*}
  C_{01} &= \Vd(z_{11},z_{21},z_{32}) =
           \{(0:0:1)\}\times\Vd(z_{32})\text{,}\\
  C_{02} &= \Vd(z_{12},z_{22},z_{31}) = 
           \Vd(z_{31})\times\{(0:0:1)\}
\end{align*}
of bidegrees $(0,1)$ and $(1,0)$, respectively.  Let $X_1$ be the blow-up of
$W_4$ in the curve $C_{01}$. This is a smooth Fano threefold of type
III.24. Moreover, let $X_4$ be the further blow-up in the curve $C_{02}$
(which is disjoint from the curve $C_{01}$ in $W_4$). This is a smooth Fano
threefold of type IV.7.  We may define an action of
$G = \SL_2 \times \Gd_\mathrm{m}$ on $W_4$ by
\begin{equation*}
  (A,t)\cdot \rleft(
  \begin{pmatrix}
    z_{11} & z_{22} \\
    z_{21} & z_{12}
  \end{pmatrix}
  ,z_{31},z_{32}\rright) = \rleft(A\cdot \begin{pmatrix}
    z_{11} & z_{22} \\
    z_{21} & z_{12}
  \end{pmatrix} \cdot
  \begin{pmatrix}
    t^{-1} & 0 \\
    0 & t
  \end{pmatrix},z_{31},z_{32}\rright),
\end{equation*}
which turns $W_4$ into a spherical variety. The following description using
the Luna--Vust theory of spherical embeddings can be easily verified.  The
lattice $\Mm$ has basis
$(\frac{1}{2}\alpha + \varepsilon_1, \frac{1}{2}\alpha - \varepsilon_1)$. We
denote the corresponding dual basis of the lattice $\Nm$ by $(d_1, d_2)$. Then
there are two colors with valuations $d_1$ and $d_2$, and the valuation cone
is given by $\Vm = \{v \in \Nm_\Qd : \langle v, \alpha \rangle \le 0\}$. Since
the curves $C_{01}$ and $C_{02}$ are $G$-invariant, the varieties $X_1$ and
$X_4$ are spherical $G$-varieties, and the blow-up morphisms
$X_4 \to X_1 \to W_4$ can be described by maps of colored fans. The following
figure illustrates this.
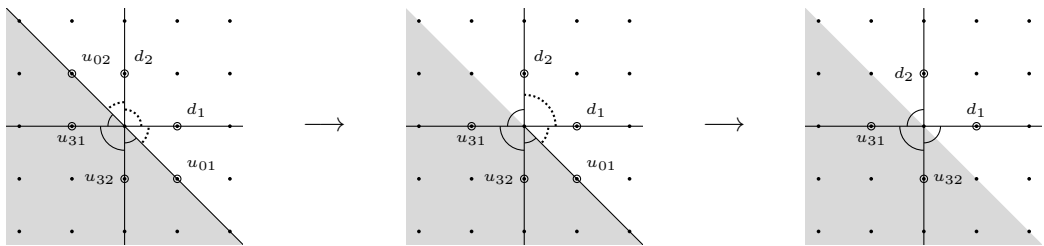
\begin{align*}
  \begin{tikzpicture}[scale=0.7]
    \clip (-2.24, -2.24) -- (2.24, -2.24) -- (2.24, 2.24) -- (-2.24, 2.24) -- cycle;
    \fill[color=gray!30] (-3, 3) -- (3, -3) -- (-3, -3) -- cycle;
    \foreach \x in {-3,...,3} \foreach \y in {-3,...,3} \fill (\x, \y) circle (1pt);
    \draw (1, 0) circle (2pt);
    \draw (0, 1) circle (2pt);
    \draw (-1, 0) circle (2pt);
    \draw (0, -1) circle (2pt);
    \draw (-1, 1) circle (2pt);
    \draw (1, -1) circle (2pt);
    \draw (0,0) -- (3,0);
    \draw (0,0) -- (0,3);
    \draw (0,0) -- (-4,4);
    \draw (0,0) -- (4,-4);
    \draw (0,0) -- (-3, 0);
    \draw (0,0) -- (0, -3);
    \path (-1, 0) node[anchor=north] {{\tiny{$u_{31}$}}};
    \path (0, -1) node[anchor=east] {{\tiny{$u_{32}$}}};
    \path (0, 1) node[anchor=south west] {{\tiny{$d_{2}$}}};
    \path (1, 0) node[anchor=south west] {{\tiny{$d_{1}$}}};
    \path (-1, 1) node[anchor=south west] {{\tiny{$u_{02}$}}};
    \path (1, -1) node[anchor=south west] {{\tiny{$u_{01}$}}};
    \begin{scope}
      \clip (0,0) -- (-1,1) -- (-1,0) -- cycle; \draw (0,0) circle (9pt);
    \end{scope}
    \begin{scope}
      \clip (0,0) -- (-1, 0) -- (0, -1) -- cycle; \draw (0,0) circle (13pt);
    \end{scope}
    \begin{scope}
      \clip (0,0) -- (0,-1) -- (1,-1) -- cycle; \draw (0,0) circle (9pt);
    \end{scope}
    \begin{scope}
      \clip (1,-1) -- (1,0) -- (0,0) -- cycle; \draw[densely dotted,thick] (0,0) circle (13pt);
    \end{scope}
    \begin{scope}
      \clip (1,0) -- (0,1) -- (0,0) -- cycle; \draw[densely dotted,thick] (0,0) circle (9pt);
    \end{scope}
    \begin{scope}
      \clip (0,1) -- (-1,1) -- (0,0) -- cycle; \draw[densely dotted,thick] (0,0) circle (13pt);
    \end{scope}
  \end{tikzpicture}
  &&
  \begin{tikzpicture}[scale=0.7]
    \clip (-0.5, -2.24) -- (0.5, -2.24) -- (0.5, 2.24) -- (-0.5, 2.24) -- cycle;
    \path (0,0) node {{$\longrightarrow$}};
  \end{tikzpicture}
  &&
  \begin{tikzpicture}[scale=0.7]
    \clip (-2.24, -2.24) -- (2.24, -2.24) -- (2.24, 2.24) -- (-2.24, 2.24) -- cycle;
    \fill[color=gray!30] (-3, 3) -- (3, -3) -- (-3, -3) -- cycle;
    \foreach \x in {-3,...,3} \foreach \y in {-3,...,3} \fill (\x, \y) circle (1pt);
    \draw (1, 0) circle (2pt);
    \draw (0, 1) circle (2pt);
    \draw (-1, 0) circle (2pt);
    \draw (0, -1) circle (2pt);
    \draw (1, -1) circle (2pt);
    \draw (0,0) -- (3,0);
    \draw (0,0) -- (0,3);
    \draw (0,0) -- (4,-4);
    \draw (0,0) -- (-3, 0);
    \draw (0,0) -- (0, -3);
    \path (-1, 0) node[anchor=north] {{\tiny{$u_{31}$}}};
    \path (0, -1) node[anchor=east] {{\tiny{$u_{32}$}}};
    \path (0, 1) node[anchor=south west] {{\tiny{$d_{2}$}}};
    \path (1, 0) node[anchor=south west] {{\tiny{$d_{1}$}}};
    \path (1, -1) node[anchor=south west] {{\tiny{$u_{01}$}}};
    \begin{scope}
      \clip (0,0) -- (-1, 0) -- (0, -1) -- cycle; \draw (0,0) circle (13pt);
    \end{scope}
    \begin{scope}
      \clip (0,0) -- (0,-1) -- (1,-1) -- cycle; \draw (0,0) circle (9pt);
    \end{scope}
    \begin{scope}
      \clip (0,0) -- (0,1) -- (-1,0) -- cycle; \draw (0,0) circle (9pt);
    \end{scope}
    \begin{scope}
      \clip (1,-1) -- (1,0) -- (0,0) -- cycle; \draw[densely dotted,thick] (0,0) circle (13pt);
    \end{scope}
    \begin{scope}
      \clip (1,0) -- (0,1) -- (0,0) -- cycle; \draw[densely dotted,thick] (0,0) circle (17pt);
    \end{scope}
  \end{tikzpicture}
  &&
  \begin{tikzpicture}[scale=0.7]
    \clip (-0.5, -2.24) -- (0.5, -2.24) -- (0.5, 2.24) -- (-0.5, 2.24) -- cycle;
    \path (0,0) node {{$\longrightarrow$}};
  \end{tikzpicture}
  &&
  \begin{tikzpicture}[scale=0.7]
    \clip (-2.24, -2.24) -- (2.24, -2.24) -- (2.24, 2.24) -- (-2.24, 2.24) -- cycle;
    \fill[color=gray!30] (-3, 3) -- (3, -3) -- (-3, -3) -- cycle;
    \foreach \x in {-3,...,3} \foreach \y in {-3,...,3} \fill (\x, \y) circle (1pt);
    \draw (1, 0) circle (2pt);
    \draw (0, 1) circle (2pt);
    \draw (-1, 0) circle (2pt);
    \draw (0, -1) circle (2pt);
    \draw (0,0) -- (3,0);
    \draw (0,0) -- (0,3);
    \draw (0,0) -- (-3, 0);
    \draw (0,0) -- (0, -3);
    \path (-1, 0) node[anchor=north] {{\tiny{$u_{31}$}}};
    \path (0, -1) node[anchor=west] {{\tiny{$u_{32}$}}};
    \path (0, 1) node[anchor=east] {{\tiny{$d_{2}$}}};
    \path (1, 0) node[anchor=south] {{\tiny{$d_{1}$}}};
    \begin{scope}
      \clip (0,0) -- (0,1) -- (-1,0) -- cycle; \draw (0,0) circle (9pt);
    \end{scope}
    \begin{scope}
      \clip (0,0) -- (-1, 0) -- (0, -1) -- cycle; \draw (0,0) circle (13pt);
    \end{scope}
    \begin{scope}
      \clip (0,0) -- (0,-1) -- (1,0) -- cycle; \draw (0,0) circle (9pt);
    \end{scope}
  \end{tikzpicture}
\end{align*}
Here the elements $u_{31} = -d_1$ and $u_{32} = -d_2$ are the
valuations of the $G$-invariant prime divisors $\Vd(z_{31})$ and
$\Vd(z_{32})$, respectively, while the elements $u_{01} = d_1-d_2$
and $u_{02} = -d_1+d_2$ are the valuations of the exceptional divisors
$E_{01}$ and $E_{02}$ over $C_{01}$ and $C_{02}$, respectively. In
particular, we see that $X_1$ is the fourth line and that $X_4$ is the
last line of Hofscheier's list.

The dotted circles in the colored fans of $X_1$ and $X_4$ specify
projective ambient toric varieties $Y_1$ and $Y_4$, respectively. From
the description of $\Sigmamax$ in Section~\ref{sec:ambient}, we deduce
that $Y_1$ and $Y_4$ are smooth, that $-K_{X_1}$ is ample on $Y_1$,
and that $-K_{X_4}$ is ample on $Y_4$. Hence
assumption~\eqref{eq:toric_smooth} holds.

\subsubsection{$X_2$ of  type III.20}
  Consider $\Pd^4_\Qd$ with coordinates
  $(z_{11} : z_{12} : z_{21} : z_{22} : z_{33})$ and the hypersurface
  $Q = \Vd(z_{11}z_{12} - z_{21}z_{22} - z_{33}^2) \subset \Pd^4_\Qd$. It
  contains the lines
  \begin{align*}
    C_{31} = \Vd(z_{12},z_{22},z_{33}), \quad  C_{32} = \Vd(z_{11},z_{21},z_{33})\text{.}
  \end{align*}
  Let $X_2$ be the blow-up of $Q$ in the lines $C_{31}$ and $C_{32}$.
  This is a smooth Fano threefold of type III.20. We may define an
  action of $G = \SL_2 \times \Gd_\mathrm{m}$ on $Q$ by
  \begin{equation*}
    (A,t)\cdot \rleft(
    \begin{pmatrix}
      z_{11} & z_{22} \\
      z_{21} & z_{12}
    \end{pmatrix}
    ,z_{33}\rright) = \rleft(A\cdot \begin{pmatrix}
      z_{11} & z_{22} \\
      z_{21} & z_{12}
    \end{pmatrix} \cdot
    \begin{pmatrix}
      t^{-1} & 0 \\
      0 & t
    \end{pmatrix},z_{33}\rright),
  \end{equation*}
  which turns $Q$ into a spherical variety. Since the lines $C_{31}$
  and $C_{32}$ are $G$-invariant, the variety $X_2$ is a spherical
  $G$-variety. Since $X_2$ is also the blow-up of $W_4$ in the curve
  $C_{33} = \Vd(z_{31}, z_{32})$, it has the same birational
  invariants as $W_4$, and the blow-up morphisms
  $Q \leftarrow X_2 \to W_4$ can be described by maps of colored fans
  as illustrated in the following picture.
  
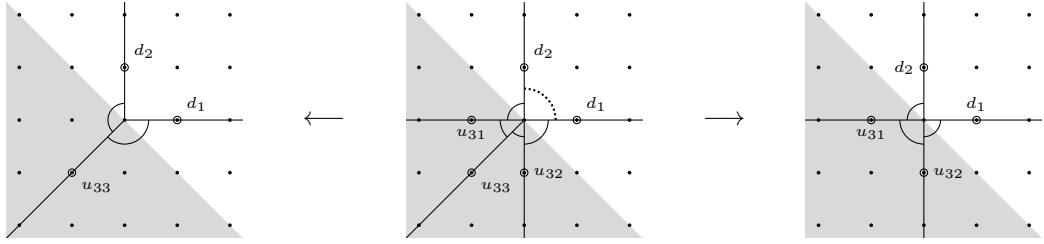
\begin{align*}
    \begin{tikzpicture}[scale=0.7]
    \clip (-2.24, -2.24) -- (2.24, -2.24) -- (2.24, 2.24) -- (-2.24, 2.24) -- cycle;
    \fill[color=gray!30] (-3, 3) -- (3, -3) -- (-3, -3) -- cycle;
    \foreach \x in {-3,...,3} \foreach \y in {-3,...,3} \fill (\x, \y) circle (1pt);
    \draw (1, 0) circle (2pt);
    \draw (0, 1) circle (2pt);
    \draw (-1, -1) circle (2pt);
    \draw (0,0) -- (3,0);
    \draw (0,0) -- (0,3);
    \draw (0,0) -- (-4,-4);
    \path (0, 1) node[anchor=south west] {{\tiny{$d_{2}$}}};
    \path (1, 0) node[anchor=south west] {{\tiny{$d_{1}$}}};
    \path (-1, -1) node[anchor=north west] {{\tiny{$u_{33}$}}};
    \begin{scope}
      \clip (0,0) -- (0,1) -- (-1,-1) -- cycle; \draw (0,0) circle (9pt);
    \end{scope}
    \begin{scope}
      \clip (0,0) -- (-2, -2) -- (2, 0) -- cycle; \draw (0,0) circle (13pt);
    \end{scope}
  \end{tikzpicture}
  &&
  \begin{tikzpicture}[scale=0.7]
    \clip (-0.5, -2.24) -- (0.5, -2.24) -- (0.5, 2.24) -- (-0.5, 2.24) -- cycle;
    \path (0,0) node {{$\longleftarrow$}};
  \end{tikzpicture}
  &&
  \begin{tikzpicture}[scale=0.7]
    \clip (-2.24, -2.24) -- (2.24, -2.24) -- (2.24, 2.24) -- (-2.24, 2.24) -- cycle;
    \fill[color=gray!30] (-3, 3) -- (3, -3) -- (-3, -3) -- cycle;
    \foreach \x in {-3,...,3} \foreach \y in {-3,...,3} \fill (\x, \y) circle (1pt);
    \draw (1, 0) circle (2pt);
    \draw (0, 1) circle (2pt);
    \draw (-1, 0) circle (2pt);
    \draw (0, -1) circle (2pt);
    \draw (-1, -1) circle (2pt);
    \draw (0,0) -- (3,0);
    \draw (0,0) -- (0,3);
    \draw (0,0) -- (-4,-4);
    \draw (0,0) -- (-3, 0);
    \draw (0,0) -- (0, -3);
    \path (-1, 0) node[anchor=north] {{\tiny{$u_{31}$}}};
    \path (0, -1) node[anchor=west] {{\tiny{$u_{32}$}}};
    \path (0, 1) node[anchor=south west] {{\tiny{$d_{2}$}}};
    \path (1, 0) node[anchor=south west] {{\tiny{$d_{1}$}}};
    \path (-1, -1) node[anchor=north west] {{\tiny{$u_{33}$}}};
    \begin{scope}
      \clip (0,0) -- (0,1) -- (-1,0) -- cycle; \draw (0,0) circle (9pt);
    \end{scope}
    \begin{scope}
      \clip (0,0) -- (-1, 0) -- (-1, -1) -- cycle; \draw (0,0) circle (13pt);
    \end{scope}
    \begin{scope}
      \clip (0,0) -- (-1,-1) -- (0,-1) -- cycle; \draw (0,0) circle (9pt);
    \end{scope}
    \begin{scope}
      \clip (0,0) -- (0,-1) -- (1,0) -- cycle; \draw (0,0) circle (13pt);
    \end{scope}
    \begin{scope}
      \clip (1,0) -- (0,1) -- (0,0) -- cycle; \draw[densely dotted,thick] (0,0) circle (17pt);
    \end{scope}
  \end{tikzpicture}
  &&
  \begin{tikzpicture}[scale=0.7]
    \clip (-0.5, -2.24) -- (0.5, -2.24) -- (0.5, 2.24) -- (-0.5, 2.24) -- cycle;
    \path (0,0) node {{$\longrightarrow$}};
  \end{tikzpicture}
  &&
  \begin{tikzpicture}[scale=0.7]
    \clip (-2.24, -2.24) -- (2.24, -2.24) -- (2.24, 2.24) -- (-2.24, 2.24) -- cycle;
    \fill[color=gray!30] (-3, 3) -- (3, -3) -- (-3, -3) -- cycle;
    \foreach \x in {-3,...,3} \foreach \y in {-3,...,3} \fill (\x, \y) circle (1pt);
    \draw (1, 0) circle (2pt);
    \draw (0, 1) circle (2pt);
    \draw (-1, 0) circle (2pt);
    \draw (0, -1) circle (2pt);
    \draw (0,0) -- (3,0);
    \draw (0,0) -- (0,3);
    \draw (0,0) -- (-3, 0);
    \draw (0,0) -- (0, -3);
    \path (-1, 0) node[anchor=north] {{\tiny{$u_{31}$}}};
    \path (0, -1) node[anchor=west] {{\tiny{$u_{32}$}}};
    \path (0, 1) node[anchor=east] {{\tiny{$d_{2}$}}};
    \path (1, 0) node[anchor=south] {{\tiny{$d_{1}$}}};
    \begin{scope}
      \clip (0,0) -- (0,1) -- (-1,0) -- cycle; \draw (0,0) circle (9pt);
    \end{scope}
    \begin{scope}
      \clip (0,0) -- (-1, 0) -- (0, -1) -- cycle; \draw (0,0) circle (13pt);
    \end{scope}
    \begin{scope}
      \clip (0,0) -- (0,-1) -- (1,0) -- cycle; \draw (0,0) circle (9pt);
    \end{scope}
  \end{tikzpicture}
\end{align*}
In particular, we see that $X_2$ is the fifth line of Hofscheier's list.

As before, the dotted circle in the colored fan of $X_2$ specifies a
projective ambient toric variety $Y_2$, which satisfies
\eqref{eq:toric_smooth}.

\subsubsection{$X_3$ of type IV.8} Consider
$W_3=\Pd^1_\Qd \times \Pd^1_\Qd \times \Pd^1_\Qd$ with coordinates
$(z_{01}:z_{02})$, $(z_{11}:z_{21})$ and $(z_{12}:z_{22})$. This is a smooth
Fano threefold of type III.27. Let $C_{31}$ be the curve
$\Vd(z_{02},z_{11}z_{12}-z_{21}z_{22})$ of tridegree $(0,1,1)$ on $W_3$. Let
$X_3$ be the blow-up of $W_3$ in $C_{31}$. This is a smooth Fano threefold of
type IV.8. We may define an action of $G = \SL_2 \times \Gd_\mathrm{m}$ on
$W_3$ by
\begin{equation*}
  (A,t)\cdot \rleft(z_{01},z_{02},
  \begin{pmatrix}
    z_{11} & z_{22} \\
    z_{21} & z_{12}
  \end{pmatrix}
  \rright) = \rleft(t\cdot z_{01}, z_{02}, A\cdot \begin{pmatrix}
    z_{11} & z_{22} \\
    z_{21} & z_{12}
  \end{pmatrix}\rright)\text{,}
\end{equation*}
which turns $W_3$ into a spherical variety.  Its Luna--Vust description is a
follows.  The lattice $\Mm$ has basis $(\alpha, \varepsilon_1)$. We denote the
corresponding dual basis of the lattice $\Nm$ by $(d, \varepsilon_1^*)$. Then
there are two colors with the same valuation $d = \frac{1}{2}\alpha^\vee$, and
the valuation cone is given by
$\Vm = \{v \in \Nm_\Qd : \langle v, \alpha \rangle \le 0\}$. Since the curve
$C_{31}$ is $G$-invariant, the variety $X_3$ is a spherical $G$-variety, and
the blow-up morphism $X_3 \to W_3$ can be described by the map of colored fans
in the figure below.

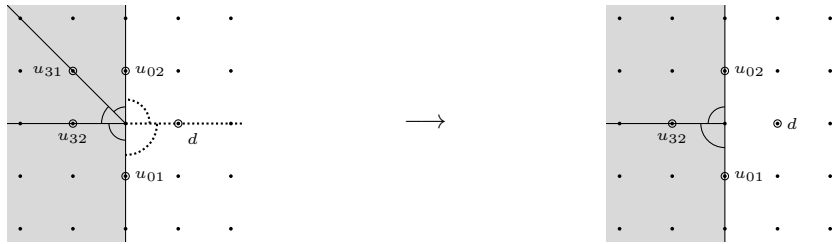
\begin{align*}
  \begin{tikzpicture}[scale=0.7]
    \clip (-2.24, -2.24) -- (2.24, -2.24) -- (2.24, 2.24) -- (-2.24, 2.24) -- cycle;
    \fill[color=gray!30] (0, 3) -- (0, -3) -- (-3, -3) -- (-3, 3) -- cycle;
    \foreach \x in {-3,...,3} \foreach \y in {-3,...,3} \fill (\x, \y) circle (1pt);
    \draw (1, 0) circle (2pt);
    \draw (0, 1) circle (2pt);
    \draw (-1, 0) circle (2pt);
    \draw (0, -1) circle (2pt);
    \draw (-1, 1) circle (2pt);
    \draw (0,0) -- (0,3);
    \draw (0,0) -- (-4,4);
    \draw (0,0) -- (-3, 0);
    \draw (0,0) -- (0, -3);
    \path (-1, 0) node[anchor=north] {{\tiny{$u_{32}$}}};
    \path (0, -1) node[anchor=west] {{\tiny{$u_{01}$}}};
    \path (0, 1) node[anchor=west] {{\tiny{$u_{02}$}}};
    \path (1, 0) node[anchor=north west] {{\tiny{$d$}}};
    \path (-1, 1) node[anchor=east] {{\tiny{$u_{31}$}}};
    \begin{scope}
      \clip (0,0) -- (0,1) -- (-1,1) -- cycle; \draw (0,0) circle (9pt);
    \end{scope}
    \begin{scope}
      \clip (0,0) -- (-1,1) -- (-1,0) -- cycle; \draw (0,0) circle (13pt);
    \end{scope}
    \begin{scope}
      \clip (0,0) -- (-1, 0) -- (0, -1) -- cycle; \draw (0,0) circle (9pt);
    \end{scope}
    \draw[densely dotted, thick] (0,0) -- (4,0);
    \begin{scope}
      \clip (1,0) -- (0,1) -- (0,0) -- cycle; \draw[densely dotted,thick] (0,0) circle (13pt);
    \end{scope}
    \begin{scope}
      \clip (0,-1) -- (1,0) -- (0,0) -- cycle; \draw[densely dotted,thick] (0,0) circle (17pt);
    \end{scope}
  \end{tikzpicture}
  &&
  \begin{tikzpicture}[scale=0.7]
    \clip (-0.5, -2.24) -- (0.5, -2.24) -- (0.5, 2.24) -- (-0.5, 2.24) -- cycle;
    \path (0,0) node {{$\longrightarrow$}};
  \end{tikzpicture}
  &&
    \begin{tikzpicture}[scale=0.7]
    \clip (-2.24, -2.24) -- (2.24, -2.24) -- (2.24, 2.24) -- (-2.24, 2.24) -- cycle;
    \fill[color=gray!30] (0, 3) -- (0, -3) -- (-3, -3) -- (-3, 3) -- cycle;
    \foreach \x in {-3,...,3} \foreach \y in {-3,...,3} \fill (\x, \y) circle (1pt);
    \draw (1, 0) circle (2pt);
    \draw (0, 1) circle (2pt);
    \draw (-1, 0) circle (2pt);
    \draw (0, -1) circle (2pt);
    \draw (0,0) -- (0,3);
    \draw (0,0) -- (-3, 0);
    \draw (0,0) -- (0, -3);
    \path (-1, 0) node[anchor=north] {{\tiny{$u_{32}$}}};
    \path (0, -1) node[anchor=west] {{\tiny{$u_{01}$}}};
    \path (0, 1) node[anchor=west] {{\tiny{$u_{02}$}}};
    \path (1, 0) node[anchor=west] {{\tiny{$d$}}};
    \begin{scope}
      \clip (0,0) -- (0,1) -- (-1,0) -- cycle; \draw (0,0) circle (9pt);
    \end{scope}
    \begin{scope}
      \clip (0,0) -- (-1, 0) -- (0, -1) -- cycle; \draw (0,0) circle (13pt);
    \end{scope}
  \end{tikzpicture}
\end{align*}
Here the elements $u_{01} = -\varepsilon_1^*$ and $u_{02} = \varepsilon_1^*$
are the valuations of the $G$-invariant prime divisors $\Vd(z_{01})$ and
$\Vd(z_{02})$, respectively, the element $u_{32} = -d$ is the valuation of the
$G$-invariant prime divisor $\Vd(z_{11}z_{12}-z_{21}z_{22})$, and
$u_{31} = -d + \varepsilon_1^*$ is the valuation of the exceptional divisor
$E_{31}$ over $C_{31}$.  This is the penultimate line of Hofscheier's list.

The dotted circles in the colored fan of $X_3$ are meant to specify a
projective ambient toric variety~$Y_3$, but since there are two colors with
the same valuation $d$, the picture is ambiguous. There are three
possibilities for which unsupported colored cones could be added to the
colored cone of $X_3$ to obtain an ambient toric variety:
\begin{enumerate}
\item $(\cone(u_{01}, d), \{D_{11}\})$ and $(\cone(u_{02}, d), \{D_{11}\})$,
\item $(\cone(u_{01}, d), \{D_{12}\})$ and $(\cone(u_{02}, d), \{D_{12}\})$, or
\item $(\cone(u_{01}, d), \{D_{11}, D_{12}\})$ and $(\cone(u_{02}, d), \{D_{11}, D_{12}\})$.
\end{enumerate}
From the description of $\Sigmamax$ in Section~\ref{sec:ambient}, we deduce
that the ambient toric variety in case $(3)$ is singular. On the other hand,
in cases $(1)$ and $(2)$, the ambient toric variety is smooth, and $-K_{X_3}$
not ample but semiample on it.  We fix $Y_3$ to be as in case $(1)$,
satisfying \eqref{eq:toric_smooth}.
  
\subsection{Cox rings and torsors}\label{sec:crt}

We proceed to compute explicitly the Cox rings $\Rm(X)$ in the examples from
Section~\ref{sec22} using Section~\ref{sec:ssr1} together with
\cite{arXiv:1408.5358} since we work over $\Qd$ here. To obtain the universal
torsor $\Tm=X_0$, we compute the set $Z_Y$ as in
Section~\ref{sec:torsors_models}. Moreover, we give simplified expressions for
$Z_X = Z_Y \cap \Spec \Rm(X)$, which can be verified using the equation
$\Phi$. Finally the anticanonical class is computed using \cite[4.1 and
4.2]{bri97} or \cite[Proposition~3.3.3.2]{adhl15}. In the case of a spherical
variety of semisimple rank one of type $T$ or $N$, this is simply the sum of
all $B$-invariant divisors.

\subsubsection{Type III.24}

We have
\begin{align*}
  \Rm(X_1) = \Qd[x_{01},x_{11},x_{12},x_{21},x_{22},x_{31},x_{32}]
  /(x_{11}x_{12}-x_{21}x_{22}-x_{31}x_{32})
\end{align*}
with $\Pic X_1 \cong \Z^3$, where
\begin{align*}
  & \deg(x_{01})=(0,0,1), \quad  \deg(x_{11})=\deg(x_{21})=(0,1,-1),\\
  &\deg(x_{12})=\deg(x_{22})=(1,0,0), \quad  \deg(x_{31})=(0,1,0),
    \quad  \deg(x_{32})=(1,0,-1)\text{.}
\end{align*}
Note that each generator $x_{ij}$ of the Cox ring corresponds to the strict
transform of $\Vd(z_{ij})$ or to the element $u_{ij}$ in
Section~\ref{IV.7}. The anticanonical class is $-K_{X_1}=(2,2,-1)$. A
universal torsor over $X_1$ is
\begin{equation*}
  \Tm_1 = \Spec\Rm(X_1) \setminus Z_{Y_1} = \Spec\Rm(X_1) \setminus Z_{X_1}\text{,}
\end{equation*}
where
\begin{align*}
  Z_{Y_1} &= \Vd(x_{11},x_{21},x_{31}) \cup \Vd(x_{11},x_{21},x_{32})
            \cup \Vd(x_{12},x_{22},x_{01}) \cup \Vd(x_{12},x_{22},x_{32})
            \cup \Vd(x_{01},x_{31})\text{,}\\
  Z_{X_1} &= \Vd(x_{11},x_{21}) \cup \Vd(x_{12},x_{22},x_{32}) \cup \Vd(x_{01},x_{31})\text{.}
\end{align*}

\subsubsection{Type III.20}

The Cox ring is
\begin{align*}
  \Rm(X_2) = \Qd[x_{11},x_{12},x_{21},x_{22},x_{31},x_{32},x_{33}]
  /(x_{11}x_{12}-x_{21}x_{22}-x_{31}x_{32}x_{33}^2)
\end{align*}
with $\Pic X_2 \cong \Z^3$, where
\begin{align*}
  & \deg(x_{11})=\deg(x_{21})=(0,1,0)\text{,}\quad  \deg(x_{12})=\deg(x_{22})=(1,0,0)\text{,}\\
  & \deg(x_{31})=(0,1,-1), \quad  \deg(x_{32})=(1,0,-1), \quad \deg(x_{33}) = (0,0,1)\text{.}
\end{align*}
The anticanonical class is $-K_{X_2}=(2,2,-1)$.
A universal torsor over $X_2$ is
\begin{equation*}
  \Tm_2 = \Spec\Rm(X_2) \setminus Z_{Y_2} = \Spec\Rm(X_2) \setminus Z_{X_2}\text{,}
\end{equation*}
where
\begin{align*}
  Z_{Y_2} &= \Vd(x_{11},x_{21},x_{31}) \cup \Vd(x_{11},x_{21},x_{33}) \cup
            \Vd(x_{12},x_{22},x_{32}) \cup \Vd(x_{12},x_{22},x_{33}) \cup \Vd(x_{31},x_{32}),\\
  Z_{X_2} &= \Vd(x_{11},x_{21},x_{31}) \cup \Vd(x_{11},x_{21},x_{33}) \cup
            \Vd(x_{12},x_{22},x_{32}) \cup \Vd(x_{12},x_{22},x_{33}) \cup \Vd(x_{31},x_{32}).
\end{align*}

\subsubsection{Type IV.8}

The Cox ring is
\begin{align*}
  \Rm(X_3) = \Qd[x_{01},x_{02},x_{11},x_{12},x_{21},x_{22},x_{31},x_{32}]
  /(x_{11}x_{12}-x_{21}x_{22}-x_{31}x_{32})
\end{align*}
with $\Pic X_3 \cong \Z^4$, where
\begin{align*}
  & \deg(x_{01})=(1,0,0,0), \quad \deg(x_{02})=(1,0,0,-1)\text{,}\\
  & \deg(x_{11})=\deg(x_{21})=(0,0,1,0)\text{,}\quad  \deg(x_{12})=\deg(x_{22})=(0,1,0,0)\text{,}\\
  & \deg(x_{31})=(0,0,0,1), \quad  \deg(x_{32})=(0,1,1,-1)\text{.}
\end{align*}
The anticanonical class is $-K_{X_3}=(2,2,2,-1)$. A universal torsor over $X_3$ is
\begin{equation*}
  \Tm_3 = \Spec\Rm(X_3) \setminus Z_{Y_3} = \Spec\Rm(X_3) \setminus Z_{X_3}\text{,}
\end{equation*}
where
\begin{align*}
  Z_{Y_3} &= \Vd(x_{11},x_{21},x_{31}) \cup \Vd(x_{11},x_{21},x_{32}) \cup \Vd(x_{12},x_{22})
            \cup \Vd(x_{02},x_{32}) \cup \Vd(x_{01},x_{02}) \cup \Vd(x_{01},x_{31}),\\
  Z_{X_3} &= \Vd(x_{11},x_{21}) \cup \Vd(x_{12},x_{22}) \cup
            \Vd(x_{02},x_{32}) \cup \Vd(x_{01},x_{02}) \cup \Vd(x_{01},x_{31}).
\end{align*}

\subsubsection{Type IV.7} 

The Cox ring is
\begin{align*}
  \Rm(X_4) = \Qd[x_{01},x_{02},x_{11},x_{12},x_{21},x_{22},x_{31},x_{32}]
  /(x_{11}x_{12}-x_{21}x_{22}-x_{31}x_{32})
\end{align*}
with $\Pic X_4 \cong \Z^4$, where
\begin{align*}
  &  \deg(x_{01})=(0,0,0,1), \quad   \deg(x_{02})=(0,0,1,0)\text{,}\\
  &  \deg(x_{11})=\deg(x_{21})=(0,1,0,-1)\text{,}\quad  \deg(x_{12})=\deg(x_{22})=(1,0,-1,0)\text{,}\\
  & \deg(x_{31})=(0,1,-1,0), \quad  \deg(x_{32})=(1,0,0,-1)\text{.}
\end{align*}
The anticanonical class is $-K_{X_4} = (2,2,-1,-1)$.  A universal torsor is
over $X_4$ is
\begin{equation*}
  \Tm_4 = \Spec\Rm(X_4) \setminus Z_{Y_4} = \Spec\Rm(X_4) \setminus Z_{X_4}\text{,}
\end{equation*}
where
\begin{align*}
  Z_{Y_4} &= \Vd(x_{11},x_{21},x_{01}) \cup \Vd(x_{11},x_{21},x_{31}) \cup \Vd(x_{11},x_{21},x_{32})\\
          &\qquad \cup \Vd(x_{12},x_{22},x_{02}) \cup
            \Vd(x_{12},x_{22},x_{31}) \cup \Vd(x_{12},x_{22},x_{32})\\
          &\qquad \cup \Vd(x_{02},x_{32}) \cup \Vd(x_{01},x_{02}) \cup \Vd(x_{01},x_{31}),\\
  Z_{X_4} &= \Vd(x_{11},x_{21}) \cup \Vd(x_{12},x_{22}) \cup
            \Vd(x_{02},x_{32}) \cup \Vd(x_{01},x_{02}) \cup \Vd(x_{01},x_{31}).
\end{align*}
Note that this is the same variety as $\Tm_3$, but with a different action of
$\Gd_{\mathrm{m},\Qd}^4$.

\subsection{Counting problems}

Applying Proposition~\ref{prop:countingproblem_abstract} to the Cox
rings of the previous section gives the following counting problems,
in which $U$ is always the subset where all Cox coordinates are
nonzero. To lighten the notation, we generally write $\{x, y\}$ to
mean $x$ or $y$, and as in the introduction, we write $N_j(B)$ for
$N_{X_j, U_j, H_j}(B)$.

\begin{cor}\label{prop:countingproblem_line_11} 
 {\rm (a)}  We have 
 \begin{align*}
  N_1(B) =   \frac{1}{8} \#\left\{\xv \in \Zd^7_{\ne 0} :
    \begin{aligned}
      &x_{11}x_{12}-x_{21}x_{22}-x_{31}x_{32}=0, \quad \max|\Pm_1(\xv)| \le B,\\
    &(x_{11},x_{21})=(x_{12},x_{22}, x_{32})=(x_{01},x_{31}) =1\\
    \end{aligned}
        \right\}\text{,}
  \end{align*}
  where
  \begin{equation*}
    \Pm_1(\xv) =\left\{
      \begin{aligned}
        & x_{31}^2x_{32}^2x_{01}, x_{32}^2x_{01}^3\{x_{11}, x_{21}\}^2,
        x_{31}^2x_{32}\{x_{12}, x_{22}\},\\
        &   x_{31} \{x_{11}, x_{21}\}  \{x_{12}, x_{22}\}^2, x_{01}  \{x_{11},
        x_{21}\}^2  \{x_{12}, x_{22}\}^2
      \end{aligned}
    \right\}.
  \end{equation*}
 \\
 {\rm (b)}  We have
   \begin{align*}
   N_2(B) =  \frac{1}{8} \#\left\{\xv \in \Zd^7_{\ne 0} :
    \begin{aligned}
      &x_{11}x_{12}-x_{21}x_{22}-x_{31}x_{32}x_{33}^2=0, \quad \max|\Pm_2(\xv)| \le B,\\
      &(x_{11},x_{21},x_{31})=(x_{11},x_{21},x_{33})=1\\
      &(x_{12},x_{22},x_{32})=(x_{12},x_{22},x_{33})=(x_{31}, x_{32})=1\\
    \end{aligned}
        \right\}\text{,}
  \end{align*}
  where
  \begin{equation*}
    \Pm_2(\xv) =\left\{
      \begin{aligned}
        & x_{32}\{x_{11}, x_{21}\}^2  \{x_{12}, x_{22}\},
        x_{32}^2x_{33}\{x_{11}, x_{21}\}^2, x_{31}\{x_{11}, x_{21}\}
        \{x_{12}, x_{22}\}^2, \\
        & x_{31}^2x_{33}\{x_{12}, x_{22}\}^2, x_{31}^2x_{32}^2x_{33}^3 \end{aligned}
    \right\}.
  \end{equation*}
 \\
 {\rm (c)}  We have 
  \begin{align*}
   N_3(B) =  \frac{1}{16} \#\left\{\xv \in \Zd^8_{\ne 0} :
    \begin{aligned}
      &x_{11}x_{12}-x_{21}x_{22}-x_{31}x_{32}=0, \quad \max|\Pm_3(\xv)| \le B,\\
    &(x_{11},x_{21})=(x_{12},x_{22})=(x_{02},x_{32})=(x_{01},x_{02})=(x_{01},x_{31})=1\\
    \end{aligned}
        \right\}\text{,}
  \end{align*}
  where
  \begin{equation*}
    \Pm_3(\xv) =\left\{
      \begin{aligned}
        &x_{02}^2x_{31}^3 x_{32}^2 , x_{01}^2x_{31}x_{32}^2,
        x_{02}^2\{x_{11},x_{21}\}^2\{x_{12},x_{22}\}^2x_{31}\\
        &x_{01}^2\{x_{11},x_{21}\}\{x_{12},x_{22}\}x_{32},
        x_{01}x_{02}\{x_{11},x_{21}\}^2\{x_{12},x_{22}\}^2
      \end{aligned}
    \right\}.
  \end{equation*}
 \\
 {\rm   (d)} We have   \begin{align*}
   N_4(B) =  \frac{1}{16} \#\left\{\xv \in \Zd^8_{\ne 0} :
    \begin{aligned}
      &x_{11}x_{12}-x_{21}x_{22}-x_{31}x_{32}=0, \quad \max|\Pm_4(\xv)| \le B,\\
    &(x_{11},x_{21})=(x_{12},x_{22})=(x_{02},x_{32})=(x_{01},x_{02})=(x_{01},x_{31})=1\\
    \end{aligned}
        \right\}\text{,}
  \end{align*}
  where
\begin{equation*}
  \Pm_4(\xv) =\left\{
    \begin{aligned}
      &x_{01} x_{02} x_{31}^2 x_{32}^2 ,
      x_{01}^2\{x_{11},x_{21}\}x_{31}x_{32}^2,
      x_{02}^2\{x_{12},x_{22}\}x_{31}^2x_{32},\\ 
      &x_{01}^2\{x_{11},x_{21}\}^2\{x_{12},x_{22}\}x_{32},
      x_{02}^2\{x_{11},x_{21}\}\{x_{12},x_{22}\}^2x_{31} ,
      x_{01}x_{02}\{x_{11},x_{21}\}^2\{x_{12},x_{22}\}^2
    \end{aligned}
  \right\}. 
\end{equation*}
 \end{cor}

 \begin{proof}
   This is a special case of
   Proposition~\ref{prop:countingproblem_abstract}.  Note that the
   coprimality conditions are derived from the expressions for $Z_X$
   (instead of $Z_Y$) from Section~\ref{sec:crt}. It can be explicitly
   verified using the equation $\Phi$ that this is correct even over
   $\Zd$ as required here.
 \end{proof}

\subsection{Application:  Proof of Theorem~\ref{dim3}}\label{appl1}

We now show how to use Theorem~\ref{manin-cor} in practice and complete the
proof of Theorem~\ref{dim3} for the varieties $X_1, \ldots, X_4$.

\subsubsection{The variety $X_4$}\label{X4}

By Corollary~\ref{prop:countingproblem_line_11}(d), we have $J=8$ torsor
variables $x_{ij}$ with $0 \leq i \leq 3$, $1 \leq j \leq 2$ satisfying the
equation
\begin{equation}\label{particular}
x_{11}x_{12} + x_{21}x_{22} + x_{31}x_{32} = 0
\end{equation}
(after changing the signs of $x_{22},x_{32}$)
with $k=3$ and $h_{ij} = 1$ for $i \geq 1$, $h_{0j} = 0$. In particular,
Remark~\ref{simpl} applies.  We have $N=17$ height conditions
with corresponding exponent matrix
\begin{equation*}
  \mathscr{A}_1 = \left(
    \begin{smallmatrix}
      1  & 2& 2& & & 2&2 &2 & 2& & & & &1 &1 & 1& 1\\
      1& & &2 &2 & & & & &2 & 2&2 &2 &1 & 1& 1& 1\\
      & & 1& & & &2 & &2 & &1 & &1 & & &2 &2 \\
      & & & 1& & 1&1 & & & 2& 2& & &2 & & 2& \\
      &1 & & & &2 & &2 & & 1& & 1& &2 &2 & & \\
      & & & &1 & & &1 &1 & & &2 &2 & &2 & & 2\\
      2 & 1& 1&2 &2 & & & & &1 & 1& 1&1 & & & & \\
      2  &2 & 2& 1& 1& 1&1 & 1& 1& & & & & & & & \\
    \end{smallmatrix}
  \right)\in \Bbb{R}_{\geq 0}^{8 \times 17}, \quad
  \mathscr{A}_2 = \left(
    \begin{smallmatrix}
      &&-1\\&&-1\\1&&-1\\1&&-1\\&1&-1\\&1&-1\\-1&-1&\\-1&-1&
    \end{smallmatrix}
  \right) \in \Bbb{R}^{8 \times 3}.
\end{equation*}
As usual, missing entries indicate zeros. We have $r=5$ coprimality conditions
with 
\begin{equation}\label{gcd1}
\begin{split}
  &S_1 = \{(1, 1), (2, 1)\}, \quad S_2 = \{(1, 2), (2, 2)\}, \quad S_3 = \{(0, 2), (3, 2)\},\\
  & S_4 = \{(0, 1), (0, 2)\}, \quad S_5 = \{(0, 1), (3, 1)\}.
\end{split}
\end{equation}
We choose 
\begin{equation}\label{tauzeta}
  {\bm \tau}^{(2)}   = (\underbrace{1, \ldots , 1}_{J_0}, \tfrac{2}{3}, \ldots,  \tfrac{2}{3}), \quad \bm \zeta = (\tfrac{1}{3}, \tfrac{1}{3}, \tfrac{1}{3}).  
\end{equation}  
(In our case $J_0 = 2$, but we will use the same definition also in other
cases later.)  Using a computer algebra system, we confirm
$C_2({\bm \tau}^{(2)} )$, $C_2((1 - h_{ij}/3)_{ij})$, and with $c_2 = 3$, we
find
$$\dim (\mathscr{H} \cap \mathscr{P}) = 3, \quad    \dim (\mathscr{H} \cap
\mathscr{P}_{ij}) = 2 \ \text{for all} \ (i,j),$$ confirming \eqref{ass2}.  We
have now checked all assumptions of Theorem~\ref{manin-cor}.
   
We show in Appendix~\ref{A} how to derive Hypothesis~\ref{H2} without computer
help and how to compute the Peyre constant in explicit algebraic terms.
  
\subsubsection{The variety $X_3$}

This is very similar to the previous case, so we can be brief. By
Corollary~\ref{prop:countingproblem_line_11}(c), we have the same torsor
variables as in the previous application satisfying \eqref{particular}.  The
corresponding exponent matrix is given by
\begin{equation*}
\mathscr{A}_1 = \left(\begin{smallmatrix}
    & 2& & & & &2 &2 &2 &2 & 1& 1&1 &1 \\
    2 & & 2& 2&2 &2 & & & & &1 & 1& 1&1 \\
    & & &2 & &2 & &1 & &1 & &2 & &2 \\
    & &2 &2 & & & 1&1 & & &2 &2 & & \\
    & &2 & & 2& &1 & &1 & & 2& &2 & \\
    & & & &2 & 2& & &1 &1 & & &2 &2 \\
    3 &1 & 1& 1&1 & 1& & & & & & & & \\
    2 &2 & & & & &1 &1 &1 &1 & & & & \\
  \end{smallmatrix}\right) \in \Bbb{R}_{\geq 0}^{8 \times 14}.
\end{equation*}
We choose ${\bm \tau}^{(2)}$ and ${\bm \zeta}$ as before and confirm \eqref{ass2} in the same
way with
$$\dim (\mathscr{H} \cap \mathscr{P}) = 3, \quad    \dim (\mathscr{H} \cap
\mathscr{P}_{ij}) = 1 \ \text{for} \ (i,j) = (0,1) \ \text{and} \ \dim
(\mathscr{H} \cap \mathscr{P}_{ij}) = 2 \ \text{otherwise}.$$
 
\subsubsection{The variety $X_1$}

Again the computations are a minor variation on the previous two cases. By
Corollary~\ref{prop:countingproblem_line_11}(a), the height matrix is
$$\mathscr{A}_1 = \left(\begin{smallmatrix} 
    1 & 3 & 3 &    &    & & & & & 1 & 1 & 1 & 1\\  
    & 2 & & &  & 1 & 1& &  &2&2&& \\
    & & & &1 & & 2& & 2& &2 & &2\\
    & & 2& & & & &1 & 1& & &2 & 2\\
    & & &1 & & 2& &2 & &2 & &2 &  \\
    2 & & & 2& 2&1 &1 &1 &1 & & & & \\
    2  & 2&2 &1 & 1& & & & & & & &  
  \end{smallmatrix}\right) \in \Bbb{R}^{7 \times 13}_{\ge 0}.$$
We make the same choice \eqref{tauzeta} for ${\bm \tau}^{(2)}$ and ${\bm \zeta}$, and confirm
\eqref{ass2} with $c_2 = 2$ and
$$\dim (\mathscr{H} \cap \mathscr{P}) = 2, \quad  \dim (\mathscr{H} \cap
\mathscr{P}_{ij}) = 0 \ \text{for} \ (i,j) =(1,2),(2,2),(3,1),\ \dim
(\mathscr{H} \cap \mathscr{P}_{ij}) = 1 \ \text{otherwise}.$$
 
\subsubsection{The variety $X_2$}

This case has some new features, as the torsor equation has a slightly
different shape. By Corollary~\ref{prop:countingproblem_line_11}(b), we have
$J_0 = 0$ and $J = 7$ torsor variables satisfying the more complicated torsor
equation
$$x_{11}x_{12} + x_{21} x_{22} + x_{31} x_{32} x_{33}^2 = 0.$$
The height matrix is given by
$$\mathscr{A}_1 = \left(\begin{smallmatrix}
    2& 2& & &2 &  &1 &1 & & & & & \\
    &1 & &1 & & & &2 & &2 & &2 &\\
    & &2 &2 & &2 & & &1 &1 & & &\\
    1& & 1& & & & 2& & 2& & 2& &\\
    & & & & & &1 &1 & 1& 1& 2& 2&2\\
    1& 1& 1& 1& 2& 2& & & & & & &2\\
    & & & & 1& 1& & & & &1 &1 & 3
  \end{smallmatrix}\right) \in \Bbb{R}^{7 \times 13}_{\geq 0},
\quad \mathscr{A}_2 = \left(\begin{smallmatrix}
    1&&-1\\1&&-1\\&1&-1\\&1&-1\\-1&-1&\\-1&-1&\\
    -2&-2&1
  \end{smallmatrix}\right) \in \Bbb{R}^{7 \times 3}.$$
Proposition~\ref{circle-method} ensures the validity of Hypothesis~\ref{H1}
with $\lambda = 1/45000$.  We have $r=5$ coprimality conditions
\begin{displaymath}
  \begin{split}
    & S_1 = \{(1, 1), (2, 1), (3, 1)\}, \quad S_2 = \{(1, 1), (2, 1), (3,
    3)\}, \quad S_3 = \{(1, 2), (2, 2), (3, 2)\},\\
    & S_4 = \{(1, 2), (2, 2), (3, 3)\}, \quad S_5 = \{(3, 1), (3, 2)\}. 
  \end{split}
\end{displaymath}
We see that \eqref{fail} holds. We  choose
$${\bm \tau}^{(2)} = (\tfrac{1}{2}, \tfrac{1}{2}, \tfrac{1}{2},
\tfrac{1}{2}, \tfrac{1}{2}, \tfrac{1}{2}, 1)$$
satisfying \eqref{tau1} and confirm $C_2({\bm \tau}^{(2)} )$,
$C_2((1 - h_{ij}/3)_{ij})$. Finally we note that $c_2 = 2$ and
compute\footnote{Dimension $-1$ indicates that the set is empty.}
\begin{displaymath}
\begin{split}
\dim(\mathscr{H} \cap \mathscr{P}) &= 2, \\
\dim(\mathscr{H} \cap \mathscr{P}_{ij}) &=
\begin{cases}
  1, &(i,j)=(3,1),(3,2),(3,3),\\
  0, &\text{otherwise},
\end{cases}
\\
\dim(\mathscr{H} \cap\mathscr{P}(1/44800, \pi)) &= -1
\end{split}
\end{displaymath}
 for the vector $(1 - h_{ij}/3)_{ij}$, and
 \begin{displaymath}
\begin{split}
\dim(\mathscr{H} \cap \mathscr{P}) &= 0, \\
\dim(\mathscr{H} \cap \mathscr{P}_{ij}) &=
\begin{cases}
  0, &(i,j)=(3,1),(3,2),\\
  -1, &\text{otherwise},
\end{cases}
\\
\dim(\mathscr{H} \cap\mathscr{P}(1/44800, \pi)) &= -1
\end{split}
\end{displaymath}
for the vector ${\bm \tau}^{(2)}$. This confirms \eqref{ass2}.

\section{Higher-dimensional examples}\label{sec:geometry_X5_X6}

\subsection{Geometry}

Consider $G = \SL_2 \times \Gd_m^r$ and, for $i = 1, \dots, r$, let
$\varepsilon_i\in \Xf(B)$ be a primitive character of $\Gd_\mathrm{m}$
composed with the natural inclusion $\Xf(\Gd_\mathrm{m}) \to \Xf(B)$ into the $i$-th
factor $\Gd_m$ of $G$. Let $T_{\SL_2} \subset \SL_2$ be a maximal torus, and let
$\chi\colon T_{\SL_2} \to \Gd_m$ be a primitive character. We consider the
subgroup
\begin{align*}
  H = \{(\lambda, \chi(\lambda), 1, \dots, 1) : \lambda \in T_{\SL_2}\}
  \subset G\text{.}
\end{align*}
Then $G/H$ is a spherical homogeneous space of semisimple rank one and type
$T$. The lattice $\Mm$ has basis
$(\frac{1}{2}\alpha + \varepsilon_1, \frac{1}{2}\alpha - \varepsilon_1,
\varepsilon_2, \dots, \varepsilon_{r})$. We denote the corresponding dual
basis of the lattice $\Nm$ by $(d_1, d_2, e_3, \dots, e_{r+1})$. There 
are two colors $D_{11}$ and $D_{12}$ with valuations $d_1$ and $d_2$,
respectively. The valuation cone is given by
$\Vm = \{v \in \Nm_\Qd : \langle v, \alpha \rangle \le 0\}$.

 \subsubsection{The fourfold $X_5$}
 Let $r = 2$, and consider the polytope in $\Nm_\Qd$ spanned by the vectors
 \begin{align*}
   d_1 &= (1, 0, 0), &d_2 &= (0, 1,0 ), & u_{31}&= (0, -1, 0), &u_{32}&= (-1, 0, 0),\\
   u_{33}&= (-1, 0, -1), &u_{01}&= (1, -1, 1), &u_{02}&= (1, -1, 0),& u_{03}&= (-1, 1, 0).
 \end{align*}
The colored spanning fan of this polytope, as defined in \cite[Remark~2.6]{gh15},
contains the following maximal colored cones:
\begin{align*}
  &(\cone(d_{1}, d_{2}, u_{33}), \{D_{11}, D_{12}\}),
  &&(\cone(d_{1}, u_{02}, u_{33}), \{D_{11}\}),
  &&(\cone(d_{2}, u_{03}, u_{33}), \{D_{12}\}),\\
  &(\cone(u_{01}, u_{02}, u_{31}), \emptyset),
  &&(\cone(u_{01}, u_{03}, u_{32}), \emptyset),
  &&(\cone(u_{01}, u_{31}, u_{32}), \emptyset),\\
  &(\cone(u_{31}, u_{32}, u_{33}), \emptyset),
  &&(\cone(u_{03}, u_{32}, u_{33}),  \emptyset),
  &&(\cone(u_{02}, u_{31}, u_{33}), \emptyset).
\end{align*}
It can be verified that each colored cone satisfies the conditions of
the smoothness criterion \cite[Th\'eor\`eme~A]{cam01}; see also
\cite[Theorem~1.2]{gag15}. Let $X_5$ be the spherical embedding of
$G/H$ corresponding to this colored fan. Then $X_5$ is a smooth Fano
fourfold with Picard number $5$.

The unsupported colored spanning fan of the polytope above (\ie
including the unsupported colored cones) specifies a projective 
ambient toric variety $Y_5$. From the description of $\Sigmamax$ in
Section~\ref{sec:ambient}, we deduce that $Y_5$ is smooth and that
$-K_{X_5}$ is ample on $Y_5$;  hence \eqref{eq:toric_smooth} holds.

\subsubsection{The fivefold $X_6$}
Let $r = 3$, and consider the polytope in $\Nm_\Qd$ spanned by the vectors
\begin{align*}
  d_1 &= (1, 0, 0, 0), & d_2 &= (0, 1, 0, 0), & u_{31} = (-1,  0,  1,  0),
  & &   u_{32} = (-1, -1,  1,  0),\\
  u_{01} &= (-1,  1, -1, -1), & u_{02} &= ( 1, -1,  0,  1), & u_{03} = (0 ,  0,  -1,  0).
\end{align*}
The colored spanning fan of this polytope contains the following maximal
colored cones:
\begin{align*}
  &(\cone(d_1, d_2, u_{01}, u_{31}), \{D_{11}, D_{12}\}),
  &&(\cone(d_1, d_2, u_{02}, u_{31}), \{D_{11}, D_{12}\}),\\
  &(\cone(d_1, u_{01}, u_{31}, u_{32}), \{D_{11}\}),
  &&(\cone(d_1, u_{02}, u_{31}, u_{32}), \{D_{11}\}),\\
  &(\cone(d_1, u_{02}, u_{03}, u_{32}), \{D_{11}\}),
  &&(\cone(d_1, u_{01}, u_{03}, u_{32}), \{D_{11}\}),\\
  &(\cone(d_2, u_{01}, u_{03}, u_{31}), \{D_{12}\}),
  &&(\cone(d_2, u_{02}, u_{03}, u_{31}), \{D_{12}\}),\\
  &(\cone(u_{02}, u_{03}, u_{31}, u_{32}), \emptyset),
  &&(\cone(u_{01}, u_{03}, u_{31}, u_{32}), \emptyset).
\end{align*}
As in the previous example, we obtain a smooth spherical Fano fivefold $X_6$
with Picard number $3$ in a smooth projective ambient toric variety $Y_6$ on
which $-K_{X_6}$ is ample.

\subsubsection{The sixfold $X_7$}

Let $r = 4$, and consider the polytope in $\Nm_\Qd$ spanned by the vectors
\begin{align*}
d_1 &= (1, 0, 0, 0, 0), &
d_2 &= (0, 1, 0, 0, 0), &
  u_{01} &= (0, 0, 1, 0, 0), &
u_{02} &= (0, 0, 0, 1, 0),\\
u_{03} &= (0, 0, 0, 0, 1), &
u_{31} &= (0, -1, 0, 0, 0), &
u_{32} &= (-1, 0, 0, 0, 1), &
u_{33} &= (-1, 0, 0, 0, 0), \\
u_{34} &= (-1, 0, -1, -1, -1), &
  u_{35} &= (-1, -1, -1, -1, -1).
\end{align*}
As above, we obtain a smooth spherical Fano sixfold $X_7$ with Picard number
$5$ in a smooth projective ambient toric variety $Y_7$ on which $-K_{X_7}$ is
ample.

\subsubsection{The sevenfold $X_8$}

Let $r = 5$, and consider the polytope in $\Nm_\Qd$ spanned by the vectors
\begin{align*}
d_1 &= (1, 0, 0, 0, 0, 0), &
d_2 &= (0, 1, 0, 0, 0, 0), &
u_{01} &= (0, 0, 1, 0, 1, 0), \\
u_{02} &= (0, 0, 0, 1, 0, 1), &
u_{03} &= (0, 0, 0, 0, 0, 1), &
u_{04} &= (0, 0, 1, 0, 0, -1), \\
u_{05} &= (0, 0, 0, 1, 0, 0), &
u_{06} &= (0, 0, 0, 0, 1, 1), &
u_{31} &= (0, -1, 0, 0, 0, 0), \\
u_{32} &= (-1, 0, -1, -1, -1, -1), &
u_{33} &= (-1, -1, 0, 0, 0, 0), &
u_{34} &= (-1, -1, -1, -1, -1, -1).
\end{align*}
As above, we obtain a smooth spherical Fano sevenfold $X_8$ with Picard number $6$ in a smooth
projective ambient toric variety $Y_8$ on which $-K_{X_8}$ is ample.

\subsection{Cox rings and torsors}

We argue as in Section~\ref{sec:crt}.
  
\subsubsection{The fourfold $X_5$}

The Cox ring is
\begin{align*}
  \Rm(X_5) = \Qd[x_{01}, x_{02}, x_{03}, x_{11}, x_{12}, x_{21}, x_{22}, x_{31}, x_{32}, x_{33}]
  /(x_{11}x_{12}-x_{21}x_{22}-x_{31}x_{32}x_{33})
\end{align*}
with $\Pic X_5 \cong \Cl X_5 \cong \Z^5$, where
\begin{align*}
&\deg(x_{01})=\deg(x_{33})= (1, 0, 0, 0, 0),\
\deg(x_{02})= (0, 1, 0, 1, 0),\
\deg(x_{03})= (0, 1, 0, 0, 0),\\
&\deg(x_{11})=\deg(x_{21})= (0, 0, 1, 0, 0),\
\deg(x_{12})=\deg(x_{22})= (0, 0, 0, 0, 1),\\
&\deg(x_{31})= (-1, 0, 0, -1, 1),\
\deg(x_{32})= (0, 0, 1, 1, 0)
\end{align*}
The anticanonical class is $ -K_{X_5} = \rleft(1, 2, 2, 1, 2\rright). $
A universal torsor over $X_5$ is
\begin{equation*}
  \Tm_5 = \Spec\Rm(X_5) \setminus Z_{X_5}\text{,}
\end{equation*}
where
\begin{align*}
  Z_{X_5} &=\Vd(x_{31}, x_{11},x_{21}) \cup
  \Vd(x_{02}, x_{12},x_{22}) \cup
  \Vd(x_{12},x_{22}, x_{31}) \cup
  \Vd(x_{32}, x_{11},x_{21})\\
  &{} \cup \Vd(x_{31}, x_{03}) \cup
  \Vd(x_{02}, x_{32}) \cup
  \Vd(x_{02}, x_{03}) \cup
  \Vd(x_{33}, x_{01}) \cup
  \Vd(x_{12},x_{22}, x_{32}) \cup
  \Vd(x_{03}, x_{11},x_{21}).
\end{align*}

\subsubsection{The fivefold $X_6$}\label{sec:cox_X6}

The Cox ring is
\begin{align*}
  \Rm(X_6) =
  \Qd[x_{01},x_{02},x_{03},x_{11},x_{12},x_{21},x_{22},x_{31},x_{32}]
  /(x_{11}x_{12}-x_{21}x_{22}-x_{31}x_{32}^2)
\end{align*}
with $\Pic X_6 \cong \Cl X_6 \cong \Z^3$, where
\begin{align*}
&\deg(x_{01}) = 
\deg(x_{02}) = (0, 0, -1),\ 
\deg(x_{03}) = (1, 0, 1),\ 
\deg(x_{11}) = \deg(x_{21}) = (1, 0, 0),\\
&\deg(x_{12}) = \deg(x_{22}) = (0, 1, 0),\ 
\deg(x_{31}) = (1, -1, 0),\ 
\deg(x_{32}) = (0, 1, 0).
\end{align*}
The anticanonical class is $-K_{X_6} = \rleft(3,1,-1\rright)$.
A universal torsor over $X_6$ is
\begin{equation*}
  \Tm_6 = \Spec\Rm(X_6) \setminus Z_{X_6}\text{,}
\end{equation*}
where
\begin{align*}
  Z_{X_6} &= \Vd(x_{01},x_{02})\cup \Vd(x_{32}, x_{12}, x_{22})\cup \Vd(x_{03}, x_{31}, x_{11}, x_{21}).
\end{align*}

\subsubsection{The sixfold $X_7$}\label{sec:cox_X7}

The Cox ring is
\begin{align*}
  \Rm(X_7) = \Qd[x_{01}, x_{02}, x_{03}, x_{11}, x_{12}, x_{21}, x_{22}, x_{31}, \dots, x_{35}]
  /(x_{11}x_{12}-x_{21}x_{22}-x_{31}x_{32}x_{33}x_{34}x_{35}^2)
\end{align*}
with $\Pic X_7 \cong \Cl X_7 \cong \Z^5$, where
\begin{align*}
&\deg(x_{01}) = \deg(x_{02}) = (-1, -1, 0, 1, 0),\
\deg(x_{03}) = (-2, -1, 0, 1, 0),\\
&\deg(x_{11}) = \deg(x_{21}) = (0, 0, 0, 1, 0),\
\deg(x_{12}) = \deg(x_{22}) = (0, 0, 0, 0, 1),\\
&\deg(x_{31}) = (1, 1, 1, -1, 1),\
\deg(x_{32}) = (1, 0, 0, 0, 0),\
\deg(x_{33}) = (0, 1, 0, 0, 0),\\
&\deg(x_{34}) = (0, 0, 1, 0, 0),\
\deg(x_{35}) = (-1, -1, -1, 1, 0).
\end{align*}
The anticanonical class is $-K_{X_7} = \rleft(-3, -2, 1, 4, 2\rright)$.
A universal torsor over $X_7$ is
\begin{equation*}
  \Tm_7 = \Spec\Rm(X_7) \setminus Z_{X_7}\text{,}
\end{equation*}
where
\begin{align*}
  Z_{X_7} &= \Vd(x_{01}, x_{02}, x_{03}, x_{34}) \cup \Vd(x_{01}, x_{02},
            x_{03}, x_{35}) \cup \Vd(x_{01}, x_{02}, x_{32}, x_{34})\\
          &\cup \Vd(x_{01}, x_{02}, x_{32}, x_{35}) \cup \Vd(x_{03}, x_{33})
            \cup \Vd(x_{11}, x_{21}, x_{32})\\
          & \cup \Vd(x_{11}, x_{21}, x_{33}) \cup \Vd(x_{12}, x_{22},
            x_{31})\cup \Vd(x_{12}, x_{22}, x_{35}) \cup \Vd(x_{31}, x_{34}). 
\end{align*}

\subsubsection{The sevenfold $X_8$}\label{sec:cox_X8}

The Cox ring is
\begin{align*}
  \Rm(X_8) = \Qd[x_{01}, \dots, x_{06}, x_{11}, x_{12}, x_{21}, x_{22},
  x_{31}, \dots, x_{34}]
  /(x_{11}x_{12}-x_{21}x_{22}-x_{31}x_{32}x_{33}^2x_{34}^2)
\end{align*}
with $\Pic X_8 \cong \Cl X_8 \cong \Z^6$, where
\begin{align*}
  &\deg(x_{01}) = (1, 1, 0, -1, 0, 0),\
    \deg(x_{02}) = (1, 1, -1, 0, 0, 0),\\
  &\deg(x_{03}) = \deg(x_{05}) = (0, 0, 1, 0, 0, 0),\
    \deg(x_{04}) = \deg(x_{06}) = (0, 0, 0, 1, 0, 0),\\
  &\deg(x_{11}) = \deg(x_{21}) = (0, 0, 0, 0, 1, 0),\
    \deg(x_{12}) = \deg(x_{22}) = (0, 0, 0, 0, 0, 1),\\
  &\deg(x_{31}) = (0, 1, 0, 0, -1, 1),\
    \deg(x_{32}) = (0, 1, 0, 0, 0, 0),\\
  &\deg(x_{33}) = (-1, -1, 0, 0, 1, 0),
    \deg(x_{34}) = (1, 0, 0, 0, 0, 0).
\end{align*}
The anticanonical class is $-K_{X_8} = \rleft(2, 3, 1, 1, 1, 2\rright)$.  A
universal torsor over $X_8$ is
\begin{equation*}
  \Tm_8 = \Spec\Rm(X_8) \setminus Z_{X_8}\text{,}
\end{equation*}
where
\begin{align*}
  Z_{X_8} &= \Vd( x_{01}, x_{02}, x_{32} ) \cup
            \Vd( x_{01}, x_{02}, x_{34}) \cup 
            \Vd( x_{03}, x_{05} ) \cup 
            \Vd( x_{04}, x_{06} )\\ & \cup
            \Vd( x_{11}, x_{21}, x_{33} ) \cup 
            \Vd( x_{12}, x_{22}, x_{31} ) \cup 
            \Vd( x_{12}, x_{22}, x_{34} )\cup 
            \Vd( x_{31}, x_{32} ).
\end{align*}

\subsection{Counting problems}

\begin{cor}\label{dim4cor}
 {\rm (a)}  We have 
  \begin{equation*}
  N_5(B) =   \frac{1}{32}\#\left\{\xv \in \Zd_{\ne 0}^{10} :
     \begin{aligned}
       &x_{11}x_{12}-x_{21}x_{22}-x_{31}x_{32}x_{33}=0, \quad \max|\Pm_5(\xv)| \le B\\
&(x_{31}, x_{11},x_{21})=
(x_{02}, x_{12},x_{22})=
(x_{12},x_{22}, x_{31})=1\\
&(x_{32}, x_{11},x_{21})=
(x_{31}, x_{03})=
(x_{02}, x_{32})=1\\
&(x_{02}, x_{03})=
(x_{33}, x_{01})=
(x_{12},x_{22}, x_{32})=
(x_{03}, x_{11},x_{21})=1
    \end{aligned}
        \right\},
  \end{equation*}
  with 
  {\begin{equation*}
    \Pm_5(\xv) =\left\{
      \begin{aligned}
        &\{x_{01},x_{33}\}^2 x_{02}^2 \{x_{12},x_{22}\} x_{31} \{x_{11},x_{21}\}^2,
        x_{32} \{x_{01},x_{33}\}^3 x_{02}^2 x_{31}^2 \{x_{11},x_{21}\} ,\\
        &x_{03} \{x_{01},x_{33}\} x_{02} \{x_{12},x_{22}\}^2 \{x_{11},x_{21}\}^2 ,
        x_{03} x_{32}^2 \{x_{01},x_{33}\}^3 x_{02} x_{31}^2 ,\\
        &x_{03}^2 x_{32} \{x_{01},x_{33}\} \{x_{12},x_{22}\}^2 \{x_{11},x_{21}\} ,
        x_{03}^2 x_{32}^2 \{x_{01},x_{33}\}^2 \{x_{12},x_{22}\} x_{31}
      \end{aligned}
    \right\}.
  \end{equation*}}\\
 {\rm (b)}  We have 
  \begin{equation*}
  N_6(B) =   \frac{1}{8}\#\left\{\xv \in \Zd_{\ne 0}^9 :
     \begin{aligned}
      &x_{11}x_{12}-x_{21}x_{22}-x_{31}x_{32}^2=0, \quad \max|\Pm_6(\xv)| \le B\\
      &(x_{01},x_{02})=(x_{32}, x_{12}, x_{22}) = (x_{03}, x_{31}, x_{11}, x_{21}) =1
    \end{aligned}
        \right\},
  \end{equation*}
  with 
  {\begin{equation*}
    \Pm_6(\xv) =\left\{
      \begin{aligned}
        &\{x_{01},x_{02}\}\{x_{12},x_{22},x_{32}\}^4x_{31}^3,
        \{x_{01},x_{02}\}\{x_{11},x_{21}\}^3\{x_{12},x_{22},x_{32}\},\\
        &\{x_{01},x_{02}\}^4x_{03}^3\{x_{12},x_{22},x_{32}\}
      \end{aligned}
    \right\}.
\end{equation*}}
 {\rm (c)}  We have 
  \begin{equation*}
  N_7(B) =   \frac{1}{32}\#\left\{\xv \in \Zd_{\ne 0}^{12} :
     \begin{aligned}
      &x_{11}x_{12}-x_{21}x_{22}-x_{31}x_{32}x_{33}x_{34}x_{35}^2=0, \quad \max|\Pm_7(\xv)| \le B\\
      &(x_{01}, x_{02}, x_{03}, x_{34}) = (x_{01}, x_{02}, x_{03}, x_{35}) =
      (x_{01}, x_{02}, x_{32}, x_{34}) =1\\
      & (x_{01}, x_{02}, x_{32}, x_{35}) = (x_{03}, x_{33}) = (x_{11}, x_{21}, x_{32}) = 1\\
      &(x_{11}, x_{21}, x_{33}) = (x_{12}, x_{22}, x_{31}) = (x_{12}, x_{22},
      x_{35}) = (x_{31}, x_{34}) = 1
    \end{aligned}
  \right\},
  \end{equation*}
  with 
  \begin{equation*}
    \Pm_7(\xv) =\left\{
      \begin{aligned}
        &x_{31}^2 x_{32} x_{33}^2 x_{34}^5 x_{35}^6 , \{x_{12},x_{22}\}^2 x_{32} x_{33}^2 x_{34}^5 x_{35}^4 ,
        \{x_{11},x_{21}\} x_{31}^2 x_{33} x_{34}^4 x_{35}^5 , \\
        &\{x_{11},x_{21}\} \{x_{12},x_{22}\}^2 x_{33} x_{34}^4 x_{35}^3 ,
        x_{03} \{x_{11},x_{21}\}^2 x_{31}^2 x_{34}^2 x_{35}^3 ,\\
        &x_{03} \{x_{11},x_{21}\}^2 \{x_{12},x_{22}\}^2 x_{34}^2 x_{35} , 
        x_{03}^2 \{x_{11},x_{21}\}^2 \{x_{12},x_{22}\}^2 x_{32} x_{34} , \\
        &x_{03}^3 \{x_{11},x_{21}\}^2 x_{31}^2 x_{32}^2 x_{35} , 
        x_{03}^3 \{x_{11},x_{21}\}^2 \{x_{12},x_{22}\} x_{31} x_{32}^2 , x_{03}^4 \{x_{12},x_{22}\}^2 x_{32}^5 x_{33}^2 x_{34} , \\
        &x_{03}^5 x_{31}^2 x_{32}^6 x_{33}^2 x_{35} , x_{03}^5 \{x_{12},x_{22}\} x_{31} x_{32}^6 x_{33}^2 , 
        \{x_{01},x_{02}\} x_{03} \{x_{11},x_{21}\}^2 \{x_{12},x_{22}\}^2 x_{34} , \\
        &\{x_{01},x_{02}\}^2 x_{03} \{x_{11},x_{21}\}^2 x_{31}^2 x_{35} , 
        \{x_{01},x_{02}\}^2 x_{03} \{x_{11},x_{21}\}^2 \{x_{12},x_{22}\} x_{31} , \\
        &\{x_{01},x_{02}\}^3 \{x_{11},x_{21}\} \{x_{12},x_{22}\}^2 x_{33} x_{34} , 
        \{x_{01},x_{02}\}^4 \{x_{12},x_{22}\}^2 x_{32} x_{33}^2 x_{34} , \\
        &\{x_{01},x_{02}\}^4 \{x_{11},x_{21}\} x_{31}^2 x_{33} x_{35} , 
        \{x_{01},x_{02}\}^4 \{x_{11},x_{21}\} \{x_{12},x_{22}\} x_{31} x_{33} , \\
        &\{x_{01},x_{02}\}^5 x_{31}^2 x_{32} x_{33}^2 x_{35} , 
        \{x_{01},x_{02}\}^5 \{x_{12},x_{22}\} x_{31} x_{32} x_{33}^2
      \end{aligned}
    \right\}.    
  \end{equation*}
  {\rm (d)}  We have 
  \begin{equation*}
    N_8(B) =   \frac{1}{64}\#\left\{\xv \in \Zd_{\ne 0}^{14} :
      \begin{aligned}      
        &x_{11}x_{12}-x_{21}x_{22}-x_{31}x_{32}x_{33}^2x_{34}^2=0, \quad \max|\Pm_8(\xv)| \le B\\
        &( x_{01}, x_{02}, x_{32} ) = ( x_{01}, x_{02}, x_{34}) = ( x_{03},
        x_{05} ) = ( x_{04}, x_{06} ) = 1\\
        &( x_{11}, x_{21}, x_{33} ) = ( x_{12}, x_{22}, x_{31} ) = ( x_{12},
        x_{22}, x_{34} ) = ( x_{31}, x_{32} ) = 1
    \end{aligned}
        \right\},
  \end{equation*}
  where $\Pm_8(\xv)$ is
  {\begin{equation*}
     \left\{
      \begin{aligned}
        &\{x_{03},x_{05}\} \{x_{04},x_{06}\} x_{31}^2 x_{32}^4 x_{33}^3 x_{34}^5 ,
        \{x_{03},x_{05}\} \{x_{04},x_{06}\} \{x_{12},x_{22}\}^2 x_{32}^4 x_{33} x_{34}^3 ,\\
        &\{x_{03},x_{05}\} \{x_{04},x_{06}\} \{x_{11},x_{21}\} \{x_{12},x_{22}\}^2 x_{32}^3 x_{34}^2 , 
        \{x_{03},x_{05}\} \{x_{04},x_{06}\} \{x_{11},x_{21}\}^3 x_{31}^2 x_{32} x_{34}^2 ,\\
        &x_{02} \{x_{03},x_{05}\}^2 \{x_{04},x_{06}\} \{x_{11},x_{21}\}^3 x_{31}^2 x_{34} ,
        x_{02}^2 \{x_{03},x_{05}\}^3 \{x_{04},x_{06}\} \{x_{11},x_{21}\} \{x_{12},x_{22}\}^2 x_{32} ,\\
        &x_{02}^2 \{x_{03},x_{05}\}^3 \{x_{04},x_{06}\} \{x_{11},x_{21}\}^2 \{x_{12},x_{22}\} x_{31} ,
        x_{02}^3 \{x_{03},x_{05}\}^4 \{x_{04},x_{06}\} \{x_{12},x_{22}\}^2 x_{32} x_{33} ,\\
        &x_{02}^4 \{x_{03},x_{05}\}^5 \{x_{04},x_{06}\} x_{31}^2 x_{33}^3 x_{34} ,
        x_{02}^4 \{x_{03},x_{05}\}^5 \{x_{04},x_{06}\} \{x_{12},x_{22}\} x_{31} x_{33}^2 ,\\
        &x_{01} \{x_{03},x_{05}\} \{x_{04},x_{06}\}^2 \{x_{11},x_{21}\}^3 x_{31}^2 x_{34} ,
        x_{01}^2 \{x_{03},x_{05}\} \{x_{04},x_{06}\}^3 \{x_{11},x_{21}\} \{x_{12},x_{22}\}^2 x_{32} ,\\
        &x_{01}^2 \{x_{03},x_{05}\} \{x_{04},x_{06}\}^3 \{x_{11},x_{21}\}^2 \{x_{12},x_{22}\} x_{31} ,
        x_{01}^3 \{x_{03},x_{05}\} \{x_{04},x_{06}\}^4 \{x_{12},x_{22}\}^2 x_{32} x_{33} ,\\
        &x_{01}^4 \{x_{03},x_{05}\} \{x_{04},x_{06}\}^5 x_{31}^2 x_{33}^3 x_{34} ,
        x_{01}^4 \{x_{03},x_{05}\} \{x_{04},x_{06}\}^5 \{x_{12},x_{22}\} x_{31} x_{33}^2
      \end{aligned}
    \right\}.    
  \end{equation*}}
\end{cor}

\begin{proof}
  This is analogous to Corollary~\ref{prop:countingproblem_line_11}.
\end{proof}

\subsection{Application: Proof of Theorem~\ref{dim4}}\label{appl2}

All cases can be proved exactly as in Section \ref{appl1}. 

\subsubsection{The variety $X_5$}  
 
By Corollary~\ref{dim4cor}(a), we have $J=10$ torsor variables $x_{ij}$
satisfying the equation
$$x_{11}x_{12} + x_{21}x_{22}  + x_{31} x_{32} x_{33} = 0.$$
We have $N=34$ height conditions with corresponding exponent matrix
$$\mathscr{A}_1 = \left(
  \begin{smallmatrix}
    & & & & & & & & & & & & & & & & &1&1&1&1&1&1&1&1&2&2&2&2&2&2&3&3&3\\
    & & & & & &1&1&1&1&1&2&2&2&2&2&2& & & & &1&1&1&1& & &2&2&2&2&1&2&2\\
    2&2&2&2&2&2&1&1&1&1&1& & & & & & &2&2&2&2&1&1&1&1&2&2& & & & &1& & \\
    & & & &1&1& & & &2&2& & & &1&2&2& & &1&1& & &2&2& & & & &2&2& & &1\\
    & &1&2& &2& & &2& &2& & &1& & &1& &2& &2& &2& &2& &1& &1& &1& & & \\
    &1& &1& & & &2&2& & &1&2&2& & & &1&1& & &2&2& & & & &2&2& & & &1& \\
    1&2& & &2& & &2& &2& & &1& & &1& &2& &2& &2& &2& &1& &1& &1& & & & \\
    1& &1& & & &2& & & & &2&1&1&2&1&1& & & & & & & & &1&1&1&1&1&1&2&2&2\\
    2&1&2&1&1&1&2& & & & &1& & &1& & &1&1&1&1& & & & &2&2& & & & &2&1&1\\
    2&1&2&1&1&1&3&1&1&1&1&3&2&2&3&2&2& & & & & & & & & & & & & & & & &
  \end{smallmatrix}\right)
, \quad \mathscr{A}_2  = \left(
  \begin{smallmatrix}
    & & -1\\
    & & -1\\
    & & -1\\
    1 & & -1\\
    1 & & -1\\
    & 1 & -1\\
    & 1 &-1\\
    -1 & -1& \\
    -1& -1& \\
    -1& -1&
  \end{smallmatrix}
\right).$$
Proposition~\ref{circle-method} gives us $\lambda = 1/34300$.  We have $r=10$
coprimality conditions, and we see immediately in this and all other cases that
\eqref{fail} holds.  We choose
\begin{equation*}
  {\bm \tau}^{(2)}   = (1, 1, 1, \tfrac{2}{3}, \ldots,  \tfrac{2}{3}) = (1 - h_{ij}/3)_{ij}. 
\end{equation*}  
We verify $C_2( {\bm \tau}^{(2)})$ and $C_2((1 - h_{ij}/3)_{ij})$ and compute and  confirm  \eqref{ass2} by 
$$\dim (\mathscr{H} \cap \mathscr{P}) = 4, \quad    \dim (\mathscr{H} \cap \mathscr{P}_{ij}) = 3,\quad \dim(\mathscr{H} \cap \mathscr{P}(1/34300,\pi))=0.$$

\subsubsection{The variety $X_6$}

By  Corollary~\ref{dim4cor}(b), we have $J= 9$ torsor variables $x_{ij}$
satisfying the equation 
$$x_{11}x_{12} + x_{21}x_{22}  + x_{31} x_{32}^2 = 0.$$  
We have $N=24$ height conditions with corresponding exponent matrix
$$\mathscr{A}_1 = \left(\begin{smallmatrix}
 & & & & & & & & & & & &1&1&1&1&1&1&1&1&1&4&4&4\\
1&1&1&1&1&1&1&1&1&4&4&4& & & & & & & & & & & & \\
 & & & & & & & & &3&3&3& & & & & & & & & &3&3&3\\
 & & & & & &3&3&3& & & & & & & & & &3&3&3& & & \\
 & & & &1&4& & &1& & &1& & & & &1&4& & &1& & &1\\
 & &3&3&3& & & & & & & & & &3&3&3& & & & & & & \\
 &4& &1& & & &1& & &1& & &4& &1& & & &1& & &1& \\
3&3& & & &3& & & & & & &3&3& & & &3& & & & & & \\
4& &1& & & &1& & &1& & &4& &1& & & &1& & &1& & 
\end{smallmatrix}\right),
 \quad \mathscr{A}_2  = \left(\begin{smallmatrix}
  & & -1\\
  & & -1\\
  & & -1\\
  1 & & -1\\
  1 & & -1\\
  & 1 & -1\\
  & 1 & -1\\
  -1 & -1& \\
  -2 & -2& 1
\end{smallmatrix}\right).$$
Proposition~\ref{circle-method} yields $\lambda = 1/34300$.  We choose
$${\bm \tau}^{(2)} = (1, 1, 1, \tfrac{1}{2}, \tfrac{1}{2}, \tfrac{1}{2},
\tfrac{1}{2}, \tfrac{1}{2}, 1)$$
satisfying \eqref{tau1}. We verify $C_2( {\bm \tau}^{(2)})$ and
$C_2((1 - h_{ij}/3)_{ij})$ and compute
\begin{displaymath}
  \begin{split}
    &\dim(\mathscr{H}\cap \mathscr{P})=2, \\
    &\dim(\mathscr{H}\cap \mathscr{P}_{ij}) = -1, (i, j) = (1,1),(2,1), \quad
    \dim(\mathscr{H}\cap \mathscr{P}_{ij}) = 1\text{ otherwise},\\
    &\dim(\mathscr{H} \cap \mathscr{P}(1/34300, \pi)) = -1 \text{ for all } \pi
  \end{split}
\end{displaymath}
for the vector $ (1 - h_{ij}/3)_{ij}$ and
\begin{displaymath}
\begin{split}
  & \dim(\mathscr{H}\cap \mathscr{P})=1, \\
  &  \dim(\mathscr{H}\cap \mathscr{P}_{ij}) =
  \begin{cases}
    1,&(i, j) = (3,1),\\
    0, &(i,j)=(0,1),(0,2),(0,3),\\
    -1, &\text{ otherwise},
  \end{cases}
  \\
  & \dim(\mathscr{H} \cap \mathscr{P}(1/34300, \pi)) = -1 \text{ for all } \pi
\end{split}
\end{displaymath}
for the vector ${\bm \tau}^{(2)}$. This confirms \eqref{ass2}.

\subsubsection{The variety $X_7$}

By  Corollary~\ref{dim4cor}(c), we have $J=12$ torsor variables $x_{ij}$
satisfying the equation 
$$x_{11}x_{12} + x_{21}x_{22}  + x_{31}x_{32}x_{33}x_{34}x_{35}^2 = 0.$$  
We have $N=80$ height conditions; the corresponding matrix
$\mathscr{A}_1$ is {\tiny
$$\left(\begin{smallmatrix}
 & & & & & & & & & & & & & & & & & & & & & & & & & & & & & & & & & & & & & & & & & & & & & & & & & & & & & & &1&1&1&1&2&2&2&2&2&2&3&3&3&3&4&4&4&4&4&4&4&4&5&5&5\\
 & & & & & & & & & & & & & & & & & & & & & & & & & & & & & &1&1&1&1&2&2&2&2&2&2&3&3&3&3&4&4&4&4&4&4&4&4&5&5&5& & & & & & & & & & & & & & & & & & & & & & & & & \\
 & & & & & & & & &1&1&1&1&1&1&2&2&2&2&3&3&3&3&3&3&4&4&5&5&5&1&1&1&1&1&1&1&1&1&1& & & & & & & & & & & & & & & &1&1&1&1&1&1&1&1&1&1& & & & & & & & & & & & & & & \\
 & & & & & &1&1&1& & & &2&2&2& & &2&2& & & &2&2&2& & & & & & & &2&2& & & &2&2&2& & &1&1& & & & & &1&1&1& & & & & &2&2& & & &2&2&2& & &1&1& & & & & &1&1&1& & & \\
 & & & &2&2& & &2& & &2& & &2& &2& &2& & &1& & &1& &2& & &1& &2& &2& & &1& & &1& &2& &2& & & &1&2& & &1& & &1& &2& &2& & &1& & &1& &2& &2& & & &1&2& & &1& & &1\\
 & &1&1& &1& & & &2&2&2& & & &2&2& & &2&2&2& & & & & & & & &2&2& & &2&2&2& & & &1&1& & & &1&1&1& & & & & & & &2&2& & &2&2&2& & & &1&1& & & &1&1&1& & & & & & & \\
 &2& &2& & & &2& & &2& & &2& &2& &2& & &1& & &1& &2& & &1& &2& &2& & &1& & &1& &2& &2& &2& &1& & & &1& & &1& &2& &2& & &1& & &1& &2& &2& &2& &1& & & &1& & &1& \\
2& &2& & & &2& & &2& & &2& & & & & & &2&1&1&2&1&1& & &2&1&1& & & & &2&1&1&2&1&1& & & & & &2&1&1& &2&1&1&2&1&1& & & & &2&1&1&2&1&1& & & & & &2&1&1& &2&1&1&2&1&1\\
1&1& & &1& & & & & & & & & & &1&1&1&1&2&2&2&2&2&2&5&5&6&6&6& & & & & & & & & & & & & & &1& & & &1& & & &1&1&1& & & & & & & & & & & & & & &1& & & &1& & & &1&1&1\\
2&2&1&1&2&1&1&1&1& & & & & & & & & & & & & & & & &2&2&2&2&2& & & & & & & & & & &1&1&1&1&2&1&1&1&2&1&1&1&2&2&2& & & & & & & & & & &1&1&1&1&2&1&1&1&2&1&1&1&2&2&2\\
5&5&4&4&5&4&4&4&4&2&2&2&2&2&2&1&1&1&1& & & & & & &1&1& & & &1&1&1&1& & & & & & &1&1&1&1&1& & & &1& & & & & & &1&1&1&1& & & & & & &1&1&1&1&1& & & &1& & & & & & \\
6&4&5&3&4&3&5&3&3&3&1&1&3&1&1& & & & &1& & &1& & & & &1& & & & & & &1& & &1& & & & & & & &1& & & &1& & &1& & & & & & &1& & &1& & & & & & & &1& & & &1& & &1& & 
\end{smallmatrix}\right),
$$}
Proposition~\ref{circle-method} yields $\lambda = 1/70000$. 
We choose 
$${\bm \tau}^{(2)} = (1, 1, 1, \tfrac{1}{2}, \tfrac{1}{2}, \tfrac{1}{2}, \tfrac{1}{2}, \tfrac{3}{4}, \tfrac{3}{4}, \tfrac{3}{4}, \tfrac{3}{4}, 1)$$
satisfying \eqref{tau1}. We verify $C_2( {\bm \tau}^{(2)})$ and $C_2((1 - h_{ij}/3)_{ij})$ and compute 
 \begin{displaymath}
\begin{split}
  &\dim(\mathscr{H}\cap \mathscr{P})=4, \\
  &\dim(\mathscr{H}\cap \mathscr{P}_{ij})= \begin{cases}
  1,&(i, j) = (0,1), (0, 2),\\
  0, &(i,j)=(1, 1), (2, 1),\\
  2, & (i, j) = (1, 2), (2, 2),\\
  3, &\text{ otherwise},
\end{cases}\\
  &\dim(\mathscr{H} \cap \mathscr{P}(1/70000, \pi)) = -1 \text{ for all } \pi
\end{split}
\end{displaymath}
for the vector $ (1 - h_{ij}/3)_{ij}$ and
\begin{displaymath}
\begin{split}
& \dim(\mathscr{H}\cap \mathscr{P})=0, \\
&  \dim(\mathscr{H}\cap \mathscr{P}_{ij}) = \begin{cases}
  0,&(i, j) = (3,1), (3, 2), (3,3), (3,4),\\
  -1, &\text{ otherwise},
\end{cases}\\
& \dim(\mathscr{H} \cap \mathscr{P}(1/70000, \pi)) = -1 \text{ for all } \pi
\end{split}
\end{displaymath}
for the vector ${\bm \tau}^{(2)}$. This confirms \eqref{ass2}. 
 
\subsubsection{The variety $X_8$}

By Corollary~\ref{dim4cor}(d), we have $J=14$ torsor variables $x_{ij}$ with
$0 \leq i \leq 3$, $J_0 = 6$, $J_1 = J_2 = 2$, $J_3 = 4$ satisfying the
equation
$$x_{11}x_{12} + x_{21}x_{22}  + x_{31}x_{32}x_{33}^2x_{34}^2 = 0$$  
with $k=3$. We have $N=156$ height conditions; it is straightforward to
extract the corresponding matrices $\mathscr{A}_1$, $\mathscr{A}_2$ from
Corollary~\ref{dim4cor}(d), which we do not spell out for obvious space
reasons.  Proposition~\ref{circle-method} yields $\lambda = 1/70000$.  We
choose
$${\bm \tau}^{(2)} = (1, 1, 1, 1, 1, 1, \tfrac{1}{2}, \tfrac{1}{2}, \tfrac{1}{2}, \tfrac{1}{2}, \tfrac{1}{2}, \tfrac{1}{2}, 1, 1)$$
satisfying \eqref{tau1}. We verify $C_2( {\bm \tau}^{(2)})$ and
$C_2((1 - h_{ij}/3)_{ij})$ and compute
\begin{displaymath}
  \begin{split}
    &\dim(\mathscr{H}\cap \mathscr{P})=5, \\
    &\dim(\mathscr{H}\cap \mathscr{P}_{ij}) = \begin{cases}
      0,&(i, j) = (1, 1), (2, 1)\\
      2, &(i,j)=(1,2),(2, 2),\\
      4, &\text{ otherwise},
    \end{cases}\\
    &\dim(\mathscr{H} \cap \mathscr{P}(1/34300, \pi)) = -1 \text{ for all } \pi
  \end{split}
\end{displaymath}
for the vector $ (1 - h_{ij}/3)_{ij}$ and
\begin{displaymath}
  \begin{split}
    & \dim(\mathscr{H}\cap \mathscr{P})=3, \\
    &  \dim(\mathscr{H}\cap \mathscr{P}_{ij}) =\begin{cases}
      -1,&(i, j) = (1, 1), (1,2),(2, 1), (2, 2),\\
      0, &(i,j)=(3,4)\\
      3, &(i,j)=(3,1), (3,2),\\
      2, &\text{ otherwise},
    \end{cases}\\
    & \dim(\mathscr{H} \cap \mathscr{P}(1/70000, \pi)) = -1 \text{ for all } \pi
  \end{split}
\end{displaymath}
for the vector ${\bm \tau}^{(2)}$. This confirms \eqref{ass2}. 

\section{A singular example}\label{sec:Xdagger}

As in Section~\ref{IV.7}, we consider the spherical $G$-variety
$W_4 = \Vd(z_{11}z_{12}-z_{21}z_{22}-z_{31}z_{32}) \subset \Pd^2_\Qd \times
\Pd^2_\Qd$.  Let $\stX^\dag \to W_4$ be the blow-up in the two disjoint
$G$-invariant curves
\begin{align*}
  C_{01} &= \Vd(z_{12},z_{22},z_{31}) = 
           \Vd(z_{31})\times\{(0:0:1)\}\text{,}\quad C_{33} = \Vd(z_{31}, z_{32})\text{.}
\end{align*}

The anticanonical divisor $-K_{\stX^\dag}$ is not ample but
semiample. Moreover,
$H^1(\stX^\dag,\Om_{\stX^\dag}) = H^2(\stX^\dag,\Om_{\stX^\dag}) = 0$ since
$\stX^\dag$ is smooth and rational. Hence $\stX^\dag$ is an almost Fano
variety. We obtain an anticanonical contraction
$\pi\colon \stX^\dag \to X^\dag$. Here $X^\dag$ is a singular Fano variety
with desingularization $\stX^\dag$. The sequence of morphisms
$W_4 \leftarrow \stX^\dag \to X^\dag$ corresponds to the following sequence of
maps of colored fans.

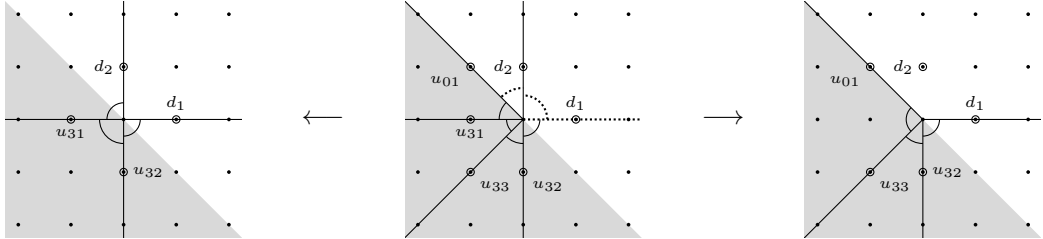
\begin{align*}
  \begin{tikzpicture}[scale=0.7]
    \clip (-2.24, -2.24) -- (2.24, -2.24) -- (2.24, 2.24) -- (-2.24, 2.24) -- cycle;
    \fill[color=gray!30] (-3, 3) -- (3, -3) -- (-3, -3) -- cycle;
    \foreach \x in {-3,...,3} \foreach \y in {-3,...,3} \fill (\x, \y) circle (1pt);
    \draw (1, 0) circle (2pt);
    \draw (0, 1) circle (2pt);
    \draw (-1, 0) circle (2pt);
    \draw (0, -1) circle (2pt);
    \draw (0,0) -- (3,0);
    \draw (0,0) -- (0,3);
    \draw (0,0) -- (-3, 0);
    \draw (0,0) -- (0, -3);
    \path (-1, 0) node[anchor=north] {{\tiny{$u_{31}$}}};
    \path (0, -1) node[anchor=west] {{\tiny{$u_{32}$}}};
    \path (0, 1) node[anchor=east] {{\tiny{$d_{2}$}}};
    \path (1, 0) node[anchor=south] {{\tiny{$d_{1}$}}};
    \begin{scope}
      \clip (0,0) -- (0,1) -- (-1,0) -- cycle; \draw (0,0) circle (9pt);
    \end{scope}
    \begin{scope}
      \clip (0,0) -- (-1, 0) -- (0, -1) -- cycle; \draw (0,0) circle (13pt);
    \end{scope}
    \begin{scope}
      \clip (0,0) -- (0,-1) -- (1,0) -- cycle; \draw (0,0) circle (9pt);
    \end{scope}
  \end{tikzpicture}
  &&
  \begin{tikzpicture}[scale=0.7]
    \clip (-0.5, -2.24) -- (0.5, -2.24) -- (0.5, 2.24) -- (-0.5, 2.24) -- cycle;
    \path (0,0) node {{$\longleftarrow$}};
  \end{tikzpicture}
     &&
    \begin{tikzpicture}[scale=0.7]
    \clip (-2.24, -2.24) -- (2.24, -2.24) -- (2.24, 2.24) -- (-2.24, 2.24) -- cycle;
    \fill[color=gray!30] (-3, 3) -- (3, -3) -- (-3, -3) -- cycle;
    \foreach \x in {-3,...,3} \foreach \y in {-3,...,3} \fill (\x, \y) circle (1pt);
    \draw (1, 0) circle (2pt);
    \draw (0, 1) circle (2pt);
    \draw (-1, 0) circle (2pt);
    \draw (-1, -1) circle (2pt);
    \draw (-1, 1) circle (2pt);
    \draw (0, -1) circle (2pt);
    \draw (0,0) -- (0,3);
    \draw (0,0) -- (-3, 0);
    \draw (0,0) -- (-3, -3);
    \draw (0,0) -- (-3, 3);
    \draw (0,0) -- (0, -3);
    \path (-1, 0) node[anchor=north] {{\tiny{$u_{31}$}}};
    \path (-1, -1) node[anchor=north west] {{\tiny{$u_{33}$}}};
    \path (0, -1) node[anchor=north west] {{\tiny{$u_{32}$}}};
    \path (-1, 1) node[anchor=north east] {{\tiny{$u_{01}$}}};
    \path (0, 1) node[anchor=east] {{\tiny{$d_{2}$}}};
    \path (1, 0) node[anchor=south] {{\tiny{$d_{1}$}}};
    \begin{scope}
      \clip (0,0) -- (1,0) -- (0,-1) -- cycle; \draw (0,0) circle (9pt);
    \end{scope}
    \begin{scope}
      \clip (0,0) -- (0, -1) -- (-1, -1) -- cycle; \draw (0,0) circle (13pt);
    \end{scope}
    \begin{scope}
      \clip (0,0) -- (-1,-1) -- (-1,0) -- cycle; \draw (0,0) circle (9pt);
    \end{scope}
    \begin{scope}
      \clip (0,0) -- (-1, 0) -- (-1, 1) -- cycle; \draw (0,0) circle (13pt);
    \end{scope}
    \draw[densely dotted, thick] (0,0) -- (4,0);
    \begin{scope}
      \clip (0,1) -- (1,0) -- (0,0) -- cycle; \draw[densely dotted,thick] (0,0) circle (13pt);
    \end{scope}
    \begin{scope}
      \clip (-1,1) -- (0,1) -- (0,0) -- cycle; \draw[densely dotted,thick] (0,0) circle (17pt);
    \end{scope}
  \end{tikzpicture}
  &&
  \begin{tikzpicture}[scale=0.7]
    \clip (-0.5, -2.24) -- (0.5, -2.24) -- (0.5, 2.24) -- (-0.5, 2.24) -- cycle;
    \path (0,0) node {{$\longrightarrow$}};
  \end{tikzpicture}
  &&
  \begin{tikzpicture}[scale=0.7]
    \clip (-2.24, -2.24) -- (2.24, -2.24) -- (2.24, 2.24) -- (-2.24, 2.24) -- cycle;
    \fill[color=gray!30] (-3, 3) -- (3, -3) -- (-3, -3) -- cycle;
    \foreach \x in {-3,...,3} \foreach \y in {-3,...,3} \fill (\x, \y) circle (1pt);
    \draw (1, 0) circle (2pt);
    \draw (0, 1) circle (2pt);
    \draw (0, -1) circle (2pt);
    \draw (-1, -1) circle (2pt);
    \draw (-1, 1) circle (2pt);
    \draw (0,0) -- (3,0);
    \draw (0,0) -- (0,-3);
    \draw (0,0) -- (-3, -3);
    \draw (0,0) -- (-3, 3);
    \path (0,-1) node[anchor=west] {{\tiny{$u_{32}$}}};
    \path (-1, -1) node[anchor=north west] {{\tiny{$u_{33}$}}};
    \path (-1, 1) node[anchor=north east] {{\tiny{$u_{01}$}}};
    \path (0, 1) node[anchor=east] {{\tiny{$d_{2}$}}};
    \path (1, 0) node[anchor=south] {{\tiny{$d_{1}$}}};
    \begin{scope}
      \clip (0,0) -- (1,0) -- (0,-1) -- cycle; \draw (0,0) circle (9pt);
    \end{scope}
    \begin{scope}
      \clip (0,0) -- (0,-1) -- (-1, -1) -- cycle; \draw (0,0) circle (13pt);
    \end{scope}
    \begin{scope}
      \clip (0,0) -- (-1,-1) -- (-1,1) -- cycle; \draw (0,0) circle (9pt);
    \end{scope}
  \end{tikzpicture}
\end{align*}

We denote by $E_{31}$ the $G$-invariant exceptional divisor contracted
by $\pi$. The singular locus of $X^\dag$ is $\pi(E_{31})$.  The dotted
circles in the colored fan of $\stX^\dag$ specify a smooth projective
ambient toric variety  $Y^\dag$ such that $-K_{\tX^\dag}$ is ample on
$Y^\dag$.

In the same way as before, a universal torsor of $\stX^{\dag}$ can be
obtained. The straightforward computations are omitted. This leads to the
following counting problem.

\begin{cor}\label{prop:countingproblem_line_8}
  We have
  \begin{align*}
    N^{\dag}(B) =   \frac{1}{16} \#\left\{\xv \in \Zd_{\ne 0}^8 :
    \begin{aligned}
      &x_{11}x_{12}-x_{21}x_{22}-x_{31}x_{32}x_{33}^2=0, \quad \max|\Pm^\dag(\xv)| \le B\\
      &(x_{11},x_{21},x_{33})=(x_{11},x_{21},x_{31})=(x_{01},x_{11},x_{21})=1\\
      &(x_{12},x_{22})=(x_{01}, x_{32})=(x_{01}, x_{33})=(x_{31}, x_{32})=1\\
    \end{aligned}
    \right\},
  \end{align*}
  with 
  \begin{equation*}
    \Pm^\dag(\xv) =\left\{
      \begin{aligned}
        &x_{01}x_{31}^2x_{32}^2x_{33}^3,
        \{x_{11},x_{21}\}x_{31}x_{32}^2x_{33}^2,
        \{x_{11},x_{21}\}^2\{x_{12},x_{22}\}x_{32},\\
        &x_{01}^3 \{x_{12},x_{22}\}^2 x_{31}^2x_{33} ,
        x_{01}^2 \{x_{11},x_{21}\} \{x_{12},x_{22}\}^2 x_{31}
      \end{aligned}
    \right\}.
  \end{equation*}
\end{cor}

By the same type of computations as before, one concludes Theorem~\ref{thm2}
from Corollary~\ref{prop:countingproblem_line_8} and Theorem~\ref{manin-cor}
applied to the almost Fano variety $\stX^\dag$.

\appendix

\section{Some explicit computations}\label{A}

We return to the variety $X_4$ discussed in Section~\ref{X4} and explain how
to obtain Hypothesis~\ref{H2} by ``bare hands'' and how to compute Peyre's
constant explicitly. We use $X_4$ as a showcase, the computations are similar
(and similarly uninspiring) in the other cases.\\
    
Recall from \eqref{X1} and \eqref{tauzeta} that for Hypothesis~\ref{H2}, we
need to show
\begin{equation}\label{analysis}
  \left.\sum_{\textbf{X}}\right.^{\ast} (X_{01}X_{02}
  (X_{11}X_{12}X_{21}X_{22}X_{31}X_{32})^{2/3})^{\alpha} \ll  B^{\alpha} (\log
  B)^2(1 + \log H)
\end{equation}
for fixed $0 < \alpha < 1$, where each $X_{ij}$ is restricted to a power of
$2$ and subject to
$$\min(X_{ij}) \leq H \quad \text{and} \quad \prod_{ij}
X_{ij}^{\alpha^\nu_{ij}} \leq B.$$ By symmetry, we can assume without loss of
generality that
\begin{equation*}
  X_{12} \geq X_{22},\quad  X_{21} \geq X_{11}.
\end{equation*}
The columns $\nu = 4, 5 $ and $\nu = 2, 3 $ of in the matrix $\mathscr{A}_1$
yield
\begin{equation}\label{restr2raw}
  X_{31}  X_{12}  \max( X_{31}X_{32},    X_{12}  X_{21}) X_{02}^2 \leq B
  ,\quad   X_{32}  X_{21} \max(  X_{31}X_{32},  X_{12}  X_{21} ) X_{01}^2 \leq
  B,
\end{equation}
respectively.  Let us first assume that
$\min(X_{ij}) \asymp \min(X_{11}, X_{22}, X_{31}, X_{32})$, \ie
$X_{01}, X_{02}$ are not the smallest parameters.  Summing over
$X_{01}, X_{02}$, we bound the $\mathbf{X}$-sum in \eqref{analysis} by
\begin{equation*}
  \sum_{\mathbf{X}} \Bigl(
  \frac{B(X_{11}X_{12}X_{21}X_{22}X_{31}X_{32})^{2/3}}{(X_{12}X_{21}X_{31}X_{32})^{1/2}\max(X_{31}X_{32},
    X_{12}X_{21})}\Bigr)^{\alpha} \leq  \sum_{\mathbf{X}}   \Bigl(
  \frac{B(X_{31}X_{32})^{1/6} (X_{21}X_{22})^{2/3}}{\max(X_{31}X_{32},
    X_{12}X_{21})^{5/6}} \Bigr)^{\alpha}   . 
\end{equation*}
Here and in similar situations, the precise summation conditions on
$\mathbf{X}$ and the variables involved will always be clear from the context.
Suppose that the minimum is taken at $X_{11}$ or $X_{22}$.  We glue together
the variables $X_{31}X_{32} = X_{3}$, say, where $X_3$ runs over powers of $2$
with multiplicity $O(\log B)$. Summing over $X_3$, the $\mathbf{X}$-sum
becomes
$$\log B \sum_{\substack{ X_{22} \leq X_{12} \leq B\\ X_{11} \leq X_{21} \leq
    B\\ \min(X_{11}, X_{22}) \leq H}}  \left(\frac{B(X_{22}
    X_{11})^{2/3}}{(X_{12}X_{21})^{2/3}}\right)^{\alpha} \ll B^{\alpha} (\log
B)^2 (1+\log H).$$
If the minimum is taken at $X_{31}$ or $X_{32}$, there are only $O(1+\log H)$
possibilities for the value of $X_3$, and we can argue in the same way.
   
Finally we treat the case where the minimum is taken at $X_{01}$ or $X_{02}$.
Without loss of generality (by symmetry), assume $X_{01} \leq X_{02}$.  We use
\eqref{restr2raw} to sum over $X_{02}$ and then sum over $X_{11} \leq X_{21}$
and $X_{22} \leq X_{12}$. In this way, we bound the $\mathbf{X}$-sum in
\eqref{analysis} by
\begin{displaymath}
  \sum_{\mathbf{X}}\Bigl(  \frac{B^{1/2}X_{01}
    (X_{11}X_{12}X_{21}X_{22}X_{31}X_{32})^{2/3}}{(X_{31}X_{12}  \max(
    X_{31}X_{32},   X_{12} X_{21}  ))^{1/2}}\Bigr)^{\alpha} \ll
  \sum_{\mathbf{X}}\Bigl(  \frac{B^{1/2}X_{01}  (X_{12}^2
    X_{21}^2X_{31}X_{32})^{2/3}}{(X_{31}X_{12}  \max(  X_{31}X_{32},   X_{12}
    X_{21}  ))^{1/2}}\Bigr)^{\alpha},
\end{displaymath}
where the sum is restricted to $X_{01}, X_{12}, X_{21}, X_{31}, X_{32}$ powers
of 2 satisfying $X_{01} \leq H$ and the second bound in \eqref{restr2raw}. We
now distinguish two cases. If $X_{31}X_{32} \geq X_{12}X_{21}$, we sum over
$X_{12} \leq X_{31}X_{32}/X_{21}$, getting
\begin{displaymath}
  \sum_{\substack{X_{01} \leq H\\ X_{32}^2X_{21}X_{31} X_{01}^2 \leq B}}
  \Bigl(B^{1/2} X_{01}X_{32} (X_{31}X_{21})^{1/2}\Bigr)^{\alpha} \ll
  \sum_{\substack{X_{01} \leq H, X_{21}, X_{31} \leq    B}}  B^{\alpha} \ll
  B^{\alpha} (\log B)^2 (1+\log H). 
\end{displaymath}
If $X_{31}X_{32} \leq X_{12}X_{21}$, we sum over $X_{31} \leq
X_{12}X_{21}/X_{32}$ instead, obtaining the same result. \\
  
Now we compute the Peyre constant. We start with the computation of the Euler
product $c_{\text{fin}}$. By \eqref{gcd1}, \eqref{gamma} and \eqref{gammaast},
we have
\begin{displaymath}
  {\bm \gamma} = ([g_4, g_5], [g_3, g_4], g_1, g_2,   g_1, g_2, g_5, g_3])\in
  \Bbb{N}^8, \quad {\bm \gamma}^{\ast} = ( g_1 g_2, g_1 g_2,  g_3g_5)\in
  \Bbb{N}^3.
\end{displaymath}
A simple computation (cf.\ Lemma~\ref{Gauss}) shows
\begin{equation*}
  \mathscr{E}_{\mathbf{b}} = \sum_{q=1}^{\infty} q^{-6} \underset{a
    \bmod{q}}{\left.\sum\right.^{\ast}} \prod_{i=1}^3 \Bigl(\sum_{x, y
    \bmod{q}} e\Bigl(\frac{a}{q}b_i xy\Bigr)\Bigr) = \sum_{q=1}^{\infty}
  \frac{\phi(q) (q, b_1)(q, b_2)(q, b_3)}{q^3}
\end{equation*}
for $\mathbf{b} \in \Bbb{N}^3$, so that
$$c_{\text{fin}} = \sum_{\mathbf{g}\in \Bbb{N}^5} \frac{\mu(\mathbf{g})}{g_1^2
  g_2^2g_3g_5[g_4, g_5][g_3, g_4]} \sum_{q=1}^{\infty} \frac{\phi(q) (q, g_1
  g_2)^2 (q, g_3g_5)}{q^3}.$$ We expand this into an Euler product, and by
brute force computation one verifies
\begin{equation*}
  c_{\text{fin}} =\prod_p \left( 1 - \frac{1}{p}\right)^4
  \left( 1 + \frac{1}{p}\right) \left( 1 + \frac{3}{p} + \frac{1}{p^2}\right).
\end{equation*}
In order to compute $c^{\ast}$ and $c_{\infty}$, we follow the argument in
Section~\ref{peyre}. We can take the rows $3, 4, 5, 6$ (\ie corresponding to
$(ij) = (11), (12), (21), (22)$) of $(\mathscr{A}_1\, \mathscr{A}_2)$ as
$Z_1, \ldots, Z_4$ in \eqref{beta}, so that
\begin{displaymath}
\begin{split}
& y_1 = w_{11} =  s_3 + 2s_7 + 2s_9 + s_{11} + s_{13} + 2s_{16} + 2s_{17} + z_1-1,\\
& y_2 = w_{12} = s_4+s_6+s_7+2s_{10} + 2s_{11}+ 2s_{14} + 2s_{16} + z_1-1,\\ 
& y_3 = w_{21} = s_2 + 2s_6 + 2s_8+s_{10}+s_{12}+2s_{14} + 2s_{15}+ z_2-1,\\
& y_4 = w_{22}  = s_5 + s_8 + s_9 + 2s_{12} + 2s_{13} + 2s_{15} + 2s_{17} + z_2-1,\\
& y_5 = s_1 + \dots + s_{17}-1.
\end{split}
\end{displaymath}
An explicit choice for a vector ${\bm \sigma}$ satisfying \eqref{1a} is for instance
$${\bm \sigma} = (\tfrac{1}{18},\tfrac{1}{18}, \tfrac{1}{18}, \tfrac{1}{18}, \tfrac{1}{18}, \tfrac{1}{18}, \tfrac{1}{18}, \tfrac{1}{18}, \tfrac{1}{18}, \tfrac{1}{18}, \tfrac{1}{18}, \tfrac{1}{18}, \tfrac{1}{18}, \tfrac{1}{18}, \tfrac{1}{12}, \tfrac{1}{12}, \tfrac{1}{18} )\in \Bbb{R}_{> 0}^{17}.$$
The linear forms $\mathscr{L}_{\iota}(\textbf{y})$ in \eqref{linear1}
containing the entries of the matrix $\mathscr{B} \in \Bbb{R}^{4 \times 5}$
are given by
\begin{displaymath}
\begin{split}
&w_{31} = y_5 + y_3 - y_2+ y_1 - y_4,\quad  w_{32} =y_5 - y_3 + y_2- y_1 + y_4 ,\\
& w_{01} = 2y_5 - y_2 - y_4, \quad w_{02} = 2y_5 - y_3 - y_1. 
\end{split}
\end{displaymath}
By contour shifts as in Section~\ref{peyre} or by the explicit formula
\eqref{cast-final}, we compute
\begin{equation*}
  c^{\ast} = \frac{1}{3!} \cdot \frac{1}{12}.
\end{equation*}
 
To compute $c_{\infty}$, we need to choose a matrix $\mathscr{C}$ as in
\eqref{mathcalC}, \ie variables $y_6, \ldots, y_{17}$ as functions of
$\mathbf{s}$. A simple possible choice is $y_{\nu} = s_{\nu}$,
$6 \leq \nu \leq 17$ (Jacobi-Determinant $-1$).  In these variables, we have
\begin{equation*}
\begin{split}
  \Bigl( \prod_{\nu=1}^{17}& s_{\nu}\Bigr)|_{y_1 = \dots = y_5 = 0} =
  \Bigl(\prod_{\nu=6}^{17} y_{\nu} \Bigr) (  2(y_6 + \dots + y_{13}) +
  3(y_{14} +y_{15} +  y_{16} + y_{17}) - 3+ 2z_1 + 2z_2)\\
  &\times   (2 y_6 + 2 y_8 + y_{10} + y_{12} + 2 y_{14 }+ 2 y_{15} + z_2 - 1)
  ( 2 y_7 + 2 y_9 + y_{11} + y_{13} + 2 y_{16} + 2 y_{17} + z_1 - 1)\\
  & \times ( y_6 + y_7 + 2 y_{10} + 2 y_{11} + 2 y_{14} + 2 y_{16} + z_1 - 1)
  (y_8 + y_9 + 2 y_{12} + 2 y_{13} + 2 y_{15} + 2 y_{17} + z_2 - 1).
\end{split}
\end{equation*}
For fixed $z_1, z_2$, the integrand is a rational function in
$y_6, \ldots, y_{17}$, and we simply shift each contour to $+\infty$ or
$-\infty$ (again it does not matter which direction we choose) and pick up the
poles. After a long computation (or a quick application of a computer algebra
system), we obtain
$$c_{\infty}  =\frac{2^8}{\pi} \int_{(1/3)}^{(2)} \mathscr{K}(z_1)
\mathscr{K}(z_2) \mathscr{K}(z_3)
\frac{2(3-z_3^2)}{(z_1-1)^2(z_2-1)^2(z_3-1)^2} \frac{\dd z_1\, \dd z_2}{(2\pi
  {\rm i})^2},$$
with $\mathscr{K}(z) = \Gamma(z) \cos(\pi z/2)$, $z_3 = 1 - z_1 - z_2$. Let us
define
\begin{displaymath}
  {\tt K}(z) = \frac{\Gamma(z) \cos(\pi z/2)}{(z-1)^2},\quad  {\tt
      K}^{\ast}(z) = \frac{2\Gamma(z) \cos(\pi z/2)(3-z^2)}{(z-1)^2}
\end{displaymath}
and let us denote by
$$\check{{\tt K}}(x) = \int_{(1/3)} {\tt K}(z) x^{-z} \frac{\dd z}{2\pi {\rm
    i}}, \quad x > 0,$$ and similarly by $\check{{\tt K}}^{\ast}$ the
corresponding inverse Mellin transforms. By \cite[6.246]{GR}, we have
$\check{{\tt K}}(x) = {\rm Si}(x)/{x}$ where
$ {\rm Si}(x) = \int_0^x \sin t \, dt/t$ is the integral sine.  To deal with
convergence issues, let
$$\mathscr{C} = (-10 - i\infty, -10 - i] \cup [-10 - i, 1/3] \cup [1/3, -10 +
i]\cup [-10 + i, -10 + i\infty).$$ Then
\begin{equation*}
  \begin{split}
    \frac{\pi}{2^8}c_{\infty} &=   \int^{(2)}_{(1/3)}{\tt K}(z_1){\tt K}(z_2)
    {\tt K}^{\ast}(1-z_1-z_2)  \frac{\dd z_1\, \dd z_2}{(2 \pi i)^2}  =
    \int^{(2)}_{(1/3)}{\tt K}(z_1){\tt K}(1-z_1-z_2) {\tt K}^{\ast}(z_2)
    \frac{\dd z_1\, \dd z_2}{(2 \pi i)^2} \\
    & = \int_0^{\infty} \check{{\tt K}}(x) \int_{(1/3)} {\tt K}(z_1) x^{-z_1}
    \frac{\dd z_1}{2\pi {\rm i}} \int_{\mathscr{C}} {\tt K}^{\ast}(z_2)
    x^{-z_2} \frac{\dd z_2}{2\pi {\rm i}} \dd x = \int_0^{\infty} \check{{\tt
        K}}(x)^2 \int_{\mathscr{C}} {\tt K}^{\ast}(z_2) x^{-z_2} \frac{\dd
      z_2}{2\pi {\rm i}} \dd x.
  \end{split}
\end{equation*}
The $z_2$-integral is also an inverse Mellin transform, but in order to avoid
convergence issues, we compute it directly by shifting the contour to the far
left and collect the poles. Comparing power series (cf.\ \cite[8.232,
8.253]{GR}), we obtain
$$\int_{\mathscr{C}}{\tt K}^{\ast}(z) x^{-z} \frac{\dd z}{(2\pi {\rm i})} =
\frac{4{\rm Si}\,x + 4 \sin x - 2x\cos x}{x}. $$ 
For this and related expressions appearing in the computation of the Peyre
constant of the varieties $X_1, \ldots, X_4$, the following lemma can be
used. Let
\begin{equation*}
  \textbf{F}(x) = \int_0^x \cos\Big(\frac{\pi t^2}{2}\Big) {\rm d}t. 
\end{equation*}

\begin{lemma}\label{sin} 
  We have
  \begin{displaymath}
    \begin{split}
      & \int_0^{\infty} \left(\frac{{\rm Si}\,x}{x}\right)^3 \dd x =
      \frac{33}{32} \pi - \frac{1}{32} \pi^3, \quad\quad \int_0^{\infty}
      \left(\frac{{\rm Si}\,x}{x}\right)^2 \frac{\sin x}{x} \dd x  =
      \frac{1}{4}\pi + \frac{\pi}{48}(21 - \pi^2),\\
      & \int_0^{\infty} \left(\frac{{\rm Si}\,x}{x}\right)^2 \cos(x) \dd x =
      \frac{\pi (12 - \pi^2)}{24}.
    \end{split}
  \end{displaymath}
  Moreover, 
  \begin{displaymath}
    \begin{split}
      &\int_{0}^{\infty} \frac{({\rm Si}\, x)^2
      }{x^2}\Bigl(\frac{\pi}{2x}\Bigr)^{1/2}
      \mathbf{F}\left(\Bigl(\frac{2x}{\pi}\Bigr)^{1/2}\right) \dd x=
      -\frac{\pi^3}{72} + \pi\left( \frac{59}{54} - \frac{4}{9}\log
        2\right),\\
      &\int_{0}^{\infty} \frac{({\rm Si}\,x)\sin x
      }{x^2}\Bigl(\frac{\pi}{2x}\Bigr)^{1/2}
      \mathbf{F}\left(\Bigl(\frac{2x}{\pi}\Bigr)^{1/2}\right) \dd x=
      \frac{\pi}{36}(25 - 12\log 2).
    \end{split}
  \end{displaymath}
\end{lemma}   

\begin{proof}
  The first integral is computed in \cite[Theorem 3]{BB0}. To compute the
  second, we observe that
  $$\int_0^{\infty} \left(\frac{{\rm Si}(x)}{x}\right)^2 \frac{ \sin(x)  }{x}
  \dd x = \int_0^1 \int_0^1 \int_0^{\infty} \frac{\sin(x) \sin(tx)
    \sin(sx)}{x^3} \dd x \frac{\dd t\, \dd s}{ts}.$$ By the residue theorem, it
  is readily seen that the inner integral equals
  \begin{displaymath}
    \begin{split}
      &\frac{\pi}{16} ((s+t+1)^2 - (s+t-1)^2\text{sgn}(s+t-1) - (s-t+1)^2
      \text{sgn}(s-t+1) - (t-s+1)^2\text{sgn}(t-s+1))\\
      & = \frac{\pi}{16}
      \begin{cases}
        -2 + 4 s - 2 s^2 + 4 t + 4 s t - 2 t^2 ,
        &  s + t \geq 1\\
        8st, & s + t \leq 1
      \end{cases}
    \end{split}
  \end{displaymath}
  for $0 \leq s, t \leq 1$, and a straightforward computation gives the
  desired result. Similarly, one computes the other integrals.
\end{proof}

The previous lemma  confirms the evaluation 
\begin{equation*}
  c_{\infty} = 32(  47 - \pi^2).
\end{equation*}

\section{Final remarks}\label{rem:not}

Here we show that $X_3, \dots, X_8, X^\dagger,\stX^\dagger$ do not belong to
any of the families of varieties described in the introduction for which
Manin's conjecture is already known. Whether or not $X_1$, $X_2$ are
biequivariant compactifications of a unipotent group, is not obvious to us,
but it is not hard to see that they are certainly neither horospherical nor
equivariant compactifications of $\Gd_\mathrm{a}^d$ nor wonderful
compactification of a semisimple group of adjoint type.

\begin{prop}\label{prop:nonsimp1}
  None of the varieties $X_3,\dots, X_8, X^\dagger,\stX^\dagger$ is isomorphic
  to a biequivariant compactification of a unipotent group.
\end{prop}

\begin{proof}
  By \cite[Proposition~1.1]{MR1906155}, the effective cone of every
  equivariant compactification of $\Gd_\mathrm{a}^3$ is simplicial.
  More generally, by \cite[Proposition~7.2]{MR3465086}, the same is true for
  biequivariant compactifications of unipotent groups.
  However, the effective cones of
  $X_3, \dots, X_8, X^\dagger,\stX^\dagger$ are not simplicial.
\end{proof}

\begin{prop}
  Neither $X_1$ nor $X_2$ is isomorphic to an equivariant
  compactification of $\Gd_\mathrm{a}^3$.
\end{prop}

\begin{proof}
  By \cite{arXiv:1802.08090}, only the first two entries of
  Table~\ref{tab:classification_spherical} are equivariant
  compactifications of $\Gd_\mathrm{a}^3$.
\end{proof}

\begin{prop}
  None of the varieties $X_1,\dots, X_8, X^\dagger,\stX^\dagger$ is
  isomorphic to a wonderful compactification of a semisimple group of
  adjoint type or to a wonderful variety covered by
  \cite[Corollary~1.5]{MR2795511}.
\end{prop}

\begin{proof}
  Over $\Qbar$, the only wonderful variety of dimension $3$ and Picard
  rank $3$ is $\Pd^1 \times \Pd^1 \times \Pd^1$; see, for instance,
  \cite{bl11}. Hence $X_1$ and $X_2$ are not wonderful {\color{blue}
    varieties}.

  Moreover, by \cite[Example~2.3.5]{bri07}, the effective cone of a
  wonderful compactification of a semisimple group of adjoint type is
  simplicial. Similarly, by \cite[Section~3.3]{MR2795511}, the
  effective cone of a wonderful variety covered by
  \cite[Corollary~1.5]{MR2795511} is simplicial. Hence the result for
  $X_3, \dots, X_8, X^\dagger,\stX^\dagger$ follows as in
  Proposition~\ref{prop:nonsimp1}.
\end{proof}

\begin{prop}
  None of the varieties $X_1,\dots, X_8, X^\dagger,\stX^\dagger$ is isomorphic
  to a horospherical variety.
\end{prop}

\begin{proof}
  By \cite[\S 6]{hofscheier} and \cite{bl11}, the varieties in
  Table~\ref{tab:classification_spherical} are not horospherical; hence
  $X_1, \dots, X_4$ are not horospherical.
    
  Now let $X$ be a complete horospherical $G$-variety. After possibly removing
  a set of codimension at least $2$, we obtain a surjective $G$-equivariant
  morphism $X \to G/P$, where $P \subseteq G$ is a parabolic subgroup and the
  fiber $Y$ is a toric variety. The fan of $Y$ is obtained from the colored
  fan of $X$ by ignoring the colors. For details, we refer to
  \cite[Section~2]{bm13}. The generators of the effective cone $\Eff G/P$ are
  a basis of the divisor class group $\Cl G/P$. Moreover, we have
  $\Rm(X) = \Rm(G/P)[X_1, \dots, X_r]$ where
  \begin{align*}
    r = \rank \Cl X - \rank \Cl G/P + \dim X - \dim G/P
    = \text{the number of rays in the fan of $Y$;}
  \end{align*}
  this follows from \cite[Theorem~4.3.2]{bri07}, see also
  \cite[Theorem~3.8]{gag14}.
    
  \begin{table}[ht]
    \centering
    \begin{tabular}[ht]{cccccc}
      \hline
      root system & parabolic subgroup & $\dim G/P$ & $\mathscr{Z}$ & $\rank \Cl G/P$ & remark\\
      \hline \hline
      $A_1$ & $\alpha_1$ & $1$ & $(2)$ & 1 & toric \\
      $A_2$ & $\alpha_1$ & $2$ & $(3)$ & 1 & toric \\
      $A_2$ & $\alpha_1, \alpha_2$ & $3$ & $(3, 3)$ & 2 & \\
      $A_3$ & $\alpha_1$ & $3$ & $(4)$ & 1 & toric \\
      $A_3$ & $\alpha_2$ & $4$ & $(6)$ & 1 & \\
      $A_3$ & $\alpha_1,\alpha_2$ & $5$ & $(4, 6)$ & 2 & \\
      $A_3$ & $\alpha_1,\alpha_3$ & $5$ & $(4, 4)$ & 2 & \\
      $A_3$ & $\alpha_1,\alpha_2,\alpha_3$ & $6$ & $(4, 4, 6)$ & 3 & \\
      $A_4$ & $\alpha_1$ & $4$ & $(5)$ & 1 & toric \\
      $A_4$ & $\alpha_2$ & $6$ & $(10)$ & 1 & \\
      $A_5$ & $\alpha_1$ & $5$ & $(6)$ & 1 & toric \\
      $A_6$ & $\alpha_1$ & $6$ & $(7)$ & 1 & toric \\

      $B_2$ & $\alpha_1$ & $3$ & $(5)$ & 1 & \\
      $B_2$ & $\alpha_2$ & $3$ & $(4)$ & 1 & toric \\
      $B_2$ & $\alpha_1,\alpha_2$ & $4$ & $(4,5)$ & 2 & \\

      $B_3$ & $\alpha_1$ & $5$ & $(7)$ & 1 & \\
      $B_3$ & $\alpha_3$ & $6$ & $(8)$ & 1 & \\
      
      $C_3$ & $\alpha_1$ & $5$ & $(6)$ & 1 & toric \\
      $C_3$ & $\alpha_3$ & $6$ & $(14)$ & 1 & \\
      
      $D_4$ & $\alpha_1$ & $6$ & $(8)$ & 1 & \\
      
      $G_2$ & $\alpha_1$ & $5$ & $(7)$ & 1 & \\
      $G_2$ & $\alpha_2$ & $5$ & $(14)$ & 1 & \\
      $G_2$ & $\alpha_1,\alpha_2$ & $6$ & $(7,14)$ & 2 & \\
      \hline
    \end{tabular}
    \caption{Flag varieties of simple groups and of dimension up to $6$}
    \label{tab:classification_flag5}
  \end{table}

  \begin{table}[ht]
    \centering
    \begin{tabular}[ht]{cccccccc}
      \hline
      root system & parabolic subgroup & $\dim G/P$ & $\mathscr{Z}$ & $\rank \Cl G/P$ & $r_{X_6}$ & $r_{X_7}$ & $r_{X_8}$\\
      \hline \hline
      $A_2$ & $\alpha_1, \alpha_2$ & $3$ & $(3, 3)$ & $2$ & $3$ & $6$ & $8$\\
      $B_2$ & $\alpha_1$ & $3$ & $(5)$ & $1$ & $4$ & $7$ & $9$\\
      
      \hline
      $A_3$ & $\alpha_2$ & $4$ & $(6)$ & $1$ & $3$ & $6$ & $8$\\
      $B_2$ & $\alpha_1,\alpha_2$ & $4$ & $(4,5)$ & $2$ & $2$ & $5$ & $7$\\
      $A_2 \times A_1$ & $\alpha_1, \alpha_2, \beta_1$ & $4$ & $(2, 3, 3)$ & $3$ & $1$ & $4$ & $6$\\
      $B_2 \times A_1$ & $\alpha_1, \beta_1$ & $4$ & $(2, 5)$ & $2$ & $2$ & $5$ & $7$\\
      
      \hline
      $A_3$ & $\alpha_1,\alpha_2$ & $5$ & $(4, 6)$ & $2$ & & $1$ & $3$ \\
      $A_3$ & $\alpha_1,\alpha_3$ & $5$ & $(4, 4)$ & $2$ & & $1$ & $3$ \\
      $B_3$ & $\alpha_1$ & $5$ & $(7)$ & $1$ & & $0$ & $2$ \\
      $G_2$ & $\alpha_1$ & $5$ & $(7)$ & $1$ & & $0$ & $2$ \\
      $G_2$ & $\alpha_2$ & $5$ & $(14)$ & $1$ & & $0$ & $2$ \\
      
      $A_3 \times A_1$ & $\alpha_2,\beta_1$ & $5$ & $(2,6)$ & $2$ & & $1$ & $3$ \\
      $B_2 \times A_1$ & $\alpha_1,\alpha_2,\beta_1$ & $5$ & $(2,4,5)$ & $3$ & & $-1$ & $0$ \\
      
      $A_2 \times A_2$ & $\alpha_1, \alpha_2, \beta_1$ & $5$ & $(3, 3, 3)$ & $3$ & & $-1$ & $0$ \\
      $B_2 \times A_2$ & $\alpha_1, \beta_2$ & $5$ & $(3, 5)$ & $2$ & & $1$ & $3$ \\
      
      $A_2 \times A_1 \times A_1$ & $\alpha_1, \alpha_2, \beta_1, \gamma_1$ & $5$ & $(2, 2, 3, 3)$ & $4$ & & $-2$ & $-1$ \\
      $B_2 \times A_1 \times A_1$ & $\alpha_1, \beta_1, \gamma_1$ & $5$ & $(2, 2, 5)$ & $3$ & & $-1$ & $0$  \\
      
      \hline 
      $A_3$ & $\alpha_1,\alpha_2,\alpha_3$ & $6$ & $(4, 4, 6)$ & $3$ & & & $4$ \\
      $A_4$ & $\alpha_2$ & $6$ & $(10)$ & $1$ & & & $6$ \\
      $B_3$ & $\alpha_3$ & $6$ & $(8)$ & $1$ & & & $6$ \\
      $C_3$ & $\alpha_3$ & $6$ & $(14)$ & $1$ & & & $6$ \\
      $D_4$ & $\alpha_1$ & $6$ & $(8)$ & $1$ & & & $6$ \\
      $G_2$ & $\alpha_1,\alpha_2$ & $6$ & $(7,14)$ & $2$ & & & $5$ \\
      
      $A_3 \times A_1$ & $\alpha_1,\alpha_2,\beta_1$ & $6$ & $(2, 4, 6)$ & $3$ & & & $4$ \\
      $A_3 \times A_1$ & $\alpha_1,\alpha_3,\beta_1$ & $6$ & $(2, 4, 4)$ & $3$ & & & $4$ \\
      $B_3 \times A_1$ & $\alpha_1,\beta_1$ & $6$ & $(2, 7)$ & $2$ & & & $5$ \\
      $G_2 \times A_1$ & $\alpha_1,\beta_1$ & $6$ & $(2, 7)$ & $2$ & & & $5$ \\
      $G_2 \times A_1$ & $\alpha_2,\beta_1$ & $6$ & $(2, 14)$ & $2$ & & & $5$ \\
      
      $A_3 \times A_2$ & $\alpha_2, \beta_1$ & $6$ & $(3, 6)$ & $2$ & & & $5$ \\
      $B_2 \times A_2$ & $\alpha_1,\alpha_2, \beta_1$ & $6$ & $(3, 4,5)$ & $3$ & & & $4$ \\

      $A_3 \times A_1 \times A_1$ & $\alpha_2, \beta_1, \gamma_1$ & $6$ & $(2, 2, 6)$ & $3$ & & & $4$ \\
      $B_2 \times A_1 \times A_1$ & $\alpha_1,\alpha_2,\beta_1,\gamma_1$ & $6$ & $(2, 2, 4,5)$ & $4$ & & & $3$ \\
      
      $A_2 \times A_2$ & $\alpha_1, \alpha_2, \beta_1, \beta_2$ & $6$ & $(3, 3, 3, 3)$ & $4$ & & & $3$ \\
      $A_2 \times B_2$ & $\alpha_1, \alpha_2, \beta_1$ & $6$ & $(3, 3, 5)$ & $3$ & & & $4$ \\
      $B_2 \times B_2$ & $\alpha_1, \beta_1$ & $6$ & $(5, 5)$ & $2$ & & & $5$ \\
      
      $A_2 \times A_3$ & $\alpha_1, \alpha_2, \beta_1$ & $6$ & $(3, 3, 4)$ & $3$ & & & $4$ \\
      $A_2 \times B_2$ & $\alpha_1, \alpha_2, \beta_2$ & $6$ & $(3, 3)$ & $3$ & & & $4$ \\
      $B_2 \times A_3$ & $\alpha_1, \beta_1$ & $6$ & $(4, 4, 5)$ & $2$ & & & $5$ \\
      $B_2 \times B_2$ & $\alpha_1, \beta_2$ & $6$ & $(4, 5)$ & $2$ & & & $5$ \\
      
      $A_2 \times A_2 \times A_1$ & $\alpha_1, \alpha_2, \beta_1, \gamma_1$ & $6$ & $(2, 3, 3, 3)$ & $4$ & & & $3$ \\
      $B_2 \times A_2 \times A_1$ & $\alpha_1, \beta_1, \gamma_1$ & $6$ & $(2, 3, 5)$ & $3$ & & & $4$ \\
      
      $A_2 \times A_1 \times A_1 \times A_1$ & $\alpha_1, \alpha_2, \beta_1, \gamma_1, \delta_1$ & $6$ & $(2, 2, 2, 3, 3)$ & $5$ & & & $2$ \\
      $B_2 \times A_1 \times A_1 \times A_1$ & $\alpha_1, \beta_1, \gamma_1, \delta_1$ & $6$ & $(2, 2, 2, 6)$ & $4$ & & & $3$ \\
      \hline
    \end{tabular}
    \caption{Nontoric flag varieties of dimension up to $6$}
    \label{tab:classification_flag6}
  \end{table}

  Table~\ref{tab:classification_flag6} contains the data of all nontoric flag
  varieties $G/P$ required here. It can be computed from
  Table~\ref{tab:classification_flag5} by forming products. The parabolic
  subgroup $P$ is described by the complement of the subset of the simple
  roots used in \cite[Theorem~8.4.3]{spr98}. It follows that the set of colors
  of $G/P$ is in bijection with the subset of simple roots given in the
  tables; see \cite[after D\'efinition 2.6]{pas08}. By
  \cite[Proposition~4.1.1]{bri07}, the rank of $\Cl G/P$ is the number of
  colors. The dimension of $G/P$ can be deduced, for instance, by
  \cite[p.~9]{tim11}. For simple $G$, it follows from
  \cite[Proposition~6.1]{gh17} that $G/P$ is toric if and only if the Dynkin
  diagram of $G$ marked with the subset of simple roots given in the tables
  appears in \cite[Lemme~2.13]{pasth}. The meaning of $\mathscr{Z}$ will be
  explained below.

  First, assume that $X^\dagger$ or $\stX^\dagger$ is isomorphic to $X$.  Then
  we have $\dim X = 3$. Recall that the effective cones of $X^\dagger$ and
  $\stX^\dagger$ are not simplicial. Since the effective cone of any flag
  variety is simplicial, we deduce $\dim G/P \le 2$. It follows that $G/P$ is
  isomorphic to a toric variety, and hence the same is true for $X$. But
  according to Section~\ref{sec:Xdagger}, the Cox rings of $X^\dagger$ and
  $\stX^\dagger$ are not polynomial rings, a contradiction.

  Next assume that $X_5$ is isomorphic to $X$. Then we have $\dim X = 4$. As
  before, we obtain $\dim G/P \le 3$ from the fact that the effective cone of
  $X_5$ is not simplicial and $\dim G/P \ge 3$ from the fact that the variety
  $X_5$ is not isomorphic to a toric one. Hence we have $\dim G/P = 3$, and
  therefore, $\rank \Cl G/P \le 3$. Moreover, we have $\dim Y = 1$ and
  therefore $r \le 2$. We obtain $\rank \Cl X \le 4$, a~contradiction to
  $\rank \Cl X_5 = 5$.

  Next assume that $X_6$ is isomorphic to $X$. Then we have $\dim X = 5$. Let
  $\mathscr{Z}(X_6)$ be the ordered tuple of the dimensions of the homogeneous
  parts of the Cox ring $\Rm(X_6)$ for the generators of the effective cone of
  $X_6$. According to Section~\ref{sec:cox_X6}, we have
  \begin{align*}
    \mathscr{Z}(X_6) = (1, 1, 2, 3)\text{.}
  \end{align*}
  As in the previous cases, we obtain $3 \le \dim G/P \le 4$. The possible
  values for $\mathscr{Z}(G/P)$ and $r = r_{X_6}$ are given in
  Table~\ref{tab:classification_flag6} (the toric cases are excluded). The
  values for $\mathscr{Z}(G/P)$ are computed using the Weyl dimension formula;
  see, for instance, \cite[Corollary~24.3]{hum80}. We have a natural
  surjective map $\phi\colon \Cl G/P \times \Zd^r \to \Cl X $ compatible with
  the $\Cl X $-grading and the finer $\Cl G/P \times \Zd^r$-grading of
  $\Rm(X)$. It maps the cone $\Eff G/P \times \Zd_{\ge 0}^r$ generated by
  $\Eff G/P$ and the degrees of $X_1, \dots, X_r$ onto $\Eff X$. Moreover, we
  have $(\Eff G/P \times \Zd_{\ge 0}^r) \cap \ker \phi = \{0\}$. It follows that
  every element of $\mathscr{Z}(X_6)$ is a sum where the summands are taken
  from the elements of $\mathscr{Z}(R/P)$ and from $r_{X_6}$ times the
  summand~$1$ and each summand may be used at most once in total. This is
  impossible for all cases in Table~\ref{tab:classification_flag6}. The same
  argument works for $X_8$, which satisfies
  \begin{align*}
    \mathscr{Z}(X_8) = (1, 1, 1, 1, 1, 1, 2, 2)
  \end{align*}
  according to Section~\ref{sec:cox_X8}.

  Finally assume that $X_7$ is isomorphic to $X$. According to Section~\ref{sec:cox_X7}, we have
  \begin{align*}
    \mathscr{Z}(X_7) = (1, 1, 1, 1, 1, 1)\text{.}
  \end{align*}
  It follows that there exists an isomorphism
  \begin{align*}
    \Rm(X_7) &\to \Rm(G/P)[X_1, \dots, X_r]\text{,} \\
    (x_{03}, x_{31}, x_{32}, x_{33}, x_{34}, x_{35})  &\mapsto (X_1, X_2, X_3, X_4, X_5, X_6)\text{.}
  \end{align*}
  After dividing out the ideal $(x_{03}, x_{31}, x_{32}, x_{33}, x_{34}, x_{35})$, we obtain
  an isomorphism
  \begin{align*}
    \Qd[x_{01}, x_{02}, x_{11}, x_{12}, x_{21}, x_{22}]/(x_{11}x_{12}-x_{21}x_{22})
    \to \Rm(G/P)[X_{7}, \dots, X_{r}]\text{.}
  \end{align*}
  This is a contradiction since the second ring is factorial by
  \cite[Proposition~1.4.1.5(i)]{adhl15}, while the first ring is not.
\end{proof}

\end{document}